\documentclass[reqno]{amsart}
\usepackage{mathrsfs}

\usepackage{amscd,amsthm,amsfonts,amssymb,amsmath}
\usepackage{amsmath,tikz-cd}
\usepackage[colorlinks=true]{hyperref}
\usepackage{fullpage,dsfont}
\usepackage[all]{xy}
\usepackage{mathtools}
\hypersetup{colorlinks, linkcolor=black, citecolor=black, urlcolor=black}

    \newcommand{\BA}{{\mathbb {A}}} 
    \newcommand{\BC}{{\mathbb {C}}} 
     \newcommand{\BF}{{\mathbb {F}}}
     \newcommand{\BH}{{\mathbb {H}}}
    \newcommand{\BI}{{\mathbb {I}}} 
     \newcommand{\BL}{{\mathbb {L}}}

    \newcommand{\BQ}{{\mathbb {Q}}} \newcommand{\BR}{{\mathbb {R}}}
     \newcommand{\BT}{{\mathbb {T}}}
     \newcommand{\BV}{{\mathbb {V}}}
     
     \newcommand{\BZ}{{\mathbb {Z}}}

    \newcommand{\CC}{{\mathcal {C}}} 
    \newcommand{\CE}{{\mathcal {E}}} \newcommand{\CF}{{\mathcal {F}}}
    \newcommand{\CG}{{\mathcal {G}}} \newcommand{\CH}{{\mathcal {H}}}
     
     \newcommand{\CL}{{\mathcal {L}}}
    \newcommand{\CM}{{\mathcal {M}}} 
    \newcommand{\CO}{{\mathcal {O}}} 
     \newcommand{\CR}{{\mathcal {R}}}
     \newcommand{\CT}{{\mathcal {T}}}
    \newcommand{\CU}{{\mathcal {U}}} 
    \newcommand{\CW}{{\mathcal {W}}} \newcommand{\CX}{{\mathcal {X}}}
     \newcommand{\CZ}{{\mathcal {Z}}}

     \newcommand{\fp}{{\mathfrak{p}}}

    \newcommand{\ab}{{\mathrm{ab}}}
    
    \newcommand{\Ann}{{\mathrm{Ann}}}
    \newcommand{\Aut}{{\mathrm{Aut}}}

    \newcommand{\cond}{\mathrm{cond^r}}
    \newcommand{\cont}{{\mathrm{cont}}}\newcommand{\cris}{{\mathrm{cris}}}
    \newcommand{\corank}{{\mathrm{corank}}}
    
    \newcommand{\can}{{\mathrm{can}}}
    
    \newcommand{\cyc}{{\mathrm{cyc}}}
    
    \newcommand{\disc}{{\mathrm{disc}}}

     \renewcommand{\div}{{\mathrm{div}}}
    \newcommand{\End}{{\mathrm{End}}} 
    
    \newcommand{\Frac}{{\mathrm{Frac}}}
    \newcommand{\Frob}{{\mathrm{Frob}}}

    \newcommand{\Gal}{{\mathrm{Gal}}} \newcommand{\GL}{{\mathrm{GL}}}
    \newcommand{\GSp}{{\mathrm{GSp}}}
    \newcommand{\Hom}{{\mathrm{Hom}}}
    \newcommand{\height}{{\mathrm{ht}}}
    \renewcommand{\Im}{{\mathrm{Im}}}
    \newcommand{\Ind}{{\mathrm{Ind}}}

    \newcommand{\Kato}{\mathrm{Kato}}
    
    \newcommand{\loc}{{\mathrm{loc}}}

    \newcommand{\ord}{{\mathrm{ord}}} \newcommand{\rank}{{\mathrm{rank}}}

    \renewcommand{\mod}{\ \mathrm{mod}\ }\renewcommand{\Re}{{\mathrm{Re}}}
    \newcommand{\rec}{{\mathrm{rec}}}

    \newcommand{\Sel}{{\mathrm{Sel}}} 
    
    \newcommand{\st}{{\mathrm{st}}}
    
    \newcommand{\sign}{{\mathrm{sign}}}
    
    \newcommand{\Spec}{{\mathrm{Spec}}}

    \newcommand{\sgn}{{\mathrm{sgn}}}

    \newcommand{\tor}{{\mathrm{tor}}}
    
    \newcommand{\ur}{{\mathrm{ur}}}

    \newcommand{\ac}{\mathrm{ac}}
    \newcommand{\isoarrow}{\xrightarrow{\sim}}
    \newcommand{\bz}{\mathbf{z}}
    
    \newcommand{\eps}{\varepsilon}
    \newcommand{\bg}{{\mathbf{g}}}
    \newcommand{\bh}{{\mathbf{h}}}
    \newcommand{\sL}{{\mathscr{L}}}
    \newcommand{\sC}{{\mathscr{C}}}
    \newcommand{\sR}{{\mathscr{R}}}
    \newcommand{\sS}{{\mathscr{S}}}
    \newcommand{\CBF}{{\mathcal{BF}}}
    
    \newcommand{\rel}{{\mathrm{rel}}}

    \font\cyr=wncyr10

    \newcommand{\Sha}{\hbox{\cyr X}}

    \newcommand{\ov}{\overline}

    \newcommand{\ra}{\rightarrow} 
    \newcommand{\bs}{\backslash}
    \newcommand{\nequiv}{\equiv\hspace{-10pt}/\ }

    \theoremstyle{plain}
    \newtheorem{thm}{Theorem}[section] \newtheorem{cor}[thm]{Corollary}
    \newtheorem{lem}[thm]{Lemma}  \newtheorem{prop}[thm]{Proposition}
    \newtheorem {conj}[thm]{Conjecture} \newtheorem{defn}[thm]{Definition}

\theoremstyle{remark} \newtheorem{remark}[thm]{Remark}
\theoremstyle{remark} 
\theoremstyle{remark} \newtheorem{example}{Example}

    \newcommand{\cO}{\mathcal O}
    \numberwithin{equation}{section}

\begin{document}

\title{Zeta elements for elliptic curves and applications}

\author{Ashay A. Burungale, Christopher Skinner, Ye Tian and Xin Wan}

\address{Ashay A. Burungale:   Department of Mathematics~\\UT Austin~\\
Austin, TX 78712} 
\email{ashayburungale@gmail.com}

\address{Christopher Skinner: Department of Mathematics, Princeton University, 
Princeton NJ 08544-1000}
\email{cmclas@princeton.edu}

\address{Ye Tian: MCM, HCMS, Academy of Mathematics and System Science, Chinese Academy of Sciences, Beijing 100190, and 
School of Mathematical Sciences, University of Chinese Academy of Sciences, Beijing 100049, China} \email{ytian@math.ac.cn}

\address{Xin Wan: Academy of Mathematics and Systems
Science, Morningside center of Mathematics, Chinese Academy of
Sciences, Beijing 100190 and 
School of Mathematical Sciences, University of Chinese Academy of Sciences, Beijing 10049} \email{xwan@math.ac.cn}

\maketitle
\begin{abstract} Let $E$ be an elliptic curve defined over $\BQ$ with conductor $N$ and 
$p\nmid 2N$ a prime. Let $L$ be an imaginary quadratic field with $p$ split. 
We prove the existence of  $p$-adic zeta element for $E$ over $L$, 
encoding two different $p$-adic $L$-functions associated to $E$ over $L$ via explicit reciprocity laws at  the primes above $p$. 
We formulate a main conjecture for $E$ over $L$ in terms of the zeta element, mediating different main conjectures in which the $p$-adic $L$-functions appear, and prove some results toward them. 

The zeta element has various applications to the arithmetic of elliptic curves. This includes a proof of 
main conjecture for semistable elliptic curves $E$ over $\BQ$ at supersingular primes $p$, as conjectured by Kobayashi in 2002. It leads to the $p$-part of the conjectural Birch and Swinnerton-Dyer (BSD) formula for such curves of analytic rank zero or one, and enables us to present the first infinite families of non-CM elliptic curves for which the BSD conjecture is true. 
We provide further evidence towards the BSD conjecture: new cases of $p$-converse to the Gross--Zagier and Kolyvagin theorem, and $p$-part of the BSD formula for ordinary primes $p$. 
Along the way, we give a proof of a conjecture of Perrin-Riou connecting Beilinson--Kato elements with rational points.

\end{abstract}

\tableofcontents

\section{Introduction}
\noindent 

A principle of Kato posits the existence of $p$-adic zeta element for a motive over a number field: an arithmetic incarnation of the associated critical $L$-values in $p$-adic \'etale cohomology. Cyclotomic units, elliptic units and Beilinson--Kato elements are primary known examples of zeta elements, the latter for an elliptic curve $E$ defined over $\BQ$. 
The aim of this paper is prove the existence of a $p$-adic zeta element for such an $E$ together with an imaginary quadratic field $L$ (cf.~Theorem~\ref{thmZ}).

The zeta element is ancillary to the Iwasawa theory of $E$ over $L$, and in turn to the arithmetic of $E$ over $\BQ$. It leads to a proof of the $\GL_2$-Iwasawa main conjecture at supersingular primes, complementing the work of Kato \cite{K} and Skinner and Urban \cite{SU} at ordinary primes: 
we establish Kobayashi's main Conjecture for semistable $E$ at supersingular primes $p$ (cf.~Theorem~\ref{thmA}).
In combination with the $p$-adic Gross--Zagier formula, in the supersingular case the main conjecture yields the $p$-part of the conjectural Birch and Swinnerton-Dyer formula if the analytic rank of $E$ is zero or one (cf.~Theorem~\ref{corA_thmA}), as well as a $p$-converse to the theorem of Gross--Zagier and Kolyvagin for $E$ (cf.~Theorem~\ref{corB_thmA}). 
In turn this leads to the first infinite families of non-CM elliptic curves for which the full BSD formula is proved to hold (cf.~Theorem~\ref{nCM}). 

Exploiting the zeta element, we also provide additional evidence towards the Birch and Swinnerton-Dyer conjecture, 
such as new cases of the $p$-part of the Birch and Swinnerton-Dyer formula (cf.~Theorem~\ref{corA'}) and $p$-converse for ordinary primes $p$ (cf.~Theorem~\ref{corB'}). 
Another application concerns a conjecture of Perrin-Riou (cf.~Theorem~\ref{Thm_PR}). 

While the main text considers weight two elliptic newforms, in this introduction we focus on the case of elliptic curves over $\BQ$.

\subsection{Arithmetic of elliptic curves} We first describe some of the applications of the zeta element to the arithmetic of elliptic curves. 

\subsubsection*{The Birch and Swinnerton-Dyer conjecture}
Let $E$ be an elliptic curve defined over the rationals. 
A fundamental arithmetic invariant of $E$ is its Mordell--Weil rank, that is the rank of the finitely generated abelian group 
$
E(\BQ).
$
As $E$ varies, this rank is typically expected to be $0$ or $1$ but is difficult to get a handle on generally. Another mysterious structure in the arithmetic of $E$ is the conjecturally finite Tate--Shafarevich group $\Sha(E_{/\BQ})$. For a prime $p$, the $p^\infty$-Selmer group $\Sel_{p^{\infty}}(E_{/\BQ})$ links these two 
via the fundamental exact sequence 
\label{Sel}
\begin{equation}\label{ex}
0 \ra E(\BQ) \otimes \BQ_{p}/\BZ_{p} \ra \Sel_{p^{\infty}}(E_{/\BQ}) \ra \Sha(E_{/\BQ})[p^{\infty}] \ra 0.
\end{equation}

On the analytic side, the primary object of interest is the complex $L$-function $L(s,E_{/\BQ})$,
with $s \in \BC$, and the fundamental analytic invariant is the analytic rank, defined as the vanishing order $\ord_{s=1}L(s,E_{/\BQ})$.

The Birch and Swinnerton-Dyer conjecture (BSD) conjecture predicts a mysterious link between the arithmetic and analytic invariants:

\begin{conj}[The Birch and Swinnerton-Dyer conjecture for ranks $0$ and $1$]\label{BSD}
Let $E$ be an elliptic curve defined over the rationals. For $r\in\{0,1\}$, the following are equivalent:
\begin{itemize}
\item[(1)] $\rank_{\BZ}E(\BQ)=r$ and $\Sha(E_{/\BQ})$ is finite.
\item[(2)] $\corank_{\BZ_{p}} \Sel_{p^{\infty}}(E_{/\BQ})=r$ for a prime $p$.
\item[(3)] $\ord_{s=1}L(s,E_{/\BQ})=r$. 
\end{itemize}
Moreover, for a prime $p$ the $p$-part of the BSD formula holds under any of the above, 
that is,
$$
\bigg{|}\frac{L^{(r)}(1,E_{/\BQ})}{\Omega_{E}R(E_{/\BQ})}\bigg{|}_{p}^{-1}
=
\bigg{|}\frac{\#\Sha(E_{/\BQ}) \cdot \prod_{q|N}c_{q}(E_{/\BQ})}{\#E(\BQ)^{2}}
\bigg{|}_{p}^{-1}
$$
for $\Omega_{E} \in \BC^{\times}$ the N\'eron period, $R(E_{/\BQ})$ the regulator of $E(\BQ)$, $N$ the conductor of $E$, 
$c_{q}(E_{/\BQ})$ the Tamagawa number at a prime $q$, and $|\cdot|_p$ the $p$-adic absolute value normalised so that $|p|_{p}=\frac{1}{p}$. 
\end{conj}

Note that (1) implies (2) by the exact sequence \eqref{ex}. That (3) implies (1) is a fundamental result towards the BSD conjecture due to 
Gross--Zagier, Kolyvagin and Rubin in the late 80's. Beginning with the work of Skinner and Zhang a decade back, the implication (2) implies (3) is referred to as a $p$-converse to the Gross--Zagier and Kolyvagin theorem: a $p$-adic criterion for $E$ to have analytic rank $r$. In this paper we prove new cases of the $p$-converse and the $p$-part of the BSD formula for both ordinary and supersingular primes, with an emphasis on the latter. 

\subsubsection{Kobayashi's supersingular main conjecture}\label{ss:Kob}
Mazur \cite{Ma} initiated the
Iwasawa theory of elliptic curves in the early 70's, formulating a main conjecture for primes $p$ of ordinary reduction. The case of primes of supersingular reduction exhibits new phenomena and a conjectural framework reflecting them remained elusive until the early 2000's, when, sparked by the work of Pollack \cite{Po}, Kobayashi \cite{Ko} formulated a signed Iwasawa main conjecture for elliptic curves at supersingular primes. 

Let $E_{/\BQ}$ be an elliptic curve of conductor $N$, and $p\nmid 2N$ a supersingular prime satisfying 
\eqref{h4}
(this is only an extra condition when $p=3$). 
Let $\BQ_\infty$ be the cyclotomic $\BZ_p$-extension of $\BQ$, $\Gamma=\Gal(\BQ_{\infty}/\BQ)$ and $\gamma_\cyc$ a topological generator of $\Gamma$. Put $\Lambda=\BZ_{p}[\![\Gamma]\!]$, viewed as a $G_\BQ$-module via the canonical projection $\Psi: G_{\BQ} \twoheadrightarrow \Gamma$. Let $\epsilon$ denote the  $p$-adic cyclotomic character $\epsilon:G_{\BQ} \ra \BZ_{p}^\times$.

For $\zeta$ a primitive $p^{t}$-th root of unity, let
$$
\psi_{\zeta}: G_{\BQ} \twoheadrightarrow \Gamma \ra \ov{\BQ}^{\times}
$$
be the finite order character induced by $\gamma_\cyc \mapsto \zeta$. 
For $t>0$, let $\psi_\zeta$ also denote the Dirichlet character of $(\BZ/p^{t+1}\BZ)^{\times}$ of $p$-power order such that the image of $\epsilon(\gamma_\cyc)\in 1+p\BZ_{p}$ maps to $\zeta$.
Let 
$$
\phi_{\zeta}: \Lambda \ra \BZ_{p}[\zeta] \subset \ov{\BQ}_{p}
$$
be the homomorphism such that $\gamma_\cyc \mapsto \zeta$.

Let $T$ denote the $p$-adic Tate module of $E$. Let 
$$
M=T(1) \otimes_{\BZ_{p}} \Lambda^{\vee}
$$
be a discrete $\Lambda$-module with the $G_{\BQ}$-action on $\Lambda^{\vee}$ via $\Psi^{-1}$.  For $\circ\in\{+,-\}$, let $$H^{1}_{\circ}(\BQ_{p},M) \subset H^{1}(\BQ_{p},M)$$ be the annihilator of Kobayashi's signed submodule 
$H^{1}_{\circ}(\BQ_{p}, T(1) \otimes_{\BZ_{p}} \Lambda)\subset H^{1}(\BQ_{p},T(1)\otimes_{\BZ_{p}}\Lambda)$ under the Pontryagin duality pairing (cf.~\S\ref{IwCoh}). 
Following Kobayashi, 
for $\Sigma$ a finite set of places of $\BQ$ containing those dividing $N\infty$, define 
\begin{equation}\label{ssQ}
S_{\circ}(E) = \ker \big{\{} H^{1}(G_{\Sigma},M) \ra \prod_{v \in \Sigma, v \nmid p} H^{1}(\BQ_{v},M) 
\times \frac{H^{1}(\BQ_{p},M)}{H^{1}_{\circ}(\BQ_{p},M)}  \big{\}}
\end{equation}
and let $X_{\circ}(E)$ denote its Pontryagin dual. Based on the work of Kato \cite{K}, Kobayashi \cite{Ko} proved that $X_{\circ}(E)$ is a torsion $\Lambda$-module.

Also for $p\nmid 2N$ a supersingular prime satisfying \eqref{h4},
Pollack \cite{Po} proved the existence of $p$-adic $L$-functions
$\mathcal{L}_{p}^{\circ}(E) \in \Lambda \simeq \BZ_{p}[\![X]\!]$, $\circ\in\{\pm\}$,
such that for $\zeta$ as above with $t=0$ or $t>0$ and even if $\circ = +$ and $t=0$ or $t>0$ and odd if $\circ=-$,
$$
\phi_{\zeta}(\mathcal{L}_{p}^{\circ}(E))=
e_{p}^\circ(\zeta)\cdot \frac{L(1, E \otimes \psi_{\zeta}^{-1})}{\Omega_{E}}
$$
with 
$$
e_{p}^+(\zeta)=
\begin{cases*}
(-1)^{\frac{t+2}{2}} \cdot \frac{p^{t+1}}{\mathfrak{g}(\psi_{\zeta}^{-1})} 
\cdot \prod_{\text{odd } m=1}^{t-1} \Phi_{p^{m}}(\zeta)^{-1}
& if $t>0$ even\\
2 & if $t=0$.
\end{cases*}
$$
and
$$ 
e_{p}^-(\zeta)=
\begin{cases*}
(-1)^{\frac{t+1}{2}} \cdot \frac{p^{t+1}}{\mathfrak{g}(\psi_{\zeta}^{-1})} 
\cdot \prod_{\text{even } m=2}^{t-1} \Phi_{p^{m}}(\zeta)^{-1}
& if $t>0$ odd\\
p-1 & if $t=0$.
\end{cases*}
$$
Here $\mathfrak{g}(\psi_{\zeta}^{-1})$ denotes the Gauss sum and $\Phi_{p^{m}}(X)$ the $p^m$-th cyclotomic polynomial.

 Kobayashi \cite{Ko} proposed the following signed Main Conjecture:

\begin{conj}\label{Kob} 
Let $E_{/\BQ}$ be an elliptic curve of conductor $N$, and $p\nmid 2N$ a supersingular prime. If $p=3$, suppose that 
\eqref{h4} holds. For $\circ\in\{+,-\}$, we have 
$$
(\CL_{p}^{\circ}(E))=\xi_{\Lambda}(X_{\circ}(E))
$$
in $\Lambda$, where $\xi_{\Lambda}(\cdot)$ denotes the $\Lambda$-characteristic ideal. 
\end{conj}

Our main result towards Kobayashi's conjecture is the following:
\begin{thm}\label{thmA}
Let $E_{/\BQ}$ be a semistable elliptic curve, and $p>2$ a supersingular prime. If $p=3$, suppose that  
\eqref{h4} holds. Then Kobayashi's Conjecture \ref{Kob} is true, i.e.
for $\circ\in\{+,-\}$, we have 
$$
(\CL_{p}^{\circ}(E))=\xi_{\Lambda}(X_{\circ}(E)).
$$
Moreover, the same holds for any quadratic twist $E^{K}:=E\otimes \chi_K$ for $\chi_K$ the character associated to a quadratic field extension $K/\BQ$ with discriminant coprime to $Np$ 
and divisible only by primes of ordinary reduction for $E$.
\end{thm}

This provides the first cases of Kobayashi's main conjecture for non-CM elliptic curves. 
\begin{remark}\noindent
\begin{itemize}
\item[(i)]The above theorem was first announced by the fourth-named author in 2014 in \cite{W}. 
The pertinent parts of this paper supersede the prior announcement, 
and the proof realises the strategy outlined therein. The preprint \cite{W} is no longer intended for publication. 
\item[(ii)] The CM case of Conjecture \ref{Kob} was established by Pollack and Rubin \cite{PoRu} in 2004. 
\end{itemize}
\end{remark}

Theorem \ref{thmA} has the following applications to the BSD conjecture. 

\subsubsection*{$p$-part of the BSD formula} 
\begin{thm}\label{corA_thmA}
Let $E_{/\BQ}$ be a semistable elliptic curve, and $p>2$ a supersingular prime. If $p=3$, suppose that  
\eqref{h4} holds.
If $\ord_{s=1}L(s,E_{/\BQ})=r\leq1$, then the $p$-part of the BSD formula is true, i.e.
$$
\bigg{|}\frac{L^{(r)}(1,E_{/\BQ})}{\Omega_{E}R(E_{/\BQ})}\bigg{|}_{p}^{-1}
=
\bigg{|}\#\Sha(E_{/\BQ}) \cdot \prod_{q|N}c_{q}(E_{/\BQ})
\bigg{|}_{p}^{-1}
$$
for $\Omega_{E} \in \BC^{\times}$ the N\'eron period and $R(E_{/\BQ})$ the regulator. Moreover, the same holds for any quadratic twist $E^K$ as in Theorem~\ref{thmA}.
\end{thm}
\noindent The proof of the $r=1$ case is based on the $p$-adic Gross--Zagier formula \cite{Ko1}. 

\subsubsection*{$p$-converse to the Gross--Zagier and Kolyvagin theorem}
\begin{thm}\label{corB_thmA}
Let $E_{/\BQ}$ be a semistable elliptic curve, and $p>2$ a supersingular prime. If $p=3$, suppose that  
\eqref{h4} holds. Then
$$
\corank_{\BZ_{p}}\Sel_{p^\infty}(E_{/\BQ})=0 \implies L(1,E_{/\BQ}) \neq 0.
$$
Moreover, the same holds for any quadratic twist $E^K$ as in Theorem~\ref{thmA}.
\end{thm}

\subsubsection*{The BSD conjecture for infinite families of non-CM curves}

\begin{thm}\label{nCM}
Let $E$ be an elliptic curve of conductor $N$ denoted by
 $46a1$, $62a1$, $66b1$, 
$69a1$, $77c1$, $94a1$, $105a1$, $106d1$,
$114b1$, $115a1$, $118c1$, $118d1$, $141b1$, $141c1$, $141e1$
or $142c1$
 in Cremona's labelling. 
Let $M>1$ be a square-free integer with $(M,N)=1$ and $E^{M}$ the quadratic twist of $E$ by the character associated to the extension $\mathbb{Q}(\sqrt{M})/\BQ$. Suppose that the following conditions hold:
\begin{itemize}
\item[(a)]
$L(1,E^{M})\not=0$, and  
\item[(b)] $E$ has ordinary reduction at the primes dividing $M$. 
\end{itemize}
Then the BSD 
conjecture holds for 
$E^{M}$, i.e. $E^{M}(\BQ)$ and $\Sha(E^{M})$ are finite, and 
 $$
 \frac{L(1,E^{M})}{\Omega_{E^{M}}}
=
\frac{\# \Sha(E^{M}) \cdot \prod_{\ell \nmid \infty} c_{\ell}(E^{M})}{\#E^{M}(\BQ)_{\tor}^2}.
$$ 
Moreover, the conditions (a) and (b) are satisfied by infinitely many $M$. 
\end{thm}

\noindent Our proof of this theorem is based on Theorem~\ref{corA_thmA} and 
prior work on the $p$-part of the BSD formula. The existence of infinitely many $M$ satisfying the conditions (a) and (b) relies on \cite{CLZ,Zi}.

\begin{remark}\noindent
\begin{itemize}
\item[(i)] The full BSD conjecture for elliptic curves without complex multiplication had previously only been established for finitely many such curves, combining theoretical results and numerical computations. 
\item[(ii)] Theorem~\ref{nCM} yields new cases of the conjecture of Flach and Morin \cite{FM} for zeta functions of arithmetic surfaces: Let $X_{/\BQ}$ 
be a principal homogeneous space of $E^{M}_{/\BQ}$ as in Theorem~\ref{nCM} and $\CX \ra \Spec(\BZ)$ a proper regular model of $X$. Then ${\rm Br}(\CX)$ is finite and
the special value conjecture \cite[Conj.~5.12]{FM} for $\zeta(\CX, s)$ at $s = 1$ holds true (cf.~\cite{FS}).
\end{itemize}
\end{remark}

\subsubsection*{Kolyvagin's conjecture}
Theorem~\ref{thmA} also has application to Kolyvagin's conjecture on the non-triviality of the Euler system of Beilinson--Kato elements and the Heegner point Kolyvagin system,
but we do not elaborate on this here and instead refer the reader to \cite{BCGS,Ki2, Ki3, Sw}.

\subsubsection{Special cases of the Birch and Swinnerton-Dyer conjecture: ordinary primes}
In the case that $p$ is a prime of ordinary reduction for $E$ we also improve on some of the existing results towards
the $p$-part of the BSD conjecture and the $p$-converse for $E$.

\subsubsection*{$p$-part of the BSD formula, bis} We prove:

\begin{thm}\label{corA'}
Let $E_{/\BQ}$ be an elliptic curve of conductor $N$, and $p \nmid 2N$ an ordinary prime. 
Suppose that the following holds.
\begin{itemize}
\item[(irr$_{\BQ}$)] The mod $p$ Galois representation $\ov{\rho}:G_{\BQ}\ra \Aut_{\BF_{p}}E[p]$ is absolutely irreducible. 
\item[(ram)] There exists a prime $\ell || N$ such that $\ov{\rho}$ is ramified at $\ell$.
\end{itemize}
If $\ord_{s=1}L(s,E_{/\BQ})=1$, then the $p$-part of the BSD formula holds, i.e.
$$
\bigg{|}\frac{L'(1,E_{/\BQ})}{\Omega_{E}R(E_{/\BQ})}\bigg{|}_{p}^{-1}
=
\bigg{|}\#\Sha(E_{/\BQ}) \cdot \prod_{q|N}c_{q}(E_{/\BQ})
\bigg{|}_{p}^{-1}
$$
for $\Omega_{E} \in \BC^{\times}$ the period and $R(E_{/\BQ})$ the regulator. 
\end{thm}
The hypothesis (ram) is not satisfied by CM curves, however our approach also applies to the CM case; see~Theorem~\ref{p-BSD}.

\subsubsection*{$p$-converse to the Gross--Zagier and Kolyvagin theorem, bis} We prove:

\begin{thm}\label{corB'}
Let $E_{/\BQ}$ be an elliptic curve of conductor $N$, and $p\nmid 2N$ an ordinary prime. 
Suppose that the following holds.
\begin{itemize}
\item[(sur$_{\BQ}$)] The mod $p$ Galois representation $\ov{\rho}:G_{\BQ}\ra \Aut_{\BF_{p}}E[p]$ is surjective. 
\item[(ram)] There exists a prime $\ell || N$ such that $\ov{\rho}$ is ramified at $\ell$.
\end{itemize}
Then 
 $$
\corank_{\BZ_{p}}\Sel_{p^{\infty}}(E_{/\BQ})=1 
 \implies \ord_{s=1} L(s,E_{/\BQ})=1.
 $$
\end{thm}

\begin{remark}
In both Theorems \ref{corA'} and \ref{corB'} the hypothesis (ram) can be replaced with the existence of an auxiliary real quadratic field as in \cite[Thm. 4]{W'}. 
\end{remark}

\subsubsection{Perrin-Riou's conjecture} 
In the early 90's Rubin \cite{Ru1} established a link between elliptic units and rational points on a CM elliptic curve. Shortly later, Perrin-Riou \cite{PR} proposed a conjectural generalisation of this phenomemenon to arbitrary elliptic curves.

 For an elliptic curve $E_{/\BQ}$ and a prime $p$, put $V=T_{p}(E)\otimes \BQ_{p}$.

Let $H^{1}_{\rm f}(\BQ_{p},V) \subset H^{1}(\BQ_{p},V)$ be the Bloch--Kato subgroup.
Let $$z_{E} \in H^{1}(\BQ,V)$$ be the $p$-adic Beilinson--Kato element arising from the image of the Beilinson--Kato element ${\bf{z}}_{\gamma}(f_{E})$ in $H^{1}(\BQ,V)$ (cf.~\S\ref{BK-rationals}) with $f_E$ the associated elliptic newform and $\gamma$ as in \S\ref{Periods}. 
The Belinson--Kato elements are constructed from special elements in the $K_2$ groups of modular curves, whose definition relies on Siegel units.
By Kato's explicit reciprocity law, 
$$
\loc_{p}(z_{E}) \in H^{1}_{\rm f}(\BQ_{p},V) \iff L(1,E_{/\BQ})=0
$$ 
for $\loc_{p}:H^{1}(\BQ,V) \ra H^{1}(\BQ_{p},V)$ the localisation (cf.~\cite[Thm.~12.5]{K}).  
Note that $H^{1}_{\rm f}(\BQ_{p},V)=E(\BQ_{p})\otimes \BQ_{p}$. 

In the case $\ord_{s=1}L(s,E_{/\BQ})>0$ Perrin-Riou conjectured
the $p$-adic Beilinson--Kato element is linked with the arithmetic of $E_{/\BQ}$ as follows.

\begin{conj}\label{PR0} 
Let $E_{/\BQ}$ be an elliptic curve and 
 $p$ a prime. 
Let $z_{E} \in H^{1}(\BQ,V)$ be the associated $p$-adic Beilinson--Kato element as above. 
Suppose that $L(1,E_{/\BQ})=0.$
Then there exists $P \in E(\BQ)$ with the following properties.

\begin{itemize}
\item[(a)] For  $\omega$ a N\'eron differential, 
$\log_{\omega}:E(\BQ_{p}) \ra \BQ_{p}$ the $p$-adic logarithm, we have
$$
\log_{\omega}(\loc_{p}(z_{E})) \doteq \log_{\omega}(P)^{2},
$$
where `$\doteq$' is equality up to  $\BQ^\times$-multiple.
 
\item[(b)] 
Moreover,  $$0 \neq P \in E(\BQ)\otimes_{\BZ}\BQ \iff  \ord_{s=1}L(s,E_{/\BQ})=1.$$ 
\end{itemize}
In particular, 
$
\loc_{p}(z_{E}) \neq 0 \iff \ord_{s=1}L(s,E_{/\BQ})=1.
$
\end{conj}

Our main result towards Conjecture \ref{PR0} is the following

\begin{thm}\label{Thm_PR}
Let $E_{/\BQ}$ be an elliptic curve of conductor $N$, and 
$p \nmid 2N$ a prime. If $p=3$ and $E$ has supersingular reduction at 3, suppose that 
\eqref{h4} holds.
Then Conjecture \ref{PR0} is true for the pair $(E,p)$.
\end{thm} 

In the text we prove a $p$-integral version of Conjecture \ref{PR0} (cf.~\eqref{up-to-uL}).

\subsection{The zeta element} We describe our results on the existence of a zeta element for an elliptic curve together with an imaginary quadratic field.
\subsubsection{Setting}
Let $E$ be an elliptic curves over $\BQ$ of conductor $N$, and $p\nmid 2N$ a prime. 
Let $\ov{\BQ}$ be an algebraic closure of $\BQ$. Fix embeddings $\iota_{\infty}:\ov{\BQ}\hookrightarrow \BC$ and $\iota_{p}:\ov{\BQ}\hookrightarrow \BC_{p}$.

Let $L$ be an imaginary quadratic field and $O_L$ its ring of integers. 
Suppose that\footnote{The main text allows $D_L$ and $N$ to have some common divisors.}
\begin{equation}\label{h1}
\text{$(D_{L},N)=1$}
\end{equation}
and 
\begin{equation}\label{h2}
\text{$(p)=v\ov{v}$ splits in $L$ with $v$ determined via $\iota_p$.}
\end{equation}
Let $L_\infty$ be the $\BZ_p^2$-extension of $L$ and $\Gamma_{L}=\Gal(L_{\infty}/L)$. 
Put $\Lambda_{L}=\BZ_{p}[\![\Gamma_{L}]\!]$ and $\Lambda_{L}^{\ur}=\Lambda_{L}\otimes_{\BZ_{p}}W(\ov{\BF}_{p})$, viewed as $G_{L}:=\Gal(\ov{\BQ}/L)$-modules via the canonical projection $G_{L}\twoheadrightarrow \Gamma_L$. 

We often assume that 
\begin{equation}\label{h3}
E[p](L)=0
\end{equation}
or
\begin{equation}\label{h_irr}
\text{$E[p]$ is an irreducible $G_L$-module.}
\end{equation}
The latter is satisfied for any supersingular prime $p>2$. Note that the former allows $p$ to be an Eisenstein prime. 

The nature of the zeta element for $E$  relies on the type of reduction at $p$. 
If $E$ has ordinary reduction at $p$, then there exists an exact sequence
$$
0 \ra T^{+} \ra T \ra T^{-} \ra 0
$$
of $\BZ_{p}[G_{\BQ_{p}}]$-modules, where 
$T$ denotes the $p$-adic Tate module of $E$ and  
$T^{-}$ is an unramified $\BZ_{p}[G_{\BQ_{p}}]$-module, $\BZ_p$-free of rank one. 
In the supersingular case we assume 
\begin{equation}\label{h4}
a_{p}(E):=p+1-\#E(\BF_p)=0,
\end{equation} referring to $p$ as supersingular. This is automatic for $p\geq 5$. 

For an ordinary or a supersingular prime $p\nmid 2N$ and $\bullet\in\{+,-\}$, put
\begin{equation}\label{n_red}
 \circ=
\begin{cases} \ord  & \text{$p$ is ordinary} \\ \bullet & \text{$p$ is supersingular}
\end{cases}
\
\cdot = 
\begin{cases}
 \emptyset  & \text{$p$ is ordinary} \\ \circ & \text{$p$ is supersingular.}
\end{cases}
\end{equation}
The zeta element lives in the Galois cohomology group 
$$
H^1_{\mathrm{rel},\circ}(\BZ[\frac{1}{p}], T(1)\hat\otimes \Lambda_{L}) = \{ 
\kappa\in H^1(O_{L}[\frac{1}{p}], T(1)\hat\otimes \Lambda_{L}) \ : \ \loc_{\ov{v}}(\kappa) \in 
H^1_{\circ}(L_{\ov{v}},T(1)\hat\otimes \Lambda_{L})\}.
$$
In the ordinary case $H^1_{\circ}(L_{\ov{v}},T(1)\hat\otimes \Lambda_{L}):=\Im\{
H^1(L_{\ov{v}},T^{+}(1)\hat\otimes \Lambda_{L}) \ra H^1(L_{\ov{v}},T(1)\hat\otimes \Lambda_{L})\}$. In the supersingular case the definition of 
$H^1_{\circ}(L_{\ov{v}},T(1)\hat\otimes \Lambda_{L})$ is based on a principle of Kobayashi. 

Let 
$$
\CL_{p}^{\rm Gr}(E_{/L})\in \Lambda_{L}^{\ur}
$$
be the associated Rankin--Selberg $p$-adic $L$-function interpolating the algebraic part of $L$-values $L(1,E_{/L}\otimes \chi)$ for certain infinite order Hecke characters $\chi$ over $L$, referred to as a Greenberg $p$-adic $L$-function. 
Note that it is a bounded measure independent of the type of reduction of $E$ at $p$. 
The following $p$-adic $L$-function does depend on the reduction. For $\cdot$ as in \eqref{n_red}, let 
$$\CL_{p}^{\cdot}(E_{/L}) \in \Lambda_L$$ be the associated Rankin--Selberg $p$-adic $L$-function interpolating a corresponding normalisation of the algebraic part of $L$-values $L(1,E_{L}\otimes \chi)$ for finite order characters $\chi$ of $\Gamma_L$. 

A conjecture of Kato \cite{K0,K1} predicts the existence of a zeta element in $H^1_{\mathrm{rel},\circ}(O_L[\frac{1}{p}], T(1)\hat\otimes \Lambda_{L})$ intertwined with the $p$-adic $L$-functions $\CL_{p}^{\rm Gr}(E_{/L})$ and $\CL_{p}^{\cdot}(E_{/L})$. 

\subsubsection{}The central result of this paper is the following. 
\begin{thm}\label{thmZ}
Let $E$ be an elliptic curve over $\BQ$ of conductor $N$ and $p\nmid 2N$ a prime. In the case $p=3$ suppose that \eqref{h4} holds if $p$ is supersingular. Let $L$ be an imaginary quadratic field satisfying the conditions \eqref{h1}, \eqref{h2} and \eqref{h3}. Then there exists a zeta element 
$$
\CZ^{\cdot}(E_{/L}) \in H^{1}_{\rel,\circ}(O_{L}[\frac{1}{p}], T(1) \hat{\otimes}\Lambda_{L})
$$
such that 
$$
{\rm Col}_{v}^{\cdot}(\loc_{v}(\CZ^{\cdot}(E_{/L})))=\CL_{p}^{\cdot}(E_{/L}),\ 
{\rm Log}_{\ov{v}}^{\cdot}(\loc_{\ov{v}}(\CZ^{\cdot}(E_{/L})))=\CL_{p}^{\rm Gr}(E_{/L}).
$$
Here ${\rm Col}_{v}^{\cdot}: H^{1}(L_{v},T(1)\otimes \Lambda_{L}) \ra \Lambda_L$ and 
${\rm Log}_{\ov{v}}^{\cdot}:  H^{1}_{\circ}(L_{\ov{v}},T(1)\otimes \Lambda_{L}) \ra \Lambda_L$ are Perrin-Riou regulator maps interpolating Bloch--Kato dual exponential and logarithm maps as in sections \ref{BFord} and \ref{BFss}.
\end{thm}

\begin{remark}
\begin{itemize}
\item[(i)] The hypotheses \eqref{h1} and \eqref{h3} are not essential. Without them, the explicit reciprocity laws merely have a slightly different appearance. 
\item[(ii)] Since the $p$-adic $L$-functions $\CL_{p}^{\cdot}(E_{/L})$ and $\CL_{p}^{\rm Gr}(E_{/L})$ are non-zero, so is the zeta element
$\CZ^{\cdot}(E_{/L})$. 
\end{itemize}
\end{remark}

\subsection{About the proofs} We begin by explaining the connections of the zeta element with the arithmetic of  elliptic curves, 
and then we give an outline of our proof of its existence. 

\subsubsection{The zeta element and main conjectures} The zeta element is ancillary to the Iwasawa theory of elliptic curves over imaginary quadratic fields, leading to 
the equivalence of apparently distant main conjectures. 

\subsubsection*{Main conjectures with $p$-adic $L$-functions}
\begin{conj}\label{Int_St} Let $E$ be an elliptic curve over $\BQ$ of conductor $N$ and $p\nmid 2N$ a prime. In the case $p=3$ suppose that \eqref{h4} holds if $p$ is supersingular. Let $L$ be an imaginary quadratic field satisfying the conditions \eqref{h1}, \eqref{h2} and \eqref{h3}. 
Then
\begin{itemize}
\item[(a)] $X_{\cdot}(E_{/L})$ is $\Lambda_{L}$-torsion, and 
\item[(b) ]$(\mathcal{L}_{p}^{\cdot}(E_{/L}))=\xi_{\Lambda_{L}}(X_{\cdot}(E_{/L})),$
an equality of ideals in $\Lambda_{L}\otimes_{\BZ_{p}} \BQ_{p}$ and even in $\Lambda_{L}$ if \eqref{h_irr} holds.
\end{itemize}
\end{conj} 
For the definition of the Selmer group $X_{\cdot}(E_{/L})$, the reader may refer to subsection \ref{ss:Sel-L}.

\begin{conj}\label{Int_Greenberg} 
Let $E$ be an elliptic curve over $\BQ$ of conductor $N$ and $p\nmid 2N$ a prime. 
Let $L$ be an imaginary quadratic field satisfying the conditions \eqref{h1} and \eqref{h2}.
Then
\begin{itemize}
\item[(a)] $X_{\rm{Gr}}(E_{/L})$ is $\Lambda_{L}$-torsion, and
\item[(b)] $(\mathcal{L}_{p}^{\rm{Gr}}(E_{/L}))=\xi_{\Lambda_{L}^{\ur}}(X_{\rm{Gr}}(E_{/L})),$ 
an equality of ideals in $\Lambda_{L}^{\ur}$.
\end{itemize}
\end{conj} 
For the definition of the Selmer group $X_{\rm Gr}(E_{/L})$, the reader may refer to subsection \ref{ss:Sel-L}.

\subsubsection*{Zeta element main conjecture}

\begin{conj}\label{Int_LLZ} 
Let $E$ be an elliptic curve over $\BQ$ of conductor $N$ and $p\nmid 2N$ a prime. In the case $p=3$ suppose that \eqref{h4} holds if $p$ is supersingular. Let $L$ be an imaginary quadratic field satisfying the conditions \eqref{h1}, \eqref{h2} and \eqref{h3}.

Let $$\CZ^{\cdot}(E_{/L}) \in H^{1}_{\rm{rel},\circ}(L, T(1) \otimes \Lambda_{L})$$ be the 
corresponding two-variable zeta element.
Then
\begin{itemize}
\item[(a)] $X_{\st,\circ}(E_{/L})$ is $\Lambda_{L}$-torsion, and 
\item[(b)] $
\xi\big{(}H^{1}_{\rm{rel},\circ}(L, T(1) \otimes \Lambda_{L})/ \Lambda_{L}
\cdot{\CZ^{\cdot}(E_{/L})}\big{)}
=\xi_{\Lambda_{L}}(X_{\st,\circ}(E_{/L})), 
$

an equality of ideals in $\Lambda_{L}\otimes_{\BZ_{p}} \BQ_{p}$ and even in 
$\Lambda_{L,\cO_{\lambda}}$ if \eqref{h_irr} holds.
\end{itemize}
\end{conj}
The explicit reciprocity law as in ~Theorem~\ref{thmZ} leads to the following.
\begin{prop}\label{Int_eq}
Let $E$ be an elliptic curve over $\BQ$ of conductor $N$ and $p\nmid 2N$ a prime. In the case $p=3$ suppose that \eqref{h4} holds if $p$ is supersingular. Let $L$ be an imaginary quadratic field satisfying the conditions \eqref{h1}, \eqref{h2} and \eqref{h3}.
Then a one-sided divisibility in any of the Conjectures \ref{Int_St}, \ref{Int_Greenberg} and \ref{Int_LLZ} 
implies the analogous divisibility in the other conjectures. 
In particular, these conjectures are equivalent.
\end{prop}

\subsubsection{Kobayashi's main conjecture: deducing Theorem \ref{thmA}} Let the setting be as in \S\ref{ss:Kob}. 

For a suitable choice of an imaginary quadratic field $L$, we first deduce the divisibility
$$
\mathcal{L}_{p}^{\rm{Gr}}(E_{/L})\big{|}
\xi_{\Lambda_{L}^{\ur}}(X_{\rm{Gr}}^{\rm ur}(E_{/L}))
$$
from the results of \cite{W1,CLW}, which rely on the
Eisenstein congruence method on the unitary group $U(3,1)$. The zeta element then yields the divisibility 
$$
\mathcal{L}_{p}^{\circ}(E_{/L})\big{|}
\xi_{\Lambda_{L}}(X_{\circ}(E_{/L}))
$$
by~Proposition~\ref{Int_eq}. Considering the cyclotomic projection, we obtain 
$$
\mathcal{L}_{p}^{\circ}(E) \mathcal{L}_{p}^{\circ}(E^L)\big{|}
\xi_{\Lambda}(X_{\circ}(E))\xi_{\Lambda}(X_{\circ}(E^{L}))
$$
for $E^L$ the quadratic twist of $E$ corresponding to the extension $L/\BQ$.

On the other hand, based on the work of Kato, Kobayashi proved that 
$$
\xi_{\Lambda}(X_{\circ}(E')) \big{|} \CL_{p}^{\circ}(E')
$$
for $E'\in\{E,E^{L}\}$ (cf.~\cite{Ko}). Hence all of the above divisibilities are equalities, concluding the proof of Theorem \ref{thmA}.
\begin{remark}
The proof of the $\GL_2$-main conjecture in the ordinary case in \cite{SU} relied on the $U(2,2)$-Eisenstein congruence for Hida families. Attempts to generalise it to the supersingular case seem to need new ideas.
The $U(3,1)$-Eisenstein congruence for certain semi-ordinary Hida families and the zeta element 
lead to the main conjecture in both the ordinary and supersingular case. 
\end{remark}

\subsubsection{Perrin-Riou's conjecture: proof}\label{ssPRo}
 Our approach is based on the two-variable zeta element, a variant of Beilinson--Kato elements, and the $p$-adic Waldspurger formula of Bertolini--Darmon--Prasanna.

We give an outline in the case $\epsilon(E_{/\BQ})=-1$.
Let $L$ be an imaginary quadratic field satisfying \eqref{h2} such that each prime dividing $N$ splits in $L$, and so
 $\epsilon(E_{/L})=-1$. 

To begin, building on \cite{K}, we introduce a  
Beilinson--Kato element over $L$:
$$
{\bf{z}}^{\cdot}_{E_{/L}} \in H^{1}_{{\rm rel},\circ}(O_{L}[\frac{1}{p}], T(1) \otimes_{\BZ_{p}} \Lambda_{L}^{\cyc})
$$
for $\Lambda_{L}^{\cyc}$ the cyclotomic Iwasawa algebra.
It is a $\Lambda_{L}^{\cyc}$-linear combination of the cyclotomic Beilinson--Kato elements 
${\bf{z}}_{E}$ and 
${\bf{z}}_{E^{L}}$
(${\bf{z}}_{E}$ is a cyclotomic deformation\footnote{In fact, Kato first constructs the $\Lambda$-adic zeta element ${\bf{z}}_{E}$ and then defines $z_{E}$.} of the Beilinson--Kato element $z_{E}$). 
The explicit reciprocity law for ${\bf{z}}^{\cdot}_{E_{/L}}$ at the place $v$ yields the cyclotomic $p$-adic $L$-function, interpolating the central $L$-values of the twists of $E_{/L}$
by the finite order Hecke characters along the cyclotomic tower.

On the other hand, we have the two-variable zeta element $\CZ^{\cdot}(E_{/L})$. 
Its explicit reciprocity law at the place $\ov{v}$ in combination with the $p$-adic Waldspurger formula gives
\begin{equation}\label{pWa}
\log_{\ov{v}}({\mathds{1}}_{L}(\loc_{\ov{v}}(\CZ^{\cdot}(E_{/L})))) \doteq
(\log_{\omega}y_{L})^{2}
\end{equation}
for $\log_{\ov{v}}$ the Bloch--Kato logarithm, 
$\mathds{1}_{L}(\cdot)$ the specialisation at the identity Hecke character, 
and $y_{L} \in E(L)$ a Heegner point arising from a modular parametrisation of $E$. 

Hence, in view of \eqref{pWa} Conjecture \ref{PR0} amounts to the comparison 
\begin{equation}\label{cZeta}
\CZ^{\cdot}(E_{/L})^{\cyc} \doteq {\bf{z}}_{E_{/L}}
\end{equation}
of cyclotomic zeta elements (here $(\cdot)^{\cyc}$ is the cyclotomic projection). 
We first establish a local analogue of  \eqref{cZeta} at the place $v$ via the explicit reciprocity laws for $\CZ^{\cdot}(E_{/L})$ and 
${\bf{z}}_{E_{/L}}$ at $v$.
The comparison \eqref{cZeta} then follows from the key global fact:
$$
\rank_{\Lambda_{L}^{\cyc}} H^{1}_{{\rm rel},\circ}(O_{L}[\frac{1}{p}], T(1) \otimes_{\BZ_{p}} \Lambda_{L}^{\cyc}) =1,
$$
which is essentially due to Kato \cite{K}.

Actually, the zeta element ${\bf z}_{E_{/L}}^\cdot$ need not be $p$-integral, in which case we slightly modify the strategy.

\subsubsection{Zeta elements over imaginary quadratic fields: existence}\label{ss:Z_st}
Our construction of the zeta element builds on the fundamental progress on Rankin--Selberg zeta elements\footnote{More specifically, the work of Lei--Loeffler--Zerbes \cite{LLZa} and its generalizations \cite{LLZb,KLZ,LZ,BL} are foundational to our construction of the zeta element.}
due to Bertolini--Darmon--Rotger \cite{BDR1,BDR2}, Lei--Loeffler--Zerbes \cite{LLZa,LLZb}, Kings--Loeffler--Zerbes \cite{KLZ}, Loeffler--Zerbes \cite{LZ} and Buyukboduk--Lei \cite{BL}. 
While these works initiate the construction, an essential difficulty arises from the presence of an Eisenstein CM Hida family for which the $p$-distinguished hypothesis fails, and the pertinent Hecke algebra is non-Gorenstein. 
This is reflected in the occurrence of singularities in the geometric construction of the zeta element, and necessitates a fine analysis of the global geometry underlying the construction, especially of the Tate lattice associated to the CM Hida family. Our analysis is quite roundabout, based on Beilinson--Kato elements and 
auxiliary Rankin--Selberg zeta elements!

We now outline the construction in the ordinary case. Let $E_{/\BQ}$ be an elliptic curve of conductor $N$ and $p\nmid 2N$ an ordinary prime. Let $L$ be an imaginary quadratic field satisfying  \eqref{h1}, \eqref{h2} and \eqref{h3}. 

Let ${\bf h}_v$ be the canonical CM family passing through the weight one theta series $\theta(\chi_L)$ associated to the quadratic character $\chi_L$ of the extension $L/\BQ$  
 with coefficients in the Iwasawa algebra $\Lambda_L^v$ of the maximal $p$-abelian extension of $L$ unramified outside $v$ as in 
\ref{CMHidaFam}.  
 Note that 
$\theta(\chi_L)$ is Eisenstein, in the sense that its associated mod $p$ residual Galois representation is reducible, hence in turn so is ${\bf h}_v$. 
The Tate lattices\footnote{in the sense of Hida} associated with the canonical CM family ${\bf{h}}_{v}$ are the $\Lambda_{L}^{v}[G_{\BQ}]$-modules  
$$
\BT = H^{1}_{\ord}(D_Lp^{\infty}) \otimes_{\BT_{D_Lp^{\infty}}^{\ord},\varphi} \Lambda_{L}^{v} \ \ \text{and}  \ \
\BH = \CH^1_\ord(D_Lp^\infty) \otimes _{\BH_{D_Lp^\infty}^{\ord},\varphi}\Lambda_L^v.
$$
Here $\BT_{D_Lp^{\infty}}^{\ord}$ (resp.~$\BH_{D_Lp^\infty}^{\ord}$) is the cuspidal (resp.~full) Hida Hecke algebra of tame level $D_L$, and $H^{1}_{\ord}(D_Lp^{\infty})$ (resp.~$\CH^1_\ord(D_Lp^\infty)$) is the ordinary part of the \'etale cohomology of the tower of closed (resp.~open) modular curves, and $\varphi: \BH_{D_Lp^\infty}^{\ord} \twoheadrightarrow \BT_{D_Lp^\infty}^{\ord} \ra \Lambda_{L}^v$ is the homomorphism corresponding to the Hecke eigenform ${\bf h}_v$. Note that the Hecke algebras $\BH_{D_Lp^\infty}^{\ord}$ and $\BT_{D_Lp^\infty}^{\ord}$, as well as the Tate modules $\BH$ and $\BT$, 
are distinct since ${\bf h}_v$ is Eisenstein. 
Moreover, a priori, neither $\BT$ nor $\BH$ are free $\Lambda_L^v$-modules. Put $\BH_{1}=\BH/\BH_{\tor}$. 

Let
$$
\mathcal{BF}(E_{/L}) \in H^1(\BZ[\frac{1}{p}], T(1)\hat\otimes  \BH_1 \hat\otimes \Lambda)
$$
be the image of  the Beilinson--Flach element $\mathcal{BF}({{f_{E}},{\bh}})$ 
associated\footnote{Here the coefficients involve the open Tate lattice $\BH_1$, instead of the closed lattice $\BT$, since the construction relies on Siegel units which are supported on cusps.} to the $p$-ordinary stabilisation of the weight two newform attached to $E$ and the CM Hida family ${\bf h}_v$. The corresponding explicit reciprocity laws \cite{LLZa,KLZ} link it to the $p$-adic $L$-functions $\CL_{p}(E_{/L})$ and $\CL_{p}^{\rm Gr}(E_{/L})$. Hence one expects  
$\mathcal{BF}(E_{/L})$ to be a candidate for the zeta element $\CZ(E_{/L})$. 
In fact, the former is a Rankin--Selberg zeta element over $\BQ$, and it would apparently be the desired zeta element for $E$ over $L$ if $\BH_1$ were integrally induced
from $G_L$ as a $G_\BQ$-representation. However, an essential obstacle is that $\BH_1$ need not be even $\Lambda_L^v$-free!

Our construction of the zeta element from $\mathcal{BF}(E_{/L})$ takes a different route, the key being the following.

\begin{thm}\label{thmZ_0}
\noindent
\begin{itemize}
\item[(a)] We have 
$$
\mathcal{BF}(E_{/L}) \in 
H^1(\BZ[\frac{1}{p}], T(1)\hat\otimes  \BT \hat\otimes \Lambda) \subset 
H^1(\BZ[\frac{1}{p}], T(1)\hat\otimes  \BH_1 \hat\otimes \Lambda).
$$
\item[(b)] The Tate lattice $\BT$ is integrally induced, i.e.  there exists an isomorphism 
$$
\BT \simeq \Ind_{G_{L}}^{G_{\BQ}}(\Lambda_{L}^{v}(\Theta_{v}))
$$
of $\Lambda[G_{\BQ}]$-modules for 
$\Theta_{v}$ the Hecke character as in \eqref{LamHec}. 
\end{itemize}
\end{thm}

In light of this theorem and the explicit reciprocity laws,
$\mathcal{BF}(E_{/L})$ does lead to the sought after zeta element $\CZ(E_{/L})$. 

Note that parts (a) and (b) of Theorem \ref{thmZ_0} have a markedly different appearance, yet their proofs are intertwined! 
A satisfying explanation that zeta element arises from the Tate lattice of the closed curve, as in part (a), has eluded us.

In the absence of the $p$-distinguished hypothesis present methods to study $\Lambda$-adic Tate modules do not apply.
 Our roundabout approach to the proof of Theorem \ref{thmZ_0} is based on Rankin--Selberg zeta elements for the pairs $(g, {\bf{h}}_{v})$ for $g$ a $p$-ordinary weight two newform: 
we introduce a preliminary two-variable zeta element for $g$ over $L$ using the former zeta element 
and an analysis of the reflexive closure of the 
$\Lambda$-adic Tate modules. 
Heuristically, this construction involves resolution of singularities arising from the underlying non-Gorenstein-ness. 
Recall that the Rankin--Selberg zeta element lives in the first Iwasawa cohomology of the 
$\Lambda_{L}^{v}[G_{\BQ}]$-module $T_{g} \otimes_{\BZ_{p}} \BH_1$ for $T_{g}$ a Tate module associated to $g$ (cf.~\cite{LLZa,KLZ}).
Since the maximal torsion-free quotient $\BH_{1}$ need not be a free $\Lambda_L^v$-module, we are  lead to consider its reflexive closure $\tilde{\BH}$ and the image of the zeta element in $T_{g} \otimes_{\BZ_{p}} \tilde{\BH}$. The analysis of the reflexive closure 
relies on Ohta's work, especially \cite{Oh1}.
 
If the property (b) fails, we show that an explicit reciprocity law for the preliminary zeta element over $L$  
would 
contradict Theorem~\ref{thmZ} (with $E$ replaced by $g$). 
Here we utilise results regarding congruence ideal for the canonical CM Hida family due to Bellaiche--Dimitrov \cite{BD} and 
 Betina--Dimitrov--Pozzi \cite{BDP'}, as well as Beilinson--Kato elments. 
In view of such an anomalous explicit reciprocity law the Greenberg $p$-adic $L$-function $\CL_p^{\rm Gr}(g_{/L})$
would vanish at the identity Hecke character over $L$  for {\it{any}} $p$-ordinary weight two newform $g$.
We arrive at a contradiction by constructing an {\it{auxiliary}} newform $g$ for a given imaginary quadratic field $L$ for which $\CL_{p}^{\rm Gr}(g_{/L})$ does not vanish at the identity Hecke character. 
The construction is based on 
the anticyclotomic non-vanishing results of Rohrlich \cite{Ro}
and the $p$-adic Waldspurger formula. 

Our proof of part (a) is based on (b) and a variant of the above strategy.

The construction in the supersingular case follows along similar lines, though we now have part (b) of Theorem \ref{thmZ_0} at our disposal. 


\subsection{Vistas}
\subsubsection{Euler systems over imaginary quadratic fields} 
In a forth-coming companion paper we expect to show that the two-variable zeta element extends to an Euler system, leading to a $p$-adic Euler system for elliptic curves $E$ over imaginary quadratic fields $L$ as in Theorem~\ref{thmZ}. The method is based on a refinement of the strategy outlined in 
\ref{ss:Z_st}. 

\subsubsection{Euler systems for $\GSp_{4}\times \GL_2$}
Our study of the $\Lambda$-adic Tate modules, especially Theorem~\ref{thmZ_0}(b),~seems to be 
 relevant in the context of Euler system for $\GSp_{4}\times \GL_2$ (cf.~\cite{HJS,LZ1}). It may shed some light on the existence of zeta elements for modular abelian surfaces over imaginary quadratic fields.

\subsubsection{Conjectures of Mazur, Kato, and Kolyvagin}
The two-variable zeta element has other applications. It is a key ingredient in the proof of Mazur's main conjecture at Eisenstein primes \cite{CGS}. It is also crucial in the recent proof of Kato's main conjecture and Perrin-Riou's Heegner point main conjecture at primes of good reduction under mild conditions \cite{BCS}, and in turn of Kolyvagin's conjecture \cite{BCGS}. It is also an ingredient in the recent proof of a result towards Kato's main conjecture at primes of additive reduction \cite{FW}. 
Some complementary application appear in \cite{Ca0,CW0,CCSS}.

\subsubsection{Sharifi's conjecture}
It would be interesting to situate our results regarding the $\Lambda$-adic Tate modules and zeta elements (cf.~Theorems~\ref{thmZ}~and~\ref{thmZ_0}) in the context of Sharifi's conjectures \cite{Shf}. These conjectures pertain to $\Lambda$-adic Tate modules in the Eisenstein case. However, the current framework assumes a Gorenstein or cyclicity hypothesis and excludes our setting.

\subsection{Related results} We include a few remarks about related results in the literature (cf.~\cite{BST}).

\subsubsection*{$p$-part of the BSD formula}
For semistable elliptic curves satisfying (irr$_{\BQ}$), the $p$-part of the BSD formula in the analytic rank one case is also due to Jetchev--Skinner--Wan with $p$ a prime of good reduction \cite{JSW} and Skinner--Zhang \cite{SZ} and Castella  \cite{Ca'} with $p>3$ a prime of multiplicative reduction. Note that the hypothesis (ram) automatically holds in the semistable case. 
The method of Jetchev--Skinner--Wan relies on the $p$-adic Waldspurger formula \cite{BDP1}. 
For supersingular primes, it gives a different approach to the $r=1$ case of Theorem~\ref{corA_thmA}. It is independent of the $p$-adic Gross--Zagier formula, but relies on 
Theorem~\ref{thmA}.

Several cases of Theorem \ref{corA'} have been obtained by Zhang \cite{Zh}
and Berti--Bertolini--Venerucci \cite{BBV}. The results are conditional on the surjectivity of $\ov{\rho}$ and impose restrictions on primes $q | N$ at which $p$ is allowed to divide the Tamagawa number $c_{q}(E_{/\BQ})$.
The case of Eisenstein primes $p$ appears in \cite{CGS} (see also~\cite{BS}). 
An analogue holds for CM elliptic curves under milder hypothesis, for example $p$ can be a prime of good supersingular reduction.
 The result is originally due to Rubin \cite{Ru2} and Kobayashi \cite{Ko'}. We provide another proof. 

In the analytic rank zero case, Kato proved that the $p$-part of the BSD formula is a consequence of Kato's main conjecture. 
We show that an analogous phenomenon persists in the analytic rank one case (cf.~\S\ref{pBSD}). 

\subsubsection*{$p$-converse theorem}
For semistable elliptic curves satisfying (irr$_{\BQ}$) and (sur), the $p$-converse was established by Skinner for good ordinary primes $p$ \cite{Sk'} under a non-triviality of the localisation at $p$. 
A similar result for $p$ multiplicative is due to Venerucci \cite{Ve2}.
In the semistable case,
Wan \cite{W0} proved the $p$-converse without assuming non-triviality of the localisation, 
 under the hypotheses of \cite{W1}. 

Several cases of Theorem \ref{corB'} have been obtained by Zhang \cite{Zh} and Skinner--Zhang \cite{SZ}. 
The results are conditional on the surjectivity of $\ov{\rho}$ and impose restrictions on primes $q | N$ at which $p$ is allowed to divide the Tamagawa number $c_{q}(E_{/\BQ})$.
An analogue holds for CM elliptic curves under milder hypothesis, for example $p$ can be a good supersingular prime. 
 The result is originally due to Rubin \cite{Ru2} and Kobayashi \cite{Ko'}. 
\subsubsection*{Perrin-Riou's conjecture} 

The split multiplicative case of Conjecture \ref{PR0} is due to Venerucci \cite{Ve1}.  In the good reduction case,
Bertolini--Darmon--Venerucci \cite{BDV} established the conjecture.  
Some cases of Conjecture \ref{PR0} have been established by Buyukboduk \cite{B}
and Buyukboduk--Pollack--Sasaki \cite{BPS}. 
More precisely, the latter authors consider the analytic rank one case with $p>2$ a good ordinary prime. 
 The result is conditional on a $p$-adic Gross--Zagier formula\footnote{an in progress work of Kobayashi.} for non-ordinary elliptic newforms with arbitrary weight and non-critical slope. 
For CM elliptic curves, a result towards Conjecture \ref{PR0} is due to Rubin \cite{Ru1} and Kato \cite[\S15]{K}: 
the analytic rank one case with $p>2$ a good ordinary prime can be deduced from \cite{Ru1} and \cite[\S15]{K}. 

\subsubsection*{The hypothesis \ref{h4}} 
The recent preprint \cite{Sprung-ss} adapts the strategy of \cite{W} (that is, the strategy fully realised in this paper) to the case $a_p\neq 0$ and not a $p$-adic unit together with the author's construction 
of `signed' bounded $p$-adic $L$-functions in this case. This existence of a suitable variant of the zeta elements $\mathcal{Z}^{\cdot}(E_{/L})$ in this case would then yield corresponding main conjectures. Unfortunately, the existence of these zeta elements is not immediate from our constructions. While their existence seems plausible to us, less clear is that they would be a straightforward consequence of the methods employed herein.

\subsection{Structure of the paper} 
In part I we study zeta elements associated to a weight two elliptic newform and an imaginary quadratic field $L$, namely 
a one-variable Beilinson--Kato element over $L$ (cf.~\S\ref{BK-overL}) and a two-variable zeta element over $L$ (cf.~\S\ref{two-variable-zeta} and \ref{two-variable-zeta-ss}). 
A link among the two is central to the paper. 
The analysis of a $\Lambda$-adic Tate module associated to a residually reducible CM Hida family with CM by $L$ constitutes a significant part (cf.~\S\ref{CMHF}--\S\ref{BFord}). The analysis is closely tied with the existence of the two variable zeta element over $L$. 
This part concludes with the proof of the Perrin-Riou conjecture (cf.~\S\ref{sMR}).

In part II we study certain Iwasawa main conjectures and the BSD conjecture 
for the elliptic newform over the rationals and the imaginary quadratic field $L$. The cyclotomic Greenberg main conjecture over the imaginary quadratic field (cf.~\S\ref{ssIMC2})  is key. In view of the comparison of the zeta elements, a recent result towards the two-variable Greenberg main conjecture over $L$ leads to a proof of Kobayashi's main conjecture at supersingular primes, with concomitant applications to the BSD conjecture (cf.~\S\ref{s:Kob}). 
In a similar vein, we establish the cyclotomic Greenberg main conjecture under a non-vanishing hypothesis 
based on the work of Kato, Skinner--Urban and Rubin (for example, see Corollary \ref{GKSU}). 
In view of the $p$-adic Waldspurger formula, the hypothesis can be verified in certain rank one situations (cf.~\eqref{locnv}). 
This leads to results towards the BSD conjecture at ordinary primes (cf.~\S\ref{MW}--\S\ref{s:pcv}).

\subsubsection*{Acknowledgement}
We are grateful to Haruzo Hida for stimulating discussions. 
We thank Laurent Clozel, John Coates, Chandrashekhar Khare, Shinichi Kobayashi, Dinakar Ramakrishnan, Sarah Zerbes, Shou-Wu Zhang and Wei Zhang for their encouraging interest. 
We also thank Patrick Allen, Christophe Breuil, Kazim Buyukboduk, Francesc Castella, Henri Darmon, Mladen Dimitrov, Matthias Flach, Olivier Fouquet, Giada Grossi, Ben Howard, Mahesh Kakde, Chan-Ho Kim, Antonio Lei, David Loeffler, Gyujin Oh, Kazuto Ota, Marco Sangiovanni Vincentelli, Anand Rajagopalan, Romyar Sharifi, Jacques Tilouine, Gisbert Wustholz and Shuai Zhai  
for instructive conversations. In addition, we thank Francesc Castella and Matthias Flach for helpful comments on the preprint.

The influence of Kazuya Kato's seminal ideas on zeta elements permeate this paper. The authors would like to express their deep admiration for his inspiring work and insights. Likewise, they are grateful to Bernadette Perrin-Riou for her pioneering work. 

During the preparation of this paper, A.B. was partially supported by the NSF grants DMS-2303864 and DMS-2302064; C.S. was partially supported by the Simons Investigator Grant \#376203 from the Simons Foundation and by the NSF grant  DMS-1901985; Y.T. was partially supported by the National Key R\&D Program of China grant no. 2023YFA1009701 and the National Natural Science Foundation of China grant no. 12288201, and X.W. was partially supported by NSFC grants 12288201, 11621061, CAS Project for Young Scientists in Basic Research grant no. YSBR-033 and 
National Key R\&D Program of China 2020YFA0712600.

\part{Zeta elements}

\section{Notation and preliminaries}\label{NotationPrelim}
In this section we recall the various objects, especially newforms, $L$-functions, periods, etc., which are fundamental 
to the main results of this paper, introducing notation that will be in force throughout. 

\subsection{Some notation}
We begin with some general notation. 

\subsubsection{The prime $p$} Throughout $p\geq 3$ will be an odd prime. Some results will also hold even for $p=2$, but restricting to odd $p$ simplifies
many arguments.  
Some of the main results of this paper are only stated for $p\geq 5$. When appropriate we include commments on this restriction.

\subsubsection{The embeddings $\iota_\infty$ and $\iota_p$}
Let $\ov{\BQ}$ be a fixed separable algebraic closure of $\BQ$ and $\ov{\BQ}_p$ a fixed separable algebraic closure of $\BQ_p$.
Let $\iota_\infty: \ov{\BQ} \hookrightarrow \BC$ and $\iota_p:\ov{\BQ} \hookrightarrow \ov{\BQ}_p$ be fixed embeddings of fields.
Via $\iota_\infty$, the complex conjugation $\tau \in \Gal(\BC/\BR)$ induces an element of order two in $G_\BQ = \Gal(\ov{\BQ}/\BQ)$, which is also denoted by $\tau$.
The embedding $\iota_p$ identifies $G_{\BQ_p}=\Gal(\ov{\BQ}_p/\BQ_p)$ with a subgroup of $G_\BQ$.

\subsubsection{Subfields of $\ov{\BQ}$}
For a subfield $K \subset \ov{\BQ}$, let $G_{K}=\Gal(\ov{\BQ}/K)$. 
For a set of places $\Sigma$ of $K$, let $K_{\Sigma} \subset \ov{\BQ}$ be the  maximal extension of $K$ unramified outside $\Sigma$ and let 
$G_{K,\Sigma}=\Gal(K_{\Sigma}/K)$. 
For a place $w$ of $K$, let $\ov{K}_{w}$ denote a fixed algebraic closure of $K_{w}$ and let $G_{K_{w}}=\Gal(\ov{K}_{w}/K_{w})$.
Let $I_{w} \subset G_{K_{w}}$ be the inertia subgroup. 
In the case that the residue field of $K_w$ is finite, let $\Frob_{w} \in G_{K_{w}}/I_{w}$ denote an {\em arithmetic} Frobenius. 
For convenience, we fix a $K$-linear embedding $\iota_w:\ov{\BQ} \hookrightarrow \ov{K}_{w}$, so that $G_{K_{w}}$ is identified with a decomposition group for $w$ in $G_{K}$.
For $K=\BQ$ and $w=p$ we take $\iota_p$ for this embedding.

\subsubsection{The character $\epsilon$} We write $\epsilon$ for the  $p$-adic cyclotomic character $\epsilon:G_{\BQ} \ra \BZ_{p}^\times$.

\subsubsection{The Iwasawa algebras $\Lambda_\CG$ and  $\Lambda$}\label{Lambda}
Let $\BQ(\mu_{p^\infty})$ be the cyclotomic extension obtained by adjoining all $p$-power
roots of unity and let $\CG = \Gal(\BQ(\mu_{p^\infty})/\BQ)$. 
Let $\BQ_{\infty}\subset \BQ(\mu_{p^\infty})$ be the $\BZ_{p}$-extension of $\BQ$ and let $\Gamma=\Gal(\BQ_{\infty}/\BQ)$. 
Let $\Delta = \Gal(\BQ(\mu_p)/\BQ)$. Then the canonical projections of $\CG$ to $\Gamma$ and $\Delta$ induce a
canonical isomorphism $\CG\xrightarrow{\sim}\Gamma\times\Delta$, which we use to identify $\CG$ with the product
$\Gamma\times\Delta$.
Let 
$$
\Lambda_\CG = \BZ_p[\![\CG]\!] \ \ \text{and} \ \ \Lambda = \BZ_p[\![\Gamma]\!].
$$
More generally, for a $p$-adically complete $\BZ_p$-algebra $R$, let
$\Lambda_{\CG,R}$ (resp.~$\Lambda_R$) denote $R[\![\CG]\!]$ (resp.~$R[\![\Gamma]\!]$). 
Note that $\Lambda_{\CG,R}$ is naturally a $\Lambda_R$-module and $\Lambda_{\CG,R}$ is identified with $\Lambda_R[\Delta]$.
Let $\gamma_\cyc \in \Gamma$ a topological generator. Then we have an $R$-isomorphism
$\Lambda_R \isoarrow R[\![T]\!]$ arising from $\gamma_\cyc \mapsto 1+T$.

Let  $\omega:\Delta\xrightarrow{\sim} (\BZ/p\BZ)^\times \isoarrow \mu_{p-1}\subset \BZ_p^\times$ 
be the composition of the canonical isomorphism with the Teichm\"uller lift. For an integer $i$, let
$\eps_i = \frac{1}{\#\Delta} \sum_{\sigma\in\Delta} \omega^{-i}(\sigma)\sigma \in \BZ_p[\Delta]$
be the idempotent associated with $\omega^i$.
Then $\Lambda_{\CG,R}$ decomposes as
$\Lambda_{\CG,R} = \oplus_{i=0}^{p-2} \Lambda_{\CG,R}^{(i)}$ with 
$\Lambda_{\CG,R}^{(i)} = \eps_i\Lambda_{\CG,R}  = \Lambda_{\CG,R}\eps_i = \Lambda_R\eps_i$.
Each $\Lambda_{\CG,R}^{(i)}$ is a subring canonically isomorphic to $\Lambda_R$, with identity element $\eps_i$.

The canonical projections
\begin{equation}\label{cycuniv}
\Psi_\CG:G_\BQ\twoheadrightarrow \CG  \ \ \text{and} \ \ \Psi: G_{\BQ} \twoheadrightarrow \Gamma 
\end{equation}
are often referred to as the canonical characters and viewed 
as taking values in $\Lambda_\CG^\times$
and $\Lambda^\times$, respectively. 
The composition of $\Psi_\CG$ with the projection to $\Lambda_{\CG}^{(i)}$ is $\omega^i\Psi$, where we also write
$\omega$ for the composition $G_\BQ\twoheadrightarrow\Delta\stackrel{\omega}{\rightarrow} \BZ_p^\times$.

\subsection{Newforms} \label{Newforms}
Let $g \in S_{2}(\Gamma_{0}(N))$ be a newform of weight two, level $N$, and trivial character.  To simplify some arguments, we will always assume that
\begin{equation}\label{p-nmid-N}
p\nmid N
\end{equation}
from \S\ref{ModularForms} onward.

\subsubsection{The Hecke field of a newform}
Let $g(q)=\sum_{n\geq 1}a_{g}(n)q^{n}$ be the $q$-expansion ($q=e^{2 \pi i \tau}$) of the newform $g$. Recall that the field
$$F=\BQ(a_{g}(n);n \geq 1),$$ 
the Hecke field of $g$, is a finite extension of $\BQ$, and 
$\cO_{F,0}=\BZ[a_{g}(n);n \geq 1] \subset F$ is an order (possibly not maximal).  Let $\cO_F$ be the ring of integers of $F$ (the maximal order).
We view $F$ as a subfield of $\ov{\BQ}$ via $\iota_\infty$.

Let $\lambda\mid p$ be the prime of $\cO_F$ determined\footnote{This choice is for convenience. All the subsequent
constructions and results for a different choice of a prime $\lambda$ can be obtained either by a different choice of the initial $\iota_p$ 
or by replacing $g$ with a Galois conjugate.} by $\iota_p$. Let $\cO=\cO_{F,(\lambda)}$ and $\cO_\lambda=\cO_{F,\lambda}$.
Let $\lambda_0= \lambda \cap \cO_{F,0}$ and let $k_0 = \cO_{F,0}/\lambda_0$ be its residue field. Let $k_\lambda = \cO/\lambda = \cO_\lambda/\lambda\cO_\lambda$;
this is a finite extension of $k_0$.

\subsubsection{Hecke algebras}\label{Hecke algebras} 
Let $\Gamma$ be either $\Gamma_{0}(N)$ or $\Gamma_{1}(N)$. 
Let $\BT_{\Gamma}$ (resp.~$\BH_\Gamma$) be the Hecke algebra for level $\Gamma$ generated over $\BZ_{(p)}$ by the  Hecke operators $T(n)$,  
$n \geq 1$, defined in \cite[\S4.9]{K},
acting on the space of cuspforms $S_2(\Gamma)$ (resp.~the space of modular forms $M_2(\Gamma)$).
Similarly, let $\BT_{\Gamma}'$ and $\BH_\Gamma'$ be the Hecke algebras for level $\Gamma$ generated over $\BZ_{(p)}$ by the dual Hecke operators $T'(n)$, also defined in \cite[\S4.9]{K}.
The Hecke algebras $\BT_\Gamma$ and $\BT'_\Gamma$ (resp.~$\BH_\Gamma$ and $\BH_\Gamma'$) act on the cuspforms (resp.~modular forms) 
of level $\Gamma$ over $\BZ_{(p)}$ and also on the cohomology of the closed (resp.~open) modular curves of level $\Gamma$ over $\BZ_{(p)}$
(cf.~\cite[\S4.9]{K}).

\subsubsection{Congruence numbers}\label{congruence}
Let $\phi_{\Gamma}:\BT_\Gamma\twoheadrightarrow  \cO_{F,0}\otimes\BZ_{(p)}$ be the $\BZ_{(p)}$-homomorphism associated with $g$. 
There is a factorisation 
$$
\phi_{\Gamma_{1}(N)}: 
\BT_{\Gamma_{1}(N)} \twoheadrightarrow \BT_{\Gamma_{0}(N)} \stackrel{\phi_{\Gamma_0}}{\twoheadrightarrow}  \cO_{F,0}\otimes\BZ_{(p)},
$$
with the first morphism being the canonical surjection arising from the inclusion $S_{2}(\Gamma_{0}(N)) \subset S_{2}(\Gamma_{1}(N))$.
Let
$\mathfrak{p}_0$ be the kernel of $\phi_{\Gamma_1(N)}$ and let $\mathfrak{m}_0$ be the kernel of the reduction 
of $\phi_{\Gamma_1(N)}$ modulo $\lambda_0$. Let $\BT_{\mathfrak{m}_0}$ be the localization of $\BT_{\Gamma_1(n)}$
at the maximal ideal $\mathfrak{m}_0$.

It will often be more convenient to work with
$\BT_{\Gamma,\cO} = \BT_{\Gamma}\otimes_{\BZ_{(p)}}\cO$ instead of 
$\BT_\Gamma$.  The $\BZ_{(p)}$-homomorphism $\phi_\Gamma$ induces an $\cO$-homomorphism
$\phi_{\Gamma,\cO}:\BT_{\Gamma,\cO}\twoheadrightarrow \cO$. Let $\mathfrak{p}$ be the kernel of $\phi_{\Gamma_1(N),\cO}$
and let $\mathfrak{m}$ be the kernel of the reduction of $\phi_{\Gamma_1(N),\cO}$ modulo $\lambda$. 
Note that $\mathfrak{p}\cap \BT_{\Gamma_1(N)} = \mathfrak{p}_0$ and $\mathfrak{m}\cap \BT_{\Gamma_1(N)} = \mathfrak{m}_0$.
Let $\BT_\mathfrak{m}$ be the localization of $\BT_{\Gamma_1(N),\cO}$ at the maximal ideal $\mathfrak{m}$, which is then an $\cO_\lambda$-algebra. 
The natural map $\BT_{\mathfrak{m}_0}\rightarrow \BT_{\mathfrak{m}}$ induces an isomorphism $\BT_{\mathfrak{m}_0}\otimes_{W(k_0)}\cO_\lambda \stackrel{\sim}{\rightarrow} \BT_{\mathfrak{m}}$,
where $W(k_0)$ is the ring of Witt vectors for the finite field $k_0$.

Let $\pi: \BT_{\mathfrak{m}} \twoheadrightarrow  \cO_{\lambda}$ be the $\cO_\lambda$-morphism arising from $\phi_{\Gamma_{1}(N),\cO}$. 
A {\em congruence number} for $g$ is an element $c_{g} \in \cO_{\lambda}$ such that 
$$
(c_{g}) = \pi(\Ann_{\mathfrak{m}}(\ker(\pi)) \subset \cO_\lambda.
$$
Note that a congruence number $c_g$ is only uniquely defined up to multiplication
by an element in $\cO_\lambda^\times$. However, we can and do choose a congruence number $c_{g} \in \cO$, which is
then algebraic and uniquely determined up to multiplication by an element of $\cO^\times$. A congruence number in $\cO$ can also be directly defined as follows.

The ring $\BT_{\Gamma_1(N),\cO_\lambda} = \BT_{\Gamma_1(N),\cO}\otimes_\cO\cO_\lambda$ is a semi-local Artinian $\cO_\lambda$-algebra and so 
is canonically identified with the product of its localisations at its (finitely many) maximal ideals. In particular, $\BT_{\mathfrak{m}}$ is a direct factor
of $\BT_{\Gamma_1(N),\cO_\lambda}$. Let $\phi_{\Gamma_1(N),\cO_\lambda}$ be the $\cO_\lambda$-linear extension of $\phi_{\Gamma_1(N),\cO}$.
It then follows that 
$$
(c_{g}) = \phi_{\Gamma_1(N),\cO_\lambda}(\Ann_{\BT_{\Gamma_{1}(N),\cO_{\lambda}}}(\phi_{\Gamma_1(N),\cO_\lambda}) \subset \cO_\lambda.
$$
Since $\Ann_{\BT_{\Gamma_{1}(N),\cO_{\lambda}}}(\ker(\phi_{\Gamma_1(N),\cO_\lambda})) = \Ann_{\BT_{\Gamma_{1}(N),\cO}}(\ker(\phi_{\Gamma_1(N),\cO}))\otimes_\cO\cO_\lambda$,
we then also have that a congruence number can be taken to be a generator of the $\cO$-ideal
$\phi_{\Gamma_1(N),\cO}(\Ann_{\BT_{\Gamma_{1}(N),\cO}}(\ker(\phi_{\Gamma_1(N),\cO})))\subset \cO$.

\subsubsection{Cohomology of modular curves}\label{cohomology}
Let $Y_{1}(N)$ (resp. $Y_{0}(N)$) be the open modular curve of level $\Gamma_{1}(N)$ (resp. $\Gamma_{0}(N)$). 
Let $X_{1}(N)=X_{1}(N)_{/\BQ}$ (resp. $X_{0}(N)$) be the modular curve of level $\Gamma_{1}(N)$ (resp. $\Gamma_{0}(N)$). We consider 
these as curves over $\BQ$ using the models in \cite[\S2.8]{K}.

For a $\BZ$-algebra $A$ let
$$
V_{A}= \frac{H^{1}(Y_{1}(N),A)}{\mathfrak{p}\cdot H^{1}(Y_{1}(N),A)}.
$$
Here $H^{1}(Y_{1}(N), A)$ denotes the Betti cohomology $H^{1}(Y_{1}(N)(\BC), A)$.
The action of the complex conjugation $\tau$ induces an involution of $H^{1}(Y_{1}(N)(\BC), A)$ and hence of $V_A$, which we also denote by $\tau$.
Let $H^{1}(Y_{1}(N), A)^\pm$ (resp.~$V_A^\pm$) be the submodule on which the action of $\tau$ is by multiplication by $\pm 1$. 
If $2$ is invertible in $A$, then $H^{1}(Y_{1}(N), A)$ (resp.~$V_A$) is a direct sum of  $H^{1}(Y_{1}(N), A)^+$ and $H^{1}(Y_{1}(N), A)^-$ (resp, $V_A^+$ and $V_A^-$). 

Since $g$ is a newform of level $N$, $\dim_{F}V_{F}=2$.
Furthermore, the maps $V_F\otimes_F F_\lambda \rightarrow V_{F_\lambda}$ and
$V_F\otimes_F\BC\rightarrow V_\BC$ induced by functoriality are both isomorphisms.
In this way $V_F$ defines an $F$-structure on each of $V_{F_\lambda}$ and $V_\BC$.

The inclusion $\cO\hookrightarrow F$ induces
an injection $(V_{\cO})_{/\tor} \hookrightarrow V_F$. Let $T_\cO$ be the image of this map. 
In particular, 
\begin{equation*}\label{lattice}
T_{\cO}=\Im(H^{1}(Y_{1}(N),\cO) \ra V_{F}) = \Im(V_\cO \ra V_{F})
\end{equation*}
This is an $\cO$-lattice in $V_F$. We analogously define $T_{\cO_{\lambda}}\subset V_{F_\lambda}$. 
Under the functorial identification of $V_{F_\lambda}$ with $V_F\otimes_F F_\lambda$, $T_{\cO_{\lambda}}$ is identified
with $T_\cO\otimes_{\cO} \cO_{\lambda}$.

\subsubsection{Galois representations}
It follows from the comparison isomorphism of Betti and \'etale cohomology that $V_{F_\lambda}$
is naturally equipped with a continuous $F_\lambda$-linear action of $G_\BQ$ and that $T_{\cO_\lambda}$ is a
$G_\BQ$-stable $\cO_\lambda$-lattice in $V_{F_\lambda}$. To simplify notation, we will let
$$
V = V_{F_\lambda} \ \ \text{and} \ \ T=T_{\cO_\lambda}.
$$
The representation  
$$
\rho: G_\BQ \rightarrow \Aut_{F_\lambda}(V)
$$
is unramified at all $\ell\nmid Np$, and for such primes
$$
\mathrm{trace}\,\rho(\Frob_\ell^{-1}) = a_g(\ell) \ \ \text{and} \ \ \det\rho(\Frob_\ell^{-1}) = \ell.
$$
In particular, 
$$
\det\rho = \epsilon^{-1}.
$$
Let 
$$
\ov{T} = T/\lambda T.
$$
We will sometimes make the following hypothesis relative to a given number field $K\subset \ov{\BQ}$:
\begin{equation}\tag{irr$_{K}$}
\text{$\ov{T}$ is an absolutely irreducible $k_\lambda[G_{K}]$-module.} 
\end{equation}
This implies, but is not implied by, the hypothesis that for a given abelian extension $M/K:$
\begin{equation}\tag{van$_{M}$}
\ov{T}^{G_M} = 0.
\end{equation}

\subsubsection{Spaces of modular forms}\label{ModularForms}
Since $p \nmid N$, the modular curve $X_{1}(N)_{/\BQ}$ extends as a smooth projective scheme $X_{1}(N)_{/\BZ_{(p)}}$.
Let $\Omega^{1}_{X_{1}(N)}$ be the sheaf of differentials, and 
for a $\BZ_{(p)}$-algebra $B$ let 
$$
S_2(\Gamma_1(N))_B = H^0(X_1(N)_{/\BZ_{(p)}},\Omega^1_{X_1(N)})\otimes_{\BZ_{(p)}} B.
$$
If $B$ is a flat $\BZ_{(p)}$-algebra, then $S_2(\Gamma_1(N))_B = H^0(X_1(N)_{/B},\Omega^1_{X_1(N)})$.
We similarly define $S_{2}(\Gamma_{0}(N))_{B}$.

Let 
$$
S_F = \frac{S_2(\Gamma_1(N))_F}{\mathfrak{p}\cdot S_2(\Gamma_1(N))_F}.
$$
Since $g$ is a newform of level $N$, $\dim_{F}S_{F}=1$.

Let 
$$
S_\cO = \Big (\frac{S_2(\Gamma_1(N))_\cO}{\mathfrak{p}\cdot S_2(\Gamma_1(N))_\cO}\Big)_{/\tor}.
$$
This is a free $\cO$-module of rank one, and 
the inclusion $\cO\hookrightarrow F$ identifies $S_\cO$ with 
an $\cO$-lattice in $S_F$.

\vskip 2mm
{\noindent{\it Good differentials.} We will say that a differential $\omega\in S_F$ is {\em good}
if it is an $\cO$-basis of $S_{\cO}$.

\subsubsection{Periods}\label{Periods}
 The Eichler--Shimura period morphism  
of \cite[\S4.10]{K} 
induces an injective $\BC$-linear map
$$
per: S_{F} \otimes_{F} \BC \ra V_{F} \otimes_{F} \BC  = V_\BC.
$$
Fix $\gamma \in V_{F}$ with 
$$
0 \neq \gamma^{\pm}=\frac{1\pm \tau}{2}\cdot \gamma \in V_F.
$$
Then $\gamma^+$, $\gamma^-$ are an $F$-basis of $V_F$ and hence a $\BC$-basis of $V_\BC$. 
For $0 \neq \omega \in S_{F}$, let 
$\Omega^{\pm}_{\omega,\gamma} \in \BC^{\times}$ be 
such that 
$$
per(\omega) = \Omega^{+}_{\omega,\gamma}  \cdot \gamma^{+} + \Omega^{-}_{\omega,\gamma}  \cdot \gamma^{-}.
$$ 
Note that the period $\Omega^{\pm}_{\omega,\gamma} $ is uniquely defined only up to multiplication by an element of $F^\times$. 

\vskip2mm
\noindent {\it{Optimal periods.}} Since $g$ is a newform, the $\cO$-modules $H^{1}(Y_{1}(N),\cO)^{\pm}[\mathfrak{p}]$ are both free of rank one. 
Here `$[\mathfrak{p}]$' denotes the submodule which is annihilated by $\mathfrak{p}$.
Let $\delta_{g}^{\pm} \in H^{1}_c(Y_{1}(N),\cO)^{\pm}[\mathfrak{p}]$ be 
an $\cO$-module generator, where  the subscript `$c$' denotes compactly supported cohomology.  We identify $\delta_g^\pm$ with their images in $V_F$.
Similarly, the $\cO$-module $H^{0}(X_{1}(N)_{/\cO}, \Omega^{1}_{X_{1}(N)})[\mathfrak{p}]$ is free of rank one. 
Let $\omega_{g} \in H^{0}(X_{1}(N)_{/\cO}, \Omega^{1}_{X_{1}(N)})[\mathfrak{p}]$ be $\cO$-module generators and let $\delta_g = \delta_g^++\delta_g^-\in V_F$.
Let $\Omega^{\pm} \in \BC^{\times}$ be defined by 
$$
per(\omega_{g}) = \Omega^{+} \cdot \delta_{g}^{+} + \Omega^{-} \cdot \delta_{g}^{-}.
$$ 
That is, $\Omega^\pm = \Omega^\pm_{\omega_g,\delta_g}$.
Note that the $\Omega^\pm$ are uniquely defined up to $\cO^\times$-multiples.
The elements $\delta_g^\pm$ map to non-trivial elements in $T_\cO^\pm$ and the differential $\omega_g$ maps to 
a non-trivial element of $S_\cO$,  so $\Omega^\pm$ is an $F^\times$-multiple of any $\Omega^\pm_{\gamma,\omega}$.
In subsequent considerations, we usually work with these optimal periods $\Omega^{\pm}$.

\begin{remark} If (irr$_\BQ$) holds, then 
$H^{1}_c(Y_{1}(N),\cO)^{\pm}[\mathfrak{p}]= H^{1}(X_{1}(N),\cO)^{\pm}[\mathfrak{p}] = H^{1}(Y_{1}(N),\cO)^{\pm}[\mathfrak{p}].$
\end{remark}

\begin{remark}
The extent to which the images of $\delta_g^\pm$ and $\omega_g$ fail to be $\cO$-bases can often be measured
by a congruence ideal (see \S\ref{IntStr} below, especially Lemma \ref{GorPer}). 
\end{remark}

\begin{remark}\label{EC-optperiod-rmk}
If $g$ is the newform associated with an isogeny class of elliptic curves over $\BQ$, then there is a curve $E_\bullet$ in the isogeny
class such that the optimal periods $\Omega^\pm$ are a $\BZ_{(p)}$-basis of the lattice of N\'eron periods of $E_\bullet$. If (irr$_\BQ$) holds, this is clear
for $E_\bullet = E_1$, the optimal quotient of $J_1(N)$ in the isogeny class. If (irr$_\BQ$) does not hold (in this case, $E_\bullet$ is a quotient of $E_1$ by an \'etale subgroup), this was proved in \cite{Wu}, see especially \cite[Thm.~4 \& Prop.~8]{Wu} and their proofs.
\end{remark}

\subsubsection{The Gorenstein condition}\label{Gorenstein}
The ring $\BT_{\mathfrak{m}}$ satisfies the {\em Gorenstein condition} if 
\begin{itemize}
\item[(i)] $\BT_{\mathfrak{m}}$ is a Gorenstein $\cO_\lambda$-algebra, that is, $\Hom_{\cO_\lambda}(\BT_{\mathfrak{m}},\cO_\lambda)$ is 
a free $\BT_{\mathfrak{m}}$-module of rank one, and
\item[(ii)] $H^1(Y_1(N),\cO)_{\mathfrak{m}}$ is a free $\BT_{\mathfrak{m}}$-module of rank two (equivalently, 
$H^1(Y_1(N),\cO)_{\mathfrak{m}}^\pm$ is a free $\BT_{\mathfrak{m}}$-module of rank one).
\end{itemize}
If $\mathfrak{m}$ is not Eisenstein, in the sense that (irr$_\BQ$) holds, 
then (ii) implies (i) (as a consequence of Poincar\'e duality).
For later reference we record the following lemma.

\begin{lem}\label{GorCond} 
If (irr$_\BQ$) holds, then $\BT_{\mathfrak{m}}$ satisfies the Gorenstein condition.
\end{lem}

\begin{proof} Since $\BT_{\mathfrak{m}} = \BT_{\Gamma_1(N),\mathfrak{m_g}}\otimes_{W(k_g)}\cO_\lambda$,
this follows immediately from \cite[Chap.~2, Cor.~1 and 2]{Wi}.
\end{proof}

\subsubsection{Integral structures and cuspidal conguence numbers}\label{IntStr}
There are natural injections 
$$
H^{0}(X_{1}(N)_{/\cO}, \Omega^{1}_{X_{1}(N)})[\mathfrak{p}] \hookrightarrow S_{\cO}
\ \ \text{and} \ \ 
H^{1}_c(Y_{1}(N),\cO)^{\pm}[\mathfrak{p}] \ra T_{\cO}^{\pm}
$$
of free, rank one $\cO$-modules.
In particular,  there are $c,c^\pm\in \cO$ such that $\frac{\omega_g}{c}$ is an $\cO$-generator of $S_\cO$
and $\gamma_g^\pm = \frac{\delta_g^\pm}{c^\pm}$ is an $\cO$-generator of $T_\cO^\pm$.
Such an element $c$ is also sometimes called a congruence number for $g$. In order to not 
confuse this with the congruence number $c_g$ defined in \S\ref{congruence}, we will refer to such a $c$ as 
a {\em cuspidal congruence number} for $g$ and to the ideal $I_{g,0} = (c) = c\cO_\lambda$ as the {\em cuspidal 
congruence ideal} of $g$. Note that $S_{\cO_\lambda} = I_{g,0}^{-1}\cdot \omega_g$.
 
\begin{lem}\label{GorPer} 
Let $c_g\in \cO$ be a generator of the congruence ideal for $g$ as in \S\ref{congruence}. 
\begin{itemize}
\item[(i)] Each of $c$ and $c^\pm$ divides $c_g$. 
\item[(ii)] If (irr$_\BQ$) holds, then we can we take $c = c^\pm = c_g$. 
In particular, 
$$
S_\cO = \cO\cdot \omega, \ \ \omega = \frac{\omega_g}{c_g}, 
$$
and
$$
T_\cO^\pm = \cO\cdot \gamma^\pm_g, \ \ \gamma_g^\pm = \frac{\delta_g^\pm}{c_g}.
$$
\end{itemize}
\end{lem}

\begin{proof}
Suppose $\omega' \in H^0(X_1(N)_{/\cO},\Omega_{X_1(N)})\subset H^0(X_1(N)_{/F},\Omega_{X_1(N)})$ projects to an $\cO$-basis of $S_\cO$ in $S_F$. Then
$\omega' = \frac{\omega_g}{c} + \omega''$ for some $c\in\cO$ and some $\omega''\in \mathfrak{p}_g H^0(X_1(N)_{/F},\Omega_{X_1(N)})$.
In particular, if $t\in \Ann_{\BT_{\Gamma_1(N),\cO}}(\phi_{\Gamma_1(N),\cO})$, then 
$$
\phi_{\Gamma_1(N),\cO}(t) \frac{\omega_g}{c} = \frac{t\cdot \omega_g}{c} = t\cdot \omega' \in H^0(X_1(N)_{/\cO},\Omega_{X_1(N)}).
$$
As $\omega_g$ is part of an $\cO$-basis of $H^0(X_1(N)_{/\cO},\Omega_{X_1(N)})$, it follows that $c\mid \phi_{\Gamma_1(N),\cO}(t)$.
As $(c_g)\subset \cO$ is the ideal generated by the $\phi_{\Gamma_1(N)}(t)$, $t\in \Ann_{\BT_{\Gamma_1(N),\cO}}(\mathfrak{p}_g)$,
it then follows that $c\mid c_g$.  An analogous argument applies to $c^\pm$.  This proves part (i).

Part (ii) is an easy consequence of the Gorenstein condition, which holds by Lemma \ref{GorCond}. 
The module $H^0(X_1(N)_{/\cO},\Omega_{X_1(N)})$ is dual to $\BT_{\Gamma_1(N),\cO}$
as a $\BT_{\Gamma_1(N),\cO}$-module. Hence there is an isomorphism of $\BT_\mathfrak{m}$-modules
$$
H^0(X_1(N)_{/\cO},\Omega_{X_1(N)})_{\mathfrak{m}} \cong \Hom_{\cO_\lambda}(\BT_{\mathfrak{m}},\cO_\lambda) \cong \BT_\mathfrak{m},
$$
the last isomorphism by part (i) of the Gorenstein condition. Then $H^0(X_1(N)_{/\cO},\Omega_{X_1(N)})_\mathfrak{m}[\mathfrak{p}]
\cong \Ann_{\BT_\mathfrak{m}}(\mathfrak{p})$ as $\BT_{\mathfrak{m}}$-modules, from which the claim that $(c) = (c_g)$ easily follows. A similar argument applies to $c^\pm$ by also appealing
to property (ii) of the Gorenstein condition.
 \end{proof}

\subsubsection{Petersson norms and periods}
Let $\omega_{g}$ denote the holomorphic differential on $X_{0}(N)$ that is the unique holomorphic extension of the differential on $Y_{0}(N)$ that pulls back to 
$2\pi i g(z)dz = g(q)\frac{dq}{q}$ under the complex uniformisation 
$$
\mathfrak{h}/\Gamma_{1}(N) \simeq Y_{1}(N)(\BC),
$$ 
for $\mathfrak{h}$ the upper half plane, $z=x+iy$ the complex variable, and $q= e^{2\pi i z}$. 
It follows from the $q$-expansion principle that we can take this to be 
the differential so denoted in \S\ref{Periods}:
For the chosen model of $X_1(N)_{/\cO}$ and since $p\nmid N$, $\omega_g\in H^0(X_1(N)_{/\cO}, \Omega_{X_1(N)})$
if and only if the $q$-expansion coefficients of $w_N\cdot g$ belong to $\cO$, where $w_N$ is the Atkin--Lehner involution (cf. \cite[\S 1.5.10]{FK});
since $g$ is a newform with trivial Nebentypus, this 
latter condition is equivalent to the $q$-expansion of $g$ having coefficients in $\cO$.

Let $\langle \cdot, \cdot \rangle$ denote the Petersson inner product on $S_{2}(\Gamma_{1}(N))_{\BC}$ 
given by 
$$
\langle g, g' \rangle =  \frac{i}{8\pi^2} \cdot \int_{X_{1}(N)(\BC)} \omega_{g} \wedge \ov{\omega_{g'}}
= \int_{\mathfrak{h}/\Gamma_{1}(N)} g(z)\ov{g'(z)} dx dy. 
$$
The congruence period of $g$ is defined to be 
$$
\Omega_{g}^{cong}=\frac{\langle g, g \rangle}{c_{g}}.
$$
This is related to the optimal periods $\Omega^\pm$ as follows.
\begin{lem}\label{PetPer}\hfill
\begin{itemize}
\item[(i)] $\Omega_g^{cong}\sim_{F^\times} -i(2\pi)^{-2} \Omega^{+} \Omega^{-}$.
\item[(ii)] If (irr$_\BQ$) holds, then 
$\Omega_{g}^{cong} \sim_{\cO^{\times}} -i(2\pi)^{-2} \cdot \Omega^{+} \Omega^{-}$.
\end{itemize}
\end{lem}
\noindent Here and throughout `$\sim_{\Box}$' denotes equality up to multiplication by an element in $\Box$.

This lemma is likely well-known, and essentially proven in \cite[Chap.~4, \S 2]{Wi} (see also \cite[\S4.4]{DDT} ).
As we have assumed that $g$ is a newform with trivial character, the proof is slightly simpler than 
in \cite{Wi}. The proof from \cite{DDT} suffices, but with $\Gamma_0(N)$ replaced by $\Gamma_1(N)$
for the purposes of this paper.

\begin{remark}\label{EC-Petersson-rmk}
If $g$ corresponds to the isogeny class of an elliptic curve $E_\bullet$ as in Remark \ref{EC-optperiod-rmk}, then 
the conclusion of part (ii) still holds even when (irr$_\BQ$) does not. The modular parameterisation $\pi_\bullet:X_1(N)\twoheadrightarrow E_1\twoheadrightarrow E_\bullet$ 
is such that the N\'eron differential $\omega_{E_\bullet}$ pulls back to a $p$-adic unit multiple of $\omega_g$, and so we have 
$$
-i(2\pi)^{-2}\Omega^+\Omega^- \sim_{\BZ_p^\times} \deg(\pi_\bullet)^{-1} \langle g,g\rangle.
$$
It remains to note that $\deg(\pi_\bullet)$ equals $c_g$ (up to a $p$-adic unit), and this follows from a simple modification of the arguments used to prove
\cite[Lem.~3.1.2]{CGS}.
\end{remark}

\subsubsection{The ordinary and supersingular cases}
In many of our results and arguments we will consider two possible cases. These are the following.

\vskip 2mm
\noindent {\it The ordinary case.}  This is the case that $a_g(p)$ is a unit in the ring of integers of $F_\lambda$ (equivalently, $\iota_p(a_g(p))$ is a $p$-adic unit), in which case we say that $g$ is ordinary.
If $g$ is ordinary, then there is exactly one root $\alpha_p$ of the Hecke polynomial $x^2-a_g(p)x+p$ that is a $p$-adic unit, that is, a unit in the ring of integers of $F_\lambda$.

In this case there is an exact sequence 
$$
0 \ra V^{+} \ra V \ra V^{-} \ra 0
$$
of $F_{\lambda}[G_{\BQ_{p}}]$-modules with $V^{\pm}$ a one-dimensional $F_{\lambda}$-vector space and 
both $V^+$ and  $V^{-}(1)$ unramified $F_{\lambda}[G_{\BQ_{p}}]$-modules. The action of $\Frob_p$ on 
$V^{-}(1)$ is just multiplication by $\alpha_p$. 

 Let $T^+ = T \cap V^+$ and $T^- = T/T^+$. Then
$T^{\pm}$ is a free $\cO_\lambda$-module of rank one, and there is also an exact sequence
$$
0 \ra T^{+} \ra T \ra T^{-} \ra 0
$$
of $\cO_{\lambda}[G_{\BQ_{p}}]$-modules.

\vskip 2mm
\noindent {\it The supersingular case.} This is the case\footnote{If $g$ is associated with an elliptic curve, then $p\nmid 2N$ being a prime of supersingular reduction need not imply $a_g(p)=0$
(but only when $p=3$). However, we adopt the terminology for convenience.} 
where $p\nmid N$ and $a_g(p)=0$. The roots of the Hecke polynomial at $p$ are just $\pm\sqrt{p}$. In this case we say that $g$ is supersingular. 
If $g$ is supersingular, then the residual representation $\ov{T}$ is irreducible as a $k_\lambda[G_{\BQ_p}]$-module (cf.~\cite[Thm.~2.6]{E}). In particular, (irr$_\BQ$) always holds in this case.

\subsubsection{Convention}\label{convention}
\noindent To distinguish objects associated with a specific newform $g$ (if necessary), we will add a subscript `$g$' in the notation (if not already present). 
This convention will also hold for all notation introduced subsequently.

\subsection{Newforms: $L$-values}\label{Newforms-LV}
Let $g\in S_2(\Gamma_0(N)$ be a newform as in \S\ref{Newforms}.

\subsubsection{$L$-functions} Let $L(s,g)$ be the usual $L$-function associated with $g$. For $\Re(s)>\frac{3}{2}$ this is just the (absolutely convergent) Dirichlet
series 
$$
L(s,g) = \sum_{n=1}^\infty a_g(n)n^{-s}.
$$
If $\chi:(\BZ/m\BZ)^\times\rightarrow \BC^\times$ is any Dirichlet character modulo $m$, let $L(s,g,\chi)$ be the $\chi$-twisted $L$-series
$$
L(s,g,\chi) = \sum_{n=1 \atop (n,m)=1}^\infty a_g(n)\chi(n) n^{-s}.
$$
If $(N,m)=1$, then $L(s,g,\chi)$ is the $L$-function of a weight $2$ newform of level $\Gamma_1(Nm^2)$ and Nebentypus $\chi^2$, denoted $g\otimes\chi$. In particular,
$L(s,g\otimes\chi) = L(s,g,\chi)$ in this case.   

\subsubsection{$L$-values: algebraicity }
Let $\chi$ be a Dirichlet character and $F_{\chi}$ the extension of the Hecke field $F$ obtained by adjoining the values of $\chi$.

\begin{thm}\label{L-Alg}
We have 
$$
\mathfrak{g}(\ov{\chi})\cdot \frac{L(1,g,\chi)}{\Omega^{\sgn(\chi)}} \in \cO_{F_{\chi},(\lambda_{\chi})}
$$
for $\lambda_{\chi}$ any prime of $F_{\chi}$ over $\lambda$.
Here $\ov{\chi}=\tau\circ \chi$ is the complex conjugate of $\chi$, $\mathfrak{g}(\cdot)$ is the usual Gauss sum, and 
$\sgn(\chi)\in \{\pm\}$ is the sign of $\chi (-1)$. 
\end{thm}
\noindent With $\cO_{F_{\chi},(\lambda_{\chi})}$ replaced by $F_\chi$, this is due to Shimura \cite{Sh}.
The result stated here is a straight-forward consequence of the definition of the optimal periods 
(cf.~\cite[\S1.6]{Va} and \cite[\S2]{Wu}).

\subsubsection{$L$-values: non-vanishing}
Let $\ell$ be a prime ($\ell=p$ is allowed). Let $\mathfrak{X}_{\BQ,\ell}^{cyc}$ be the set of Dirichlet characters with $\ell$-power conductor and $\ell$-power order. 
We recall the following non-vanishing result for the central $L$-values in the vertical family arising from $\mathfrak{X}_{\BQ,\ell}^{cyc}$, which is due to Rohrlich \cite{Ro}.

\begin{thm}\label{NVIw}
Let 
$\ell$ be a prime.
Then, 
$$
L(1,g, \chi) \neq 0
$$
for all but finitely many $\chi \in \mathfrak{X}_{\BQ,\ell}^{cyc}$.
\end{thm}

We also have the following mod $p$ non-vanishing result, essentially
due to  Stevens \cite[Thm. 2.1]{Ste} (cf. Vatsal \cite[Rem. 1.12]{Va}).

\begin{thm}\label{NVHz} 
Suppose that (irr$_\BQ$) holds. Let $M$ be a positive integer. Let $\Sigma$ be a finite set of primes not
dividing $pNM$. Then, for a given choice of sign $\eps=\pm$
there exists a Dirichlet character $\chi$ with $\sign\chi= \eps$ and such that $(pNMM_\Sigma, \cond(\chi))=1$, where $M_\Sigma = \prod_{\ell\in\Sigma}\ell$,  and 
$$
v_{p}\bigg{(}\mathfrak{g}(\ov{\chi})\cdot \frac{L^\Sigma(1,g, \chi)}{\Omega^{\sgn(\chi)}}\bigg{)} = 0.
$$
\end{thm}
\noindent 
Here $L^\Sigma(s,g,\chi)$ is the incomplete $L$-series obtained by omitting the Euler factors at the primes $\ell\in\Sigma$,
and $v_p$ is the normalized valuation on $\ov{\BQ}_p$ (so $v_p(p)= 1$).

To deduce Theorem \ref{NVHz} from the results in \cite{Va} it is enough to show that the canonical periods (in the sense of {\em op.~cit.}) for the newform $g$ are the same (up to units in $\cO^\times$) as 
the canonical periods for the eigenform 
$$
g_\Sigma = \sum_{n=1\atop (n,M_\Sigma)=1}^\infty a_g(n) q^n \in S_2(\Gamma_0(NM_\Sigma^2)).
$$
This identification of periods is essentially \cite[Thm.~4.2]{Di} (more precisely, the proof of this theorem). 
In the case that $g$ is ordinary this identification of periods is also proved in \cite[\S4.1]{HV} (see also
\cite[Thm.~3.6.2]{EPW} and its proof in \S3.8 of \cite{EPW}).

\subsubsection{$\GL_2$-type abelian varieties}\label{newforms-AV}
For the newform $g$ as above, let $A_{g}$ be an associated $\GL_2$-type abelian variety over $\BQ$ with dimension $[F:\BQ]$ such that 
$\CO_F \subset \End_{\BQ}(A_{g/\BQ})$ and 
$$
L(s,A_{g}) = \prod_{\sigma: F \hookrightarrow \BC} L(s,g^{\sigma}). 
$$
Here $g^{\sigma}$ denotes the $\sigma$-conjugate.

\subsection{Newforms: $p$-adic Hodge theory}\label{Newforms-pH}
Let $g \in S_{2}(\Gamma_{0}(N))$ be a newform as in \S\ref{Newforms}.
Recall  that $G_{\BQ_p}$ has been identified with a subgroup of $G_\BQ$ via the embedding $\iota_p$. In this way, we view $G_{\BQ_p}$
as acting on $V$ and $T$ in the following constructions. 

\subsubsection{$D_{dR}$ , $D_{cris}$, and spaces of modular forms}\label{Dcris}
The $G_{\BQ_{p}}$-representation $V$ is de Rham (cf.~\cite[\S9.4]{K}). 
Let $D_{dR}(V(n))=(V(n)\otimes_{\BQ_p} B_{dR})^{G_{\BQ_p}}$ be the Dieudonn\'e module of the Tate-twist $V(n)$. 
Then $D_{dR}(V(n))$ is a two-dimensional $F_\lambda$-space with an
exhaustive and decreasing filtration $(D^{i}_{dR}(V(n)))_{i \in \BZ}$. 
The \'etale-de Rham comparison isomorphism yields a canonical identification (cf. \cite[\S11.2]{K} and \cite[\S1.7.1]{FK})
\begin{equation}\label{CmpdR-V}
\bigg{(}H^1_{dR}(X_1(N)/\BQ)\otimes F)/\mathfrak{p} (H^1_{dR}(X_1(N)/\BQ)\otimes F)\bigg{)}\otimes_F F_\lambda \simeq D_{dR}(V)
\end{equation}
respecting filtrations. In particular, the image of $S_F\otimes_F F_\lambda$ in the right-hand side is identified with $D^1_{dR}(V)$:
\begin{equation}\label{CmpdR}
S_F\otimes_F{F_\lambda}\simeq D^1_{dR}(V).
\end{equation}

Since $p \nmid N$, the $G_{\BQ_{p}}$-representation $V(n)$ is also crystalline. 
Let $D_{cris}(V(n))= (V(n)\otimes_{\BQ_p}B_{cris})^{G_{\BQ_p}}$ be the crystalline Dieudonn\'e module, which in this case is a two-dimensional $F_{\lambda}$-space
equipped with the crystalline Frobenius automorphism $\varphi_{cris}$. There are canonical identifications
\begin{equation}\label{CmpCris}
D_{dR}(V(n)) \simeq D_{cris}(V(n))
\end{equation}
of $F_\lambda$-spaces. Note that the left-hand side is equipped with an $F_\lambda$-filtration, the Hodge filtration, which induces a filtration on
the right-hand side:  $0\subsetneq D_{crys}^{1-n}(V(n)) \subsetneq D_{crys}^{-n}(V(n))$. 

Let $T_0 = \mathrm{im}\{H^1(X_1(N),\CO_\lambda) \rightarrow V\}$. This is a sublattice of $T$ and equal to $T$ if (irr$_\BQ$) holds. 
Since $X_1(N)$ is a curve, $p\nmid N$, and $p>2$, the image 
of $H^1(X_1(N)_{/\BZ_{(p)}},\Omega^\bullet_{X_1(N)})\otimes \cO_\lambda$ in
the left-hand side of \eqref{CmpdR-V} is identified via \eqref{CmpCris} with the strongly divisible $\cO_\lambda$-lattice 
$D_\lambda(T_0)\subset D_{cris}(V)$ that corresponds to $T_0$ (cf.~\cite[\S14.22]{K}) and is contained
in the strongly divisible lattice $D_\lambda(T)$.  In particular, $D_\lambda^1(T_0) = D_\lambda(T_0)\cap D^1_{cris}(V)$ is
identified with $S_{\cO}\otimes_{\cO}\cO_\lambda\subset S_F\otimes_F F_\lambda$ and contained in 
$D_\lambda^1(T) = D_\lambda(T)\cap D^1_{cris}(V)$. In particular, 
\begin{equation}\label{CmpdR-Int}
\text{$S_{\cO}\otimes_{\cO}\cO_\lambda \subset D^1_\lambda(T)$, with equality holding if (irr$_\BQ$) holds.}
\end{equation}

\begin{remark}\label{TateTwist-rmk}
 The Dieudonn\'e module $D_{cris}(\BQ_p(1))\simeq D_{dR}(\BQ_p(1))$ can be naturally identified with $\BQ_p$ so that the strongly divisible
$\BZ_p$-lattice in $D_{cris}(\BQ_p(1))$ 
corresponding to the $\BZ_p$-lattice $\BZ_p(1)\subset \BQ_p(1)$ is identified with $\BZ_p$. From this it follows that there are natural 
identifications $D_{\Box}(V)\simeq D_{\Box}(V(n)) = D_{\Box}(V)\otimes_{\BQ_p}D_{\Box}(\BQ_p(1))^{\otimes n}$ of $F_\lambda$-spaces such that 
$D_{\Box}^1(V)$ is identified with $D_{\Box}^{1-n}(V(n))$, $\Box = cris$ or $dR$.  
Under these identifications the strongly divisible $\cO_\lambda$-lattice associated with $T$ is identified
with that associated with $T(n)$.  
\end{remark}

\begin{remark}\label{D(T)-rmk}
If $g$ is ordinary, then the notion of $\omega\in S_F$ being good defined in \S\ref{ModularForms} agrees with that
in \cite[\S17.5]{K} for the lattice $T_0$. To see this, note that $\omega$ is good in the former sense if and only if $\omega$ is identified with an $\cO_\lambda$-generator of $D_\lambda^1(T_0) = D_\lambda(T_0) \cap D^1_{cris}(V)$, which is identified with 
$D_\lambda^0(T_0(1)) = D_\lambda(T_0(1)) \cap D^0_{cris}(V(1))$ (see Remark \ref{TateTwist-rmk}). 
Here $D_\lambda(T_0(n))$ is the strongly divisible lattice corresponding to $T_0(n)$. But when $g$ is ordinary,
the latter is naturally isomorphic to $D(T_0^-(1)) = (T_0^-(1)\otimes W(\ov{\BF}_p))^{G_{\BQ_p}}$, where
$T_0^- = T_0 \cap V^-$ and 
$W(\ov{\BF}_p)$ is the ring of Witt vectors of the algebraic closure $\ov{\BF}_p$ of $\BF_p$ (this is naturally
identified with the completion $\widehat{\BZ}_p^\ur$ of the ring of integers of the maximal unramified extension of $\BQ_p$).
Being good in the sense of \cite[\S17.5]{K} just means being identified with an $\cO_\lambda$-basis
of $D(T_0^-(1))$. In particular, if (irr$_\BQ$) holds, then $\omega$ being good in the sense of  \S\ref{ModularForms}
is the same as being good for the lattice $T$ in the sense of  \cite[\S17.5]{K}.
\end{remark}

\subsubsection{Coleman maps}\label{Coleman}
Let $\alpha$ be a root of $x^2-a_g(p)x+p$ such that  
\begin{equation}\label{non-crit}
\ord_{p}(\alpha) < 1.
\end{equation}
Let 
\begin{equation}\label{CanPai}
[\cdot, \cdot ] : D_{dR}(V(2)) \times (F_{\lambda}(\alpha) \otimes_{F_{\lambda}}  D_{cris}(V)) \ra F_{\lambda}(\alpha)\otimes_{F_\lambda}D_{cris}(F_\lambda(1)) = F_\lambda(\alpha)
\end{equation}
be the canonical pairing coming from \eqref{CmpCris} together with the canonical pairing $V(2)\times V \ra F_\lambda(1)$. 
(Here we have used that there is a canonical identification $V(2) \simeq V^*(1)$, arising from Poincar\'e duality as well as an identification $D_{cris}(F_\lambda(1)) = F_\lambda$
as in Remark \ref{TateTwist-rmk}.)  For $0 \neq \omega \in S_{F}$, let $\eta_{\omega} \in F_{\lambda}(\alpha) \otimes_{F_{\lambda}}  D_{cris}(V)$ be the unique element such that  
(i) $\varphi_{cris}(\eta_{\omega})=\alpha \cdot \eta_{\omega}$, and  (ii) $[\omega, \eta_{\omega}] = 1$
(cf.~\cite[Thm. 16.6]{K}). Here we view $\omega \in D_{dR}^1(V) \simeq D_{dR}^{-1}(V(2))$ via \eqref{CmpdR}
(see also Remark \ref{TateTwist-rmk}).

For $h\geq 1$, let
$$
\mathfrak{K}_{h,F_{\lambda}(\alpha)}
= \bigg{\{} \sum_{n \geq 0} c_{n}T^{n} \in F_{\lambda}(\alpha)[\![T]\!] \bigg{|} \lim_{n \ra \infty} |c_{n}|_{p}  \cdot n^{-h} =0\bigg{\}}
\ \ \text{and} \ \ \mathfrak{K}_{\infty,F_{\lambda}(\alpha)} = \cup_{h\geq 1} \mathfrak{K}_{h,F_{\lambda}(\alpha)}.
$$
Under the identification of $\Lambda$ with $\BZ_p[\![T]\!]$ as in \S\ref{Lambda},
$\mathfrak{K}_{h,F_{\lambda}(\alpha)}$ is a $\Lambda$-module.
Let
$$
Col_{\eta_{\omega},\CG} : H^{1}(\BQ_{p}, V(1) \otimes_{\BZ_{p}} \Lambda_{\CG}) \ra 
\mathscr{H}_{1,F_\lambda(\alpha)} = \mathfrak{K}_{1,F_{\lambda}(\alpha)}\otimes_{\BZ_p}\BZ_p[\Delta]
$$
be the Coleman map, which is the composition
$$
H^{1}(\BQ_{p}, V(1) \otimes_{\BZ_{p}} \Lambda_\CG) \isoarrow H^{1}(\BQ_{p}, V(2) \otimes_{\BZ_{p}} \Lambda_\CG) \stackrel{\mathfrak{L}_{\eta_{\omega}}}{\rightarrow}
\mathscr{H}_{1,F_{\lambda}(\alpha)},
$$
for $\mathfrak{L}_{\eta_{\omega}}$ as in \cite[Thm. 16.4]{K} (due to Perrin-Riou).
Here $G_{\BQ_{p}} \subset G_{\BQ}$ acts on the Iwasawa algebra $\Lambda_\CG$ via the inverse of the canonical character $\Psi_\CG$ (see \eqref{cycuniv})
and the first map in the composition is the isomorphism induced by the $G_{\BQ_p}$-isomorphism $V(1) \otimes_{\BZ_{p}} \Lambda_\CG \isoarrow V(2)\otimes_{\BZ_{p}} \Lambda_\CG$,
given by 
 $v\otimes\gamma \mapsto (v\otimes(\zeta_{p^{n}}))\otimes\epsilon^{-1}(\gamma)\gamma$ for 
 $\gamma\in\CG$.
 
Here
$(\zeta_{p^{n}})_n \in \BQ_p(1))$ is a fixed compatible system of $p^n$th roots of unity (this same choice is used in the definition of $\mathfrak{L}_{\eta_{\omega}}$).
The map $\mathfrak{L}_{\eta_{\omega}}$ is a $\Lambda_\CG$-morphism, but
$Col_{\eta_{\omega},\CG}$ is not due to the presence of the Tate twist in its definition. Instead,
$Col_{\eta_{\omega}, \CG}$ satisfies
$$
Col_{\eta_{\omega},\CG}(\lambda\cdot ) = \epsilon^{-1}(\lambda)Col_{\eta_{\omega},\CG}(\cdot), \ \ 
\lambda\in\Lambda_\CG.
$$

Let $\iota_\epsilon: \Lambda \isoarrow \Lambda,  \ \ \gamma\mapsto \epsilon(\gamma)\gamma$.
Restricting $Col_{{\eta_\omega},\CG}$ to the direct summand 
corresponding to $\Lambda_{\CG}^{(0)}$ (which is canonically identified
with $\Lambda$) and then composing with the isomorphism obtained by 
tensoring with $\otimes_{\Lambda,\iota_\epsilon}\Lambda$ yields a 
$\Lambda$-module Coleman map
\begin{equation}\label{Col-Lambda}
Col_{\eta_{\omega}} : H^{1}(\BQ_{p}, V(1) \otimes_{\BZ_{p}} \Lambda) 
\ra \mathfrak{K}_{1,F_{\lambda}(\alpha)}\otimes \eps_1 \stackrel{id\otimes 1}{\isoarrow}
(\mathfrak{K}_{1,F_{\lambda}(\alpha)}\otimes \eps_1)\otimes_{\Lambda,\iota_\epsilon}\Lambda
\simeq \mathfrak{K}_{1,F_{\lambda}(\alpha)}\otimes_{\Lambda,\iota_\epsilon}\Lambda.
\end{equation}

Recall that in \eqref{Col-Lambda}, $G_{\BQ_p}$ is acting on $\Lambda$ via the inverse
of the canonical character $\Psi$.

\vskip 2mm
\noindent{\it{The ordinary case.}}
If $g$ is ordinary, then we define
\begin{equation}\label{ordCyc}
H^{1}_{\ord}(\BQ_{p}, T(1) \otimes_{\BZ_{p}} \Lambda)
= \Im (H^{1}(\BQ_{p}, T^{+}(1) \otimes_{\BZ_{p}} \Lambda) \ra H^{1}(\BQ_{p}, T(1) \otimes_{\BZ_{p}} \Lambda)).
\end{equation}
We also put
\begin{equation}\label{over-ord}
H^1_{/\ord}(\BQ_p,T(1) \otimes_{\BZ_{p}} \Lambda) = \frac{H^{1}(\BQ_{p}, T(1) \otimes_{\BZ_{p}} \Lambda)}
{H^{1}_{\ord}(\BQ_{p}, T(1)\otimes_{\BZ_{p}} \Lambda)}.
\end{equation} 

For $0\neq \omega\in S_F$, the Coleman map $Col_{\eta_{\omega}}$ induces an injection
\begin{equation}\label{ordCol-Inj}
Col_{\eta_{\omega}}: H^{1}_{/\ord}(\BQ_{p}, T(1) \otimes_{\BZ_{p}} \Lambda) \hookrightarrow F_\lambda\otimes_{\BZ_p}\Lambda \subset \mathfrak{K}_{1,F_{\lambda}(\alpha)}\otimes_{\Lambda,\iota_\epsilon}\Lambda,
\end{equation}
and if (irr$_\BQ$) holds and $\omega \in S_{F}$ is good then the image is contained in $\Lambda_{\cO_\lambda}$ with finite index.  This follows from the corresponding properties for $Col_{\eta_\omega,\CG}$ from
\cite[Prop. 17.11]{K}.

\vskip 2mm
\noindent{\it The supersingular case.}  
In this case, for $\circ = \pm$ let 
$$
H^{1}_{\circ}(\BQ_p, T(1) \otimes_{\BZ_{p}} \Lambda)\subset H^1(\BQ_p,T(1)\otimes_{\BZ_p}\Lambda)
$$
be the {\em signed} submodule defined as in \cite[\S4.4]{L} (see also \cite[\S7]{Ko1}).  
As in \cite{K}, the definitions and constructions in \cite{L} and \cite{Ko1} are made with 
$\Lambda$ replaced by $\Lambda_{\CG}$. The module above is then the summand corresponding
to the factor $\Lambda_{\CG}^{(0)} = \Lambda$ of $\Lambda_\CG$.
We also put
\begin{equation}\label{over-pm}
H^{1}_{/\circ}(\BQ_{p}, T(1) \otimes_{\BZ_{p}} \Lambda)=
\frac{H^{1}(\BQ_{p}, T(1) \otimes_{\BZ_{p}} \Lambda)}
{H^{1}_{\circ}(\BQ_{p}, T(1) \otimes_{\BZ_{p}} \Lambda)}.
\end{equation}

For $0\neq \omega\in S_F$ there is a signed 
Coleman map
$$
Col^{\circ}_{\omega}:  H^{1}(\BQ_{p}, T(1) \otimes_{\BZ_{p}} \Lambda)\rightarrow F_\lambda\otimes_{\BZ_p}\Lambda,
$$
defined in \cite[\S3.4.2]{L}.  This is a $Frac(\mathfrak{K}_{\infty,F_{\lambda}(\alpha)})$-linear
combination of the Coleman maps $Col_{\eta_{\omega}}$ associated with the two roots $\alpha = \pm\sqrt{-p}$ of the Hecke polynomial $x^2+p$ at $p$.
By \cite[Cor.~4.11]{L} 
the kernel of $Col^{\circ}_{\omega}$ is $H^1_\circ(\BQ_p,T(1) \otimes_{\BZ_{p}}\Lambda)$, and there is an induced injection
\begin{equation}\label{pmCol-Inj}
 Col^{\circ}_{\omega}:  H^{1}_{/\circ}(\BQ_{p}, T(1) \otimes_{\BZ_{p}} \Lambda)\hookrightarrow F_\lambda\otimes_{\BZ_p}\Lambda
\end{equation}
with image in $\lambda^{s_\circ}\Lambda_{\cO_\lambda}$ for some $s_\circ\in \BZ$.
If $\omega$ is good, then the image is contained in $\Lambda_{\cO_\lambda}$ with finite index
(cf.~\cite[Thm.~2.14 and Rmk.~2.15]{BL}). As in the ordinary case, these results just follow from the corresponding
results with $\Lambda$ replaced by $\Lambda_{\CG}$.  We emphasize that $Col^\circ_\omega$ is a $\Lambda$-module map.

\subsection{Newforms: Quadratic twists}  Let $g\in S_2(\Gamma_0(N)$ be a newform as in \S\ref{Newforms}.
In this subsection we compare many of the previous definitions and constructions for $g$ and certain of its quadratic twists.
Here in particular the conventions of \S\ref{convention} will be in use.

\subsubsection{The Imaginary quadratic field $L$}\label{IQF}
Throughout, $L \subset \ov{\BQ}$ will denote an imaginary quadratic field of discriminant $-D_{L} < 0$.
 
Write $N=N_LN_L'$ where $N_L$ (resp~$N_L'$) is a product of primes that divide (resp.~do not divide) $D_L$.
We will always suppose that
\begin{equation}\label{ord}
\text{$p$ splits in $L$: $(p)=v\ov{v}$ with $v$ determined via $\iota_p$}
\end{equation}
and that
\begin{equation}\label{sqfree}
\text{$N_L$ is squarefree.}
\end{equation}
While neither condition is always needed, each is an essential feature of some of our later arguments. 

Let $\chi_L:\BZ/D_L\BZ\rightarrow \{\pm 1\}$ be the primitive quadratic character of conductor $D_L$ associate with $L$. We also write
$\chi_L$ for the quadratic character of $G_\BQ$ associated with $L$. These characters are identified by the reciprocity map of class field theory.

The chosen embedding $\iota_p$ determines an isomorphism $L_v\isoarrow \BQ_p$, using which we may take $\ov{\BQ}_p$ to be the chosen separable algebraic closure of $L_v$
and $\iota_p$ as the chosen embedding of $\ov{\BQ}$ into $\ov{L}_v$.  In particular, $G_{L_v}\subset G_L$ is identified with $G_{\BQ_p}\subset G_\BQ$.
Similarly, the composite $\iota_p\circ\tau:\ov{\BQ}\hookrightarrow \ov{\BQ}_p$ determines an isomorphism $L_{\bar v}\isoarrow\BQ_p$, using which we identify $G_{L_{\bar v}}\subset G_L$ with 
$\tau G_{\BQ_p} \tau^{-1}\subset G_\BQ$.

\subsubsection{The quadratic twist $g'=g\otimes\chi_L$}
Let $g'$ be the twist of $g$ by $\chi_L$, which has $q$-expansion
$$
g' (q) = \sum_{n \geq 1} \chi_{L}(n) \cdot a_{f}(n)q^{n}.
$$
Since $N_L$ is squarefree,
$g'=g\otimes\chi_L$ is a newform of level $\mathrm{lcm}(N,D_L^2) = N_L'D_L^2$, weight $2$, and trivial Nebentypus.  
The Hecke field of $g'$ is the same as that of $g$.

\subsubsection{Quadratic twists I: periods}\label{QuadraticTwist-I}
We record the following relation between the periods of $g$ and  $g'$.

\begin{lem}\label{PerTw}\hfill
\begin{itemize}
\item[(i)] We have
$$
\mathfrak{g}(\chi_{L}) \cdot \Omega_{g'}^{\pm} \sim_{F^{\times}} \Omega_{g}^{\mp}.
$$
\item[(ii)] Suppose that (irr$_\BQ$) holds. 
Then
$$
\mathfrak{g}(\chi_{L}) \cdot \Omega_{g'}^{\pm} \sim_{\cO^{\times}} \Omega_{g}^{\mp}.
$$
\end{itemize}
\end{lem}
\noindent Here, as before, `$\sim_{\Box}$' denotes equality up to an element in $\Box$.

\begin{proof}  
Let $\ell \nmid ND_{L}$ be a prime. By Theorem \ref{NVIw} there exists $\chi \in \mathfrak{X}_{\BQ,\ell}^{cyc}$ such that
$$
L(1, g',\chi) \neq 0.
$$
As $L(1,g',\chi) = L(1,g,\chi_L\chi)$, it then follows from Theorem \ref{L-Alg}(i) applied to both 
$L(1,g',\chi)$ and $L(1,g,\chi_L\chi)$ together with the factorisation 
$\mathfrak{g}(\chi_L\ov{\chi}) = \mathfrak{g}(\chi_L)\mathfrak{g}(\ov{\chi})$ of Gauss sums that 
$$
\mathfrak{g}(\chi_{L}) \cdot \Omega_{g'}^{+} \sim_{F_{\chi}^{\times}} \Omega_{g}^{-}.
$$
 
Let $\ell' \nmid ND_L\ell$ be another prime.
The same argument as above shows that there exists $\chi' \in \mathfrak{X}_{\BQ,\ell'}^{cyc}$ such that
$$
\mathfrak{g}(\chi_{L}) \cdot \Omega_{g'}^{+} \sim_{F_{\chi'}^{\times}} \Omega_{g}^{-}.
$$
As $F_{\chi} \cap F_{\chi'} = F$,
it follows that $\mathfrak{g}(\chi_L) \cdot \Omega_{g'}^+ \sim_{F^\times} \Omega_{g}^-$. 

That  $\mathfrak{g}(\chi_L)\cdot \Omega_{g'}^- \sim_{F^\times} \Omega_{g}^+$ can be seen by
reversing the roles of $g$ and $g'$ in the preceding argument and noting that 
$L^{\{q\mid D_L\}}(1,g,\chi)= L(1,g',\chi_L\chi)$, 
so $L(1,g,\chi_\ell)\sim_{F_{\chi}^\times} L(1,g',\chi_L\chi_\ell)$, and that $\mathfrak{g}(\chi_L)^2 = \pm D_L$.

Suppose now that (irr$_\BQ$) holds. 
Let $\varepsilon$ be a sign.
By Theorem \ref{NVHz}, there exists a Dirichlet character $\chi$ of sign $-\varepsilon$ such that $(pND_{L}, \cond(\chi))=1$ and 
$$
v_{p}\bigg{(}\mathfrak{g}(\ov{\chi})\cdot \frac{L(1,g',\chi)}{\Omega_{g'}^{-\varepsilon}}\bigg{)} = 0.
$$
Since by Theorem \ref{L-Alg}(ii),
$$
\mathfrak{g}(\chi_L\ov{\chi})\cdot \frac{L(1,g,\chi_L\chi)}{\Omega_{g}^{\varepsilon}} \in \cO_{F_{\chi},(\lambda_{\chi})},
$$
and since $L(1,g',\chi) = L(1,g,\chi_L\chi)$ and $\mathfrak{g}(\chi_L \ov{\chi}) = \mathfrak{g}(\chi_L)\mathfrak{g}(\ov{\chi})$,
it follows that
$$
\mathfrak{g}(\chi_{L}) \cdot \frac{\Omega_{g'}^{-\varepsilon}} {\Omega^{\varepsilon}_g} \in \cO_{F_{\chi},(\lambda_{\chi})}.
$$
Part (ii) then follows upon reversing the roles of $g$ and $g'$ and $\varepsilon$ and $-\varepsilon$ in this argument and using that $\mathfrak{g}(\chi_L)^2 = \pm D_L \in \cO^\times$.
\end{proof}

\begin{remark}
For an alternative approach to proving the preceding lemma, see \cite[Lem. 9.6]{SZ}.
\end{remark}

\begin{remark}\label{EC-periodtwist-rmk}
If $g$ corresponds to the isogeny class of an elliptic curve $E_\bullet$ as in Remark \ref{EC-optperiod-rmk}, then 
the conclusion of part (ii) of Lemma \ref{PerTw} still holds even when (irr$_\BQ$) does not. 
This is a consequence of the main result of \cite{VivekPal}. 
\end{remark}

\subsubsection{Quadratic twists II: identifications and rigidifications}\label{QuadraticTwist-II}

In view of the characterising properties of the $\lambda$-adic Galois representations $V_g$
and $V_{g'}$,
there is an isomorphism
\begin{equation}\label{GalTw}
V_{g} \otimes \chi_{L} \simeq V_{g'}
\end{equation}
of $F_{\lambda}[G_{\BQ}]$-modules.

If (irr$_{\BQ}$) holds, then any two $G_\BQ$-stable $\cO_\lambda$ lattices in $V_g$ (or $V_{g'}$) are scalar multiples of one another. In particular,
there is an isomorphism
\begin{equation}\label{GalTw-Int}
T_{g} \otimes \chi_{L} \simeq T_{g'}
\end{equation}
of $\cO_{\lambda}[G_{\BQ}]$-modules.

For subsequent arguments it will be convenient to choose an isomorphism \eqref{GalTw} or even an isomorphism \eqref{GalTw-Int} (which then induces an isomorphism \eqref{GalTw}).
\vskip 2mm
\noindent{\it{Rigidifications.}}
The choice of an isomorphism \eqref{GalTw} can be rigidified as in the following lemma.

\begin{lem}\label{RigIso}\hfill
\begin{itemize}
\item[(a)] Let
\begin{equation}\label{RigMod}
S_{F,g} \isoarrow S_{F,g'}
\end{equation}
be a fixed isomorphism of one-dimensional $F$-vector spaces. There exists a unique isomorphism 
\begin{equation}\label{RigGal}
V_{g} \otimes \chi_{L} \isoarrow V_{g'}
\end{equation}
of $F_{\lambda}[G_{\BQ}]$-modules such that the diagram 
$$\begin{tikzcd}[row sep=2.5em]
& 
S_{F,g} \otimes_{F} F_{\lambda}     \arrow[r,"\sim"] \arrow[d,"\wr"]&
 D^1_{dR}(V_{g})        \arrow[d,"\wr"]\\
&
S_{F,g'} \otimes_{F} F_{\lambda}  \arrow[r,"\sim"]&
D^1_{dR}(V_{g'}) \\
\end{tikzcd}$$
commutes. 
Here the horizontal isomorphisms are induced via \eqref{CmpdR}, and the first (resp.~second) vertical isomorphism is induced via \eqref{RigMod} (resp. \eqref{RigGal}).
\item[(b)] Suppose that (irr$_{\BQ}$) holds.  Let
\begin{equation}\label{RigMod-Int}
S_{\cO,g} \isoarrow S_{\cO,g'}
\end{equation}
be a fixed isomorphism of rank one $\cO$-modules. There exists a unique isomorphism 
\begin{equation}\label{RigGal-Int}
T_{g} \otimes \chi_{L} \isoarrow T_{g'}
\end{equation}
of $\cO_{\lambda}[G_{\BQ}]$-modules such that the diagram 
$$\begin{tikzcd}[row sep=2.5em]
& 
S_{\cO,g} \otimes_{\cO} \cO_{\lambda}    \arrow[r,"\sim"] \arrow[d,"\wr"]&
D_\lambda^1(T_g)  \arrow[d,"\wr"]\\
&
S_{\cO,g'} \otimes_{\cO} \cO_{\lambda} \arrow[r,"\sim"]&
D^1_\lambda(T_{g'}) \\
\end{tikzcd}$$
commutes.
Here the horizontal isomorphisms are induced via \eqref{CmpdR-Int}, and the first (resp.~second) vertical isomorphism is induced via \eqref{RigMod-Int} (resp. \eqref{RigGal-Int}).
\end{itemize}
\end{lem}

\begin{proof}
Let $h:V_{g} \otimes \chi_{L} \isoarrow V_{g'}$ be any isomorphism of $F_{\lambda}[G_{\BQ}]$-modules. 
Then $h$ induces 
$D^{1}_{dR}(V_{g}) \simeq D^{1}_{dR}(V_{g'})$
and hence an isomorphism
$S_{F,g} \otimes_{F} F_{\lambda} \simeq S_{F,g'} \otimes_{F} F_{\lambda}$
by \eqref{CmpdR}. As these are all one-dimensional $F_\lambda$-spaces, this last isomorphism must be
$a$ times the one induced via \eqref{RigMod} for some $a \in F_{\lambda}^{\times}$. 
Then $a^{-1}\cdot h$ is the desired isomorphism. The uniqueness follows from $V_{g}$ being an irreducible $F_{\lambda}[G_{\BQ}]$-module. 
This proves part (a). 

In light  of \eqref{CmpdR-Int} and and the fact that (irr$_{\BQ}$) implies that all $G_\BQ$-stable lattices in $V_{g}$ are scalar multiples of one another, 
essentially the same argument proves part (b).
\end{proof}

\begin{remark}
Lemma \ref{RigIso} is a variant of \cite[Lem. 15.11(1)]{K}.
\end{remark}

\vskip 2mm
\noindent{\it Optimal rigidifications.}
Let $c_g$ and $c_{g'}$ be fixed generators of the congruence ideals of $g$ and $g'$, respectively.
In the rest of this paper, we choose the isomorphism \eqref{RigMod} so that 
\begin{equation}\label{RigMod-Opt}
S_F\isoarrow S_F', \ \ \frac{\omega_{g}}{c_{g}} \mapsto \frac{\omega_{g'}}{c_{g'}}.
\end{equation}
By Lemma \ref{RigIso} this choice determines an isomorphism 
\begin{equation}\label{RigGal-Opt}
Tw:V_g\otimes\chi_L\isoarrow V_{g'}.
\end{equation}
If (irr$_{\BQ}$) also holds, then by Lemma \ref{GorPer} both $\frac{\omega_{g}}{c_{g}}$ and $\frac{\omega_{g'}}{c_{g'}}$
are good in the sense defined of \S\ref{ModularForms}, and so the map \eqref{RigMod-Opt} is determined by an 
an isomorphism \eqref{RigMod-Int}. In this case, the map
$Tw$ arises from an isomorphism \eqref{GalTw-Int}.

\subsubsection{Quadratic twists III: Coleman maps}\label{QuadraticTwist-III}
We assume that \eqref{RigMod} has been fixed (say as in \eqref{RigMod-Opt}).
Let $Tw:V\otimes\chi_L\isoarrow V_{g'}$ be the isomorphism as in \eqref{RigGal-Opt} (so as in Lemma \ref{RigIso}(a)).

Since $p$ splits in $L$ by hypothesis, $a_g(p) = a_{g'}(p)$ and we may choose 
$$
\alpha=\alpha_g=\alpha_{g'}
$$
to be a root of the Hecke polynomial at $p$ as in \S\ref{Coleman}.

For $0 \neq \omega \in S_{F,g}$, let $\eta_{\omega} \in F_{\lambda}(\alpha) \otimes_{F_{\lambda}}  D_{cris}(V_{F_{\lambda},g})$ be the element in \S\ref{Coleman} 
Let $\omega' \in S_{F,g'}$ be the image of $\omega$ under the fixed isomorphism \eqref{RigMod}.
Then the pair $(\alpha,\omega')$ gives rise to 
$$
\eta_{\omega}' = \eta_{\omega'} \in F_{\lambda}(\alpha) \otimes_{F_{\lambda}}  D_{cris}(V_{F_{\lambda},g'})
$$ 
as in \S\ref{Coleman}. 
The isomorphism $Tw$ induces an identification 
$D_{cris}(V_{g})\simeq D_{cris}(V_{g'})$, and $\eta_{\omega}'$ is just the image of $\eta_{\omega}$ under this identification.

The isomorphism $Tw$ also induces isomorphisms
$$
Tw_*:H^1(\BQ_p,(T_g(n)\otimes\chi_L)\otimes_{\BZ_p}\Lambda)\otimes_{\BZ_p}\BQ_p \isoarrow H^1(\BQ_p,T_{g'}(n)\otimes_{\BZ_p}\Lambda)\otimes_{\BZ_p}\BQ_p,
$$
and even without tensoring with $\BQ_p$ if $Tw:T_g\otimes\chi_L\isoarrow T_{g'}$.
The Coleman maps for $g$ and $g'$ are related via $Tw_*$ as follows.

\begin{lem}\label{ColTw} 
Let $0\neq \omega \in S_F$. 
\begin{itemize} 
\item[(a)] We have $Col_{\eta_{\omega}} = Col_{{\eta}_{\omega}'} \circ Tw_*$. 
\item[(b)] In the supersingular case, $Col^\circ_{{\omega}} = Col^\circ_{{\omega}'} \circ Tw_*$.
\end{itemize}
\end{lem}

\begin{proof}
Part (a) easily from the commutativity of the diagram in Lemma \ref{RigIso}(a) and the definition of the Coleman maps.
The key point is that $\mathfrak{L}_{\eta_{\omega}}=\mathfrak{L}_{\eta_{\omega}'}\circ{Tw_*}$, 
which follows from the functorial properties of Perrin-Riou's logarithm and the choices of $\eta_{\omega}$ and $\eta_{\omega}'$.
Part (b) follows from part (a) and the definition of the signed Coleman map $Col_{\omega}^\circ$ (resp.~$Col^\circ_{\omega'}$) 
as a $Frac(\mathfrak{K}_{1,F_{\lambda}(\alpha)})$-linear
combination of the Coleman maps $Col_{\eta_{\omega}}$ associated with the two roots of the Hecke polynomial at $p$.
\end{proof}

We have corresponding identifications of the ordinary and signed submodules.

\begin{lem}\label{CohTw}
The isomorphism $Tw_*$ induces an isomorphism
$$Tw_*:H^1_\ord(\BQ_p,T_g(1)\otimes_{\BZ_p}\Lambda)\otimes_{\BZ_p}\BQ_p \ \isoarrow H^1_\ord(\BQ_p,T_{g'}(1)\otimes_{\BZ_p}\Lambda)\otimes_{\BZ_p}\BQ_p
\ \ \text{if $g$ is ordinary,}
$$
and an isomorphism
$$Tw_*:H^1_\circ(\BQ_p,T_g(1)\otimes_{\BZ_p}\Lambda)\otimes_{\BZ_p}\BQ_p \isoarrow H^1_\circ(\BQ_p,T_{g'}(1)\otimes_{\BZ_p}\Lambda)\otimes_{\BZ_p}\BQ_p 
\ \ \text{if $g$ is supersingular.}
$$
If $Tw:T_g\otimes\chi_L\isoarrow T_{g'}$, then these are isomorphisms without tensoring with $\BQ_p$.
\end{lem}
 
\begin{proof} This is an easy consequence of the definitions of the cohomology groups and the map $Tw_*$. Alternatively, this
follows from Lemma \ref{ColTw} together with these subgroups being the kernels of the corresponding Coleman maps.
\end{proof}

\section{Beilinson--Kato elements}\label{Beilinson-Kato}
In this section, we recall the Beilinson--Kato elements of Kato \cite{K}, introduce a variant defined with respect to an auxiliary imaginary quadratic field,
and describe their connections to $p$-adic $L$-functions as well as related results about Iwasawa cohomology groups.

Let $g\in S_2(\Gamma_0(N))$ be a newform as in \S\ref{Newforms} and let $L$ be an imaginary quadratic field as in \S\ref{IQF}.
As before, let $g'=g\otimes\chi_L$ be the twist of $g$ by $\chi_L$.
We work with the rigidifications associated with the choice of isomorphism \eqref{RigMod-Opt} and in particular with the consequential
isomorphism \eqref{RigGal-Opt}.

\subsection{Some Iwasawa cohomology groups} \label{IwCoh}
For $K$ a number field with ring of integers $\cO_K$ and $M= T=T_g$, $T_g\otimes\chi_L$, or $T_{g'}$, we define

$$
H^1(\cO_K[\frac{1}{p}], M(1)\otimes_{\BZ_p}\Lambda) = H^1(G_{K,\Sigma}, M(1)\otimes_{\BZ_p}\Lambda)
$$
for any finite set $\Sigma$ of places of $K$ that contains all the primes dividing $pND_L\infty$ (just those dividing $pN$ is enough for $M=T_g$). 
Here $G_K$ acts on $\Lambda$ via the inverse $\Psi^{-1}$ of the canonical character (see \S\ref{Lambda}). 
This is a finitely-generated $\Lambda_{\cO_\lambda}$-module.
It is a torsion-free $\Lambda_{\cO_\lambda}$-module if (irr$_L$) holds.

\begin{remark}\label{Iw-rmk} 
As the notation indicates, $H^1(\cO_K[\frac{1}{p}], M(1)\otimes_{\BZ_p}\Lambda)$ is independent of the choice of $\Sigma$,
and the classes in $H^1(\cO_K[\frac{1}{p}], M(1)\otimes_{\BZ_p}\Lambda)$ are unramified at all finite places not dividing $p$
(cf.~the proof of \cite[Lem.~8.5]{K}). The key point is that $H^1(G_{K,\Sigma}, M(1)\otimes_{\BZ_p}\Lambda) = 
\varprojlim_n H^1(\Gal(K_\Sigma/K(\zeta_{p^n})),M(1))^{(0)}$, where the superscript $(0)$ denotes the 
$\Lambda_\CG$-summand on which $\Delta$ acts trivially.
\end{remark}

The natural action of $\tau$ (which restricts to the non-trivial automorphism of $L$)
induces a decomposition 
\begin{equation}\label{spl}
H^1(\cO_L[\frac{1}{p}],T(1)\otimes_{\BZ_p}\Lambda) = H^1(\cO_L[\frac{1}{p}],T(1)\otimes_{\BZ_p}\Lambda)^+\oplus
H^1(\cO_L[\frac{1}{p}],T(1)\otimes_{\BZ_p}\Lambda)^-,
\end{equation}
where $H^1(\cO_L[\frac{1}{p}],T(1)\otimes_{\BZ_p}\Lambda)^\pm$ is the $\Lambda_{\cO_\lambda}$-submodule on which
$\tau$ acts as multiplication by $\pm 1$. Then restriction to $G_L$ yields identifications 
\begin{equation}\label{spl+}
H^1(\BZ[\frac{1}{p}],T(1)\otimes_{\BZ_p}\Lambda) \isoarrow H^1(\cO_L[\frac{1}{p}],T(1)\otimes_{\BZ_p}\Lambda)^+
\end{equation}
and
\begin{equation}\label{spl-}
H^1(\BZ[\frac{1}{p}],(T(1)\otimes\chi_L)\otimes_{\BZ_p}\Lambda) \isoarrow H^1(\cO_L[\frac{1}{p}],T(1)\otimes_{\BZ_p}\Lambda)^-.
\end{equation}

Since the rigidifications fix an isomorphism $Tw:V\otimes\chi_L = V_g\otimes\chi_L \isoarrow V_{g'}$,
we view an element of $H^1(\BZ[\frac{1}{p}],T_{g'}(1)\otimes_{\BZ_p}\Lambda)\otimes_{\BZ_p}\BQ_p$ as an element of
$H^1(\BZ[\frac{1}{p}],(T(1)\otimes\chi_L)\otimes_{\BZ_p}\Lambda)\otimes_{\BZ_p}\BQ_p$ and hence as an element of 
$H^1(\cO_L[\frac{1}{p}],T(1)\otimes_{\BZ_p}\Lambda)^-\otimes_{\BZ_p}\BQ_p$. 
If (irr$_\BQ$) holds, then the same holds without tensoring with $\BQ_p$.

\subsubsection{Local conditions}
The fixed choices of decomposition groups at $w=v,\bar v$ determine isomorphisms $\iota_w:H^1(L_w,T(1)\otimes_{\BZ_p}\Lambda)\isoarrow H^1(\BQ_p,T(1)\otimes_{\BZ_p}\Lambda)$. For $w=v$ this is just the identity map, corresponding to $G_{L_v} = G_{\BQ_p}\subset G_\BQ$. For $w=\bar{v}$, $G_{L_{\bar v}} = \tau G_{\BQ_p} \tau^{-1} \subset G_\BQ$, and 
the isomorphism $\iota_{\bar v}$ is that determined by the map on cycles $c \mapsto (\sigma \mapsto \tau^{-1}c(\tau^{-1}\sigma\tau))$.  

Suppose $g$ is either ordinary or in the supersingular case. Let
$$
H^1_\Box(L_w,T(1)\otimes_{\BZ_p}\Lambda) = \iota_w^{-1}\big{(}H^1_\Box(\BQ_p,T(1)\otimes_{\BZ_p}\Lambda)\big{)}, \ \ \Box = 
\begin{cases} \ord  & \text{$g$ is ordinary} \\ \circ & \text{$g$ is supersingular.}
\end{cases}
$$
We then define global Iwasawa cohomology groups with local restrictions:
\begin{equation}\label{Cyc-global}
H^1_\Box(\BZ[\frac{1}{p}], T(1)\otimes_{\BZ_p}\Lambda) = \bigg{\{} \kappa\in H^1(\BZ[\frac{1}{p}], T(1)\otimes_{\BZ_p}\Lambda) \ : \ 
\loc_p(\kappa) \in H^1_\Box(\BQ_p, T(1)\otimes_{\BZ_p}\Lambda)\bigg{\}},
\end{equation}
and 
\begin{equation}\label{CycL-global}
H^{1}_{rel,\Box}(\cO_{L}[\frac{1}{p}], T(1) \otimes_{\BZ_{p}} \Lambda)
= \bigg{\{} \kappa \in H^{1}(\cO_{L}[\frac{1}{p}], T(1) \otimes_{\BZ_{p}} \Lambda) \ : \ 
\loc_{\bar v}(\kappa) \in H^1_{\Box}(L_{\bar v},T(1)\otimes_{\BZ_p}\Lambda)\bigg{\}},
\end{equation}
where $\Box$ has the same meaning as above.

\subsubsection{Coleman maps at $v$ and $\bar{v}$}\label{Col-L}  We transport the Coleman map to $H^1(L_w,T(1)\otimes_{\BZ_p}\Lambda)$ via $\iota_w$: For $0\neq\omega\in S_F$ we let
$$
Col_{\eta_{\omega},w} = Col_{\eta_{\omega}}\circ \iota_w.
$$
Similarly, if $g$ is supersingular we let 
$$
Col^\circ_{{\omega},w} = Col^\circ_{{\omega}}\circ \iota_w.
$$
We make analogous definitions of Coleman maps on $H^1(L_w,T_{g'}\otimes_{\BZ_p}\Lambda)$.  These are related as follows,

\begin{lem} \label{ColTw-L}
Let $0\neq \omega\in S_F$.
\begin{itemize}
\item[(a)] We have $Col_{\eta_{\omega}, v} = Col_{\eta'_{\omega},v}\circ Tw_*$ and $Col_{\eta_{\omega},\bar{v}} = - Col_{\eta'_{\omega},\bar{v}}\circ Tw_*$.
\item[(b)] In the supersingular case we have $Col^\circ_{{\omega}, v} = Col^\circ_{{\omega'},v}\circ Tw_*$ and $Col^\circ_{{\omega},\bar{v}} 
= - Col^\circ_{{\omega'},\bar{v}}\circ Tw_*$.
\end{itemize}
\end{lem}

\begin{proof} By Lemma \ref{ColTw}(a) we have 
$$
Col_{\eta_{\omega}, w} = Col_{\eta_{\omega}}\circ \iota_{w,g} = Col_{\eta'_{\omega}}\circ Tw _* \circ \iota_{w,g} = Col_{\eta'_{\omega},w}\circ \iota_{w,g'}^{-1}\circ Tw_*\circ \iota_{w,g}.
$$
Part (a) then follows from $\iota_{v,g'}^{-1}\circ Tw_*\circ \iota_{v,g} = Tw_*$ and $\iota_{{\bar v},g'}^{-1}\circ Tw_*\circ \iota_{{\bar v},g} = - Tw_*$. The change in sign arises from the twisting by
$\tau$ in the definition of the maps $\iota_{{\bar v},g}$ and $\iota_{{\bar v},g'}$: the action of $\tau$ on $V_{g'} \simeq V_g\otimes\chi_L$ is identified with $-1$ times the action of $\tau$ on $V_g$.

Part (b) follows from part (a).
\end{proof}

\subsection{Beilinson--Kato elements}\label{Beilinson--Kato-elements}
We recall Kato's zeta elements (the Beilinson--Kato elements) and define a variant over the imaginary quadratic field $L$.

\subsubsection{Over the rationals}\label{BK-rationals}
In \cite{K} Kato introduced a collection of elements 
$$
\bz_\gamma(g) \in H^{1}(\BZ[\frac{1}{p}], T(1)\otimes_{\BZ_{p}} \Lambda) \otimes_{\BZ_{p}} \BQ_{p}
$$
(the Beilinson--Kato elements) associated with the newform $g$ and elements $\gamma\in V$. Among their properties are the following: 
\begin{itemize}
\item[(i)] The map
$$
V \ra H^{1}(\BZ[\frac{1}{p}], T(1)\otimes_{\BZ_{p}} \Lambda) \otimes_{\BZ_{p}} \BQ_{p}, \ \ \gamma \mapsto \bz_{\gamma}(g),
$$
is an $F_\lambda$-linear homomorphism.
\item[(ii)] For $\gamma \in V$ with $\gamma^+,\gamma^-\neq 0$, 
$$
\bz_{\gamma}(g) \neq 0.
$$
\item[(iii)] If (irr$_{\BQ}$) holds, then
$$
\bz_{\gamma}(g) \in H^{1}(\BZ[\frac{1}{p}], T(1)\otimes_{\BZ_{p}} \Lambda) 
\ \text{for} \ \gamma \in T.
$$
\end{itemize}
A characterising property of the morphism $\gamma \mapsto {\bz}_{\gamma}(g)$ in terms of the Bloch--Kato dual exponential maps
is given in \cite[Thm. 12.5 (1)]{K}.  Other properties are recalled below (cf.~\S\S\ref{BK-underCol} and \ref{padicL-I}).

More precisely, Kato defined elements $\bz_\gamma^{(p)}\in H^1(\BZ[\frac{1}{p}], T\otimes_{\BZ_p} \Lambda_{\CG})$. 
The element we denote $\bz_\gamma(g)$ is the image of Kato's $\bz_\gamma^{(p)}$ under the composition of 
the maps induced by the isomorphism $T\otimes_{\BZ_p} \Lambda_{\CG}\isoarrow T(1)\otimes_{\BZ_p} \Lambda_{\CG}$
and the projection $T(1)\otimes_{\BZ_p} \Lambda_{\CG}\twoheadrightarrow T(1)\otimes_{\BZ_p} \Lambda_{\CG}^{(0)} = 
T(1)\otimes_{\BZ_p} \Lambda$ (for the first isomorphism see the definition of the Coleman maps in \S\ref{Coleman}).

\begin{remark}\label{EC-Katoselement}
Both \cite[Thm. 5.2 iv)]{Ko} and \cite{Wu} address further the issue of 
the integrality of the Beilinson--Kato elements ${\bf{z}}_{\gamma}(f)$ (property (iii)).
In particular, if $g$ corresponds to the isogeny class of an elliptic curve $E_\bullet$ as in Remark \ref{EC-optperiod-rmk}, then 
$\bz_\gamma(g) \in H^{1}(\BZ[\frac{1}{p}], T(1)\otimes_{\BZ_{p}} \Lambda)$ for $\gamma\in T$ even if (irr$_\BQ$) does not hold (see \cite[Thm.~13]{Wu}).
\end{remark}

\subsubsection{Under Coleman maps}\label{BK-underCol}
Let $\alpha$ be as in \eqref{non-crit}. Let $0\neq \omega \in S_F$.
For $\gamma \in V_{F}$ let
\begin{equation}\label{ColBK}
L_{\alpha,\omega,\gamma}(g)=Col_{\eta_{\omega}}(\loc_{p}({\bf{z}}_{\gamma}(g))) \in 
\mathfrak{K}_{1,F_\lambda(\alpha)}\otimes_{\Lambda,\iota_\epsilon}\Lambda.
\end{equation} 
 In the supersingular case, let
\begin{equation}\label{ColBK-pm}
L^\circ_{\omega,\gamma}(g)=Col^\circ_{{\omega}}(\loc_{p}({\bf{z}}_{\gamma}(g))).
\end{equation}
As we recall in \S\ref{padicL} below, Kato's explicit reciprocity law yields the following.

\begin{lem}\label{ColBK-lem} \hfill 
\begin{itemize}
\item[(i)] If $\gamma^\pm\neq 0$, then $L_{\alpha,\omega,\gamma}(g) \neq 0$, and furthermore, if $g$ is supersingular, then
we also have $L^\circ_{\omega,\gamma}(g) \neq 0$.
\item[(ii)] If $g$ is ordinary, then 
$L_{\alpha,\omega,\gamma}(g) \in F_\lambda\otimes_{\BZ_p}\Lambda$, and if $g$ is supersingular,
then $L^\circ_{\omega,\gamma}(g) \in F_\lambda\otimes_{\BZ_p}\Lambda.$
\item[(iii)] Suppose (irr$_{\BQ}$) holds and $\omega$ is good. Suppose also that $\gamma\in T$.
If $g$ is ordinary, then 
$L_{\alpha,\omega,\gamma}(g) \in \Lambda_{\cO_{\lambda}}$, and if $g$ is supersingular,
then $L^\circ_{\omega,\gamma}(g) \in \Lambda_{\cO_{\lambda}}$.

\end{itemize}
\end{lem}

\begin{proof}
Part (i) is a consequence of ${\bf{z}}_{\gamma}(g)$ being non-trivial (property (i) of the Beilinson--Kato elements), 
the relation of ${\bf{z}}_{\gamma}(g)$ with twisted central $L$-values $L(1,g,\chi)$ of $g$ \cite[Thm. 12.5 (2)]{K}, 
and the interpolation property of the Coleman map \cite[Thm. 16.4 (i)]{K}. (This is made more explicit in  \S\ref{padicL} below.)
Part (ii) is a consequence of the Coleman map being bounded \eqref{ordCol-Inj}, \eqref{pmCol-Inj}.
Part (iii) is a consequence of the integrality of $\bz_{\gamma}(g)$ (property (iii)) and the integrality of the Coleman maps (cf.~\S\ref{Coleman}).
\end{proof}

\subsubsection{Over $L$}\label{BK-overL}  Let $g' = g\otimes\chi_L$. We choose $\alpha_{g'} = \alpha_g = \alpha$ to satisfy \eqref{non-crit}. 
For $0\neq \omega\in S_{F,g}$, let $\omega' \in S_{F,g'}$ be the image of $\omega$ under the fixed isomorphism \eqref{RigMod-Opt},
as in \S\ref{QuadraticTwist-III}.  Let $\gamma\in V=V_g$ and $\gamma'\in V'=V_{g'}$.
To the quadruple $(\alpha,\omega,\gamma,\gamma')$ we associate elements in the Iwasawa cohomology groups as follows.

\vskip 2mm
\noindent {\it{The ordinary case.}} Suppose $g$ is ordinary. We let 
$$
\bz^\ord_{\alpha,\omega,\gamma,\gamma'}(g_{/L}) = \frac{1}{2}\bigg{(}L_{\alpha,\omega'_F, \gamma'}(g')\bz_{\gamma}(g) + L_{\alpha,\omega, \gamma}(g)\bz_{\gamma'}(g') \bigg{)}\in
H^1(\cO_L[\frac{1}{p}], T(1)\otimes_{\BZ_p}\Lambda)\otimes_{\BZ_p}\BQ_p.
$$
Here we are using that $L_{\alpha,\omega'_F, \gamma'}(g'),L_{\alpha,\omega, \gamma}(g) \in \Lambda_{\cO_\lambda}\otimes_{\BZ_p}\BQ_p$ by
Lemma \ref{ColBK-lem}(ii) and
that we can view the element $\bz_{\gamma'}(g')$ as belonging to $H^1(\cO_L[\frac{1}{p}], T(1)\otimes_{\BZ_p}\Lambda)^-\otimes_{\BZ_p}\BQ_p$
as in \S\ref{IwCoh}.

\vskip 2mm
\noindent {\it{The supersingular case.}} Suppose $g$ is supersingular. We let
$$
\bz^\circ_{\alpha,\omega,\gamma,\gamma'}(g_{/L}) = \bz^\circ_{\omega,\gamma,\gamma'}(g_{/L}) = \frac{1}{2}\bigg{(} L^\circ_{\omega'_F, \gamma'}(g')\bz_{\gamma}(g) + L^\circ_{\omega, \gamma}(g)\bz_{\gamma'}(g')\bigg{)} \in
H^1(\cO_L[\frac{1}{p}], T(1)\otimes_{\BZ_p}\Lambda)\otimes_{\BZ_p}\BQ_p.
$$
Here we are using that $L^\circ_{\omega'_F, \gamma'}(g'),L^\circ_{\omega, \gamma}(g) \in \Lambda_{\cO_\lambda}\otimes_{\BZ_p}\BQ_p$ by
Lemma \ref{ColBK-lem}(ii) and
that we can view the elements $\bz_{\gamma'}(g')$ as belonging to $H^1(\cO_L[\frac{1}{p}], T(1)\otimes_{\BZ_p}\Lambda)^-\otimes_{\BZ_p}\BQ_p$
as in \S\ref{IwCoh}. Though the construction does not depend on the choice of $\alpha$, it is sometimes  included in the notation for convenience. 

\vskip 2mm
It follows from property (iii) of the Beilinson--Kato classes along with Lemma \ref{ColBK-lem} that 
\begin{equation}\label{BKL-Int}
\bz^{\Box}_{\alpha,\omega,\gamma,\gamma'}(g_{/L}) \in H^1(\cO_L[\frac{1}{p}],T(1)\otimes_{\BZ_p}\Lambda) \ \ \text{if (irr$_\BQ$) holds, $\omega$ is good, $\gamma\in T_g$, $\gamma'\in T_{g'}$,}
\ \ \Box = \ord,\circ.
\end{equation}

\vskip 2mm
These elements belong to the global Iwasawa cohomology groups defined in \S\ref{CycL-global}.

\begin{lem} \label{BK-lem}
Let $0\neq\omega\in S_F$. Let $\gamma\in V_g$ and $\gamma'\in V_{g'}$.  Then
$$
\bz^\Box_{\alpha,\omega,\gamma,\gamma'}(g_{/L}) \in H^1_{rel,\Box}(\cO_L[\frac{1}{p}],T(1)\otimes_{\BZ_p}\Lambda)\otimes_{\BZ_p}\BQ_p, \ \ \
\Box = \begin{cases} \ord  & \text{$g$ is ordinary} \\ \circ & \text{$g$ is supersingular.} \end{cases}
$$
Furthermore, if (irr$_\BQ$) holds, $\omega$ is good, and $\gamma\in T_g$, $\gamma'\in T_{g'}$ then this inclusion holds without tensoring with $\BQ_p$.

\end{lem}

\begin{proof} Let $\bz = \bz^\Box_{\alpha,\omega,\gamma,\gamma'}(g_{/L})$. 
Suppose first that $g$ is ordinary. We have 
\begin{equation*}\begin{split}
Col_{\eta_{\omega},\bar{v}}(\bz) & = \frac{1}{2}\bigg{(}L_{\alpha,\omega'_F, \gamma'}(g') Col_{\eta_{\omega},\bar{v}}(\loc_{\bar v}(\bz_{\gamma}(g))) + L_{\alpha,\omega, \gamma}(g) Col_{\eta_{\omega},\bar{v}}(\loc_{\bar v}(\bz_{\gamma'}(g'))) \bigg{)} \\
& = \frac{1}{2}\bigg{(}L_{\alpha,\omega'_F, \gamma'}(g') Col_{\eta_{\omega},\bar{v}}(\loc_{\bar v}(\bz_{\gamma}(g))) - L_{\alpha,\omega, \gamma}(g) Col_{\eta_{\omega}',\bar{v}}(
\loc_{\bar v}(\bz_{\gamma'}(g')))\bigg{)} \\
& = \frac{1}{2}\bigg{(}L_{\alpha,\omega'_F, \gamma'}(g') Col_{\eta_{\omega}}(\loc_p(\bz_{\gamma}(g))) - L_{\alpha,\omega, \gamma}(g) Col_{\eta_{\omega}'}(\loc_p(\bz_{\gamma'}(g'))) \bigg{)} \\
& = \frac{1}{2}\bigg{(}L_{\alpha,\omega'_F, \gamma'}(g')L_{\alpha,\omega, \gamma}(g) - L_{\alpha,\omega, \gamma}(g)L_{\alpha,\omega'_F, \gamma'}(g')  \bigg{)}= 0.
\end{split}
\end{equation*}
In the first line (as in the definition of $\bz$) we are viewing $\bz_{\gamma'}(g')$ as belonging to $H^1(\cO_L[\frac{1}{p}],T(1)\otimes_{\BZ_p}\Lambda)\otimes_{\BZ_p}\BQ_p$ via $res |_{G_L}\circ Tw_*^{-1}$.
The second equality then follows from Lemma \ref{ColTw-L}. The third follows from the second by the definitions of the Coleman maps for $\bar v$.
This proves the lemma for $g$ ordinary. The proof for the supersingular case is the same, replacing the Coleman maps with the signed Coleman maps.

The inclusion without tensoring with $\BQ_p$ then follows from \eqref{BKL-Int}.
\end{proof}

\begin{remark}\label{EC-Katoelement-overL}
Suppose $g$ corresponds to the isogeny class of an elliptic curve $E_\bullet$ as in Remark \ref{EC-optperiod-rmk} but (irr$_\BQ$) does not hold. 
Even though $z_\gamma(g)$ and $z_\gamma'(g')$ are integral (see Remark \ref{EC-Katoselement}), we cannot immediately conclude the same of 
$\bz^\ord_{\alpha.\omega,\gamma,\gamma'}(g_{/L})$. 
This is because we have not shown in this case that the rigidification $Tw$ identifies $T_g$ with $T_g'$ (an identification of $T_g'$ with a sublattice of $T_g$ would be enough).
However, we will later see that the integrality of $\bz^\ord_{\alpha.\omega,\gamma,\gamma'}(g_{/L})$ in this case is a consequence of the integrality of related Beilinson--Flach classes
(cf. Remark \ref{EC-Katoelement-overL-int}). \end{remark}

\subsection{Consequences for Iwasawa cohomology}
We explain some consequences of the non-vanishing of the Beilinson-Kato elements $\bz_\gamma(g)$ and $\bz_{\alpha,\omega,\gamma,\gamma'}(g/L)$ under the Coleman maps.

\subsubsection{Over the rationals}

Kato \cite{K} shows that the classes $\bz_\gamma(g)$ are essentially the base layer of an Euler system, from which he deduces a number of consequences
for the global Iwasasa cohomology groups $H^{1}(\BZ[\frac{1}{p}], T(1)\otimes_{\BZ_{p}} \Lambda)$.  
The following theorem is an immediate consequence of \cite[Thm. 12.4]{K} and plays a crucial role in our subsequent arguments.

\begin{thm}\label{cycIwQ} \hfill
\begin{itemize}
\item[(a)] The global Iwasawa cohomology group $H^{1}(\BZ[\frac{1}{p}], T(1)\otimes_{\BZ_{p}} \Lambda)$ is a torsion-free $\Lambda_{\cO_{\lambda}}$-module. Moreover, 
$H^{1}(\BZ[\frac{1}{p}], T(1)\otimes_{\BZ_{p}} \Lambda) \otimes_{\BZ_{p}} \BQ_{p}$ is a free $(\Lambda_{\cO_{\lambda}} \otimes_{\BZ_{p}} \BQ_{p})$-module of rank one.
\item[(b)] If (Van$_{\BQ}$) holds, then $H^{1}(\BZ[\frac{1}{p}], T(1)\otimes_{\BZ_{p}} \Lambda)$ is a free $\Lambda_{\cO_{\lambda}}$-module of rank one.
\end{itemize}
\end{thm}

\begin{remark}\label{EC-free-over-Lambda} 
If (Van$_\BQ$) holds but not (irr$_\BQ$) holds, then part (b) can be deduced from (a) just as in the argument at the end of the proof of \cite[Lem.~9]{Wu}.
As a consequence, if $g$ corresponds to the isogeny class of an elliptic curve $E_\bullet$ as in Remark \ref{EC-optperiod-rmk} such that $E_\bullet[p](\BQ) = 0$, then the conclusion of part (b) of Theorem \ref{cycIwQ} holds (see also the first paragraph of the the proof of \cite[Thm.~13]{Wu}).
\end{remark}

\subsubsection{Over $L$}
Combining Theorem \ref{cycIwQ} for $T_g$ and $T_{g'}$  with \eqref{spl}, \eqref{spl+}, and \eqref{spl-}, we immediately conclude the following:

\begin{thm}\label{cycIwL} \hfill
\begin{itemize}
\item[(a)] The global Iwasawa cohomology group $H^{1}(\cO_L[\frac{1}{p}], T(1)\otimes_{\BZ_{p}} \Lambda)$ is a torsion-free $\Lambda_{\cO_{\lambda}}$-module. Moreover, 
$H^{1}(\cO_L[\frac{1}{p}], T(1)\otimes_{\BZ_{p}} \Lambda) \otimes_{\BZ_{p}} \BQ_{p}$ is a free $(\Lambda_{\cO_{\lambda}} \otimes_{\BZ_{p}} \BQ_{p})$-module of rank two.
\item[(b)] If (Van$_{\BQ}$) holds (so in particular if (irr$_{\BQ}$) holds), then $H^{1}(\cO_L[\frac{1}{p}], T(1)\otimes_{\BZ_{p}} \Lambda)$ is a free $\Lambda_{\cO_{\lambda}}$-module of rank two.
\end{itemize}
\end{thm}

We note that the hypotheses on $g$ and $L$ imply that if (Van$_\BQ$) holds for $\overline{T}_g$, then
(Van$_L$) holds for $\overline{T}_g$ and (Van$_\BQ$) holds for $\overline{T}_{g'}$.

\begin{remark}\label{EC-free-over-Lambda-L} 
If (Van$_{\BQ}$) holds for $\overline{T}_g$ but (irr$_{\BQ}$) does not, then $T_g\otimes\chi_L$ may be not be identified with $T_{g'}$. However, the freeness of 
$H^{1}(\BZ[\frac{1}{p}], T_{g'}(1)\otimes_{\BZ_{p}} \Lambda)$ as a $\Lambda_{\CO_\lambda}$-module can be used along with (Van$_\BQ$) for $\overline{T}_{g'}$
to conclude that $H^{1}(\BZ[\frac{1}{p}], (T_{g}(1)\otimes\chi)\otimes_{\BZ_{p}} \Lambda)$ is also a free $\Lambda_{\CO_\lambda}$-module of rank 1, whence
the conclusion of part (b) of the theorem. The proof of this latter freeness is just as in Remark \ref{EC-free-over-Lambda}.

As a consequence, if $g$ corresponds to the isogeny class of an elliptic curve $E_\bullet$ as in Remark \ref{EC-optperiod-rmk} 
and $E_\bullet[p](\BQ) = 0$, then the conclusion of part (b) of Theorem \ref{cycIwL} holds.
\end{remark}

Combining Theorem \ref{cycIwL} with the non-vanishing of the Coleman maps on the zeta elements $\bz_\gamma(g)$ 
(see Lemma \ref{ColBK})(i)), we can deduce an analog of Theorem \ref{cycIwQ} for $H^1_{rel,\Box}(\cO_L[\frac{1}{p}], T(1)\otimes_{\BZ_p}\Lambda)$.

\begin{thm}\label{cycIwL-Box} Suppose $g$ is ordinary or in the supersingular case. Let $\Box = \ord$ or $\circ$, accordingly.
\begin{itemize}
\item[(a)] The global Iwasawa cohomology group $H^{1}_{rel,\Box}(\cO_L[\frac{1}{p}], T(1)\otimes_{\BZ_{p}} \Lambda)$ is a torsion-free $\Lambda_{\cO_{\lambda}}$-module. Moreover, 
$H^{1}_{rel,\Box}(\cO_L[\frac{1}{p}], T(1)\otimes_{\BZ_{p}} \Lambda) \otimes_{\BZ_{p}} \BQ_{p}$ is a free $(\Lambda_{\cO_{\lambda}} \otimes_{\BZ_{p}} \BQ_{p})$-module of rank one.
\item[(b)] If (Van$_{\BQ}$) holds, then $H^{1}_{rel,\Box}(\cO_L[\frac{1}{p}], T(1)\otimes_{\BZ_{p}} \Lambda)$ is a free $\Lambda_{\cO_{\lambda}}$-module of rank one.
\end{itemize}
\end{thm}

\begin{proof} 
Since $\Lambda_{\cO_\lambda}\otimes_{\BZ_p}\BQ_p$ (resp.~$\Lambda$) is a regular ring of dimension two (resp.~a complete local regular ring of dimension two), 
to prove part (a) (resp. part (b)) it suffices to show that
$H^{1}_{rel,\Box}(\cO_L[\frac{1}{p}], T(1)\otimes_{\BZ_{p}} \Lambda)$ is the kernel of a $\Lambda_{\cO_\lambda}$-morphism
$C:H^{1}(\cO_L[\frac{1}{p}], T(1)\otimes_{\BZ_{p}} \Lambda) \rightarrow \Lambda_{\cO_\lambda}$ with 
non-zero image.  

Suppose first that $g$ is ordinary. Then we can take the map $C$ to be $Col_{\eta_{\omega},\bar{v}}\circ\loc_{\bar{v}}$ for some
$0\neq\omega\in S_F$ and $\alpha$ as in \eqref{non-crit}. Replacing $\omega$ with a non-zero scalar multiple if necessary,
it follows from \eqref{ordCol-Inj} that $C$ takes values in $\Lambda_{\cO_\lambda}$ and that $H^{1}_{rel,\Box}(\cO_L[\frac{1}{p}], T(1)\otimes_{\BZ_{p}} \Lambda)$
is the kernel of $C$. For $\gamma\in V$ with $\gamma^\pm\neq 0$, we have 
$C(\bz_\gamma(g)) = L_{\alpha,\omega,\gamma}(g)$, which is non-zero by Lemma \ref{ColBK}(i), so the theorem follows in this case.

Similarly, if $g$ is in the supersingular case, then the map $C=Col^\circ_{{\omega},\bar{v}}\circ\loc_{\bar{v}}$,
for a suitable $\omega$, has the desired properties. 
\end{proof}

\begin{thm}\label{cycIwL-Box-II} Suppose $g$ is ordinary or in the supersingular case. Let $\Box = \ord$ or $\circ$, accordingly.
Let $0\neq \omega\in S_F$, let $\alpha = \alpha_g = \alpha_{g'}$ satisfy \eqref{non-crit}, and let $\gamma\in V_g$ and $\gamma'\in V_{g'}$ such that 
$\gamma^\pm,(\gamma')^\pm \neq 0$.
\begin{itemize}
\item[(a)] The class $\bz^\Box_{\alpha,\omega,\gamma,\gamma'}(g_{/L}) \in H^1_{rel,\Box}(\cO_L[\frac{1}{p}],T(1)\otimes_{\BZ_p}\Lambda)\otimes_{\BZ_p}\BQ_p$
is not $(\Lambda_{\cO_\lambda}\otimes_{\BZ_p}\BQ_p)$-torsion. 
\item[(b)] The class $\bz^\Box_{\alpha,\omega,\gamma,\gamma'}(g_{/L}) \in H^1_{rel,\Box}(\cO_L[\frac{1}{p}],T(1)\otimes_{\BZ_p}\Lambda)$
is not $\Lambda_{\cO_\lambda}$-torsion if (irr$_\BQ)$ holds, $\omega$ is good, $\gamma\in T_g$, and $\gamma'\in T_{g'}$.\end{itemize}
\end{thm}

\begin{proof}  By Theorem \ref{cycIwL-Box} it suffices to show that $\bz=\bz^\Box_{\alpha,\omega,\gamma,\gamma'}(g_{/L})$ is non-zero.
This follows by evaluating $\loc_v(\bz)$ under $Col_{\eta_{\omega},v}$ if $g$ is ordinary and under $Col^\circ_{{\omega},v}$ in the supersingular case. 
A similar argument as in the proof of Lemma \ref{BK-lem}
shows that
$$
Col_{\eta_{\omega},v}(\bz) = L_{\alpha,\omega'_F, \gamma'}(g')L_{\alpha,\omega, \gamma}(g) \neq 0,
$$ 
which proves the ordinary case, and that in the supersingular case
$$
Col^\circ_{{\omega},v}(\bz) = L^\circ_{\omega'_F, \gamma'}(g')L^\circ_{\omega, \gamma}(g) \neq 0.
$$ 
\end{proof}

\subsection{Connections to $p$-adic $L$-functions}\label{padicL}
We recall the cyclotomic $p$-adic $L$-functions of $g$ and $g'$ and their relations with the Beilinson--Kato elements.

\subsubsection{Characters in cyclotomic towers}\label{characters-cyc}
For $\zeta$ a primitive $p^{t}$-th root of unity, let
$$
\psi_{\zeta}: G_{\BQ} \twoheadrightarrow \Gamma \ra \ov{\BQ}^{\times}
$$
be the finite order character induced by $\gamma_\cyc \mapsto \zeta$. 
For $t>0$, let $\psi_\zeta$ also denote the Dirichlet character of $(\BZ/p^{t+1}\BZ)^{\times}$ of $p$-power order such that the image of $\epsilon(\gamma_\cyc)\in 1+p\BZ_{p}$ maps to $\zeta$.
Let 
$$
\phi_{\zeta}: \Lambda \ra \BZ_{p}[\zeta] \subset \ov{\BQ}_{p}
$$
be the homomorphism such that $\gamma_\cyc \mapsto \zeta$. 
There is also an extension of this to a homomorphism $\mathfrak{K}_{1,F_\lambda(\alpha)}\otimes_{\Lambda,\iota_\epsilon}\Lambda \ra \ov{\BQ}_{p}$ sending $g(T)\in\mathfrak{K}_{1,\BQ_{p}}$ to the value of the power series $g(T)$  
at $T=\epsilon(\gamma_\cyc)\zeta-1$

\subsubsection{$p$-adic L-functions}\label{padicL-I} 
Let $\alpha$ be root of $x^2-a_g(p)x + p$ satisfying \eqref{non-crit}. Let $0\neq\omega\in S_F$ and $\gamma\in V$ with $\gamma^\pm\neq 0$,
and let $\Omega^{\pm}_{\omega,\gamma} \in \BC^\times$ be the periods defined by $per(\omega) = \Omega_{\omega,\gamma}^+\cdot \gamma^+ + \Omega_{\omega,\gamma}^- \cdot \gamma^-$ as in \S\ref{Periods}.
Then there exists a cyclotomic $p$-adic $L$-functions as follows (cf. \cite[Thm. 16.2]{K}).

\begin{thm}\label{pcycQ}
There exists an unique
$\mathcal{L}_{\alpha,\omega,\gamma}(g) \in \mathfrak{K}_{1,\BQ_{p}}\otimes_{\Lambda,\iota_\epsilon}\Lambda \ra \ov{\BQ}_{p}$
such that 
$$
\phi_{\zeta}(\mathcal{L}_{\alpha,\omega,\gamma}(g))=
e_{p}(\zeta)\cdot \frac{L(1, g \otimes \psi_{\zeta}^{-1})}{\Omega_{\omega,\gamma}^+}, \ \ \ 
e_{p}(\zeta)= \begin{cases*}
\alpha^{-(t+1)} \cdot \frac{p^{t+1}}{\mathfrak{g}(\psi_{\zeta}^{-1})} & if $\zeta \neq 1$\\
(1-\frac{1}{\alpha})^{2} & else.
\end{cases*}.
$$
Moreover, if $g$ is ordinary, $\omega$ is good, and $\gamma\in T$, then 
$
\mathcal{L}_{\alpha,\omega,\gamma}(g) \in \Lambda_{\cO_\lambda}.
$
\end{thm}

These $p$-adic $L$-functions are related to the Beilinson--Kato elements via the Coleman maps \cite[Thm. 16.6]{K}):

\begin{thm}\label{ERLBKI}
Let $\alpha$ be root of $x^2-a_g(p)x + p$ satisfying \eqref{non-crit}. We have
$$
Col_{\eta_{\omega}}(\loc_{p}({\bf{z}}_{\gamma}(g))) = \mathcal{L}_{\alpha,\omega,\gamma}(g)
$$
In other words, $L_{\alpha,\omega,\gamma}(g)=\mathcal{L}_{\alpha,\omega,\gamma}(g)$.
\end{thm}

\begin{remark}\label{NV-rmk1}
This theorem summarizes the relation between the Beilinson--Kato elements and the $L$-values $L(1,g,\chi)$
for $\chi\in \mathfrak{X}_{\BQ,p}^\cyc$. In particular, the non-vanishing of $\mathcal{L}_{\alpha,\omega,\gamma}(g)$,
and hence of $L_{\alpha,\omega,\gamma}(g)$ by Theorem \ref{ERLBKI}, is an obvious consequence of Theorem \ref{NVIw}
(due to Rohrlich).
\end{remark}

\vskip2mm
\noindent{\it Supersingular case.} Suppose $g$ is supersingular. In this case,
there are two {\em signed} $p$-adic $L$-functions \cite{Po}.

\begin{thm}\label{pcycQss} 
Suppose $a_p(g) = 0$. Let $\circ=\pm$. 
There exists an unique 
$\mathcal{L}_{\omega,\gamma}^{\circ}(g) \in \BQ_p\otimes_{\BZ_p}\Lambda_{\cO_\lambda}$,
such that 
$$
\phi_{\zeta}(\mathcal{L}_{\omega,\gamma}^{\circ}(g))=
e_{p}^\circ(\zeta)\cdot \frac{L(1, g \otimes \psi_{\zeta}^{-1})}{\Omega_{\omega,\gamma}^{+}}
$$
for
$$
e_{p}^+(\zeta)=
\begin{cases*}
(-1)^{\frac{t+2}{2}} \cdot \frac{p^{t+1}}{\mathfrak{g}(\psi_{\zeta}^{-1})} 
\cdot \prod_{\text{odd } m=1}^{t-1} \Phi_{p^{m}}(\zeta)^{-1}
& if $t>0$ even\\
2 & if $t=0$.
\end{cases*}
$$
and
$$ 
e_{p}^-(\zeta)=
\begin{cases*}
(-1)^{\frac{t+1}{2}} \cdot \frac{p^{t+1}}{\mathfrak{g}(\psi_{\zeta}^{-1})} 
\cdot \prod_{\text{even } m=2}^{t-1} \Phi_{p^{m}}(\zeta)^{-1}
& if $t>0$ odd\\
p-1 & if $t=0$.
\end{cases*}
$$
\end{thm}
\noindent Here $\Phi_{p^n}(X)$ is the $p^n$th cyclotomic polynomial.

These $p$-adic $L$-functions are also related to the Beilinson--Kato classes \cite[\S3.3.1]{L}:

\begin{thm}\label{ERLBKI-ss}
We have
$$
Col_{\omega}^{\circ}(\loc_{p}({\bf{z}}_{\gamma}(g))) = \mathcal{L}_{\omega, \gamma}^{\circ}(g).
$$
In other words, $L^\circ_{\omega,\gamma}(g) = \mathcal{L}_{\omega, \gamma}^{\circ}(g)$.
Furthermore, if (irr$_\BQ$) holds, $\omega$ is good, and $\gamma\in T$, then $\mathcal{L}_{\omega, \gamma}^{\circ}(g)\in \Lambda_{\cO_\lambda}$.
\end{thm}

The conclusion that $\mathcal{L}_{\omega, \gamma}^{\circ}(g)\in \Lambda_{\cO_\lambda}$ when (irr$_\BQ$) holds, $\omega$ is good, and $\gamma\in T$
is just a consequence of the displayed relation together with property (iii) of the Beilinson--Kato elements and the 
integrality of the image of $Col^\circ_{\omega}$ when $\omega$ is good (cf.~\S\ref{Coleman}).

\begin{remark}\label{NV-rmk2}
The non-vanishing of $\mathcal{L}_{\omega,\gamma}^\circ(g)$,
and hence of $L^\circ_{\alpha,\omega,\gamma}(g)$ by Theorem \ref{ERLBKI-ss}, is also an obvious consequence of Theorem \ref{NVIw}.
\end{remark}

\subsubsection{$p$-adic L-functions over $L$} \label{sspLcy} 
Let $g' = g\otimes\chi_L$. We choose $\alpha_{g'} = \alpha_g = \alpha$ to satisfy \eqref{non-crit}. 
For $0\neq \omega\in S_{F,g}$, let $\omega' \in S_{F,g'}$ be the image of $\omega$ under the fixed isomorphism \eqref{RigMod-Opt},
as in \S\ref{QuadraticTwist-III}.  Let $\gamma\in V=V_g$ and $\gamma'\in V'=V_{g'}$.
To the quadruple $(\alpha,\omega,\gamma,\gamma')$ we associate a $p$-adic $L$-function:
$$
\mathcal{L}_{\alpha,\omega,\gamma,\gamma'}(g_{/L}) =
\mathcal{L}_{\alpha,\omega,\gamma}(g) \cdot \mathcal{L}_{\alpha,\omega',\gamma'}(g') \in \mathfrak{K}_{1,F_\lambda(\alpha)}\otimes_{\Lambda,\iota_\epsilon}\Lambda.
$$

If $g$ is ordinary, then $\mathcal{L}_{\alpha,\omega,\gamma,\gamma'}(g_{/L})$ belongs to $\BQ_p\otimes_{\BZ_p}\Lambda_{\cO_\lambda}$,
and even to $\Lambda_{\cO_\lambda}$ if (irr$_\BQ$) holds, $\omega$ is good ,and $\gamma\in T_g$, $\gamma'\in T_{g'}$.

If $g$ is supersingular, then we similarly, define
$$
\mathcal{L}^\circ_{\omega,\gamma,\gamma'}(g_{/L}) =
\mathcal{L}^\circ_{\omega,\gamma}(g) \cdot \mathcal{L}^\circ_{\omega',\gamma'}(g') \in \BQ_p\otimes_{\BZ_p}\Lambda_{\cO_\lambda}.$$
If (irr$_\BQ$) holds, $\omega$ is good ,and $\gamma\in T_g$, $\gamma'\in T_{g'}$, then $\mathcal{L}^\circ_{\omega,\gamma,\gamma'}(g_{/L})\in\Lambda_{\cO_\lambda}$.

Combining Theorems \ref{ERLBKI} and \ref{ERLBKI-ss} with the definitions of the Beilinson--Kato elements $\bz^\Box_{\alpha,\omega,\gamma,\gamma'}(g_{/L})$
easily yields the following:

\begin{thm}\label{ERLBKII}\hfill
\begin{itemize}
\item[(a)] If $g$ is ordinary, then
$$
Col_{\eta_{\omega},v}(\loc_v(\bz^\ord_{\alpha,\omega,\gamma,\gamma'}(g_{/L}))) = \mathcal{L}_{\alpha,\omega,\gamma,\gamma'}(g_{/L}).
$$
\item[(b)] If $g$ is supersingular, then 
$$
Col^\circ_{{\omega},v}(\loc_v(\bz^\circ_{\alpha,\omega,\gamma,\gamma'}(g_{/L}))) = \mathcal{L}^\circ_{\omega,\gamma,\gamma'}(g_{/L}).
$$
\end{itemize}
\end{thm}

\begin{remark} In view of Theorem \ref{pcycQ} and Theorem \ref{pcycQss}, the cyclotomic $p$-adic $L$-functions 
$\mathcal{L}_{\alpha,\omega,\gamma,\gamma'}(g_{/L})$ and  $\mathcal{L}_{\omega,\gamma,\gamma'}^{\circ}(g_{/L})$
are characterised by an analogous interpolation property with respect to the $L$-values $L(1,g,\psi_\zeta^{-1})L(1,g,\chi_L\psi_\zeta^{-1})$. 
It can be thus viewed as a cyclotomic $p$-adic L-function for the base change of the newform $g$ to the imaginary quadratic field $L$. 
\end{remark}

\begin{remark}\label{NVcyc}
The cyclotomic $p$-adic $L$-functions $\mathcal{L}_{\alpha,\omega,\gamma,\gamma'}(g_{/L})$ and 
$\mathcal{L}_{\omega, \gamma,\gamma'}^{\circ}(g_{/L})$ are non-zero since the factors in their definition are non-zero
(see Remarks \ref{NV-rmk1} and \ref{NV-rmk2}).
\end{remark}

\section{CM Hida families}\label{CMHF}
In this section we study the $\Lambda$-adic Tate-module associated with certain Hida families of CM newforms.

\subsection{The families}
We define the CM Hida families of interest after a few preliminaries.

\subsubsection{Imaginary quadratic fields, again}\label{IQF-II}
Let $L$ be an imaginary quadratic field as in \S\ref{IQF}.  Let $c\in\Gal(L/\BQ)$ be the non-trivial automorphism. 
Let $w_{L}$ be the number of roots of unity in $L$ and let $h_L$ be the class number of $L$.
Let $N_{L/\BQ}: L^{\times} \ra \BQ^{\times}$ be the norm map.
Let $\BA$ (resp. $\BA^{(\infty)}$) be the adeles (resp. finite adeles) over $\BQ$, and let $\BA_L = \BA\otimes_{\BQ} L$ (resp. $\BA_{L}^{(\infty)} = \BA^{(\infty)}\otimes_\BQ L$) be
the adeles (resp. finite adeles) over $L$.
For a place $w$ of $L$, let $\varpi_{w} \in L_{w}^{\times}$ be a uniformiser. 
Let $\rec_{L_{w}}:L_{w}^{\times} \ra G_{L_{w}}^{\ab}$ be the reciprocity map of local class field theory normalised so that uniformisers map to lifts of the arithmetic Frobenius $\Frob_w$. 
Similarly, let $\rec_{L}:\BA_{L}^{\times}/L^{\times} \ra G_{L}^{\ab}$ be the reciprocity map of global class field theory normalised so that $\rec_{L}|_{L_{w}^{\times}}=\rec_{L_{w}}$. 
For $x\in \BA_{L}$ and $\ell$ a prime of $\BQ$, let $x_{\ell}=(x_w)_{w\mid\ell}$ be the $\ell$-component. 
For each non-zero fractional ideal $\mathfrak{a}$ of $L$, let $x_{\mathfrak{a}} \in (\BA_{L}^{(\infty)})^{\times}$ be a finite id\'ele so that $\ord_{w}(x_{\mathfrak{a},w})=\ord_{w}(\mathfrak{a})$ for each finite place 
$w$ with $\ord_{w}(\mathfrak{a})\neq 0$ and $x_{\mathfrak{a},w} = 1$ for all other places $w$. 

\subsubsection{Hecke characters}\label{Hecke-char}
Let $\psi: \BA_{L}^{\times}/L^{\times} \ra \BC^\times$ be an algebraic Hecke character over $L$.
By this we mean that there is a pair of integers $(m,n)$ such that the restriction of $\psi$ to the identity component of $(L\otimes \BR)^\times$
is given by the composition $\rho_\infty$ of the algebraic character $\rho: (L\otimes\BR)^{\times} \ra 
(\ov{\BQ}\otimes\BR)^{\times}$, $\rho(x\otimes r) = (x\otimes r)^m (\bar{x}\otimes r)^n$ for $\bar{x} = c(x)$,
with the homomorphism $(\ov{\BQ}\otimes\BR)^{\times}\rightarrow \BC^\times$ induced by $\iota_\infty$.
We call the pair $(m,n)$ the infinity type of $\psi$.
The values of the character 
$$
\BA_{L}^{\times} \ra \BC^{\times},  \text{ $\alpha \mapsto \rho_\infty(\alpha_{\infty})^{-1} \psi(\alpha)$},
$$
generate a number field $F_{\psi}$. In what follows, we regard $F_{\psi} \subset \ov{\BQ}$ via the embedding $\iota_\infty$.
Let 
$\psi^c =\psi \circ c$ be the conjugate Hecke character. The infinity type of 
$\psi^c$ is $(n,m)$.

Let $\rho_p$ be the composition of the algebraic character $\rho:(L\otimes\BQ_p)^{\times} \ra 
(\ov{\BQ}\otimes\BQ_p)^{\times}$, $\rho(x\otimes r) = (x\otimes r)^m (\bar{x}\otimes r)^n$, with the homomorphism $(\ov{\BQ}\otimes\BQ_p)^{\times}\rightarrow \ov{\BQ}_p^\times$ induced by $\iota_p$.
The $p$-adic Galois character associated to $\psi$ is given by
$$
\rho_{\psi}: G_{L} \ra \ov{\BQ}_{p}^{\times}, 
\text{ $\rho_{\psi}(\sigma)=\rho_p(\alpha_{p})\rho_\infty(\alpha_{\infty})^{-1}\psi(\alpha)$ for $\sigma=\rec_{L}(\alpha)$}.
$$
Here we view $\rho_\infty(\alpha_{\infty})^{-1}\psi(\alpha)$ as belonging to $\ov{\BQ}_p$ by regarding $F_{\psi}$
as a subfield of $\ov{\BQ}_{p}$ via the embedding $\iota_p$. 
This representation is Hodge--Tate at each place $w\mid p$ of $L$: the Hodge-Tate weight\footnote{We adopt geometric conventions for Hodge--Tate weights. In particular, the Hodge--Tate weight of the 
$p$-adic cyclotomic character $\epsilon$ is $-1$.} 
at $v$ is $m$ and the Hodge--Tate weight at $\bar v$ is $n$. 

Serre \cite{Se} proved that $\psi \mapsto \rho_{\psi}$ is a bijection between the algebraic Hecke characters of $\BA_L^\times$ and the $p$-adic Galois characters of $G_L$ that are unramified outside a finite set of places and are Hodge--Tate at the places above $p$.

If $\psi$ has infinity type $(m,n)$, then the restriction of $\psi$ to $\BA^\times\subset \BA_L^\times$ equals
$\epsilon_\psi |\cdot|^{m+n}$ for some finite order character $\epsilon_\psi$.

\subsubsection{CM modular forms}
Let $\psi$ be an algebraic Hecke character over $L$ with infinity type $(1-k,0)$  for some integer $k\geq 1$ and conductor $\mathfrak{f}_\psi$.
Let $N_\psi = D_LN_{L/\BQ}(\mathfrak{f}_{\psi})$ and let $\chi_\psi = \chi_{L}\epsilon_{\psi}$ (which we view as a Dirichlet character modulo $N_\psi$).
Let $\theta_{\psi} \in S_{k}(\Gamma_{1}(N_\psi),\chi_{\psi})$ be the corresponding CM modular form which has $q$-expansion 
\begin{equation}\label{defCM}
\theta_{\psi}(q)=\sum_{0 \neq \mathfrak{a}\subset \cO_{L}, (\mathfrak{a},\mathfrak{f}_\psi) =1} \psi(x_\mathfrak{a}) q^{N_{L/\BQ}(\mathfrak{a})},
\end{equation}
and so 
\begin{equation}\label{CM-LF}
L(s,\theta_{\psi}) = L(s,\psi). 
\end{equation}
The Hecke field $F_{\theta_{\psi}}$ is a subfield of $F_{\psi}$.

Suppose $(\mathfrak{f}_\psi,p) = 1$. Then $\theta_\psi$ is $p$-ordinary:
The coefficient of $q^p$ is $\psi_v(p)+\psi_{\bar{v}}(p)$ and the valuation of $\psi_v(p)$ (with respect to $\iota_p$) is $k-1$ while the valuation of $\psi_{\bar{v}}(p)$ is $0$. Note that if $k=1$, then the coefficient of $q^p$ need not be a $p$-adic unit, but the Hecke polynomial at $p$ has a $p$-adic unit root at the prime above $p$ determined by $\iota_p$.

Let $\lambda\mid p$ be the place of $F_\psi$ determined by the embedding $\iota_p$. 
The representation $\rho_\psi$ takes values in $F_{\psi,\lambda}$.
If $V_{F_{\psi,\lambda}, \theta_\psi}$ is the $p$-adic Galois representation associated to
$\theta_\psi$ as in \cite[\S\S6.3,8.3,15.10]{K} (denoted $V_{F_{\psi,\lambda}}(\theta_\psi)$ in {\it op.~cit.}), then 
\begin{equation}\label{defInd}
V_{F_{\psi,\lambda}, \theta_\psi}^\vee \simeq \Ind^{G_{\BQ}}_{G_{L}} (\rho_{\psi}).
\end{equation}

\subsubsection{Some $\BZ_p$-extensions of $L$}\label{Zp-ext}
Let $L_\infty$ be the $\BZ_{p}^2$-extension of $L$.  Let $L_{\infty}^{\cyc} \subset L_{\infty}$ (resp.~$L_{\infty}^{\ac} \subset L_{\infty}$) be the cyclotomic (resp.~anticyclotomic) $\BZ_p$-extension of $L$. 
For a prime $w\mid p$, let $L_{\infty}^{w}\subset L_\infty$ be the $\BZ_p$-extension unramified away from $w$.  
For $\Box= \emptyset$, $\cyc$, $\ac$, or $w$, let $\Gamma_L^\Box = \Gal(L_\infty^\Box/L)$.
Then $\Gamma_L = \Gamma_L^\emptyset \cong \BZ_p^2$ and 
$\Gamma_L^\Box \cong\BZ_p$ if $\Box\neq \emptyset$.
Let $\ov{\Gamma}_L^{w} \subset \Gamma_L^{w}$ be the image of $G_{L_{w}} \subset G_{L}$
under the canonical projection $G_{L} \twoheadrightarrow \Gamma_L^{w}$. Let $h_p \geq 0$ be the integer such that
$\ov{\Gamma}_L^{w}=(\Gamma_L^{w})^{p^{h_p}}$. 
If $p \nmid h_{L}$, then $h_p=0$.

The group $\Gal(L/\BQ) = \{c,1\}$ acts on $\Gamma_L$ via conjugation by any lift of $c$ to $\Gal(L_\infty/\BQ)$. 
Let $\Gamma_L^\pm\subset\Gamma_L$ be the $\BZ_p$-summand of
rank one on which this action is just multiplication by $\pm 1$. The subgroup $\Gamma_L^+$ (resp.~$\Gamma_L^-$) is mapped
isomorphically onto $\Gamma_L^\cyc$ (resp. $\Gamma_L^\ac$) via the canonical projection from $\Gamma_L$.

Let $\gamma_\Box \in \Gamma_L^\Box$, $\Box = \pm$, $\cyc$, $\ac$, or $w$, be a topological generator. 
In light of the canonical isomorphisms $\Gamma_L^+\isoarrow\Gamma_L^\cyc\isoarrow \Gamma$ and $\Gamma_L^-\isoarrow\Gamma_L^\ac$ we can and do choose these so that $\gamma_+\in\Gamma_L^+$ 
and $\gamma_\cyc\in \Gamma^\cyc_L$ are 
identified with the previously chosen $\gamma_\cyc\in\Gamma$, $\gamma_-$ is identified with $\gamma_\ac$, 
$c\gamma_vc^{-1} = \gamma_{\bar v}$, $\gamma_+$ maps to $\gamma_w^{p^{h_p}/2}$ ($w=v,\bar v$), and 
$\gamma_-$ maps to $\gamma_v^{1/2}$ and to $\gamma_{\bar v}^{-1/2}$.
Note that $\gamma_w^{p^{h_p}}$ is a topological generator of $\ov{\Gamma}_L^w$.

\subsubsection{More Iwasawa algebras}\label{MoreIwasawa}
For $\Box = \emptyset$, $\pm$, $\cyc$, $\ac$, or $w$, let
$$
\Lambda^\Box_L = \BZ_p[\![\Gamma^\Box_L]\!].
$$
Similarly, for a $p$-adically complete $\BZ_p$-algebra $R$ let $\Lambda_{L,R}^\Box = R[\![\Gamma_L^\Box]\!]$.  

For $\Box\neq \pm$, let
$$
\Psi_L^\Box: G_L \twoheadrightarrow \Gamma_L^\Box
$$
be the canonical projection. We view these as $(\Lambda^\Box_L)^\times$-valued characters of $G_L$ (the canonical characters).
The canonical projections $\Gamma_L\twoheadrightarrow \Gamma_L^\Box$ induce ring homomorphisms $\Lambda_L\twoheadrightarrow\Lambda_L^\Box$.
Similarly, the canonical projection $\Gamma_L\twoheadrightarrow\Gamma$ induces a homomorphism $\Lambda_L\twoheadrightarrow \Lambda$.
The canonical projection $\Gamma^\cyc_L\twoheadrightarrow\Gamma$ is an isomorphism, inducing an identification $\Lambda_L^\cyc\isoarrow \Lambda$;
the projection $\Lambda_L\twoheadrightarrow \Lambda$ factors through this isomorphism.
Canonical characters map to canonical characters under
all these ring homomorphisms.

There are identifications
$$
\Lambda_{L,R}^\Box\isoarrow R[\![T_\Box]\!],  \ \ \gamma_\Box \mapsto 1+T_\Box.
$$
Under these, the canonical
identification of $\Lambda_{L,R}^\cyc$ with $\Lambda_R$ becomes the isomorphism $R[\![T_\cyc]\!] \simeq R[\![T]\!]$, $T_\cyc\mapsto T$.

Let $(\Box,\Box') = (\cyc,w)$ or $(\cyc,\ac)$. Then the canonical projections induce $\Gamma_L \isoarrow \Gamma_L^\Box \times \Gamma_L^{\Box'}$,
and hence an isomorphism of rings 
\begin{equation}\label{LambdaL-fact}
\Lambda_{L,R} \simeq \Lambda^\Box_{L,R}\hat\otimes_{R}\Lambda^{\Box'}_{L,R}.
\end{equation} 
These induce isomorphisms
$$
\Lambda_{L,R} \simeq R[\![T_\cyc, T_\ac]\!]  \ \ \text{and} \ \ \Lambda_{L,R} \simeq R[\![T_\cyc,T_w]\!].
$$
Note that with respect to the first (resp.~second) of these isomorphisms, the
projection $\Lambda_{L,R}\twoheadrightarrow\Lambda_R$ is just the map 
$T_\cyc\mapsto T$, $T_\ac\mapsto 0$
(resp.~$T_w\mapsto 0$) on power series rings.

The inclusion 
$\Gamma_L^+\times\Gamma_L^- \isoarrow \Gamma_L$
induces an isomorphism 
$\Lambda_{L,R}^+\hat\otimes_R\Lambda_{L,R}^- \isoarrow \Lambda_L$
and hence an isomorphism
$$
R[\![T_+,T_-]\!]\simeq\Lambda_{L,R}.
$$
The isomorphism $\Lambda_{L,R} \simeq R[\![T_\cyc, T_\ac]\!]$ is then the map
$T_+\mapsto T_\cyc$, $T_-\mapsto T_\ac$ on power series rings.
Similarly, the isomorphism $\Lambda_{L,R} \simeq R[\![T_\cyc,T_v]\!]$ is  the map
$T_+\mapsto (1+T_\cyc)(1+T_v)^{p^{h_p}/2}-1$, $1+T_-\mapsto (1+T_v)^{1/2}$.

\subsubsection{A $\Lambda_L^v$-family of $CM$ forms}\label{CMHidaFam}
Let 
\begin{equation}\label{LamHec}
\Theta_{v}: \BA_{L}^{\times} \ra \Gamma_{L}^{v}
\end{equation}
be the composition of the reciprocity map $\rec_{L}:\BA_{L}^{\times} \ra G_{L}^{\ab}$ with the canonical projection $G_{L}^{\ab} \twoheadrightarrow \Gamma_{L}^v$.
Let ${\bf{h}}_{v}$ be the  
family of CM forms associated to $\Theta_v$, with $q$-expansion given by
$$
{\bf{h}}_{v}(q) = \sum_{n \geq 1} {\bf{b}}(n) q^{n} \in \Lambda_{L}^{v}[\![q]\!], \ \  
{\bf{b}}(n)= \sum_{{\mathfrak{a}\subset \cO_{L}, {v} \nmid \mathfrak{a}, N_{L/\BQ}(\mathfrak{a})=n} } \Theta_{v}(x_{\mathfrak{a}}).
$$
We refer to ${\bf{h}}_{v}$ as the  canonical CM family with CM by $L$.

For a $\BZ_p$-algebra homomorphism $\psi: \Lambda_{L}^{v} \ra \ov{\BQ}_{p}$ 
such that $\psi(\gamma_v^{h_p}) = \epsilon(\gamma_+)^{m}$ with $m \equiv 0 \mod{p-1}$, $m\geq 0$,
the specialisation $h_{v,\psi}$ with the $q$-expansion  
$$
h_{v,\psi}(q)=\sum_{n \geq 1} \psi({\bf{b}}(n)) q^{n} 
$$
is the $p$-ordinary stabilisation of a CM newform of weight $k=m+1$ and level $\Gamma_{1}(D_{L})$. 
Indeed, $\psi|_{\Gamma_L^v}$ -- viewed as a character of $G_L$ --  is visibly unramified outside $p$ and has Hodge--Tate weight $-m=1-k$ (resp. $0$) at $v$ (resp. $\bar{v}$). The corresponding algebraic Hecke character over $L$ is readily seen to be unramified with infinity type $(1-k,0)$. Occasionally we also denote this algebraic Hecke character by $\psi$ as well.
It follows that ${\bf{h}}_{v}$ is a $p$-ordinary $\Lambda_{L}^{v}$-adic modular form of tame level $D_{L}$ (see Section \ref{Hida-Hecke} below).

\subsection{The Tate lattices}\label{sTatel}
We analyse certain lattices in the Galois representation associated with the CM family ${\bf{h}}_{v}$.

\subsubsection{The Hida--Iwasawa algebras $\Lambda_D$ and $\Lambda_0$}
Let 
$$
\Lambda_D = \BZ_p[\![\BZ_p^\times]\!] \ \ \text{and} \ \ \Lambda_{0}=\BZ_p[\![1+p\BZ_p]\!].
$$ 
Since $\BZ_p^\times = \mu_{p-1} \times 1+p\BZ_p$, there is a canonical identification 
$\Lambda_D = \Lambda_0[\mu_{p-1}]$. Then $\Lambda_D = \oplus_{i=0}^{p-2} \Lambda_D^{(i)}$, where
$\Lambda_D^{(i)}$ is the direct summand on which multiplication by the group element $[\zeta]\in \mu_{p-1}$
acts as $\zeta^i$. In particular, $\Lambda_D^{(i)} = \delta_i\cdot\Lambda_D = \Lambda_0\cdot\delta_i$, where 
$\delta_i = \frac{1}{(p-1)} \sum_{\zeta\in \mu_{p-1}}\zeta^{-i} [\zeta] \in \BZ_p[\mu_{p-1}]$.

 We fix the $\BZ_p$-isomorphism $\Lambda_\CG\isoarrow \Lambda_D$ sending $g\in\CG$ to $[\epsilon(g)]\in \BZ_p^\times\subset \Lambda_D$.
 This identifies $\Lambda$ with $\Lambda_0$ (via $\gamma_\cyc\mapsto [\epsilon(\gamma_\cyc)]$).
This also identifies the idempotent $\eps_i \in \Lambda_{\CG}$ with $\delta_{i}\in\Lambda_D$.
Via this identification the character $\Psi_D:G_\BQ\rightarrow \Lambda_D$, $g\mapsto [\epsilon(g)] \in \BZ_p^\times$ is
identified with the canonical character $\Psi_\CG$.

\subsubsection{Hida's Hecke algebras}\label{Hida-Hecke} 
Let $M$ be a positive integer such that $p\nmid M$.
Let 
$$
\BT_{Mp^\infty}^{\ord} = \varprojlim_{r} \BT_{\Gamma_{1}(Mp^{r})}^{' \ord} \ \ \text{and} \ \ \BH_{Mp^\infty}^{\ord} = \varprojlim_{r} \BH_{\Gamma_{1}(Mp^{r})}^{' \ord}
$$
be Hida's ordinary Hecke algebras, where $\BT_{\Gamma_{1}(Mp^{r})}^{' \ord}$ (resp.~$\BH_{\Gamma_{1}(Mp^{r})}^{' \ord}$) is the $p$-ordinary direct summand of the Hecke algebra 
$\BT_{\Gamma_{1}(Mp^{r})/\BZ_{p}}'$ (resp.~$\BH_{\Gamma_{1}(Mp^{r})}'$) 
acting on cuspforms (resp.~on modular forms). 
Recall that there is a Hecke operator $U_{p}'$ in both $\BT_{Mp^\infty}^{\ord}$ and $\BH_{Mp^\infty}^{\ord}$ \cite[Def. 2.4.3]{KLZ}. The operator $U_p'$ projects to the $T_p'$-operator in $\BT_{\Gamma_{1}(Mp^{r})}^{' \ord}$
and $\BH_{\Gamma_{1}(Mp^{r})}^{' \ord}$. 
Both $\BT_{Mp^\infty}^{\ord}$ and $\BH_{Mp^\infty}^{\ord}$ are $\Lambda_D$-algebras, with the group element $[a]\in \BZ_p^\times \subset \Lambda_D^\times$ acting 
on modular forms of level $Np^r$ via the inverse of the diamond operator $\langle a'\rangle$ for $a' = (a \mod p^r,1) \in (\BZ/p^r\BZ \times\BZ/M\BZ)^\times$. 
The inclusion of the spaces of cuspforms in the spaces of modular forms induces a surjection $\BH_{Mp^\infty}^\ord\twoheadrightarrow \BT_{Mp^\infty}^\ord$
of $\Lambda_D$-algebras.
Moreover, $\BT_{Mp^\infty}^{\ord}$ and  $\BH_{Mp^\infty}^{\ord}$ are finite, free $\Lambda_{0}$-modules.

Let $M_{\Lambda_{D}}^{\ord}$ be the space of $p$-ordinary $\Lambda_D$-adic modular forms over $\BZ_{p}$ with tame level $M$,
and let $S_{\Lambda_{D}}^{\ord}\subset M_{\Lambda_{D}}^{\ord}$ be the subspace of $\Lambda_D$-adic cuspforms \cite[\S7.4]{KLZ}. 
The former is an $\BH_{Mp^\infty}^\ord$-module and the latter a $\BT_{Mp^\infty}^\ord$-module. 
There is a canonical duality
\begin{equation}\label{HidaDu}
\Hom_{\Lambda_{D}}(\BT_{Mp^{\infty}}^{\ord}, \Lambda_{D}) \simeq S_{\Lambda_{D}}^{\ord}
\end{equation}
of $\BT_{Mp^{\infty}}^{\ord}$-modules \cite[Thm.~2.5.3]{Ohta}.

The canonical CM family $\bf{h}_v$ belongs to the space $S_{\Lambda_L^v}^\ord = S_{\Lambda_D}^\ord\otimes_{\Lambda_D}\Lambda_L^v$ for $M=D_L$. Here we view $\Lambda_L^v$ as a $\Lambda_D$-algebra via the $\BZ_p$-homomorphism $\varphi_v:\Lambda_D\rightarrow \Lambda_L^v$ sending 
$[\epsilon(\tilde\gamma)]$ to $\epsilon(\tilde\gamma)^{-1}\gamma_v$ for all all lifts $\tilde\gamma\in G_L$ of $\gamma_\cyc\in \Gamma^\cyc$.
Note that $\varphi_v$ factors through the projection $\Lambda_D\twoheadrightarrow \Lambda_D^{(-1)} = \Lambda_0\cdot\delta^{(-1)}$.

\subsubsection{The $\Lambda_D$-adic Tate modules: generalities}
Let 
$$
H^{1}_{\ord}(Mp^{\infty}) = \varprojlim_{r} H^{1}_{\ord}(X_{1}(Mp^{r})_{/\ov{\BQ}}, \BZ_{p}(1)) \ \ \text{and} \ \ \CH^{1}_{\ord}(Mp^{\infty}) = \varprojlim_{r} H^{1}_{\ord}(Y_{1}(Mp^{r})_{/\ov{\BQ}}, \BZ_{p}(1))
$$
be the $\Lambda_D$-adic Tate modules.  
Here 
the subscript `$\ord$' refers to the $p$-ordinary direct summand of the first \'etale cohomology arising from the dual $p$-ordinary idempotent (cf.~\cite[\S7.2]{KLZ}). Note that $H^{1}_{\ord}(Mp^{\infty})$ 
(resp.~$\CH^{1}_{\ord}(Mp^{\infty})$) has a natural structure of a $\BT_{Mp^{\infty}}^{\ord}[G_{\BQ}]$-module (resp.~a $\BH_{Mp^{\infty}}^{\ord}[G_{\BQ}]$-module).
We recall some of the fundemental properties of these Tate modules \cite{Oh1}, \cite{Oh2}.

\begin{thm}\label{ordTate}
There exists a commutative diagram 
$$
\begin{tikzcd}[row sep=2.5em]
0 \arrow[r,] & \mathcal{F}^{+}H^{1}_{\ord}(Mp^{\infty}) \arrow[r,] \arrow[d,"="] & H^{1}_{\ord}(Mp^{\infty}) \arrow[r,] \arrow [d,hook] & \mathcal{F}^{-}H^{1}_{\ord}(Mp^{\infty}) \arrow[r,]\arrow[d,hook] & 0 \\
0 \arrow[r,] & \mathcal{F}^{+}\CH^{1}_{\ord}(Mp^{\infty}) \arrow[r,] & \CH^{1}_{\ord}(Mp^{\infty}) \arrow[r,] & \mathcal{F}^{-}\CH^{1}_{\ord}(Mp^{\infty}) \arrow[r,] & 0 
\end{tikzcd}
$$
of $\BH_{Mp^\infty}^\ord[G_{\BQ_p}]$-modules,
where the top and bottom lines are short exact sequences, the second vertical arrow is induced by the canonical inclusions $H^1(X_1(Mp^r),\BZ_p(1)) \hookrightarrow H^1(Y_1(Mp^r),\BZ_p(1))$, and 
furthermore:
\begin{itemize}
\item[(i)] The $G_{\BQ_{p}}$-action on $\mathcal{F}^{+}H^{1}_{\ord}(Mp^{\infty}) = \mathcal{F}^{+}\CH^{1}_{\ord}(Mp^{\infty})$ is such that $I_{p}$ acts via the 
character $\epsilon\Psi_D$. 
Moreover, 
$$
\mathcal{F}^{+}H^{1}_{\ord}(Mp^{\infty}) = \mathcal{F}^{+}\CH^{1}_{\ord}(Mp^{\infty})\simeq \BT_{Mp^{\infty}}^{\ord}
$$
as $\BH_{Mp^{\infty}}^{\ord}$-modules. 

\item[(ii)] The $G_{\BQ_p}$-action on the quotients $\mathcal{F}^{-}H^{1}_{\ord}(Mp^{\infty})$ and $\mathcal{F}^{-}\CH^{1}_{\ord}(Mp^{\infty})$ is unramified with an arithmetic Frobenius acting via the Hecke operator $U_{p}'$. Moreover, there exist isomorphisms
$$
\mathcal{F}^{-}H^{1}_{\ord}(Mp^{\infty}) \simeq S_{\Lambda_D} \ \ \text{and} \ \ \mathcal{F}^{-}\CH^{1}_{\ord}(Mp^{\infty}) \simeq M_{\Lambda_D}
$$
of $\BT_{Mp^{\infty}}^{\ord}$-modules and $\BH_{Mp^{\infty}}^{\ord}$-modules, respectively, such that the third vertical arrow corresponds to the inclusion $S_{\Lambda_D} \hookrightarrow M_{\Lambda_D}$.
\end{itemize}
\end{thm}

{\color{blue}} For $p\geq 5$, this theorem is explained in \cite[\S\S1.7-1.8]{FK}. It is also proved in {\cite{Oh1,Oh2}} (but see the remarks in \cite[\S\S1.7.15-16]{FK}). See also \cite[Thm.~7.2.3]{KLZ}. The arguments in these papers likely apply to the $p=3$ case as well. The results of \cite{Ca} explicitly covers some parts of the $p=3$ cases. An alternate proof that also includes the $p=3$ case is included in \cite{SV-S-Ohta}.

From (a twisted) Poincare duality it is possible to define a perfect $\Lambda_D$-duality for $H^1_\ord(Mp^\infty)$:

\begin{thm}\label{ordTate-dual}
There exists a perfect pairing 
$$
(\cdot,\cdot): H^1_\ord(Mp^\infty)\times H^1_\ord(Mp^\infty)\rightarrow \Lambda_D
$$
of $\Lambda_D$-modules satisfying $(t\cdot x, y) = (x,t\cdot y)$ for all $x,y\in H^1_\ord(Mp^\infty)$ and $t\in \BT_{Mp^\infty}^\ord$.
Furthermore, the pairing $(\cdot,\cdot)$ induces isomorphisms
$$
\Hom_{\Lambda_D}(H^1_\ord(Mp^\infty)^\pm,\Lambda_D) \simeq H^1_\ord(Mp^\infty)^\mp
$$
as $\BT^\ord_{Mp^\infty}$-modules,
where the superscript $\pm$ denotes the submodule on which the action of complex conjugation is by $\pm 1$.
\end{thm}

The existence of the perfect pairing $(\cdot,\cdot)$ is explained in \cite[\S1.6]{FK}, see especially \cite[\S1.6.5]{FK}. The consequence for the $\pm$-subspaces for the action of $c$ follows from \cite[\S1.6.3(3)]{FK} (see also \cite[Thm.~7.2.3(v)]{KLZ}).

\subsubsection{The Tate lattices for ${\bf h}_{v}$}\label{Tate-lattices}
Let $\varphi:\BT_{D_Lp^{\infty}}^{\ord} \ra \Lambda_{L}^v$ be the $\Lambda_{D}$-homomorphism corresponding to the canonical CM family ${\bf{h}}_{v}$ (which corresponds to a branch of a Hida family in the sense of \cite[\S7]{KLZ}). This extends the homomorphism $\varphi_v:\Lambda_D\rightarrow\Lambda_L^v$,  
and it induces a  $\Lambda_D$-homomorphism $\BH_{D_Lp^{\infty}}^{\ord} \ra \Lambda_{L}^v$ by composition with the canonical $\Lambda_D$-surjection
$\BH^\ord_{D_Lp^\infty}\twoheadrightarrow \BT_{D_Lp^{\infty}}^{\ord}$. We continue to write $\varphi$ for this last homomorphism.

The Tate lattices associated with the canonical CM family ${\bf{h}}_{v}$ are the $\Lambda_{L}^{v}[G_{\BQ}]$-modules  
$$
\BT = H^{1}_{\ord}(D_Lp^{\infty}) \otimes_{\BT_{D_Lp^{\infty}}^{\ord},\varphi} \Lambda_{L}^{v} \ \ \text{and}  \ \
\BH = \CH^1_\ord(D_Lp^\infty) \otimes _{\BH_{D_Lp^\infty}^{\ord},\varphi}\Lambda_L^v.
$$
We similarly define the $\Lambda_{L}^{v}[G_{\BQ_{p}}]$-modules
$$
\BT^{\pm}=\mathcal{F}^{\pm} H^{1}_{\ord}(D_Lp^{\infty}) \otimes_{\BT_{D_Lp^{\infty}}^{\ord},\varphi} \Lambda_{L}^{v}
\ \ \text{and} \ \ 
\BH^{\pm}=\mathcal{F}^{\pm} \CH^{1}_{\ord}(D_Lp^{\infty}) \otimes_{\BH_{D_Lp^{\infty}}^{\ord},\varphi} \Lambda_{L}^{v}.
$$
By Theorem \ref{ordTate} there is a commutative diagram
of $\Lambda_L^v[G_{\BQ_p}]$-modules
\begin{equation}\label{CM-ord}
\begin{tikzcd}[row sep=2.5em]
\BT^{+} \arrow[r,] \arrow[d,"="] & \BT \arrow[r,] \arrow[d,] &  \BT^{-} \arrow[r,] \arrow[d,]& 0 \\
\BH^{+} \arrow[r,] & \BH \arrow[r,]  &  \BH^{-} \arrow[r,] & 0 
\end{tikzcd}
\end{equation}
with exact rows.
 Furthermore, 
\begin{equation}\label{subrk}
\text{$\BT^{+} = \BH^+$  is a free $\Lambda_{L}^v$-module of rank one.}
\end{equation}

Let ${\bf{F}}=\Frac(\Lambda_{L}^{v})$ be the field of fractions of $\Lambda_L^v$ and let 
$$
\BV = \BT\otimes_{\Lambda_L^v}{\bf{F}} = \BH\otimes_{\Lambda_L^v}{\bf{F}}.
$$
Comparing the traces of Frobenius elements
$\Frob_\ell$, $\ell\nmid D_Lp$,
shows that there is an isomorphism
\begin{equation}\label{CMgen}
\BV
\simeq \Ind_{G_{L}}^{G_{\BQ}}({\bf{F}}(\Psi_L^{v}))
\end{equation}
of ${\bf{F}}[G_{\BQ}]$-modules. 
In particular, $\BT\otimes_{\Lambda_{L}^{v}} {\bf{F}}$  is an irreducible ${\bf{F}}[G_\BQ]$-module, and the $G_{\BQ}$-action factors through $\Gal(L_{\infty}/\BQ)$.

\begin{remark}\label{HidaGal-rmk} The CM family $\bf{h}_v$ corresponds to a branch of a Hida family
$\bf{h}$ in the sense of \cite[\S7]{KLZ}. In particular, $\Lambda_L^v$ equals $\Lambda_{{\bf{h}}_v}$ (in the notation
of {\it op.~cit.}) and the module $\BH$ (resp.~$\BT$) then equals
$M({\bf h})^*\otimes_{\Lambda_{\bf{h}}}\Lambda_{{\bf h}_v}$ 
(resp.~$M({\bf h})^*_{\mathrm{par}}\otimes_{\Lambda_{\bf{h}}}\Lambda_{{\bf h}_v}$).
It follows that for $\psi:\Lambda_L^v\rightarrow\ov{\BQ}_p$ as in \S\ref{CMHidaFam},
$\BH\otimes_{\Lambda_L^v,\psi} F_{h_{v,\psi},\lambda} \cong M_{F_{h_v,\psi},\lambda}(h_{v,\psi})^*$,
where the right-hand side is as in \cite[\S2.8]{KLZ}. Comparing the normalizations in \cite[\S2.8]{KLZ} and \cite[\S\S6.3,8.3,15.10]{K}
shows that $M_{F_{h_v,\psi},\lambda}(h_{v,\psi})^* \cong V_{h_{v,\psi},\lambda}^\vee \cong 
\Ind_{G_L}^{G_{\BQ}} (\psi\circ\Psi_L^v)$, which implies \eqref{CMgen}. 
\end{remark}

The $\Lambda_{L}^v$-modules $\BT$ and $\BH$ may not be torsion-free, so we let 
$$
\BT_{1}=\BT/\BT_{\tor}  \ \ \text{and} \ \ \BH_1 = \BH/\BH_{\tor}.
$$
These are just the respective images of $\BT$ and $\BH$ in $\BT\otimes_{\Lambda_L^v} \bf{F}$ and $\BH\otimes_{\Lambda_L^v} \bf{F}$. In particular,
$$
\BT_1 \subseteq \BH_1 \subset \BV.
$$
Both $\BT_1$ and $\BH_1$ are torsion-free $\Lambda_L^v$-modules of rank two (but not obviously free). 
Let $\BT_1^-$ be the quotient of $\BT_1$ by the image of $\BT^+$, and let $\BH^-$ be the analogous quotient of $\BH_1$.
There are exact sequences
$$
\BT^{+} \ra \BT_{1} \ra \BT_{1}^{-} \ra 0 \ \  \text{and} \ \ \BH^{+} \ra \BH_{1} \ra \BH_{1}^{-} \ra 0 
$$
of $\Lambda_{L}^{v}[G_{\BQ_{p}}]$-modules.  It is not a priori clear that either $\BT_{1}^-$ or $\BH_1^-$ is torsion-free. 

\begin{lem}\label{CMTate} \hfill
\begin{itemize}
\item[(i)] The $\Lambda_{L}^{v}[G_{\BQ_{p}}]$-module morphisms $\BT^{+} \ra \BT_1$ and $\BH^{+} \ra \BH_1$ are non-zero and thus injective.
\item[(ii)] The submodules $\BT^{+}$ and $\BH^+$ are $G_{L}$-stable with $G_{L}$ acting via $\Psi_L^v$. 
\item[(iii)] The action of $G_{L_v}$ on the quotients $\BT_{1}^{-}$ and $\BH_1^-$ is via $\Psi_L^{v,c}$.
Moreover, for $\sigma \in G_{L_{v}}$ and $t \in \BT_{1}$ and $h\in \BH_1$,
$$
\sigma \cdot t - \Psi_L^{v,c} (\sigma) t \in \BT^{+} \ \ \text{and} \ \ \sigma \cdot h - \Psi_L^{v,c}(\sigma) h \in \BH^{+}. 
$$
\end{itemize}
\end{lem}

\begin{proof} Recall that $G_{L_v} = G_{\BQ_{p}} \subset G_{\BQ}$, so $I_v = I_p$.
It then follows from Theorem \ref{ordTate}(i) that $I_v$ acts trivially on $\BT^-$ and $\BH^-$. It follows from Theorem \ref{ordTate}(ii) and
the definitions of $\Psi_L^v$ and the canonical CM family that $I_{v}$ acts by $\Psi_L^v$ on $\BT^{+} = \BH^+$. 
As it is also clear from \eqref{CMgen} that $I_v$ does not act trivially on all of $\BV$, 
the image of $\BT^+$ in $\BV$, and hence in $\BT_1$, must be non-zero. The same applies to the image of $\BH^+$ in $\BH_1$.
Injectivity then follows from \eqref{subrk}. This proves part (i).

It follows from \eqref{CMgen} that $\BV$ has a unique $\bf{F}$-submodule of rank one on which $G_{L}$ acts via $\Psi_L^v$,
and the quotient of this module is unramified at $v$. It then follows from part (i) that $\BT^+$ and $\BH^+$ are $\Lambda_L^v$-lattices in this rank one submodule and hence necessarily
$G_L$-stable (with $G_L$ acting via $\Psi_L^v$). This is part (ii).

That $G_{L_v}$ acts on $\BT_1^-$ and $\BH_1^-$ via $\Psi_L^{v,c}$ follows from Theorem \ref{ordTate}(i). Part (iii) follows from this.
\end{proof}

\vskip 2mm
\noindent{\it Questions.} It is natural to ask:
\begin{equation}\label{Q1}\tag{Q1}
\text{Does the top exact sequence of \eqref{CM-ord} split as an exact sequence of $\Lambda_{L}^{v}$ or $\Lambda_{L}^{v}[G_{\BQ_{p}}]$-modules?}
\end{equation}
\begin{equation}\label{Q2}\tag{Q2}
\text{Are $\BT_1$ and $\BT_1^-$ free $\Lambda_L^v$-modules (of rank two and one, respectively)?}
\end{equation}
Determining the answers to these questions occupies much of the rest of this section as well as the next.
As explained in the introduction, these answers lie at the heart of this paper (and many subsequent applications).
The answer to the analogous questions for CM Hida families satisfying a $p$-distinguished hypothesis is known to be `yes' (see for example \cite[Thm. 3.4]{L}).

\subsubsection{Reflexive closure of the Tate lattices}\label{ssNot}
Recall that for a finite $\Lambda_{L}^{v}$-module $M$, the reflexive closure $\widetilde{M}$ of $M $ is 
$$
\widetilde{M}=\bigcap_{P \in \height_{1}(\Lambda_{L}^{v})} (M/M_{\tor})_{P} \subset M\otimes_{\Lambda_{L}^{v}} {\bf{F}},
$$
where $\height_{1}(\cdot)$ is the set of height one prime ideals \cite[Def. 23.1 \& Lem. 23.18]{St}. 
As $\Lambda_{L}^{v}$ is a two-dimensional regular local ring, $\widetilde{M}$ is a free $\Lambda_{L}^v$-module. 
There is a canonical morphism 
\begin{equation}\label{rfm}
\varphi_{M}: M \ra \widetilde{M}
\end{equation}
of $\Lambda_{L}^{v}$-modules. The kernel of $\varphi_M$ is the $\Lambda_L^v$-torsion submodule $M_\tor$, and the cokernel is a pseudo-null $\Lambda_L^v$-module 
(that is, its localization at any $P \in \height_{1}(\Lambda_{L}^{v})$ is zero). In particular, the cokernel has finite order.

Since $\BT_1$ and $\BH_1$ are torsion-free $\Lambda_L^v$-modules of rank two, it follows that
$$
\widetilde\BT = \widetilde\BT_1 \ \ \text{and} \ \ \widetilde\BH = \widetilde \BH_1, \ \text{and these are free $\Lambda_L^v$-modules of rank two.}
$$
We also have 
$$
\text{$\BT^+ = \widetilde\BT^+= \widetilde\BH^+ = \BH^+$, $\widetilde\BT^- = \widetilde\BT^-_1$, and $\widetilde\BH^- = \widetilde\BH_1^-$ are free $\Lambda_L^v$-modules of rank one.}
$$
Note that $\widetilde\BT$ and $\widetilde\BH$ are $\Lambda_L^v[G_L]$-modules, with the $G_L$-action arising from that on $\BT_1$ and $\BH_1$, respectively.
Similarly, $\widetilde\BT_1^-$ and $\widetilde\BH_1$ are $\Lambda_L^v[G_{L_v}]$-modules; clearly, both $\widetilde\BT_1^-$ and $\widetilde\BH_1^-$ are isomorphic to $\Lambda_L^v(\Psi_L^{v,c})$.
The induced sequences $\BT^+ \rightarrow \widetilde\BT \rightarrow \widetilde \BT_1^-\rightarrow 0$ 
and $\BH^+ \rightarrow \widetilde\BH \rightarrow \widetilde \BH_1^-\rightarrow 0$ of $\Lambda_L^v[G_{L_v}]$-modules need not be exact. 
However, we do have the following.

\begin{lem}\label{subchar} For $t \in \widetilde{\BT}$, $h\in \widetilde\BH$, and $\sigma \in G_{L_{v}}$, we have
$$
\sigma \cdot t - \Psi_L^{v,c} (\sigma) t \in \bigcap_{P \in \height_{1}(\Lambda_{L}^{v})} \BT^{+}_{P} = \widetilde\BT^+ = \BT^{+} \ \ \text{and} \ \ 
\sigma \cdot h -\Psi_L^{v,c} (\sigma) h \in \bigcap_{P \in \height_{1}(\Lambda_{L}^{v})} \BH^{+}_{P} = \widetilde\BH^+ = \BH^{+}.
$$
\end{lem}
\noindent This is an immediate consequence of Lemma \ref{CMTate}(iii). 

\begin{lem}\label{LocInd}
For $P \in \height_{1}(\Lambda_{L}^v)$ with $P \neq (T_{v})$, there is an isomorphism
$$
\widetilde{\BT}_{P}
\simeq \Ind_{G_{L}}^{G_{\BQ}}(\Lambda_{L,P}^{v}(\Psi_L^v))
$$
of $\Lambda_{L,P}^{v}[G_{\BQ}]$-modules. 
\end{lem}
\noindent Here $T_{v}\in \Lambda_L^v$ is the variable as in \S \ref{MoreIwasawa}, and so $(T_v)$ is the kernel of the augmentation map $\Lambda_L^v\twoheadrightarrow \BZ_p$. 

\begin{proof} 
The claim in the lemma is equivalent to $\widetilde\BT\simeq \Lambda_{L,P}^{v}(\Psi_L^v)\oplus \Lambda_{L,P}^{v}(\Psi_L^{v,c})$ as $\Lambda_{L,P}^v[G_L]$-modules.
Since this splitting holds over $\bf{F}$ by \eqref{CMgen}, it also holds over $\Lambda_{L,P}^v$ if $\Psi_L^v\nequiv \Psi_L^{v,c}\mod{P}$. 

Let $\chi_P = \Psi_L^v\mod{P}$. If $\chi_P= \chi_P^c$, then the fixed field of the kernel of $\chi_P$ is contained in $L^v\cap L^\cyc = L$. That is, $\chi_P$ is trivial, and 
so $P=(T_v)$.
\end{proof}

\vskip 2mm
\noindent{\it{More questions}.} It is natural to ask:
\begin{equation}\label{Q3}\tag{Q3}
\text{Is $\widetilde{\BT} \simeq \Ind_{G_{L}}^{G_{\BQ}}(\Lambda_{L}^{v}(\Psi_L^v))$?}
\end{equation}
Or even
\begin{equation}\label{Q4}\tag{Q4}
\text{Is $\widetilde{\BT} \simeq \Ind_{G_{L}}^{G_{\BQ}}(\BT^+)$?}
\end{equation}

Answering 
these questions is key to our answering questions \eqref{Q1} and \eqref{Q2} (and to subsequent applications).

\subsubsection{Saturation of $\BT^+$ and $\BH^+$}\label{saturation}
A priori, the $G_{L}$-stable $\Lambda_{L}^{v}$-submodule $\BT^{+}=\BH^+$ might not be saturated in $\widetilde\BT$ or $\widetilde\BH$.
Let
\begin{equation}\label{SatClo}
\BT_{v} = (\BT^{+} \otimes_{\Lambda_{L}^{v}} {\bf{F}}) \cap \widetilde{\BT} \ \ \text{and} \ \ \BH_{v} = (\BH^{+} \otimes_{\Lambda_{L}^{v}} {\bf{F}}) \cap \widetilde{\BH}
\end{equation}
be their respective saturations.  
Note that it is not a priori clear that $\BT_v$ equals $\BH_v$. 
Note also that $\BT_v$ and $\BH_v$ are $G_L$-stable submodules and that the $G_L$-action is via $\Psi_L^v$.
Let also
$$
\BT_{\bar v} = c\cdot \BT_v \subset \widetilde\BT \ \ \text{and} \ \ \widetilde\BT^v = \widetilde\BT/\BT_v.
$$
We similarly define $\BH_{\bar v}$ and $\widetilde\BH^v$.
Both  $\BT_{\bar v}$ and $\BH_{\bar v}$ are $G_L$-stable submodules with $G_L$-action via $\Psi_L^{v,c}$.

\begin{lem}\label{satC}\hfill
\begin{itemize}
\item[(i)] The submodules $\BT_{v}$ and $\BH_v$ are free $\Lambda_{L}^{v}$-modules of rank one. In particular, both are isomorphic to $\Lambda_L^v(\Psi_L^v)$ as
$\Lambda_L^v[G_L]$-modules.
\item[(ii)] The submodule $\BT_v$ (resp.~$\BH_v$) is a $\Lambda_{L}^v$-summand of $\widetilde \BT$ (resp,~$\widetilde\BH$).  
Consequently, $\widetilde\BT^v$ and $\widetilde\BH^v$ are both free $\Lambda_L^v$-modules of rank one, and both are isomorphic to 
$\Lambda_L^v(\Psi_L^{v,c})$ as $\Lambda_L^v[G_L]$-modules. 
\item[(iii)] The image of $\BT_{\bar v}$ (resp.~$\BH_{\bar v}$) in the quotient $\widetilde\BT^v$ (resp.~$\widetilde\BH^v$) is either
$\widetilde\BT^v$ or $T_v\cdot \widetilde\BT^v$ (resp.~$\widetilde\BH^v$ or $T_v\cdot \widetilde\BH^v$).
\end{itemize}
\end{lem}
\noindent Here again, $T_{v}\in\Lambda_L^v$ is as in \S\ref{MoreIwasawa}. 

\begin{proof}
Part (i) is clear if we can show that $\BT_v$ and $\BH_v$ are free $\Lambda_L^v$-submodules.
We have $\BT_v \subset \widetilde\BT_v\subset \widetilde\BT$. Since the quotient $\widetilde\BT_v/\BT_v$ is pseudo-null, $\widetilde\BT_v$ is contained in the saturation of $\BT_v$ in $\widetilde\BT$.
But $\BT_v$ is saturated in $\widetilde\BT$ by definition. Hence $\BT_v = \widetilde\BT_v$, and 
so $\BT_v$ is a free $\Lambda_L^v$-module. The same argument applies to $\BH_v$.

Let $\{e_{+}, e_{-}\} \subset \widetilde{\BT}$ be a $\Lambda_{L}^{v}$-basis such that 
$$
c\cdot e_{\pm} = \pm e_{\pm}.
$$
Let $e_{v}$ be a basis of $\BT_{v}$. Then $e_{\bar v} = c\cdot e_v$ is a basis of the $\Lambda_{L}^v$-module 
$$
\BT_{\bar{v}} = c \cdot \BT_{v} \subset \widetilde{\BT}.
$$
Note that by part (i), $\BT_{\bar v} \simeq \Lambda_L^v(\Psi_L^{v,c})$ as a $\Lambda_L^v[G_L]$-module.

There exist $a,b\in \Lambda_L^v$ such that
$$
e_{v} = ae_{+}+be_{-} \ \ \text{and} \ \  e_{\bar{v}} = c\cdot e_{v} = ae_{+}-be_{-}.
$$
So for $\gamma\in G_L$, we have 
$$
(\gamma-\Psi_L^{v,c}(\gamma)) e^+  = \frac{\Psi_L^v(\gamma)-\Psi_L^{v,c}(\gamma)}{2a}\cdot e_v \ \ 
\text{and} \ \ 
(\gamma-\Psi_L^{v,c}(\gamma)) e^-  = \frac{\Psi_L^v(\gamma)-\Psi_L^{v,c}(\gamma)}{2b}\cdot e_v.
$$
Since $\BT_v$ is saturated, this means
$$
a, b\mid (\Psi_L^v(\gamma)-\Psi_L^{v,c}(\gamma)) \ \text{for all} \ \gamma\in\Gamma_L.
$$
We claim that the ideal $I=\langle \Psi_L^v(\gamma)-\Psi_L^{v,c}(\gamma) \ : \ \gamma\in\Gamma_L\rangle\subset \Lambda_L^v$ is just $(T_v)$.
To see this, note that the character $\Psi_L^v\mod{I}$ is just the composition 
$$
G_L\twoheadrightarrow \Gamma_L \twoheadrightarrow \Gamma_L^v \ra (\Lambda_L^v/I)^\times,
$$
where the arrows are all the canonical projections. The kernel of this character contains $\Gamma^-$, so the fixed field of the kernel
is contained in $L_\infty^\cyc\cap L_\infty^v = L$. This means that $I$ must be the augmentation ideal $(T_v)$ of $\Lambda_L^v$.

It then follows that $a,b\mid T_v$. On the other hand, since $e_v$ is a basis of a saturated module, $a$ and $b$ cannot have a common divisor. 
So at least one of $a$ and $b$ is a unit in $\Lambda_L^v$. This implies that $e_v$ has non-trivial image in $\widetilde\BT/\mathfrak{m}_L^v\widetilde\BT$,
where $\mathfrak{m}_L^v\subset\Lambda_L^v$ is the maximal ideal, which implies part (ii) for $\BT_v$.  

If both $a$ and $b$ are units, then $\widetilde \BT/(\BT_v\oplus\BT_{\bar v}) = 0$. While if only one of $a$ and $b$ is a unit, then the other is a unit times $T_v$,
whence $\widetilde \BT/(\BT_v\oplus\BT_{\bar v}) \simeq \Lambda_L^v/T_v\Lambda_L^v$. This, together with part (ii), implies part (iii) for $\BT_{\bar v}$.

The same arguments applied to $\BH_v\subset \widetilde\BH$ yield the claims in (ii) and (iii) for $\BH_v$ and $\BH_{\bar v}$.
\end{proof}

Since $\BT^+=\BH^+\subseteq \BT_v\subseteq\BH_v$, and these are all free $\Lambda_L^v$-modules of rank one, we have
$$
\BT^+ = C_v\cdot \BT_v, \ \ \BT_v = D_v\cdot\BH_v, \ \ \BH^+ = C_v D_v\cdot \BH_v,  \ \ \text{for some} \ \ C_v, D_v\in\Lambda_L^v.
$$
We will eventually show that $C_v$ is a unit (that is, $\BT^+ = \BT_v$) and $D_v\mid T_v$. As a first step we have the following. 

\begin{lem}\label{SatClo-2} 
Let $p^{h_p}$ be the index of $\ov{\Gamma}_v$ in $\Gamma_v$ as in \S\ref{Zp-ext}. Then $C_vD_v\mid ((1+T_v)^{p^{h_p}}-1)$.
\end{lem}

\begin{proof} Arguing as in the proof of part (ii) of Lemma \ref{satC}, but replacing $\widetilde\BT$ with $\widetilde\BH$ and and using that $\sigma\cdot h - \Psi_L^{v,c}(\sigma)h \in \BH^+$ for all 
$\sigma\in G_{L_v}$ and $h\in\widetilde \BH$ by Lemma \ref{subchar}, we find that 
$$
C_vD_v \mid (\Psi_L^v(\gamma)-\Psi_L^{v,c}(\gamma)) \ \ \text{for all $\gamma\in G_{L_v}$}.
$$
Let  $J=\langle \Psi_L^{v}(\gamma)-\Psi_L^{v,c}(\gamma) \ : \ \gamma\in G_{L_v}\rangle\subset \Lambda_L^v$ and consider $\Psi_L^v\mod J$, which is just the composite character
$$
G_L\twoheadrightarrow \Gamma_L \twoheadrightarrow \Gamma_L^v \rightarrow (\Lambda_L^v/J)^\times.
$$
Since $\Psi_L^{v,c}(I_v) = 1$, it follows that $\langle \Psi_L^v(\gamma) -1 \ : \ \gamma\in I_v\rangle\subset J$.
So the fixed field $L_v$ of the kernel of $\Psi_L^v\mod J$ is unramified at both $v$ and $\bar v$, and so contained in $L_\infty^v \cap L_\infty^{\bar v}\subset L_\infty^\ac$.
It follows that for all $\gamma\in G_{L_v}$, $\gamma$ equals $c\gamma c^{-1}$ on $L_v$ (by definition of $J$) and equals $c\gamma^{-1}c^{-1}$ on $L_v$ (since $L_v \subset L_\infty^\ac$).
Hence $G_{L_v}$ fixes $L_v$. That is, $J$ contains the ideal $\langle \gamma - 1 \ : \ \gamma\in \ov\Gamma_v\rangle = ((1+T_v)^{p^{h_p}}-1)$, from which the desired divisibility follows.
\end{proof}

To prove that $D_v\mid T_v$ and $C_v$ is a unit, we make use of some of the results proved in \cite{BD} and \cite{BDP'}.

\subsubsection{Some consequences of \cite{BD} and \cite{BDP'}}
Let $P\in\height_{1}(\Lambda_{L}^v)$ be a height one prime dividing $(1+T_v)^{p^{h_p}}-1$. 
Then $P=(\Phi_{p^r}(1+T_v))$ for some $0\leq r\leq h_p$, where $\Phi_{p^r}(X)$ is the $p^r$th cyclotomic polynomial.
Let $\zeta_{p^r}$ be a $p^r$th root of unity. Let $P_r\subset\Lambda_L^v\otimes_{\BZ_p}\ov{\BQ}_p$ be the prime above $P$ that is the kernel of the homomorphism
$\lambda_r:\Lambda_L^v\rightarrow \ov{\BQ}_p$, $T_v\mapsto (\zeta_{p^r} -1)$.  
Note that $P_r\cap (\Lambda_0\otimes_{\BZ_p}\ov{\BQ}_p)$ is just the prime $(X)$, and that the natural maps induce isomorphisms
$$
\ov{\BQ}_p[\![X]\!] = {(\Lambda_0\otimes_{\BZ_p}\ov{\BQ}_p)}^\wedge_{(X)} \isoarrow {(\Lambda_L^v\otimes_{\BZ_p}\ov{\BQ}_p)}^\wedge_{P_r},
$$
where the superscript '$\wedge$' denotes completion. 

Let $Q\subset\BH^\ord_{D_Lp^\infty}\otimes_{\BZ_p}\ov{\BQ}_p$ be the kernel of the composition
$$
\varphi_r:\BH^\ord_{D_Lp^\infty}\otimes_{\BZ_p}\ov{\BQ}_p\stackrel{\varphi\otimes id}{\longrightarrow}\Lambda_L^v\otimes_{\BZ_p}\ov{\BQ}_p\stackrel{\lambda_r}{\rightarrow}\ov{\BQ}_p.
$$
Note that the homomorphism $\varphi_r$ factors through $\BT^\ord_{D_Lp^\infty}\otimes_{\BZ_p}\ov{\BQ}_p$ (since $\varphi$ factors through $\BT_{D_Lp^\infty}^\ord$).
Moreover, $\varphi_r$ gives the eigenvalues for the action of the Hecke operators on the specialization $g_{v,r} = g_{v,\lambda_r}$ of the canonical CM family $\bf{h}_v$.
The specialization $g_{v,r}$ is the $p$-stabilization of a weight one CM  eigenform of level $D_L$. 
If $r=0$, $g_{v,r}$ is an Eisenstein series associated with 
the pair of characters $1,\chi_L$. If $r\neq 0$, $g_{v,r}$ is a CM cuspform associated to the Hecke character 
$\psi_r = \Psi_L^v \mod P_r$ (note that $\psi_r^c = \psi^{-1}_r \neq \psi_r$).

Let 
$$
\CR = {(\BH^\ord_{D_Lp^\infty}\otimes_{\BZ_p}\ov{\BQ}_p)}^\wedge_{Q} \ \ \text{and} \ \ \CR^0 = {(\BT^\ord_{D_Lp^\infty}\otimes_{\BZ_p}\ov{\BQ}_p)}^\wedge_{Q}.
$$
These are $\ov{\BQ}_p[\![X]\!]$-algebras. Let also
$$
M = {(M_{\Lambda_0}\otimes_{\BZ_p}\ov{\BQ}_p)}^\wedge_{Q}, \ \ \text{and} \ \ S = {(S_{\Lambda_0}\otimes_{\BZ_p}\ov{\BQ}_p)}^\wedge_{Q}.
$$
To apply results of \cite{BDP'} and \cite{BD}, we first explain the connection between these objects and those considered in {\it op.~cit.}

Suppose first that $r=0$. Then we take $N=D_L$ and $\phi = \chi_L$ in \cite{BDP'}. 
The ring ${(\Lambda_0\otimes_{\BZ_p}\ov{\BQ}_p)}^\wedge_{(X)}$  (which is just $\ov{\BQ}_p[\![X]\!]$) is identified with the completion of the local ring of weight space
at the point corresponding to the character $1\times \chi_L$ of $\BZ_p^\times\times (\BZ/D_L\BZ)^\times$, that is, the ring $\Lambda$ of \cite[\S3.4]{BDP'}. 
The ring $\CR$ (resp. $\CR^0$ )is just the ring denoted $\CT^{\rm{full}}$ 
(resp. $\CT^{\rm{full}}_{\rm{cusp}}$) in \cite[\S4.2]{BDP'} (cf.~\cite[\S3.2]{BD}). By \cite[Props.~4.4 \& 5.5]{BDP'}, $\CT = \CT^{\rm{full}}$ and $\CT_{\rm{cusp}} = \CT^{\rm{full}}_{\rm{cusp}}$,
so $\CR  = \CT$ and $\CR^0 = \CT_{\rm{cusp}}$.
The $\CR$-module $M$ is just the module $M^\dagger_{\mathfrak{m}_f}$ of \cite[\S4.2]{BDP'} for $f = g_{v,0}$, and the $\CR^0$-module $S$ is just $S^\dagger_{\mathfrak{m}_f}$.
The canonical CM family ${\bf{h}}_v$ belongs to $S$. There are two Hida--Eisenstein families ${\bf{E}}_1, {\bf{E}}_2 \in M$ associated with the  pair of characters $1,\chi_L$ (denoted $\mathcal{E}_{1,\chi_L}$ and $\mathcal{E}_{\chi_L,1}$ in \cite{BDP'}). 
From the results in \cite{BDP'} we then conclude the following.
\begin{prop}\label{Str-prop1} If $P=(T_v)$, then
\begin{itemize}
\item[(i)] $S = \ov{\BQ}_p[\![X]\!] \cdot {\bf{h}}_v$,
\item[(ii)] $M = \ov{\BQ}_p[\![X]\!] \cdot {\bf{h}}_v \oplus \ov{\BQ}_p[\![X]\!] \cdot\frac{{\bf{E}}_1-{\bf{h}}_v}{X}\oplus \ov{\BQ}_p[\![X]\!] \cdot\frac{{\bf{E}}_2-{\bf{h}}_v}{X}$.
\item[(iii)] Let $I_{i}\subset \CR$ be the annihilator of ${\bf{E}}_i$. Then $I_i\cdot {\bf{h}}_v = X\ov{\BQ}_p[\![X]\!]\cdot {\bf{h}}_v$.
\end{itemize}
\end{prop}

\begin{proof} 
Parts (i) and (ii) follow from \cite[Cor.~4.5 \& Prop.~5.4(ii)]{BDP'}.

Part (iii) follows from \cite[Prop.~4.2]{BDP'}, after noting 
that the analysis in {\it op.~cit.} also applies with the Eisenstein family $\mathcal{E}_{\chi_L,1}$ replacing $\mathcal{E}_{1,\chi_L}$.
\end{proof}

Suppose now that $0<r\leq h_p$. We take $\psi = \psi_r$ (the conductor is $\mathfrak{c}_\psi = (1)$) 
and $f = g_{v,r}$ in \cite{BD}. Again ${(\Lambda_0\otimes_{\BZ_p}\ov{\BQ}_p)}^\wedge_{(X)}= \ov{\BQ}_p[\![X]\!]$ is identified with the completion of the local 
ring of weight space $\CW$ at the weight of $f$, that is, with the ring $\Lambda$ as in \cite[\S3.1]{BD}. 
The ring $\CR^0$ is just the completion of the local ring of the full eigencurve $\CE_{\rm{full}}$ of tame level $N=D_L$ at the point corresponding to $f$ (cf. the first paragraph of
\cite[\S 3.2]{BD}), which we denote by $\CT^{\mathrm{full}}$, adopting the superscript `full' from \cite{BelD,BDP'}.  
Let $\CT$ be the completion of the local ring of the eigencurve at the point corresponding to $f$. 
The argument from \cite[\S7]{BelD} is readily adapted to the case at hand to show that $\CT = \CT^{\rm{full}}$. That is, $\CR^0=\CT$.

\begin{prop}\label{Str-prop2} If $P\neq (T_v)$, then the minimal number of $\CR^0$-generators of $S$ is at most $3$.
\end{prop}

\begin{proof} 
We first note that $S \simeq \Hom_{\ov{\BQ}_p[\![X]\!]}(\CR^0, \ov{\BQ}_p[\![X]\!])$ 
as $\CR^0$-modules. This follows from the duality \eqref{HidaDu}.  In particular, it suffices 
to show that $\Hom_{\ov{\BQ}_p[\![X]\!]}(\CR^0, \ov{\BQ}_p[\![X]\!])$ is generated by at most $3$ elements as an $\CR^0$-module.
A description of $\CR^0$ is provided by \cite[Thm.~B]{BD}, and a straight-forward case-by-case analysis then yields the desired bound.
\end{proof}

From Proposition \ref{Str-prop1} we conclude the following:
\begin{prop}\label{Sat-prop1} Suppose $P=(T_v)$.  The following hold:
\begin{itemize} 
\item[(i)] $\BT^+_P = \BT_{v,P}$, $\BT_P = \BT_{1,P}$ and either $\BH^+_P = \BH_{v,P}$ or $\BH^+_P = T_v\cdot \BH_{v,P}$;
\item[(ii)] $(\BH_1^-/\BT_1^-)_P \simeq (\Lambda_{L,P}^v/P)^s$ for $s=1$ or $2$;
\item[(iii)] If $\BH^+_P = \BH_{v,P}$, then $(\BH_1/\BT_1)_P \simeq (\BH_1^-/\BT_1^-)_P \simeq (\Lambda_{L,P}^v/P)$;
\item[(iv)] If $\BH^+_P = T_v\cdot \BH_{v,P}$, then $\BT_{1,P} = T_v\cdot \BH_{1,P}$.
\end{itemize}
\end{prop}
\noindent Note that (i) means that $T_v\nmid C_v$. 

\begin{proof}
Let 
$$
\CH = {(\CH^1_\ord(D_Lp^\infty)\otimes_{\BZ_p}\ov{\BQ}_p)}^\wedge_Q,  \ \ H = {(H^1_\ord(D_Lp^\infty)\otimes_{\BZ_p}\ov{\BQ}_p)}^\wedge_Q,
$$
$$
\CH^\pm = ({\CF^\pm\CH^1_\ord(D_Lp^\infty)\otimes_{\BZ_p}\ov{\BQ}_p)}^\wedge_Q \ \ \text{and} \ \ H^\pm = {(\CF^\pm H^1_\ord(D_Lp^\infty)\otimes_{\BZ_p}\ov{\BQ}_p)}^\wedge_Q.
$$
Note that $H^+ =  \CH^+ \simeq \CR^0$, $H^- \simeq S \simeq \Hom_{\ov{\BQ}_p[\![X]\!]}(\CR^0, \ov{\BQ}_p[\![X]\!])$, 
$\CH^- \simeq M$, and there exists a commutative diagram of $\CR$-modules
$$
\begin{tikzcd}[row sep=2.5em]
0 \arrow[r,] & H^+ \arrow[r,] \arrow[d,"="] & H \arrow[r,] \arrow[d,hook] & H^- \arrow[r,]\arrow[d,hook] & 0 \\
0 \arrow[r,] & \CH^+ \arrow[r,] & \CH  \arrow[r,] & \CH^- \arrow[r,] & 0 
\end{tikzcd}
$$
induced by the commutative diagram in Theorem \ref{ordTate}. Since
$$
\CH\otimes_{\CR,\varphi_r}\ov{\BQ}_p[\![X]\!]= \widehat\BH_P\otimes_{\widehat\Lambda_{L,P}}\ov{\BQ}_p[\![X]\!],  \ H\otimes_{\CR,\varphi_r}\ov{\BQ}_p[\![X]\!]= \widehat\BT_P\otimes_{\widehat\Lambda_{L,P}}\ov{\BQ}_p[\![X]\!],
$$
$$
\CH^\pm\otimes_{\CR,\varphi_r}\ov{\BQ}_p[\![X]\!] = \widehat\BH^\pm_P\otimes_{\widehat\Lambda_{L,P}}\ov{\BQ}_p[\![X]\!], \ \ \text{and} \ \ 
H^\pm\otimes_{\CR,\varphi_r}\ov{\BQ}_p[\![X]\!] = \widehat\BT^\pm_P\otimes_{\widehat\Lambda_{L,P}}\ov{\BQ}_p[\![X]\!],
$$
applying $\otimes_{\CR,\varphi_r}\ov{\BQ}_p[\![X]\!]$ to the preceding commutative diagram yields
a commutative diagram of $\ov{\BQ}_p[\![X]\!]$-modules 
$$
\begin{tikzcd}[row sep=2.5em]
0 \arrow[r,] & (\widehat\BT^+_P\otimes_{\widehat\Lambda_{L,P}^v}\ov{\BQ}_p[\![X]\!]\simeq\ov{\BQ}_p[\![X]\!]) \arrow[r,] \arrow[d,"="] & \widehat\BT_P\otimes_{\widehat\Lambda_{L,P}^v}\ov{\BQ}_p[\![X]\!] \arrow[r,] \arrow[d,hook] & \widehat\BT^-_P\otimes_{\widehat\Lambda_{L,P}^v}\ov{\BQ}_p[\![X]\!]  \arrow[r,]\arrow[d,hook] & 0 \\
0 \arrow[r,]& (\widehat\BH^+_P\otimes_{\widehat\Lambda_{L,P}^v}\ov{\BQ}_p[\![X]\!]\simeq\ov{\BQ}_p[\![X]\!]) \arrow[r,] & \widehat\BH_P\otimes_{\widehat\Lambda_{L,P}^v}\ov{\BQ}_p[\![X]\!]  \arrow[r,] & \widehat\BH^-_P\otimes_{\widehat\Lambda_{L,P}^v}\ov{\BQ}_p[\![X]\!]  \arrow[r,] & 0 
\end{tikzcd}
$$
The exactness on the left of the rows was established in Lemma \ref{CMTate}(i).

From Proposition \ref{Str-prop1}(i) it follows that $\widehat\BT_P^-\otimes_{\widehat\Lambda_{L,P}^v}\ov{\BQ}_p[\![X]\!]\simeq S\otimes_{\CR,\varphi_r}\ov{\BQ}_p[\![X]\!]\cong \ov{\BQ}_p[\![X]\!]$. 
It then follows from the exactness of the top row of the above
commutative diagram that $\widehat\BT_P\otimes_{\widehat\Lambda_{L,P}^v}\ov{\BQ}_p[\![X]\!]$ is a free $\ov{\BQ}_p[\![X]\!]$-module of rank two and that $\widehat\BT_P^+\otimes_{\widehat\Lambda_{L,P}^v}\ov{\BQ}_p[\![X]\!]$ is a direct summand, that is,
$\widehat\BT_P^+\otimes_{\widehat\Lambda_{L,P}^v}\ov{\BQ}_p[\![X]\!] = \widehat\BT_{v,P}\otimes_{\widehat\Lambda_{L,P}^v}\ov{\BQ}_p[\![X]\!]$. 
As $\ov{\BQ}_p[\![X]\!]$ is faithfully flat over $\widehat\Lambda_{L,P}^v$ and 
so over $\Lambda_{L,P}^v$, it then follows
that $\BT_P$ is a free $\Lambda_{L,P}^v$-module of rank $2$ and $\BT_P^+ = \BT_{v,P}$. Part (i) follows from this together with Lemma \ref{SatClo-2}

From Proposition \ref{Str-prop1}(ii),(iii) it follows that 
$\widehat\BH_P\otimes_{\widehat\Lambda_{L,P}^v}\ov{\BQ}_p[\![X]\!]\simeq M\otimes_{\CR,\varphi_r}\ov{\BQ}_p[\![X]\!] \simeq \ov{\BQ}_p[\![X]\!]\oplus \ov{\BQ}_p$ such that the image of 
$\widehat\BT_P\otimes_{\widehat\Lambda_{L,P}^v}\ov{\BQ}_p[\![X]\!]\simeq S\otimes_{\CR,\varphi_r}\ov{\BQ}_p[\![X]\!]$ is $X \ov{\BQ}_p[\![X]\!]$. Hence
$(\widehat\BH^-_P/\widehat\BT^-_P)\otimes_{\widehat\Lambda_{L,P}^v}\ov{\BQ}_p[\![X]\!] \cong \ov{\BQ}_p^2$ and
$(\widehat\BH^-_{/\tor,P}/\widehat\BT^-_{/\tor,P})\otimes_{\widehat\Lambda_{L,P}^v}\ov{\BQ}_p[\![X]\!] \cong \ov{\BQ}_p$.
It follows that $\BH^-_P/\BT^-_P \cong (\Lambda_{L,P}^v/T_v\Lambda_{L,P}^v)^2$
and $\BH^-_{/\tor,P}/\BT^-_{/\tor,P} \cong (\Lambda_{L,P}^v/T_v\Lambda_{L,P}^v)$.
Since $\BH^-/\BT^-\twoheadrightarrow \BH^-_{1}/\BT^-_{1}\twoheadrightarrow \BH^-_{/\tor}/\BT^-_{/\tor}$,
part (ii) follows. 

Parts (iii) and (iv) follow from (i) and (ii) and the fact that $\BH_{1,P}$ is a free $\Lambda_{L,P}^v$-module of rank two.

\end{proof}

Using Proposition \ref{Str-prop2} we deduce:
\begin{prop}\label{Sat-prop2}
If $P\neq (T_v)$, then $\BT^+_P = \BT_{v,P}$.
\end{prop}

\begin{proof} We will show that $\widehat\BT_P^+\otimes_{\widehat\Lambda_{L,P}^v}\ov{\BQ}_p[\![X]\!] = \widehat\BT_{v,P}\otimes_{\widehat\Lambda_{L,P}^v}\ov{\BQ}_p[\![X]\!]$,
from which the proposition follows.
Since $\widehat\BT_P^+\otimes_{\widehat\Lambda_{L,P}^v}\ov{\BQ}_p[\![X]\!]\simeq \ov{\BQ}_p[\![X]\!]$, it suffices to show that the minimal number of $\ov{\BQ}_p[\![X]\!]$-generators of $\widehat\BT_P\otimes_{\widehat\Lambda_{L,P}^v}\ov{\BQ}_p[\![X]\!]$ is 
one more than the minimal number of $\ov{\BQ}_p[\![X]\!]$-generators of $\widehat\BT_P^-\otimes_{\widehat\Lambda_{L,P}^v}\ov{\BQ}_p[\![X]\!]$. 
This is equivalent to the minimal number of $\CR^0$-generators of $H$ being one more than the minimal number of $\CR^0$-generators of $H^-$, where $H$ and $H^\pm$ are defined just as in 
the proof of Proposition \ref{Sat-prop1}. 

Let $H^{c=\pm} = {(H^1_\ord(D_Lp^\infty)^\pm\otimes_{\BZ_p}\ov{\BQ}_p)}^\wedge_Q$. It then follows from Theorem \ref{ordTate-dual} that
$$
H = H^{c=+} \oplus H^{c=1} \ \ \text{and} \ \ H^{c=\pm} \simeq \Hom_{\ov{\BQ}_p[\![X]\!]}(H^{c=\mp},\ov{\BQ}_p[\![X]\!])
$$
as $\CR^0$-modules. Note that $H^{c=\pm}$ is necessarily non-zero for both signs.

Suppose the minimal number of $\CR^0$-generators of $H$ equals the minimal number of $\CR^0$-generators of $H^-$. Since $H^-\simeq S$,  by Proposition \ref{Str-prop2},
the latter is at most 3. Hence for some choice of sign $\epsilon$, the minimal number of generators of $H^{c=\epsilon}$ is 1. Comparing $\ov{\BQ}_p[\![X]\!]$-ranks then yields
$H^{c=\epsilon} \cong \CR^0$ as $\CR^0$-modules, and so $H^{c=-\epsilon}\cong \Hom_{\ov{\BQ}_p[\![X]\!]}(\CR^0,\ov{\BQ}_p[\![X]\!])$ as $\CR^0$-modules. But this implies that the minimal 
number of $\CR^0$-generators of $H$ is one more than the minimal number of $\CR^0$-generators of $H^{c=-\epsilon}\cong \Hom_{\ov{\BQ}_p[\![X]\!]}(\CR^0,\ov{\BQ}_p[\![X]\!])
\simeq S \simeq H^-$, a contradiction.
\end{proof}

Finally, from Propositions \ref{Sat-prop1} and \ref{Sat-prop2} we immediately deduce the following.

\begin{prop}\label{Sat-prop3} The following hold:
\begin{itemize}
\item[(i)] $\BT^+ = \BT_v$ (that is, $C_v$ is a unit);
\item[(ii)] $\BH^+ = \BH_v$ or $\BH^+ = T_v\cdot \BH_v$ (that is, $D_v\mid T_v$);
\item[(iii)] If $\BH^+ = \BH_v$ then $\widetilde\BH/\widetilde\BT \simeq \Lambda_L^v/(T_v)$;
\item[(iv)] If $\BH^+ = T_v\cdot \BH_v$, then $\widetilde\BT = T_v\cdot \widetilde\BH$.
\end{itemize}
\end{prop}

\subsubsection{Some results on submodules of induced representations}\label{ind-sec}
Let $\BV$ be as in \eqref{CMgen} and let $\BL \subset \BV$ be a $G_\BQ$-stable free $\Lambda_L^v$-submodule of rank two.  We will say that $\BL$ is an {\em induced module}
if $\BL\simeq \Ind_{G_L}^{G_\BQ}(\Lambda_L^v(\Psi_L^v))$.

Let $\BV_v \subset \BV$ be the unique $\mathbf{F}$-line on which
$G_L$-acts as $\Psi_L^v$ and let $\BL_v = \BL\cap \BV_v$. 
Since $\BL=\widetilde\BL$ as $\BL$ is a free $\Lambda_L^v$-module, $\BL_v = \widetilde\BL\cap \BV_v$ and so $\BL_v$ is saturated in $\widetilde\BL$,
from which it follows that $\BL_v= \widetilde\BL_v$. In particular, $\BL_v$ is a free $\Lambda_L^v$-module of rank one.  So $\BL_v\simeq \Lambda_L^v(\Psi_L^v)$.

Let
$$
\BL' = \BL_v \oplus c\cdot\BL_v \subseteq \BL.
$$
Then $\BL' $ is $G_\BQ$-stable. In particular, $\BL'\simeq \Ind_{G_L}^{G_{\BQ}}(\Lambda_L^v(\Psi_L^v))$, that is, $\BL'$ is an induced module.

\begin{lem}\label{ind-lem1} We have
\begin{itemize}
\item[(i)] $\BL \simeq \Ind_{G_L}^{G_{\BQ}}(\Lambda_L^v(\Psi_L^v))$ if and only if $\BL' = \BL$, and
\item[(ii)] if $\BL' \neq \BL$, then 
$\BL/\BL' \simeq \Lambda_L^v/T_v\Lambda_L^v$ and $\BL' = T_v\BL+\BL_v = T_v\BL + \BL_v'$.
\end{itemize}
\end{lem}

\begin{proof} 
If $\BL' = \BL$ then it is immediate that $\BL \simeq \Ind_{G_L}^{G_{\BQ}}(\Lambda_L^v(\Psi_L^v))$.
On the other hand, if $\BL \simeq \Ind_{G_L}^{G_{\BQ}}(\Lambda_L^v(\Psi_L^v))$ then it follows
from the definition of $\BL_v$ that $\BL$ is the $G_\BQ$-represenation induced from the  $G_L$-stable submodule $\BL_v$,
that is, $\BL = \BL_v\oplus c\cdot\BL_v = \BL'$. This proves part (i) of the lemma.

The proof of part (ii) is similar to that of Lemma \ref{satC}.
Let $e_+, e_-\in \BL$ be a $\Lambda_L^v$-basis such that $c$ acts on $e_\pm$ as $\pm 1$.
Let $e_v \in \BL_v$ be a $\Lambda_L^v$-generator. Then $e_v = a e_+ + b e_-$ for some
$a,b\in\Lambda_L^v$. Arguing as in the proof of {\it loc.cit.}~ shows that $a,b\mid T$ and that
at least one of $a,b$ is a unit. As $e_{\bar v}  = c\cdot e_v = ae_+ - b e_-$ is a generator of $c\cdot \BL_v$,
it follows that $a e_+, be_- \in \BL'$ and these generate $\BL'$. In particular, 
$$
\BL/\BL' \cong \Lambda_L^v/a\Lambda_L^v \oplus \Lambda_L^v/b\Lambda_L^v
$$
from which the first claim of part (ii) follows.
Since $a e_+$ and $be_-$ are both contained in the $\Lambda_L^v$-module
generated by $e_v = a e_++be_-$, $T_ve_+$, and $T_v e_-$ (which is just
$T_v\BL + \BL_v$) the second claim also holds. 
\end{proof}

Let 
$$
\BL'' = \frac{1}{T_v}\BL'.
$$
By Lemma \ref{ind-lem1}, $T_v\BL \subset \BL'$. Hence $\BL \subset\BL''$.

\begin{lem}\label{ind-lem2} The following hold: 
\begin{itemize}
\item[(i)] $\BL''\simeq \Ind_{G_L}^{G_\BQ}(\Lambda_L^v(\Psi_L^v))$,
\item[(ii)] $\BL_v = T_v \BL''_v$,
\item[(iii)] if $\BL'\neq\BL$, then the inclusion $\BL\subset \BL''$ induces an isomorphism
$\BL/\BL_v\isoarrow \BL''/\BL''_v$ and satisfies $\BL''/\BL\simeq\Lambda_L^v/T_v\Lambda_L^v$.
\end{itemize}
\end{lem}

\begin{proof} Parts (i) and (ii) are immediate from the definition of $\BL''$.
Suppose $\BL'\neq \BL$. Then by Lemma \ref{ind-lem1}(ii), $\BL' = T_v \BL + \BL'_v$. It follows that
$\BL'' = \BL + \BL''_v$ and hence that the map $\BL/\BL_v \ra \BL''/\BL''_v$ is a surjection.
Since both $\BL/\BL_v$ and $\BL''/\BL_v''$ are free $\Lambda_L^v$-modules of rank one,
the map is also an injection. This proves the first claim of part (iii). The second follows from
the first in combination with part (ii) and the snake lemma applied to the commutative diagram
$$
\begin{tikzcd}[row sep=2.5em]
0 \arrow[r,] & \BL_v \arrow[r,]\arrow[d,hook] & \BL \arrow[r,] \arrow[d,hook] &  \BL/\BL_v \arrow[r,]\arrow[d,"="] & 0 \\
0 \arrow[r,] & \BL''_v \arrow[r,] & \BL''  \arrow[r,] & \BL''/\BL''_v \arrow[r,] & 0.
\end{tikzcd}
$$ 
\end{proof}

\begin{lem}\label{ind-lem3}
Suppose $\BL_1\subsetneq\BL_2\subset \BV$ are $G_\BQ$-stable free $\Lambda_L^v$-submodules of rank two such that $\BL_{1,v} = \BL_{2,v}$. 
Then $\BL_1 = \BL_{1,v}\oplus c\cdot\BL_{1,v} \simeq \Ind_{G_L}^{G_\BQ}(\Lambda_L^v(\Psi_L^v))$ and 
$\BL_2/\BL_1 \simeq \Lambda_L^v/T_v\Lambda_L^v$. In particular,
the inclusion $\BL_1\subset\BL_2$ induces an identification of
$\BL_1/\BL_{1,v}$ with $T_v(\BL_2/\BL_{2,v})$.
\end{lem}

\begin{proof} Let $\BL' = \BL_{1,v}\oplus c\cdot\BL_{1,v} = \BL_{2,v}\oplus c\cdot\BL_{2,v}$. Then $\BL'\subseteq\BL_1\subsetneq\BL_2$.
Let $\BL'' = \frac{1}{T_v}\BL'$. 
It follows from Lemma \ref{ind-lem1}(ii) applied to $\BL_1$ and $\BL_2$ that $\BL_1 \subsetneq\BL_2 \subseteq \BL'' = \frac{1}{T_v}\BL'$. 
Moreover, if $\BL_1\neq \BL'$, then it would follow from Lemma \ref{ind-lem2}(iii) that $\BL''/\BL_1 \simeq \Lambda_L^v/T_v\Lambda_L^v\simeq \BL''/\BL_2$.
But combined with the surjection $\BL''/\BL_1\twoheadrightarrow \BL''/\BL_2$, 
this would then imply that $\BL_1 = \BL_2$, a contradiction.  So it must be that $\BL_1=\BL'$, as claimed in the lemma.
That $\BL_2/\BL_1\simeq \Lambda_L^v/T_v\Lambda_L^v$ then follows from Lemma \ref{ind-lem1}(ii).
\end{proof}

\subsubsection{Is $\widetilde\BT$ induced?}
One of the important ancillary results of this paper will show that
\begin{equation}\label{Ind-eq}\tag{{\bf Ind}}
\begin{split}
\bullet \ & \BT^+\oplus c\cdot \BT^+ \isoarrow \BT_v\oplus c\cdot\BT_v = \BT_1 = \widetilde\BT \\
\bullet \ & \BT_1 = \widetilde\BT \simeq \Ind_{G_L}^{G_\BQ}(\Lambda_L^v(\Psi_L^v)) \\ 
\bullet \ & \text{the $\Lambda_L^v[G_{\BQ_p}]$-surjection $\BT\twoheadrightarrow \BT_1$ splits} \\
\bullet \ & \text{$\BT_1/\BT^+ = \widetilde\BT/\BT_v$ is identified with $T_v(\widetilde\BH/\BH_v)$ in $\widetilde\BH/\BH_v = \widetilde\BH/\BH^+$.}
\end{split}
\end{equation}
Note that this will provide a positive answer to each of the questions \eqref{Q1}--\eqref{Q4}.

By combining the results in Section \ref{ind-sec} with Proposition \ref{Sat-prop3}, we deduce the following result towards establishing \eqref{Ind-eq}.

\begin{prop}\label{notind-prop1} The statements \eqref{Ind-eq} hold unless 
$\BH^+ = T_v\BH_v$, $\widetilde\BT = T_v \widetilde\BH$, and $\widetilde\BT \neq \BT_v\oplus c\cdot\BT_v$.
\end{prop}

\begin{proof} By Proposition \ref{Sat-prop3}(ii), either $\BH^+ = \BH_v$ or $\BH^+ = T_v\BH_v$. Suppose that the former holds. 
Then \eqref{Ind-eq} follows from combining Proposition \ref{Sat-prop3}(i),(iii) with 
Lemma \ref{ind-lem3}. The key point is that there are inclusions
$$
\BT_v\oplus c\cdot \BT = \BT^+ \oplus c\cdot\BT^+ \subseteq \BT_1 \subseteq \widetilde\BT\subsetneq
\widetilde\BH,
$$
so that the equality of the first and third modules (which follows from Lemma \ref{ind-lem3}) implies
equality of the first three.

Suppose then that $\BH^+ = T_v\BH_v$. Then by Proposition \ref{Sat-prop3}(iv),
$\widetilde\BT = T_v\widetilde\BH$.  There are two cases to consider: (1) $\widetilde\BT = \BT_v\oplus c\cdot\BT_v$ and (2) $\widetilde\BT = \BT_v\oplus c\cdot\BT_v$.  In case (1), that \eqref{Ind-eq} holds follows immediately from Proposition \ref{Sat-prop3}(i), so  $\BT_v\oplus c\cdot\BT_v = \BT^+\oplus c\cdot\BT^+\subseteq\BT_1$. These leaves case (2), which is exactly the exception in the proposition.
 \end{proof}

To help with the eventual exclusion of the exceptional case in Proposition \ref{notind-prop1},
we record the following application of the results from Section \ref{ind-sec}.

\begin{lem}\label{ind-I-lem} There exists an induced lattice 
$\widetilde \BT \subset \BI\subsetneq \widetilde\BH$ such that 
either (a) $\BT = \BI$ and \eqref{Ind-eq} holds, or 
(b) $\BT\neq \BI$, \eqref{Ind-eq} does not hold, and 
the inclusions
$\widetilde \BT \subsetneq \BI\subsetneq \widetilde\BH$ determine
identifications
\begin{itemize}
\item[(i)] $\BT^+ = \BT_v = T_v\BI_v = T_v\BH_v$,
\item[(ii)] $\widetilde\BT^- = \widetilde\BT/\BT^+ \isoarrow \BI/\BI_v\isoarrow T_v(\widetilde\BH/\BH_v)$.
\end{itemize}
\end{lem}

\begin{proof} If \eqref{Ind-eq} holds, then possibility (a) clearly holds with $\BI = \widetilde\BT$. 
So suppose \eqref{Ind-eq} does not hold. 
 Let $\widetilde\BT' = \BT_v\oplus c\cdot\BT_v\subset \widetilde\BT$, and let
$\BI = \widetilde\BT'' = \frac{1}{T_v}\widetilde\BT'$. Part (i) is immediate from the definitions.
Since $\widetilde\BT\neq \widetilde\BT'$, it follows from Lemma \ref{ind-lem2}(iii) that 
$\widetilde\BT/\BT^+ = \widetilde\BT/\BT_v\isoarrow \BI/\BI_v$.
Similarly, $\BI = \widetilde\BH'$ but $\widetilde\BH\neq \widetilde\BH'$
(else $\widetilde\BT = \widetilde\BT'$ since $\widetilde\BT = T_v\widetilde\BH$), so 
part (ii) follows from Lemma \ref{ind-lem1}(ii).
\end{proof}

\subsection{Congruence ideal}\label{CMcong}
Another key ingredient in our later arguments is the congruence ideal of the canonical CM family ${\bf{h}}_v$.

 \subsubsection{The ideal(s)}\label{ssCdef}
The (inverse) congruence ideal $I_{{\bf{h}}_{v}} \subset {\bf{F}}$ associated with ${\bf{h}}_{v}$ is the fractional 
$\Lambda_{L}^{v}$-ideal 
characterised by the existence of a surjective 
$\Lambda_{L}^{v}$-morphism 
$$
 M_{\Lambda_{D}}^{\ord}\otimes_{\BH_{D_Lp^{\infty}}^{\ord},\varphi} \Lambda_{L}^{v} \twoheadrightarrow I_{{\bf{h}}_{v}}, \ \  {\bf{h}}_{v} \mapsto 1.
 $$ 
 Let $\widetilde{I}_{{\bf{h}}_{v}}\subset {\bf{F}}$ be the reflexive closure of $I_{{\bf{h}}_v}$. Then 
 $\widetilde{I}_{{\bf{h}}_{v}}$ is a principal fractional $\Lambda_L^v$-ideal of the form $\widetilde{I}_{{\bf{h}}_{v}} = \frac{1}{H_v}\Lambda_L^v$ for some $H_v\in \Lambda_L^v$.
Such an $H_v$ is typically referred to as a congruence power series (well-defined only up to a unit of $\Lambda_L^v$).  
In the following we identify a congruence power series for $\bh_v$.

We similarly define the cuspidal congruence ideal $I^{\mathrm{cusp}}_{{\bf{h}}_{v}} \subset {\bf{F}}$
to be the fractional $\Lambda_L^v$-ideal such that there exists a surjective $\Lambda_{L}^{v}$-morphism
$$
 S_{\Lambda_{D}}^{\ord}\otimes_{\BT_{D_Lp^{\infty}}^{\ord},\varphi} \Lambda_{L}^{v} \twoheadrightarrow I^{\mathrm{cusp}}_{{\bf{h}}_{v}}, \ \  {\bf{h}}_{v} \mapsto 1.
 $$ 
Clearly $I^{\mathrm{cusp}}_{{\bf{h}}_{v}}\subset I_{{\bf{h}}_{v}} \subset {\bf{F}}$. Let $H_v^{\mathrm{cusp}}\in\Lambda_L^v$ be 
such that $1/H_v^{\mathrm{cusp}}$ is a $\Lambda_L^v$-generator of the reflexive closure $\tilde I_{\bh_v}^{\mathrm{cusp}}$ of $I_{\bh_v}^{\mathrm{cusp}}$. 
Then $H_v^{\mathrm{cusp}}\mid H_v$ in $\Lambda_L^v$. We identify a cuspidal congruence power series $H_v^{\mathrm{cusp}}$ along the way to 
identifying $H_v$.

\subsubsection{The Katz $p$-adic L-function} \label{ssKpL}
By results of Katz (see \cite{HT0}), we have the following $p$-adic 
$L$-function for the imaginary quadratic field $L$.

\begin{thm}\label{pKatzL}
There exists an unique $\mathcal{L}_{v}(L) \in \Lambda_{L}^{v,\ur}= \Lambda_{L,W(\ov{\BF}_p)}^v$ such that 
for any continuous $W(\ov{\BF}_p)$-linear homomorphism $\theta: \Lambda_{L}^{v,\ur} \ra \ov{\BQ}_{p}$
such that $\theta(\gamma_v^{h_p}) = \epsilon(\gamma^+)^m$ for some $m\equiv 0 \mod{(p-1)}$, $m\geq 0$,  
$$
\frac{\theta(\mathcal{L}_{v}(L))}{\Omega_{p}^{2m}} = 
(1-p^{m}\psi_\theta(\varpi_{\bar v})^{-2}) \cdot (1-p^{m-1}\psi_\theta(\varpi_{\bar v})^{-2})\frac{w_L}{2}\cdot
\frac{\pi^{m-1}\cdot m!} {\sqrt{D_L}^{m-1}} \cdot \frac{L(1,\psi_\theta/\psi_\theta^c)}{\Omega_{\infty}^{2m}},
 $$
where 
$\psi_\theta$ is the algebraic Hecke character such that $\sigma_{\psi_\theta}$ is the (composite) $p$-adic character $G_L \twoheadrightarrow \Gamma_L^v\stackrel{\theta}{\rightarrow} \ov{\BQ}_p^\times$, and 
$(\Omega_{p},\Omega_{\infty}) \in W(\overline{\BF}_{p})^{\times} \times \BC^\times$ are CM periods over $L$ as in 
\cite[\S4.5]{JSW}.
\end{thm}

\subsubsection{The cuspidal congruence ideal and the Katz $p$-adic L-function} 
The connection between the (cuspidal) congruence ideal associated to the canonical CM Hida family ${\bf{h}}_v$ and the Katz $p$-adic L-function for $L$ is given by 
the following.
\begin{thm}\label{excz} 
Let $H^{\mathrm{cusp}}_{v}$ be a cuspidal congruence power series for ${\bf{h}}_v$. 
Then 
$$
\big{(}\frac{h_{L}}{w_{L}} \cdot \mathcal{L}_{v}(L)\big{)} = (H^\mathrm{cusp}_{v}) \subset \Lambda_{L}^{v,\ur}.
$$
\end{thm}

\begin{proof}
By \cite[Thm.~I]{HT1} and \cite{BuHs}, $\frac{h_{L}}{w_{L}} \cdot \mathcal{L}_{v}(L) \mid H_{v} \text{ in } \Lambda_{L}^{v,\ur}.$
Conversely, by \cite[Thm.~1.4.7]{HT} and \cite{Ru}, 
$H_{v} \mid \frac{h_{L}}{w_{L}} \cdot  \mathcal{L}_{v}(L) \text{ in } \Lambda_{L,P}^{v,\ur}$
for $(T_v)\neq P \in \height_{1}(\Lambda_{L}^{v,\ur})$.
It thus suffices to show that 
$
\ord_{(T_v)}(H_{v}) \neq 0.
$ 

By the definition of the congruence ideal (see \S\ref{ssCdef}), $\ord_{(T_v)}(H_{v})=0$ if (and only if) there is a unique cuspidal Hida family of tame level $D_{L}$ passing through the weight one specialisation of ${\bf{h}}_{v}$.
In view of the $q$-expansion (see \S\ref{CMHidaFam}), such a weight one specialisation is nothing but the $p$-ordinary stabilisation of the theta series corresponding to the quadratic Dirichlet character $\chi_{L}$. The desired uniqueness then follows from \cite[Thm. A(i)]{BDP'}.
\end{proof}

\begin{remark}
The weight one specialisation of ${\bf{h}}_{v}$ as above corresponds to a `trivial zero' in the sense of \cite[Def. 1.5.2]{HT}. 
Furthermore, the non-vanishing of the congruence power series at the identity does not  
follow directly from \cite{HT} (for example, \cite[Thm. 1.4.7]{HT}) and \cite{Ru}):
in \cite{HT} it is assumed throughout that the underlying CM family satisfies a $p$-distinguished hypothesis that does not hold for $\bh_v$.
\end{remark}

\subsubsection{Congruence power series} It is relatively straightforward to identify 
a congruence power series $H_v$ in terms of a cuspidal congruence power series $H_v^{\mathrm{cusp}}$:
\begin{thm}\label{excz-2}
Let $H^{\mathrm{cusp}}_{v}$ be a cuspidal congruence power series for ${\bf{h}}_v$. Then $H_v = T_v\cdot H_v^\mathrm{cusp}$
is a congruence power series for $\bh_v$.  In particular, 
$$
\big{(}\frac{h_{L}}{w_{L}} \cdot T_v\cdot \mathcal{L}_{v}(L)\big{)} = (H_{v}) \subset \Lambda_{L}^{v,\ur}.
$$
\end{thm}

\begin{proof} 
We clearly have $\ord_P(H_v) = \ord_P(H_v^{\mathrm{cusp}})$ for every height one prime $P\in ht_1(\Lambda_L^v)$ such that
there are no congruences modulo $P$ between $\bh_v$ and an Eisenstein family. In particular, this equality holds for all $P\neq (T_v)$
as $(T_v)$ is the only height one prime such that $\Ind_{G_L}^{G_\BQ}(\Psi_P)$, $\Psi_P = \Psi_L^v \mod P$, is not irreducible
(and so the only prime for which such a congruence could possibly exist). That $\ord_{(T_v)}(H_v) = 1$ follows
from \cite[Prop.~5.4(ii)]{BDP'} (see also Proposition \ref{Str-prop1}(ii)). This proves that $H_v = T_v\cdot H_v^{\mathrm{cusp}}$ is a
congruence power series. The second claim of the theorem then follows from Theorem \ref{excz}.
\end{proof}

In light of Theorem \ref{excz-2} we find it convenient to set
\begin{equation}\label{nrm}
\mathcal{H}_{v}=\frac{h_{L}}{w_{L}} \cdot     T_v \cdot \mathcal{L}_{v}(L) \in \Lambda_{L}^{v,\ur},
\end{equation}
so that $(H_v) = (\mathcal{H}_v)$ in $\Lambda_L^{v,\ur}$.

\section{Zeta elements over imaginary quadratic fields: the ordinary case}\label{BFord} 
In this section we introduce the Beilinson--Flach element associated with an ordinary weight two newform and the canonical 
CM Hida family 
and recall its associated explicit reciprocity laws and their connection with $p$-adic $L$-functions.
Our Beilinson--Flach element is just a variant of the Rankin--Selberg zeta elements constructed
by Loeffler and Zerbes together with Lei \cite{LLZa}, \cite{LLZb} and Kings \cite{KLZ}.
We use these elements (together with the choice of a suitable auxiliary newform) to finish answering 
the questions \eqref{Q1}-\eqref{Q4} raised in \S\ref{CMHF}. With these answers in hand, 
we then define a two-variable zeta element over the CM field for an ordinary newform.
In Section \ref{BFss} below we define
the analog of these Beilinson--Flach classes and zeta elements for supersingular newforms and record the corresponding
explicit reciprocity laws.

We continue with the notation and conventions introduced in \S\S\ref{NotationPrelim}--\ref{CMHF}.
Throughout this section we assume $g$ to be ordinary at $p$.

\subsection{Some more Iwasawa cohomology}
For $M = \BT_1, \tilde\BT_1, \BH_1$, or $\tilde\BH_1$ as in \S\ref{CMHF}, let 
$$
H^1(\BZ[\frac{1}{p}], T(1)\hat\otimes  M\hat\otimes \Lambda) = 
H^1(G_{\BQ,\Sigma}, T(1)\hat\otimes  M\hat\otimes \Lambda)
$$
where $\Sigma$ is any finite set of primes containing all $\ell\mid ND_Lp\infty$.
These cohomology groups are independent of the choice of $\Sigma$ and consist of classes unramified
outside $p$, as the notation suggest (cf.~Remark \ref{Iw-rmk}).

Recall that $M$ is a $\Lambda_L^v$-module and there is a $\Lambda_L^v[G_{\BQ_p}]$-filtration 
$0\rightarrow M^+\rightarrow M \rightarrow M^-\rightarrow 0$.
Since $g$ is ordinary at $p$, there is also a $\cO[G_{\BQ_p}]$-filtration $0\ra T^+\ra T\ra T^-\ra 0$. 

\begin{lem}\label{local-inj-lem}  For $M = \BT_1$ or $\BH_1$, 
the natural maps
$$
H^1(\BQ_p,T^-(1)\hat\otimes M^+\hat\otimes \Lambda) 
\rightarrow H^1(\BQ_p,T^-(1)\hat\otimes M\hat\otimes \Lambda)
$$
and
$$
H^1(\BQ_p,T^+(1)\hat\otimes M^-\hat\otimes \Lambda) 
\rightarrow H^1(\BQ_p,T(1)\hat\otimes M^-\hat\otimes \Lambda)
$$
are injective.
\end{lem}

\begin{proof} The kernel in both instances is the image of $H^0(\BQ_p,T^-(1)\hat\otimes M^-\hat\otimes \Lambda)$.
Since the $G_{\BQ_p}$-actions on $T^-(1)$ and $M^-$ are unramified while the action on $\Lambda$ is totally ramified,
it is easily seen that $H^0(I_p, T^-(1)\hat\otimes M^-\hat\otimes \Lambda)=0$ and hence
that $H^0(\BQ_p,T^-(1)\hat\otimes M^-\hat\otimes \Lambda)=0$.
\end{proof}

By Lemma \ref{local-inj-lem} we can identify $H^1(\BQ_p,T^+(1)\hat\otimes  M^- \hat\otimes \Lambda)$ with its image
in $H^1(\BQ_p,T(1)\hat\otimes  M^- \hat\otimes \Lambda)$, and similarly for
$H^1(\BQ_p,T^-(1)\hat\otimes  M^+ \hat\otimes \Lambda)$.
We then let
$$
H^1_{\mathrm{rel},\ord}(\BZ[\frac{1}{p}], T(1)\hat\otimes  M\hat\otimes \Lambda) = \{ 
\kappa\in H^1(\BZ[\frac{1}{p}], T(1)\hat\otimes  M\hat\otimes \Lambda) \ : \ \loc_p(\kappa) \in 
H^1(\BQ_p,T^+(1)\hat\otimes  M^- \hat\otimes \Lambda)\}.
$$
The reason for using the subscripts `$\mathrm{rel}$' and `$\ord$' will be made clear later (see Section \ref{two-variable-zeta} below).

\subsection{The Beilinson--Flach element} 
Let $\alpha$ be the unit root of $x^2-a_p(g)x+p$.  We will say that $g$ is {\em anomalous} if
\begin{equation}\label{anom}\tag{{anom}}
\alpha^2\equiv 1 \mod\lambda  \ \ \text{and} \ \ \text{(irr$_\BQ$) does not hold.}
\end{equation} 
Let $g_\alpha(z) = g(z) - \alpha^{-1}pg(pz)\in S_2(\Gamma_0(Np))$ be the $p$-stabilisation of
$g$ corresponding to $\alpha$. In the notation of \cite{KLZ} there is an isomorphism 
$(Pr^\alpha)^*:M_{F_\lambda}(g_\alpha)^* \isoarrow M_{F_\lambda}(g)^* = V(1)$ 
that maps $M_{\cO_\lambda}(g_\alpha)^*\hookrightarrow M_{\cO_\lambda}(g)^* = T(1)$,
and the latter is an isomorphism if $g$ is not anomalous \cite[Prop.~7.3.1]{KLZ}. 
It follows from the proof of {\it op.~cit.}~that the image of $M_{\cO_\lambda}(g_\alpha)^*$ in $T(1)$ contains 
$(1-\alpha^2)T(1)$ in all cases.

The form $g_\alpha$ is the specialisation of a Hida family ${\bf g}$ in the sense of \cite[\S7]{KLZ}.
Let $M({\bf g})^*$ be the $G_\BQ$-module defined in \cite[Def.~7.2.5]{KLZ}. As explained in \cite[\S7.3]{KLZ}, there is a specialisation
map $M({\bf g})^*\otimes_{\Lambda_{\bf g}}\cO_\lambda \twoheadrightarrow M_{\cO_\lambda}(g_\alpha)^*$, so composing
with $(Pr^\alpha)^*$ yields a specialisation map 
\begin{equation}\label{HidaFam-sp}
M({\bf{g}})^*\otimes_{\Lambda_{\bf{g}}}\cO_\lambda \rightarrow T(1)
\end{equation} 
whose image is a sublattice of finite index.
Here $\Lambda_{\bf{g}}$ is a localisation of Hida's ordinary Hecke algebras $\BH_{Np^\infty}^\ord$ (see \S\ref{Hida-Hecke})
at a maximal ideal and the tensor product is with respect to the local $\BZ_p$-homomorphism $\Lambda_{\bf{g}}\rightarrow \cO$
that sends $T_\ell'$ to $a_\ell(g)$ for all $\ell\neq p$ and maps $U_p'$ to $\alpha$.  
This specialisation map \eqref{HidaFam-sp} is surjective if $g$ is not anomalous. If $g$ is anomalous,
then \eqref{HidaFam-sp} might not be surjective, but the image does contain
$(1-\alpha^2)T(1)$.

The canonical CM family $\bh_v$ corresponds to a branch (in the sense of \cite[\S7.5]{KLZ}) of a Hida family $\bh$ of tame level $D_L$,
and it follows from the definition of $\BH_1$ and $M(\mathbf{h})^*$ that there is a specialisation map 
\begin{equation}\label{CM-sp}
M({\bf h})^*\otimes_{\Lambda_{\bf h}}\Lambda^v_L \twoheadrightarrow  \BH_1.
\end{equation}
More precisely, $\Lambda_{\bf{h}}$ and $M({\bf h})^*$ are the respective localisations of 
$\BH_{D_Lp^\infty}^{\ord}$ and $\CH^1_\ord(D_Lp^\infty)$ at a maximal ideal of $\BH_{D_Lp^\infty}^{\ord}$
and the tensor product in \eqref{CM-sp} is with respect to the map induced by
$\varphi:\BH_{D_Lp^\infty}^{\ord}\rightarrow \Lambda^v_L$ from \S\ref{Tate-lattices}.
In particular, the left-hand side of \eqref{CM-sp} is just $\BH$ and the map is just
the quotient modulo $\Lambda_L^v$-torsion.

Fix an integer $c >1$ such that $(c,6ND_Lp) = 1$.
Let 
$_{c}\mathcal{BF}^{{\bf g},{\bf h}}_1 \in 
H^1(\BZ[\frac{1}{p}], M({\bg})^*\hat\otimes  M({\bh)}^*\hat\otimes \Lambda)$
be the class associated in \cite[Def.~8.1]{KLZ} to the Hida families $\bf g$ and $\bf h$.  Recall our convention
that $G_\BQ$ acts on $\Lambda$ via the inverse of the canonical character $\Psi$. 
Let
$$
_{c}\mathcal{BF}(g_{/L}) \in H^1(\BZ[\frac{1}{p}], T(1)\hat\otimes  \BH_1\hat\otimes \Lambda)
$$
be the image of $_{c}\mathcal{BF}^{{\bg},{\bh}}_1$ under the map induced by the specialisations
\eqref{HidaFam-sp} and \eqref{CM-sp}.

\subsubsection{Local properties at $p$.} We record some important local properties
of these Beilinson--Flach classes.

\begin{lem}\label{BF-locp-lem}
The image of $\loc_p(_{c}\mathcal{BF}(g_{/L}))$ in $H^1(\BQ_p,T^-(1)\hat\otimes \BH_1^-\hat\otimes \Lambda)$
is $0$. 
\end{lem}

\begin{proof} In the proof of \cite[Prop.~8.1.7]{KLZ} it is shown that the image of $_{c}\mathcal{BF}^{{\bf g},{\bf h}}_1$
in $H^1(\BQ_p,\mathscr{F}^{--}M({\bf g}\otimes{\bf h})^*\hat\otimes\Lambda)$ is $0$
(notation as in {\em~loc.~cit.}). Since
$\mathscr{F}^{--}M({\bf g}\otimes{\bf h})^*\subset M({\bf g})^*\hat\otimes  M({\bf h})^*$ projects onto 
$T^-(1)\hat\otimes  \BH_1^-$ under the specialisation maps \eqref{HidaFam-sp} and \eqref{CM-sp},
the lemma follows.
\end{proof}

As an immediate consequence of this lemma and Lemma \ref{local-inj-lem} we have:

\begin{cor}\label{BF-locp-cor} \hfill
\begin{itemize}
\item[(i)] The image of $\loc_p(_{c}\mathcal{BF}(g_{/L}))$ 
in  $H^1(\BQ_p, T(1)\hat\otimes \BH_1^-\hat\otimes \Lambda)$ is contained in
$H^1(\BQ_p, T^+(1)\hat\otimes \BH_1^-\hat\otimes \Lambda)$. In particular, 
$_{c}\mathcal{BF}(g_{/L})\in H^1_{\mathrm{rel},\ord}(\BZ[\frac{1}{p}], T(1)\hat\otimes \BH_1\hat\otimes \Lambda)$.
\item[(ii)] The image of $\loc_p(_{c}\mathcal{BF}(g_{/L}))$ in $H^1(\BQ_p, T^-(1)\hat\otimes \BH_1\hat\otimes \Lambda)$ is contained in $H^1(\BQ_p, T^-(1)\hat\otimes \BH^+\hat\otimes \Lambda)$.
\end{itemize}
\end{cor}

\subsection{Explicit reciprocity laws}  In this section we recall the explicit reciprocity laws 
for the Beilinson--Flach elements from \cite{KLZ}, which connect the elements
$_{c}\mathcal{BF}(g_{/L})$ to two-variable $p$-adic $L$-functions.

 Let $\sR=\CO_\lambda\hat\otimes \Lambda_L^v\hat\otimes \Lambda$.
 We will consider two sets $\Xi^{(I)}$ and $\Xi^{(II)}$ of continuous characters 
 $\chi: \Gamma_L^v\hat\otimes\Gamma \rightarrow \overline{\BQ}^\times_p$. Each such character determines
 a continuous $\CO_\lambda$-homomorphism $\phi_\chi:\sR\rightarrow \overline{\BQ}_p$, where
 $\CO_\lambda$ is identified with the completion of $\iota_p(\CO)$ in $\overline{\BQ}_p$.
 Note that any character $\chi:\Gamma_L^v\hat\otimes\Gamma \rightarrow \overline{\BQ}^\times_p$ 
 is determined by the pair $(\chi_1,\chi_2) = (\chi|_{\Gamma_L^v},\chi|_{\Gamma})$. Then  the first set of homomorphisms is
 \begin{equation}\label{eq:ch-1}
 \Xi^{(I)}:= \{ \chi = (1,\chi_2) \ : \ \text{$\chi_2 = \psi_\zeta$ for some $p^t$th-root of unity $\zeta$}\},
 \end{equation} and the second set of homomorphisms is 
 \begin{equation}\label{eq:ch-2}
 \Xi^{(II)}: = \left\{ \chi = (\chi_1,\chi_2) \ : \ 
 \begin{matrix} \text{$\chi_1(\gamma_v^{h_p}) = \epsilon(\gamma_+)^m$ for some $m\equiv 0\mod p-1$, $m>0$} \\
 \text{$\chi_2(\gamma_\cyc) = \zeta \epsilon(\gamma_\cyc)^n$ for some $p^t$th-root of unity $\zeta$ and some $0\leq n\leq m$}
 \end{matrix}\right\}.
 \end{equation}

\subsubsection{The explicit reciprocity law I} As we explain below, from the constructions in \cite[\S\S8,10]{KLZ} we obtain an injective
$\sR$-morphism 
\begin{equation}\label{ERI-map}
\sC:H^1(\BQ_p,T^-(1)\hat\otimes \BH^+\hat\otimes \Lambda)\hookrightarrow {J_g}\otimes_{\CO_\lambda}\sR,
\end{equation}
where $J_g\subset F_\lambda$ is a certain fractional ideal 
such that 
$$
\text{$J_g  = \frac{1}{\alpha(1-p\alpha^{-2})(1-\alpha^{-2})\lambda_N(g) c_g}\CO \ $ if \ (irr$_\BQ$) holds},
$$
where $c_g$ is the congruence number of $g$ as in Section \ref{congruence} and 
$\lambda_N(g) = \pm 1$ is the eigenvalue of $g$ for the usual Atkin--Lehner involution $w_N$ of level $N$.
We let 
$$
\sL_{p,c}(g/L) = \sC(\loc_p(_{c}\mathcal{BF}(g/L))) \in {J_g}\otimes_{\CO_\lambda}\sR.
$$
This makes sense in light of Corollary \ref{BF-locp-cor}(ii).

The map $\sC$ is related to the Coleman map $Col_{\eta_{\omega_g}}$ as follows.

\begin{prop}\label{Col-cycsp-prop}
The reduction of the map $\sC$ modulo $\gamma_v-1$ equals the composition 
$$
H^1(\BQ_p,T^-(1)\hat\otimes \BH^+\hat\otimes \Lambda)
\stackrel{\mod \gamma_v-1}{\twoheadrightarrow} H^1(\BQ_p, T^-(1)\hat\otimes \Lambda)
\stackrel{\frac{1}{\alpha(1-p\alpha^{-2})(1-\alpha^{-2})\lambda_N(g)} Col_{\eta_{\omega_g}}}{\longrightarrow} J_g\otimes_{\CO_\lambda}\Lambda_{\CO_\lambda}.
$$
\end{prop}

\noindent The first map in this proposition depends on an identification $\BH^+/(\gamma_v-1)\BH^+ \simeq \BZ_p$ determined by part of the data defining
the map $\sC$ (the reduction modulo $\gamma_v-1$ of the map $\omega_{{\bf h}_v}$ recalled below).

The element $\sL_{p,c}(g/L)$ is related to $L$-values by the following explicit reciprocity law.

\begin{thm}[Explicit Reciprocity Law I]\label{ERLI-thm}
The element $\sL_{p,c}(g/L)$ satisfies:
For $\chi\in \Xi^{(I)}$, 
$$
\phi_\chi(\sL_{p,c}(g/L)) = (c^2-\langle c\rangle \psi_\zeta^2(c)\chi_L(c))\frac{\mathcal{E}(\zeta)}{(1-p\alpha^{-2})(1-\alpha^{-2})\lambda_N(g)}
\frac{L(1,f\otimes \psi_\zeta^{-1})L(1,f\otimes\chi_L\psi_\zeta^{-1})}{\pi^2(-i)2^3\langle g,g\rangle},
$$
where
$$
\mathcal{E}(\zeta) = \begin{cases} 
\alpha^{-2(t+1)} \frac{p^{2(t+1)}}{\mathfrak{g}(\psi_\zeta^{-1})^2} & \zeta \neq 1 \\
(1 - \frac{1}{\alpha})^4 & \text{else}.
\end{cases}
$$
\end{thm}
\noindent Here $\langle c\rangle = (1+p)^{\log_p(c)}\in 1+p\BZ_p$. 

The proofs of both Proposition \ref{Col-cycsp-prop} and Theorem \ref{ERLI-thm} are given below.

\begin{remark}\label{ERLI-rmk}
As explained below, $\sL_{p,c}(g/L)$ is essentially the specialization at $g_\alpha$ of the $p$-adic Rankin-Selberg $L$-function
constructed by Hida for 
the branches ${\bf a}$ and ${\bf h_v}$ of the Hida families ${\bg}$ and ${\bh}$. As such it satisfies an interpolation formula as in 
Theorem \ref{ERLI-thm} for all characters $\chi:\Gamma_L^v\hat\otimes\Gamma \rightarrow \ov{\BQ}_p^\times$ of finite order. 
However, the general interpolation formula is more complicated to write and not necessary for the purposes of this paper.
\end{remark}

Let $\sS = \CO_\lambda\hat\otimes \Lambda_L^v = \Lambda_{L,\CO_\lambda}^v$, so $\sR= \sS\hat\otimes \Lambda$.
Let $\nu:\sR\isoarrow \sR$ be the $\sS$-algebra automorphism such that 
\begin{equation}\label{eq:tw-1}
\nu(a\otimes b\otimes [\gamma_\cyc]) = a \otimes b [\gamma_+]^2\otimes [\gamma_\cyc] = a\otimes b \epsilon\Psi_D(\tilde\gamma_+)\otimes [\gamma_\cyc],
\end{equation}
for any lift $\tilde\gamma^+\in G_L$ of $\gamma_L^+\in\Gamma_L$.  (Here we view $\Psi_D$ as $\Lambda_L^v$-valued via the map 
$\varphi_v$ fixed at the end of Section \ref{Hida-Hecke}.)
If $M$ is any profinite $\sS[G_{\BQ_p}]$-module, then there is an isomorphism
of profinite $\sS[G_{\BQ_p}]$-modules
\begin{equation}\label{unram-twist}
 M\hat\otimes \Lambda = M\hat\otimes_\sS\sR \stackrel{id\otimes\nu}{\isoarrow} M(\epsilon^{-1}\Psi_D^{-1})\hat\otimes_\sS\sR
 = M(\epsilon^{-1}\Psi_D^{-1})\hat\otimes \Lambda,
 \end{equation}
 where the $G_{\BQ_p}$-action on the $\Lambda$-factor is by the inverse of the canonical character $\Psi$.

To define $\sC$ and see that the claims in the preceding proposition and theorem hold, we recall
a construction in \cite{KLZ}. In particular, for any unramified profinite $\BZ_p[G_{\BQ_p}]$-module $M$, 
\cite[Thm.~8.2.3]{KLZ} provides a homomorphism of $\Lambda$-modules
$$
\CL_M:H^1(\BQ_p,M\hat\otimes \Lambda)\rightarrow D(M)\hat\otimes I^{-1}\Lambda
$$
that is functorial in $M$.
Here $D(M) = (M\hat\otimes  W(\overline{\BF}_p))^{G_{\BQ_p}}$ and $I = (\gamma_\cyc-\epsilon^{-1}(\gamma_\cyc))$. 
More precisely, this is the restriction of the map denoted $\CL_M$ in {\it op.~cit.} to the $\Lambda_{\CG}^{(0)}$-summand.
Taking $M=(T^-(1)\hat\otimes \BH^+)(\epsilon^{-1}\Psi_D^{-1}) = T^-(1)\hat\otimes \BH^+(\epsilon^{-1}\Psi_D^{-1})$
and pre-composing with the isomorphism
$$H^1(\BQ_p,T^-(1)\hat\otimes \BH^+\hat\otimes \Lambda)\isoarrow H^1(\BQ_p,T^-(1)\hat\otimes \BH^+(\epsilon^{-1}\Psi_D^{-1})\hat\otimes \Lambda)$$ 
coming from \eqref{unram-twist} yields 
a homomorphism  
$$
\CC_{g,\bh_v}: H^1(\BQ_p,T^-(1)\hat\otimes \BH^+\hat\otimes \Lambda)
\hookrightarrow D(T^{-}(1)\hat\otimes  \BH^+(\epsilon^{-1}\Psi_D^{-1}))\hat\otimes \Lambda \subset D(T^{-}(1)\hat\otimes  \BH^+(\epsilon^{-1}\Psi_D^{-1}))\hat\otimes 
I^{-1}\Lambda.
$$
That the image of $\CL_M$ lies in $D(M)\hat\otimes \Lambda$ in this case follows from the fourth bullet point of 
\cite[Thm.~8.2.3]{KLZ} and the fact that $H^0(\BQ_p,T^-(1)\hat\otimes \BH^+)=0$. This same vanishing implies that
$\CL_M$ is injective in this case by the third bullet point of {\it loc.~cit.}, hence so is $\CC_{g,\bh_v}$.
The map $\sC$ is then defined to be the composition of $\CC_{g,\bh_v}$ with an isomorphism
 $$
 \psi_{g,{\bf h_v}}:D(T^{-}(1)\hat\otimes  \BH^+(\epsilon^{-1}\Psi_D^{-1}))\hat\otimes \Lambda \isoarrow {J_g}\otimes_{\CO_\lambda}\sR,
 $$
 defined as follows.

Following the notation of \cite{KLZ}, let ${\bf a}$ be the (new, cuspidal) branch of the Hida family $\bg$ containing the $p$-stabilization $g_\alpha$ of $g$,
and let $\eta_{\bf g}:D(\mathscr{F}^-M(\bg)^*)\otimes_{\Lambda_{\bg}}\Lambda_{\bf a}\rightarrow I_{\bf a}$ be the
$\Lambda_{\bf a}$ homomorphism of \cite[Prop.~10.1.1 part 2]{KLZ}. Here $I_{\bf a}$ is the congruence ideal for the branch $\bf a$
as in \cite[7.7.1]{KLZ}.  Specializing to $g_\alpha$ as in \cite[Prop.~10.1.1 part 2(b)]{KLZ} yields a commutative diagram 
\begin{equation}\label{tilde-eta-g}
\begin{tikzcd}[row sep=2.5em]
& D(\mathscr{F}^-M(\bg)^*)\otimes_{\Lambda_{\bg}}\Lambda_{\bf a} \arrow[r,"{\eta_{\bf a}}"] \arrow[d,] & I_{\bf a} \arrow[d,] \\
 D_{dR}^0(M_{F_\lambda}(g_\alpha)^*) \arrow[r, "="]\arrow[d,"(Pr^\alpha)_*"] & D(\mathscr{F}^-M_{F_\lambda}(g_\alpha)^*) \arrow[r,"\simeq"] \arrow[d,"(Pr^\alpha)_*"] & F_\lambda \arrow[d,"="] \\
 D_{dR}^0(V^-(1)) \arrow[r, "="] & D(V^-(1)) \arrow[r,"\tilde\eta_g"] & F_\lambda,
\end{tikzcd}
\end{equation}
where the top vertical arrows are induced via functoriality from the specialisation map 
$M({\bf g})^*\otimes_{\Lambda_{\bf a}}\CO_\lambda\twoheadrightarrow M_{\CO_\lambda}(g_\alpha)^*$ (the corresponding
homomorphism $\Lambda_{\bf a}\rightarrow \CO_\lambda$ extends to a homomorphism $I_{\bf a}\rightarrow F_\lambda$),
the left horizontal arrows are just the identifications as in Remark \ref{D(T)-rmk}, and the middle right horizontal isomorphism is
such that the composition of the middle horizontal arrows is given by {\it loc.~cit.}.  
We note that by functoriality the image of the composition of the middle
two vertical arrows is in $D(T^-(1))$.
The map $\tilde\eta_g$ is defined to be 
the isomorphism making this diagram commute. From the definition of 
$\eta_{\bg}$ and the proof of \cite[Prop.~10.1.1 part 2(b)]{KLZ} we find that 
\begin{equation}\label{eta-eq}
\tilde\eta_g = [\cdot,\frac{1}{\alpha(1-p\alpha^{-2})(1-\alpha^{-2})\lambda_N(g)}\eta_{\omega_g}]: D_{dR}(V^-(1)) = D(T^-(1))\otimes_{\CO_\lambda}F_\lambda \rightarrow F_\lambda,
\end{equation}
where $[\cdot,\cdot]$ is the pairing as in \eqref{CanPai}. 
We define $J_g$ to be the image of $D(T^-(1))$ under
$\tilde\eta_g$.  Recall that the image of 
$S_{\CO_\lambda}$ under the deRham-\'etale comparison map $S_{F_\lambda} = D_{dR}(V^-(1)) \simeq D(V^-(1))$ is contained in $D(T^-(1))$, with equality if (irr$_\BQ$) holds
(see Remark \ref{D(T)-rmk}). By Lemma \ref{GorPer} an $\CO_\lambda$-generator of $S_{\CO_\lambda}$ is of the form $\omega_g/c$ for some $c\in \CO$ that divides $c_g$ and we may take $c=c_g$ if (irr$_\BQ$) holds. It
follows that if (irr$_\BQ$) holds, then $J_g$ is generated by $\tilde\eta_g(\omega_g/c_g) = (\alpha(1-p\alpha^{-2})(1-\alpha^{-2})\lambda_N(g)c_g)^{-1}$.

Still following the notation of \cite{KLZ},  the map $\omega_{\bf h}$ of \cite[Prop.~10.1.1 part 1]{KLZ}
induces a $\Lambda_L^v$-homomorphism 
\begin{equation}\label{eq:D-sub}
\omega_{{\bf h}_v}: D(\BH^+(\epsilon^{-1}\Psi_D^{-1}))\isoarrow \Lambda_L^v.
\end{equation}
Here we have used that $\mathscr{F}^+M({\bf h})^*\otimes_{\Lambda_{\bf h}}\Lambda_L^v = \BH^+$ and that the twist by $\epsilon^{-1}\Psi_D^{-1}$ is 
then identified with the $(-1-{\bf k})$-twist of {\it loc.~cit.} (and that $\varepsilon_{\bf h}|_{G_{\BQ_p}}= 1$ since $p$ splits
in the quadratic field $L$). 

The map $\psi_{g,{\bh}_v}$ is defined to be the composition of isomorphisms
$$
 \psi_{g,{\bf h_v}}:D(T^{-}(1)\hat\otimes  \BH^+(\Psi_D^{-1}))\hat\otimes \Lambda  = D(T^{-}(1))\hat\otimes  D(\BH^+(\Psi_D^{-1}))\hat\otimes \Lambda
 \stackrel{\tilde\eta_g\otimes\omega_{{\bf h}_v}\otimes id}{\longrightarrow} {J_g}\otimes_{\CO_\lambda}\sR \stackrel{id\otimes\nu^{-1}}{\longrightarrow} {J_g}\otimes_{\CO_\lambda}\sR,
 $$
and $\sC$ is then defined to be the injective map
$$
\sC =  \psi_{g,{\bf h_v}}\circ\CC_{g,\bf{h}_v} : H^1(\BQ_p,T^-(1)\hat\otimes \BH^+\hat\otimes \Lambda) \hookrightarrow 
 {J_g}\otimes_{\CO_\lambda}\sR.
 $$
 This final map is an $\sR$-homomorphism.

\begin{proof}[Proof of Proposition \ref{Col-cycsp-prop}]
The map $\omega_{\bh_v}$ factors through a $\Lambda_L^v$-isomorphism $D(\BH^+(\epsilon^{-1}\Psi_D^{-1}))\cong \BH^+(\epsilon^{-1}\Psi_D^{-1})$
(see also \cite[Prop.~1.7.6]{FK}).  Since $\BH^+$ is a free $\Lambda_L^v$-module of rank one, $\omega_{\bh_v}$ determines a $\Lambda_L^v$-basis of 
$\BH^+(\epsilon^{-1}\Psi_D^{-1})$. In particular,  $\omega_{\bh_v}$ determines isomorphisms
$\BH^+/(\gamma_v-1)\BH^+ \simeq \BH^+(\Psi_D^{-1})/(\gamma_v-1) \BH^+(\Psi_D^{-1})\simeq \BZ_p$. 
It then follows from functoriality that $\sC \,\mod \gamma_v-1$ factors as 
the composition 
$$
H^1(\BQ_p,T^-(1)\hat\otimes \BH^+\hat\otimes \Lambda) 
\stackrel{\mod \gamma_v-1}{\twoheadrightarrow} H^1(\BQ_p,T^-(1)\hat\otimes \Lambda) \stackrel{\CL_{T^-(1)}}{\rightarrow} D(T^{-1})\hat\otimes\Lambda
\stackrel{\tilde\eta_g\otimes id}{\rightarrow} J_g\hat\otimes\Lambda,$$
where the first map is just the projection induced by the isomorphism $\BH^+/(\gamma_v-1)\BH^+ \simeq \BZ_p$ from above.
The characterization \eqref{eta-eq} of $\tilde\eta_g$ in terms of $\eta_{\omega_g}$ combined with 
the fact that $\CL_{T^-(1)}$ is just Perrin-Riou's `big logarithm map'  for $V^-(1)$ (see the second bullet point of \cite[Thm.~8.2.3]{KLZ})
shows -- by comparing with the definition of the Coleman map $Col_{\eta_{\omega_g}}$ in terms of this big logarithm map -- that the above displayed composition of maps equals 
$\frac{1}{\alpha(1-p\alpha^{-2})(1-\alpha^{-2})\lambda_N(g)} Col_{\eta_{\omega_g}}$.
\end{proof}

\begin{proof}[Proof of Theorem \ref{ERLI-thm}]
It follows from the definition of $_c\mathcal{BF}(g/L)$ and the commutativity of the diagram \eqref{tilde-eta-g}
that $\sL_{p,c}(g/L)$ is -- in the notation of \cite[\S10]{KLZ} -- just the image of $\langle \CL(_c\mathcal{BF}_1^{{\bf g},{\bf h}}), \eta_{\bf a}\otimes\omega_{\bf h}\rangle \in I_{{\bf a}} \hat\otimes\Lambda_{\bf h}^{\rm{cusp}}\hat\otimes\Lambda$ under the projection
$$
I_{{\bf a}} \hat\otimes\Lambda_{\bf h}^{\rm{cusp}}\hat\otimes\Lambda \twoheadrightarrow J_g\hat\otimes \Lambda_L^v\hat\otimes\Lambda = J_g\otimes_{\CO_\lambda}\sR,
$$
induced by the specialization map $\Lambda_{\bf a} \rightarrow \CO_\lambda$ corresponding to $g_\alpha$
and the projection $\Lambda_{\bh}^{\rm cusp} \twoheadrightarrow \Lambda_{\bh_v} = \Lambda_L^v$.
The claim in the theorem is then an easy consequence of \cite[Thms.~2.7.4, 7.7.2, 10.2.2]{KLZ}
(see also the comment about weight one specialisations following \cite[Thm.~7.7.2]{KLZ}). Here we have
used that the weight one specialisation of $\bh_v$ is the ordinary stabilisation of the weight one Eisenstein series $E_1(1,\chi_L)$
for the characters $1$ and $\chi_L$, which means that the Rankin-Selberg $L$-function
$L(s,g,E_1(1,\chi_L),\psi_\zeta^{-1})$ equals $L(s,g\otimes\psi_\zeta^{-1})L(s,g\otimes\chi_L\psi_\zeta^{-1})$.
\end{proof}

\subsubsection{The explicit reciprocity law II}  
The second explicit reciprocity law essentially arises from exchanging the roles of $g$ and $\bf{h}_v$ in the preceding analysis.
As we explain, there is an injective $\sR$-homomorphism
$$
\sL: H^1(\BQ_p,T^+(1)\otimes\tilde\BH^-\hat\otimes\Lambda) \hookrightarrow \tilde I_{\bf h_v}\otimes_{\Lambda_L^v} \sR
$$
defined analogously to $\sC$.  Here $I_{\bf{h}_v}$ is the congruence ideal for the canonical CM family introduced in Section \ref{CMcong} and
$\tilde I_{\bh_v}$ is its reflexive closure.
We let 
$$
\sL_{p,c}^{Gr}(g/L)= \sL(\loc_p(_c\mathcal{BF}(g/L))) \in \tilde I_{\bf h_v}\otimes_{\Lambda_L^v} \sR.
$$
This makes sense by Corollary \ref{BF-locp-cor}(i). These satisfy:

\begin{thm}[Expicit Reciprocity Law II]\label{ERLII-thm}
For $\chi\in \Xi^{(II)}$,
$$
\phi_\chi(\sL_{p,c}^{Gr}(g/L)) = \lambda_{D_L}(h_{v,\chi_1}^0)(c^2-\psi_\zeta^2(c)\chi_L(c)\langle c\rangle^{2n-m+1)} ) \CE(\chi) \frac{n!(n-1)!}{\pi^{2n+1}(-i)^{m-1}2^{2n+m+1}}\frac{L(1+n, g, h_{v,\chi_1}^0,\psi_\zeta^{-1}\omega^n)}{\langle h_{v,\chi_1}^0,h_{v,\chi_1}^0\rangle},
$$
where 
$$
\CE(\chi) = \begin{cases}  
\frac{(1-\frac{p^n}{\chi_1(\varpi_{\bar v})\alpha})(1-\frac{p^n}{\chi_1(\varpi_{\bar v})\beta})
(1-\frac{p^{m}\alpha}{p^{n+1}\chi_1(\varpi_{\bar v})\alpha})(1-\frac{p^{m}\beta}{p^{n+1}\chi_1(\varpi_{\bar v})\alpha})}
{(1-p^{m-1}\chi_1^{-2}(\varpi_{\bar v}))(1-p^{m}\chi_1^{-2}(\varpi_{\bar v}))} & \zeta = 1, n\equiv 0\mod p-1 \\
\frac{(p^{t+1}/\mathfrak{g}(\psi_\zeta^{-1}\omega^n))^2 (p^{2n-1}/\chi_1(\varpi_{\bar v}))^{t+1}}
{(1-p^{m-1}\chi_1^{-2}(\varpi_{\bar v}))(1-p^{m}\chi_1^{-2}(\varpi_{\bar v}))}
 & \text{else}.
\end{cases}
$$
\end{thm}

\noindent Here $h_{v,\chi_1}^0$ is the newform of level $D_L$ and weight $2m+1$ whose ordinary $p$-stabilisation is $h_{v,\chi_1}$.
The $L$-function $L(s,g,h_{v,\chi_1}^0,\psi)$ is the $\psi$-twist of the usual Rankin--Selberg $L$-functions (cf.~\cite[\S2.7]{KLZ}).
Also, we identify the Galois character $G_L\rightarrow\Gamma_L^v\stackrel{\chi_1}{\rightarrow} \BQ_p^\times$ 
with an algebraic Hecke character of infinity type $(-2m,0)$ as in Section \ref{Hecke-char}, which we continue to denote by $\chi_1$.
The hypotheses on $\chi_1$ ensure that this algebraic Hecke character is unramified at each prime above $p$, and we have 
denoted by $\varpi_{\bar v}$ a uniformiser at $\bar v$ (note that $\chi_1(\varpi_{\bar v})$ is a $p$-adic unit).
Here $t\geq 0$ is such that $\zeta$ is a primitive $p^t$th root of unity. Finally, $\langle c \rangle  = (1+p)^{\log_p c}\in 1+p\BZ_p$.

\begin{remark}\label{ELRII-rmk}
Just as for Theorem \ref{ERLI-thm}, the proof of Theorem \ref{ERLII-thm} 
essentially identifies $\sL_{p,c}^{Gr}(g/L)$ as the specialization at $g_\alpha$ of a $p$-adic Rankin-Selberg $L$-function
constructed by Hida for 
the branches ${\bf h_v}$ and ${\bf a}$ of the Hida families ${\bf h}$ and ${\bf g}$. As such it also satisfies an interpolation formula 
for a larger collection of characters than just $\Xi^{(II)}$. However --  as for $\sL_{p,c}$ -- 
the general interpolation formula is more complicated to write and not necessary for the purposes of this paper.
\end{remark}

\begin{remark}\label{Gr-rmk}
The superscript `$Gr$' on $\sL_{p,c}^{Gr}(g/L)$ is intended to reference Greenberg. The use of this notation is motivated
by $\sL_{p,c}^{Gr}(g/L)$ being essentially a $p$-adic $L$-function for a $p$-adic family of Galois representations 
satisfying the Panchishkin condition, much as considered in \cite{Gr1}. 
\end{remark}

The definition of $\sL$ is analogous to that of $\sC$. 
Let $\iota_\epsilon:\Lambda_\CG\isoarrow \Lambda_\CG$ be the isomorphism such that $\gamma\mapsto \epsilon(\gamma)\gamma$ for all $\gamma\in\CG$.
From the isomorphism $T^+(1)\otimes\Lambda_{\CG} \stackrel{id\otimes\iota_{\epsilon}}{\longrightarrow} T^+\otimes\Lambda_{\CG}$ we obtain an
injective homomorphism
$$
\CL_{g,{\bf h}_v}: H^1(\BQ_p,T^+(1)\otimes\tilde\BH^-\hat\otimes\Lambda)\isoarrow H^1(\BQ_p,T^+\otimes\tilde\BH^-\hat\otimes\Lambda_{\CG}^{(-1)})
\stackrel{\CL_{T^+\otimes\tilde\BH^-}}{\hookrightarrow} D(T^+\otimes\tilde\BH^-)\hat\otimes\Lambda_{\CG}^{(-1)}\cong D(T^+\otimes\tilde\BH^-)\hat\otimes\Lambda,
$$
just as we did $\CC_{g,{\bf h}_v}$. Then $\sL$ is the composition of $\CL_{g,{\bf h}_v}$ with an isomorphism
$$
\xi_{g,{\bf h}_v}: D(T^+\otimes\tilde\BH^-)\hat\otimes\Lambda \isoarrow \tilde I_{\bh_v}\otimes_{\Lambda_L^v} \sR
$$
defined as follows.

We let $\eta_{\bh_v}:D(\BH_1^-)\isoarrow I_{\bh_v}$ be the $\Lambda_{L^v}$-map as in \cite[Prop.~10.1.1 part 2(b)]{KLZ}.
Here we are using that $\eta_{\bh_v}$ from {\it op.~cit.}~factors through the quotient
$D(\mathscr{F}^-M(\bh)^*)\otimes_{\Lambda_{\bh}}\Lambda_L^v = D(\BH^-)\twoheadrightarrow D(\BH^-_1)$ by functoriality.  
The map $\eta_{\bh_v}$ then induces a map on reflexive closures 
$\tilde\eta_{\bh_v}: \widetilde{D(\BH^-_1))}\isoarrow \tilde I_{\bh_v}$. Since 
$\widetilde{D(\BH^-_1)}  = D(\tilde\BH^-_1) = D(\tilde\BH^-)$
by functoriality (this is easily seen from the fact that there is a natural identification $D(M)\simeq M$ for profinite unramified
$\BZ_p[G_{\BQ_p}]$-modules $M$ that is functorial in $M$ \cite[Prop.~1.7.6]{FK}, this yields a 
$\Lambda_L^v$-isomorphism $\tilde\eta_{\bh_v}:D(\tilde\BH^-) \isoarrow \tilde I_{\bh_v}$
of free $\Lambda_L^v$-modules of rank one
and a commutative diagram
\begin{equation}\label{eta-hv-eq}
\begin{tikzcd}[row sep=2.5em]
D(\mathscr{F}^-M(\bh)^*)\otimes_{\Lambda_{\bh}}\Lambda_L^v \arrow[r,"{\eta_{\bh_v}}"] \arrow[d,] & I_{\bh_v} \arrow[d,hook] \\
 D(\tilde\BH^-) \arrow[r,"\tilde \eta_{\bh_v}"] &  \tilde I_{\bh_v},
\end{tikzcd}
\end{equation}
where the left vertical arrow is induced by functoriality. 

Let $\omega_{g,\alpha}: D(V^+) \rightarrow F_\lambda$ denote the map
$[\omega_g,-]:D(V^+) = D_{cris}^0(V) \isoarrow F_\lambda$,
where $[-,-]$ is the pairing from \eqref{CanPai}. The restriction to $D(T^+)$ is mapped isomorphically onto $\CO_\lambda$, and it follows 
from \cite[Prop.~10.1.1 part 1]{KLZ} that $\omega_{g,\alpha}$ fits into a commutative
diagram
\begin{equation}\label{omega-g-eq}
\begin{tikzcd}[row sep=2.5em]
D(\mathscr{F}^+M(\bh)^*(\Psi_D^{-1})) \arrow[rr,"\omega_{\bg}"] \arrow[d,] & & \Lambda_{\bf a} \arrow[d,] \\
D(\mathscr{F}^+M_{F_\lambda}(g_\alpha)^*) \arrow[rr, "{[(Pr_\alpha)^*(\omega_g),-]}"] \arrow[d, "(Pr_\alpha)_*"]  & & F_\lambda \arrow[d,"="] \\
D(V^+)  \arrow[rr, "\omega_{g,\alpha}"] && F_\lambda,
\end{tikzcd}
\end{equation}
where the vertical arrows come from the specialisation maps.  We note that by functoriality, the image of the composition of the left vertical arrows is
contained in $D(T^+)$.

The maps $\xi_{g,{\bf h}_v}$ and $\sL$ are then defined to be 
$$
 \xi_{g,{\bf h_v}}:D(T^{+}\hat\otimes  \tilde\BH^-)\hat\otimes \Lambda  
 \stackrel{\omega_g\otimes\tilde\eta_{{\bf h}_v}\otimes id}{\longrightarrow} 
 {\CO_\lambda}\hat\otimes  \tilde I_{\bh_v}\hat\otimes \Lambda\stackrel{id\otimes id \otimes\iota_\epsilon^{-1}}{\longrightarrow}
 {\CO_\lambda}\hat\otimes  \tilde I_{\bh_v}\hat\otimes \Lambda  = \tilde I_{\bh_v}\otimes_{\Lambda_L^v} \sR
 $$
 and
$$
\sL =  \xi_{g,{\bf h_v}}\circ\CL_{g,\bf{h}_v} : H^1(\BQ_p,T^+(1)\hat\otimes \tilde\BH^-\hat\otimes \Lambda) \rightarrow 
 \tilde I_{\bh_v}\otimes_{\Lambda_L^v} \sR.
 $$
 The former is an isomorphism and the latter is an injective $\sR$-homomorphism.
 
 \begin{proof}[Proof of Theorem \ref{ERLII-thm}]
 It follows from the definition of $_c\mathcal{BF}(g/L)$ and the commutativity of the diagrams \eqref{eta-hv-eq} and \eqref{omega-g-eq}
that $\sL_{p,c}^{Gr}(g/L)$ is - in the notation of \cite[\S10]{KLZ} -- the image of 
$\langle _c\CL(\mathcal{BF}_1^{{\bf g},{\bf h}}), \omega_{\bf g}\otimes\eta_{\bf h_v}\rangle \in \Lambda_{\bf g}^{\rm{cusp}} \hat\otimes I _{\bf h}\hat\otimes\Lambda$
under the projection
$
\Lambda_{\bf g}^{\rm{cusp}} \hat\otimes I _{\bf h}\hat\otimes\Lambda \rightarrow 
\CO_\lambda \hat\otimes \tilde I _{\bf h_v}\hat\otimes\Lambda = \tilde I_{\bh_v}\otimes_{\Lambda_L^v} \sR
 $ induced by the map $\Lambda_{\bf g}^{\rm{cusp}} \rightarrow \CO_\lambda$ underlying the specialisation map \eqref{HidaFam-sp}.
 The claim in the theorem is then an easy consequence of \cite[Thms.~2.7.4, 7.7.2, 10.2.2]{KLZ}.
 \end{proof}
 
\subsubsection{Integral normalisations}
Let $\tilde c_g \in \CO$ be such that $(\alpha(1-p\alpha^{-2})(1-\alpha^{-2})\lambda_N(g)\tilde c_g)^{-1}$ is a $\CO_\lambda$-generator of $J_g$. As noted above, we can, and do, take 
$\tilde c_g = c_g$ if (irr$_\BQ$) holds, where $c_g\in \CO$ is a congruence number as in Section \ref{congruence}.  We then put 
$$
\sC^{\mathrm{int}}= \tilde c_g \cdot \sC: H^1(\BQ_p, T^-(1)\hat\otimes \Lambda_L^v\hat\otimes\Lambda) \rightarrow \CO_\lambda\hat\otimes \Lambda_L^v\hat\otimes\Lambda = \sR.
$$

Let $\CH_L \in \Lambda_L^{\ur}$ be as in \eqref{nrm}. Let $\CO_\lambda^{\ur}$ be the completion of the ring of integers of the 
maximal unramified extension of $F_\lambda$ (this is just the compositum of $\CO_\lambda$ and $W(\bar{\BF}_p)$).
We put
$$
\sL^\mathrm{int} = \CH_v\cdot \sL:H^1(\BQ_p,T^+(1)\hat\otimes \tilde \BH^-\hat\otimes\Lambda)\otimes_{\Lambda_{L,\CO_\lambda}^v}
\Lambda_{L,\CO_\lambda^\ur}^v \rightarrow \sR^\ur,
$$
where  $\sR^\ur= \CO_\lambda^\ur\hat\otimes\Lambda_L^v\hat\otimes\Lambda = \sR\otimes_{\Lambda_{L,\CO_\lambda}^v} \Lambda_{L,\CO_\lambda^\ur}^v$.

We also normalise the elements $_c\mathcal{BF}(g/L)$.  Let $r_c = \log_p(c) \in \BZ_p$. and let
$$
\mathbf{c} = (c^2 - \chi_L(c)\otimes\langle c\rangle \gamma_{v}^{-r_c}\otimes\gamma_{\cyc}^{2r_c}) \in \sR.
 $$
Modulo the maximal ideal of $\sR$, $\mathbf{c}$ is congruent to $c^2-\chi_L(c)$. 
As $p$ is odd, we can therefore choose $c$ so that $\mathbf{c}\in \sR^\times$. Henceforth we assume that $c$ satisfies this.
We then put
$$
\mathcal{BF}(g/L) = \mathbf{c}^{-1}\cdot {_c\mathcal{BF}(g/L)} \in H^1(\BQ,T(1)\hat\otimes \tilde\BH\hat\otimes \Lambda).
$$
We also put
$$
\sL_p(g/L) = \sC^\mathrm{int}(\loc_p(\mathcal{BF}(g/L)))\in \sR  \ \ \text{and} \ \ \sL_p^{Gr}(g/L)=\sL^{\mathrm{int}}(\loc_p(\mathcal{BF}(g/L))) \in \sR^{\ur}.
$$

We record the following versions of our two explicit reciprocity laws (Theorems \ref{ERLI-thm} and \ref{ERLII-thm}) and their consequences.

\begin{thm}[Explicit Reciprocity Law I$'$]\label{ERLIint-thm}
The element $\sL_{p}(g/L)\in \sR$ satisfies:
For $\chi\in \Xi^{(I)}$ with $\chi_1=1$, 
$$
\phi_\chi(\sL_{p}(g/L)) = \mathcal{E}(\zeta) \frac{L(1,f\otimes \psi_\zeta^{-1})L(1,f\otimes\chi_L\psi_\zeta^{-1})}{\pi^2(-i)2^3\Omega_g},
$$
where
$$
\Omega_g = \frac{\langle g,g\rangle}{\tilde c_g} \ \ 
\text{and} \ \
\mathcal{E}(\zeta) = \begin{cases} 
\alpha^{-2(t+1)} \frac{p^{2(t+1)}}{\mathfrak{g}(\psi_\zeta^{-1})^2} & \zeta \neq 1 \\
(1 - \frac{1}{\alpha})^4 & \text{else}.
\end{cases}
$$
\end{thm}

Given $\chi\in \Xi^{(II)}$ let 
$\psi_\chi$ be the algebraic Hecke character of $L$ with infinity type $(n-m,n)$ and such that
$\sigma_{\psi_\chi}$ is the composition 
$G_L\twoheadrightarrow \Gamma_L^v\times\Gamma
\stackrel{\chi_1\times \chi_2^{-1}}{\rightarrow}\overline{\BQ}_p^\times$.
Then the second explicit reciprocity law can be rewritten as follows.

\begin{thm}[Explicit Reciprocity Law II$'$]\label{ERLIIint-thm}
The element $\sL_p^{Gr}(g/L)\in \sR^\ur$ satisfies: 
For $\chi\in \Xi^{(II)}$, 
$$
\phi_\chi(\sL_{p}^{Gr}(g/L)) = 
- w_L (\chi_1(\gamma_v)-1)\CE'(\chi) \frac{n!(n-1)!\pi^{2m-2n-1}}{2^{2n-2m+2}D_L^{m/2}}\cdot \Omega_p^{2m}\frac{L(1, g, \psi_\chi)}{\Omega_\infty^{2m}},
$$
where 
$$
\CE'(\chi) = 
\begin{cases}
(1-a(p)\psi_\chi(\varpi_{\bar v})^{-1}p^{-1} + \psi_\chi(\varpi_{\bar v})^{-2}p^{-1})^2 & \zeta = 1, n\equiv 0\mod p-1 \\
(p^{t+1}/\mathfrak{g}(\psi_\zeta^{-1}\omega^n))^2 (p^{2n-1}/\chi_1(\varpi_{\bar v}))^{t+1}
& \text{else}.
\end{cases}
$$
\end{thm}

\noindent Here $(\Omega_p,\Omega_\infty)$ are the CM periods as in Theorem \ref{pKatzL},
and 
$$
L(s,g,\psi_\chi) = \sum_{\mathfrak{a}, (\mathfrak{a},\mathfrak{f}_{\psi_\chi})=1}
a_g(N(\mathfrak{a})) \psi_\chi(x_{\mathfrak{a}}) N(\mathfrak{a})^{-s} = L(s,g,h_{\chi_1}^0, \psi_{\zeta}^{-1}\omega^n).
$$
The key to this rewrite of the second explicit reciprocity law is the following formula (essentially due to Shimura -- see \cite[\S 7]{HT1}), which expresses ${\langle h_{v,\chi_1}^0,h_{v,\chi_1}^0\rangle}$
in terms of $L(1,\chi_1/\chi_1^c)$:
$$
{\langle h_{v,\chi_1}^0,h_{v,\chi_1}^0\rangle} = \frac{h_L m! D_L^{1/2} }{w_L 2^{m-1}(2\pi)^{m+1}} L(1,\chi_1/\chi_1^c).
$$
Combined with the interpolation formula for the Katz $p$-adic $L$-function $\sL_v(L)$ from Theorem \ref{pKatzL}, this easily yields
the formula in Theorem \ref{ERLIIint-thm}.

\begin{remark}\label{BFint-rmk} {\it A priori}, the class $\mathcal{BF}(g/L)$ depends on the choice of an auxiliary integer $c$. However, as
the notation suggest, the class can be shown to be independent of this choice. This in fact follows from 
Theorem \ref{ERLIIint-thm}, the injectivity of $\sL^{\mathrm{int}}$, and the Zariski density of the specialisations $\phi_\chi$, $\chi\in \Xi^{(II)}$. 
\end{remark}

\subsubsection{Comparisons with other $p$-adic $L$-functions}
Our application of the Beilinson--Flach elements to the conjecture of Perrin-Riou and other problems stems in large part from being able to realize various $p$-adic $L$-functions as specialisations of the elements $\sL_p(g/L)$ and $\sL_p^{Gr}(g/L)$. 
\medskip

\noindent{\it Comparison with cyclotomic $L$-functions.} From Theorem \ref{ERLIint-thm} together with Proposition \ref{Col-cycsp-prop} we conclude the following:

\begin{prop}\label{ERLIint-prop} 
Let $g' = g\otimes\chi_L$. Let $0\neq \omega\in S_{F,g}$, and let $\gamma\in V_{F,g}$ and $\gamma'\in V_{F,g'}$
such that $\gamma^\pm\neq 0$ and $(\gamma')^\pm\neq 0$. 
There exists a constant $c(\omega,\gamma,\gamma') \in F^\times$ such that 
\begin{itemize}
\item[(i)] $\sL_p(g/L) \mod (\gamma_v-1) = c(\omega,\gamma,\gamma')\mathfrak{g}(\chi_L)^{-1}
\CL_{\alpha,\omega,\gamma,\gamma'}(g/L) \in  \Lambda_{\CO_\lambda} = \sR/(\gamma_v-1)\sR$,
\item[(ii)] if (irr$_\BQ$) holds, $\omega\in S_{g,\CO}$ is good, and $\gamma\in T_{g,\CO}$ and $\gamma'\in T_{g',\CO}$
are such that $\gamma^\pm$ is an $\CO$-basis of $T_{\CO,g}^\pm$ and $(\gamma')^\pm$ is an $\CO$-basis of $T_{\CO,g'}^\pm$,
then $c(\omega,\gamma,\gamma') \in \CO^\times$.
\end{itemize}
\end{prop}
\noindent In part (i), $\CL_{\alpha,\omega,\gamma,\gamma'}(g/L)$ is the cyclotomic $p$-adic $L$-function for $g$ over $L$ as in 
Section \ref{sspLcy}.

\begin{proof}
By comparing Theorems \ref{ERLIint-thm} and \ref{pcycQ} we see that
$$\sL_p(g/L) \mod (\gamma_v-1) = c(\omega,\gamma,\gamma')\mathfrak{g}(\chi_L)^{-1} 
\CL_{\alpha,\omega,\gamma}(g)\CL_{\alpha,\omega',\gamma'}(g') = 
c(\omega,\gamma,\gamma')\mathfrak{g}(\chi_L)^{-1}\CL_{\alpha,\omega,\gamma,\gamma'}(g/L),
$$
where $\omega'$ is the image of $\omega$ under the isomorphism \eqref{RigMod-Opt} and 
$$ 
c(\omega,\gamma,\gamma') = \frac{\Omega_{\omega,+}\Omega_{\omega',+}\mathfrak{g}(\chi_L)}{\pi^2(-i)2^3 \Omega_g}.
$$
Here $\Omega_{\omega,\pm}$ (resp.~$\Omega_{\omega',\pm}$) are the periods determined by $\omega$ and $\gamma^\pm$
(resp.~by $\omega'$ and $(\gamma')^\pm$) as in Section \ref{Periods}.
From the definitions of the periods and Lemmas \ref{PerTw}  and \ref{PetPer} we have 
$$
\Omega_{\omega,+}\sim_{F^\times}\Omega_g^+, \ \ \Omega_{\omega',+}\sim_{F^\times}\Omega_{g'}^+ \sim_{F^\times}\mathfrak{g}(\chi_L)^{-1}\Omega_g^-, \
\ \text{and} \ \ 
\Omega_g \sim_{F^\times} (4\pi^2 i)^{-1}\Omega_g^+\Omega_g^-.
$$
So $c(\omega,\gamma,\gamma')\sim_{F^\times} 1$, proving part (i). If (irr$_\BQ$) holds and $\omega$, $\gamma$, and $\gamma'$ are as in part (ii), then each $\sim_{F^\times}$ can be replaced with $\sim_{\CO^\times}$, yielding part (ii).
\end{proof} 

\medskip

\noindent{\it Comparison with the $L$-function of Bertolini--Darmon--Prasanna.}
If the pair $(g,L)$ satisfies the Heegner hypothesis:
\begin{equation}\label{Heeg}\tag{Heeg}
\ell\mid N \implies \text{$\ell$ splits in $L$},
\end{equation} 
then Bertolini, Darmon, and Prasanna \cite{BDP1} constructed an anticyclotomic $p$-adic $L$-function interpolating the special values
$L(1,g,\psi)$ for $\psi$ a Hecke character with infinity type $(-n,n)$:

\begin{thm}\label{pBDPL}
Suppose \eqref{Heeg} holds. There exists $\sL_v^{BDP}(g/L) \in \Lambda_{L,\CO^\ur}^\ac = \CO^\ur[\![\Gamma_L^\ac]\!]$
such that for $\chi:\Gamma_L^\ac\rightarrow\ov{\BQ}_p^\times$ with $\chi(\gamma_\ac^{h_p}) = \epsilon(\gamma)^n$ for some positive integer
$n>0$, $n\equiv 0 \mod p-1$, 
$$
\phi_{\chi^{-1}}(\sL_v^{BDP}(g/L)) = (1-a(p)\psi_\chi(\varpi_{\bar v})^{-1}p^{-1} + \psi_\chi(\varpi_{\bar v})^{-2}p^{-1})^2
\cdot w_L \frac{n!(n-1)!(2\pi)^{2n-1}}
{4 D_L^n} \Omega_p^{4n}\frac{L(1,g,\psi_\chi)}{\Omega_\infty^{4n}}.
$$
\end{thm}
\noindent Here again $(\Omega_p,\Omega_\infty)$ are the CM periods as in Theorem \ref{pKatzL} 
and 
$\psi_\chi$ is the algebraic Hecke character of $L$ with infinity type $(-n,n)$ such that
$\sigma_{\psi_\chi}$ is the composition 
$G_L\twoheadrightarrow \Gamma_L^\ac \stackrel{\chi}{\rightarrow}\overline{\BQ}_p^\times$.
In fact, $\sL_v^{BDP}(g/L)$ is only constructed as a continuous function in \cite{BDP1}. That it is in fact a measure
(and in particular an element of $\CO^\ur[\![\Gamma_L^\ac]\!]$) is shown in \cite{CH}.

\begin{remark} The newform $g$ is not required to be ordinary in \cite{BDP1}: Theorem \ref{pBDPL} holds
in the both the ordinary and supersingular cases. 
\end{remark}

Of crucial importance to us is the following remarkable formula of Bertolini, Darmon, and Prasanna \cite[Thm.~5.13]{BDP1}, which
generalizes and extends a theorem of Rubin \cite{Ru1} for the CM case:

\begin{thm}\label{BDPformula} Suppose \eqref{Heeg} holds.
The image of $\sL_v^{BDP}(g/L)$ under the specialisation corresponding
to the trivial character ${\bf 1}: \Gamma_L^\ac\rightarrow \BQ_p$ (which is 
 the augmentation map $\Lambda_{L,\CO^\ur}^v\twoheadrightarrow \CO^\ur$) is
$$
\phi_{\bf 1}(\sL_v^{BDP}(g/L)) = (1-a(p)p^{-1}+p^{-1})^2 (\log_{\omega_g}(y_L))^2,
$$
where $\log_{\omega_g}$ is the $p$-adic logarithm on $J_0(N)$ for the differential $\omega_g$ and $y_L \in J_0(N)(L)$ is a Heegner point.
\end{thm}

Comparing Theorems \ref{pBDPL} and \ref{ERLIIint-thm} yields:
\begin{prop}\label{GRL=BDPL} Suppose \eqref{Heeg} holds. 
The image of $\sL_p^{Gr}(g/L)$ under the map
$\phi_\ac:\sR^\ur \twoheadrightarrow \Lambda_{L,\CO^\ur}^{ac}$ induced from the homomorphism
$\Gamma_L^v\times\Gamma\twoheadrightarrow \Gamma_L^\ac$, $\gamma_v \mapsto \gamma_\ac^2$ and 
$\gamma\mapsto \gamma_\ac^{h_p}$, equals $-T_\ac^2$ times the image of $\sL_v^{BDP}(g/L)$ under the involution 
$\iota_\ac$ of $\Lambda_L^\ac$ induced by $\gamma_\ac\mapsto \gamma_\ac^{-1}$:
$$
\phi_\ac(\sL_p^{Gr}(g/L)) = -((1+T_\ac)^2-1) \cdot \iota_\ac(\sL_v^{BDP}(g/L)).
$$
\end{prop}

Recall that if the Heegner hypothesis \eqref{Heeg} holds, then the root number 
$\epsilon(g/L)$ equals $-1$. When this root number is $+1$ we have:
\begin{prop}\label{GRLvan-prop} Suppose $\epsilon(g/L)=+1$. Let $\phi_\ac$ be
as in Proposition \ref{GRL=BDPL}. Then 
$
\phi_\ac(\sL^{Gr}(g/L)) = 0.
$
\end{prop}

\begin{proof} Let $\chi:\Gamma_L^\ac \rightarrow \BQ_p^\times$ be a character such that
$\chi(\gamma_\ac^{h_p}) = \epsilon(\gamma_\cyc)^n$ for some integer $n>0$, $n\equiv 0 \mod (p-1)$.
Then it follows from Theorem \ref{ERLIIint-thm} that $\phi_\chi(\phi_\ac(\sL^{Gr}(g/L)))$ is a multiple
of $L(1,g,\chi)$. The hypothesis that $\epsilon(g/L) = +1$ together with $(p,ND_L) = 1$ implies that the root number $\eps(1,g,\chi) = -1$ (see \cite[Lem.~3.1(2)]{CST}).  
 In particular, $L(1,g,\chi) = 0$. 
Since the (kernels of) specialisations at such characters $\chi$ are Zariski dense in $\CO_\lambda^\ur[\![\Gamma_\ac]\!]$, this implies that $\phi_\ac(\sL^{Gr}(g/L)) = 0$.
\end{proof}

\subsection{First applications: $\BT_1$ satisfies \eqref{Ind-eq} and $\CBF(g/L)$ arises from $\BT_1$} 
In this section we prove the following theorems:
\begin{thm}\label{BT-thm1} The $\Lambda_L^v[G_L]$-module $\BT_1$ satisfies \eqref{Ind-eq}.
\end{thm}
\begin{thm}\label{BT-thm2}
The Beilinson--Flach element $\CBF(g/L)$ belongs to the submodule
$H^1_{\mathrm{rel},\ord}(\BQ,T(1)\hat\otimes\BT_1\hat\otimes\Lambda)$ of 
$H^1_{\mathrm{rel},\ord}(\BQ,T(1)\hat\otimes\BH_1\hat\otimes\Lambda)$.
\end{thm}
\noindent Our proofs of Theorems \ref{BT-thm1} and \ref{BT-thm2} are intertwined.

Let $\widetilde\BT\subset\BI\subsetneq \widetilde\BI$ be an induced lattice as in Lemma \ref{ind-I-lem}.
By Proposition \ref{Sat-prop3}(iii),(iv), the $\Lambda_L^v = \BZ_p[\![T_v]\!]$-module $\widetilde\BH/\widetilde\BT$ is 
annihilated by $T_v$. It follows that as a $\BZ_p[G_\BQ]$-modules $\widetilde\BH/\widetilde\BT$, $\BI/\widetilde\BT$, and
$\widetilde\BH/\BI$ are all quotients of $\BZ_p\oplus\BZ_p(\chi_L)$.
Consequently, as $H^0(\BQ,T(1)\hat\otimes\Lambda) = 0 = H^0(\BQ,T(1)\hat\otimes\Lambda(\chi_L))$,
the inclusions $\widetilde\BT \subset\BI\subset \widetilde\BH$ induce inclusions
$H^1(\BQ,T(1)\hat\otimes\widetilde\BT\hat\otimes\Lambda) \hookrightarrow 
H^1(\BQ,T(1)\hat\otimes\BI\hat\otimes\Lambda) \hookrightarrow 
H^1(\BQ,T(1)\hat\otimes\widetilde\BH\hat\otimes\Lambda),$ 
by which we view the first two as submodules of the third.

Let $\CBF\in H^1(\BQ,T(1)\hat\otimes\widetilde\BH\hat\otimes\Lambda)$ be the image of $\CBF(g/L)$ under the map induced by the inclusion
$\BH_1\subset \widetilde\BH$. 
The image of $\loc_p(\CBF)$ in $H^1(\BQ_p,T(1)^-\hat\otimes(\widetilde\BH/\widetilde\BH_v)\hat\otimes\Lambda)$
is the image of $\loc_p(\CBF(g/L))$ under the induced map $H^1(\BQ_p,T(1)^-\hat\otimes \BH_1^-\hat\otimes\Lambda)
\rightarrow H^1(\BQ_p,T(1)^-\hat\otimes(\widetilde\BH/\widetilde\BH_v)\hat\otimes\Lambda)$.
So it follows from Lemma \ref{BF-locp-lem} that 
$$
\CBF\in H^1_{\mathrm{rel},\ord}(\BQ,T(1)\hat\otimes\widetilde\BH\hat\otimes\Lambda)
= \{ c\in H^1(\BQ,T(1)\hat\otimes\widetilde\BH\hat\otimes\Lambda) \ : \ \loc_p(c) = 0 \in H^1(\BQ_p,T(1)^-\hat\otimes(\widetilde\BH/\widetilde\BH_v)\hat\otimes\Lambda)\}.
$$
As a step toward proving the above theorems, we show:
\begin{prop}\label{BF-I-prop} The Beilinson--Flach element $\CBF$ satisfies
$$
\CBF\in H^1_{\mathrm{rel},\ord}(\BQ,T(1)\hat\otimes\BI\hat\otimes\Lambda)
= \{ c\in H^1(\BQ,T(1)\hat\otimes\BI\hat\otimes\Lambda \ : \ \loc_p(c) = 0 \in H^1(\BQ_p,T(1)^-\hat\otimes(\BI/\BI_v)\hat\otimes\Lambda)\}.
$$
Moreover, if \eqref{Ind-eq} does not hold, then $\CBF \in T_v H^1_{\mathrm{rel},\ord}(\BQ,T(1)\hat\otimes\BI\hat\otimes\Lambda)$.
\end{prop}

\begin{proof}
Since $T_v$ annihilates $\widetilde\BH/\BI$, we certainly have $\CBF' = T_v\CBF\in H^1(\BQ,T(1)\hat\otimes\BI\hat\otimes\Lambda)$.
We also have that $G_{\BQ_p}$ acts trivially on $\widetilde\BH/\BI$, from which it follows easily that 
$H^0(\BQ_p,T^-(1)\hat\otimes(\widetilde\BH/(\BI+\widetilde\BH_v)\hat\otimes\Lambda) = 0$ and hence that
$H^1(\BQ_p,T^-(1)\hat\otimes(\BI/\BI_v)\hat\otimes\Lambda)\hookrightarrow H^1(\BQ_p,T^-(1)\hat\otimes(\widetilde\BH/\widetilde\BH_v)\hat\otimes\Lambda)$.
Consequently, $\CBF'$ belongs to the submodule $H^1_{\mathrm{rel},\ord}(\BQ,T(1)\hat\otimes\BI\hat\otimes\Lambda)$.

Since $\BI = \BI_v\oplus c\cdot\BI_v$
we have $H^1(\BQ, T(1)\hat\otimes\BI\hat\otimes\Lambda) = H^1(L,T(1)\hat\otimes\BI_v\hat\otimes\Lambda)$ by Shapiro's Lemma, which yields
an identification $H^1_{\mathrm{rel},\ord}(\BQ,T(1)\hat\otimes\BI\hat\otimes\Lambda) = H^1_{\mathrm{rel},\ord}(L,T(1)\hat\otimes\BI_v\hat\otimes\Lambda)$, where 
$H^1_{\mathrm{rel},\ord}(L,T(1)\hat\otimes\BI_v\hat\otimes\Lambda) = \{ c\in H^1_{\mathrm{rel},\ord}(L,T(1)\hat\otimes\BI_v\hat\otimes\Lambda) \ : \
\loc_p(c) = 0 \in H^1(L_{\bar v}, T^-(1)\hat\otimes\BI_v\hat\otimes\Lambda)$. 
This induces
an injection
\begin{equation}\label{H1-L-eq2}
H^1_{\mathrm{rel},\ord}(\BQ,T(1)\hat\otimes\BI\hat\otimes\Lambda)/T_v H^1_{\mathrm{rel},\ord}(\BQ,T(1)\hat\otimes\BI\hat\otimes\Lambda) 
\hookrightarrow H^1_{\mathrm{rel},\ord}(L,T(1)\hat\otimes\Lambda)
\end{equation}
after choosing an isomormophism $\BI_v\simeq \Lambda_L^v$ (and hence an isomorphism $\BI_v/T_v\BI_v\simeq\BZ_p$).
Let $\ov{\CBF}'\in H^1_{\mathrm{rel},\ord}(L,T(1)\hat\otimes\Lambda)$ denote the image of $\CBF'$ under \eqref{H1-L-eq2}. 

We now consider the image of $\loc_{v}(\ov{\CBF}')$ under the Coleman map $Col_{\eta_\omega,v}$.
Since either $\BT^+ = \BI_v$ or $\BT^+ = T_v\cdot\BI_v$ and $\BH^+ = \BT^+$, we have a commutative diagram
\begin{equation}\label{H1-L-eq3}
\begin{tikzcd}[row sep=2.5em]
H^1_{\mathrm{rel},\ord}(\BQ,T(1)\hat\otimes\BI\hat\otimes\Lambda) \arrow[r,"\eqref{H1-L-eq2}"]  \arrow[d,"\loc_p"] & H^1_{\mathrm{rel},\ord}(L,T(1)\hat\otimes\Lambda) \arrow[d,"\loc_v"] & \\
 H^1(\BQ_p,T^-(1)\hat\otimes\BI_v\hat\otimes\Lambda) \arrow[r,"/T_v"] & H^1(L_v,T^-(1)\hat\otimes\Lambda) \arrow[r,"Col_{\eta_\omega,v}"] & 
 \Lambda_{\CO_\lambda}\otimes_{\BZ_p}\BQ_p \\
 H^1(\BQ_p,T^-(1)\hat\otimes\BH^+\hat\otimes\Lambda) \arrow[r,"/T_v"] \arrow[u,]  & H^1(\BQ_p,T^-(1)\hat\otimes\Lambda) \arrow[r,"Col_{\eta_\omega}"]  \arrow[u,"(*)"] & 
 \Lambda_{\CO_\lambda}\otimes_{\BZ_p}\BQ_p \arrow[u,"(**)"] . 
\end{tikzcd}
\end{equation}
Here we have fixed the isomorphism $\BI_v\simeq\Lambda_L^v$ to restrict to the isomorphism $\BT^+=\BH^+\simeq\Lambda_L^v$ determined
by $\omega_{\bh_v}$ as in the proof of Proposition \ref{Col-cycsp-prop}. This in turn induces an isomorphism $\BI_v/T_v\BI_v\simeq \BZ_p$,
which we take as the isomorphism for \eqref{H1-L-eq2}. The lower veritical arrows in \eqref{H1-L-eq3} are induced from the inclusion
$\BT^+=\BH^+ \subset\BI_v$. In particular, $(*)$ and $(**)$ are both the identity map if $\BT^+ = \BT_v$ (so if \eqref{Ind-eq} holds) and are both the zero
map if $\BT^+=T_v\BI_v$ (so if \eqref{Ind-eq} does not hold).
Since $\loc_p(\CBF') \in H^1(\BQ_p,T^-(1)\hat\otimes\BI_v\hat\otimes\Lambda)$ is the image of 
$T_v\loc_p(\cdot\CBF(g/L)) \in  H^1(\BQ_p,T^-(1)\hat\otimes\BH^+\hat\otimes\Lambda)$, it follows from the commutativity of \eqref{H1-L-eq3} that
$Col_{\eta_\omega,v}(\loc_{v}(\ov{\CBF}')) = 0$.  

By Theorem \ref{cycIwL-Box}, $H^1_{\mathrm{rel},\ord}(L,T(1)\hat\otimes\Lambda)$ is a torsion-free $\Lambda_{\CO_\lambda}$-module and 
$H^1_{\mathrm{rel},\ord}(L,T(1)\hat\otimes\Lambda)\otimes_{\BZ_p}\BQ_p$ is a free $\Lambda_{\CO_\lambda}\otimes_{\BZ_p}\BQ_p$-module of rank one.
As $Col_{\eta_\omega,v}\circ\loc_v$ is non-zero on
$H^1_{\mathrm{rel},\ord}(L,T(1)\hat\otimes\Lambda)$ (see the proof of Theorem \ref{cycIwL-Box-II}), we have an injection
$
Col_{\eta_\omega,v}\circ\loc_v : H^1_{\mathrm{rel},\ord}(L,T(1)\hat\otimes\Lambda)  \hookrightarrow \Lambda_{\CO_\lambda}\otimes_{\BZ_p}\BQ_p$.
Since $Col_{\eta_\omega,v}(\loc_{v}(\ov{\CBF}')) = 0$, it then follows that $\ov{\CBF}'=0$ and hence that 
$\CBF'\in T_v H^1_{\mathrm{rel},\ord}(\BQ,T(1)\hat\otimes\BI\hat\otimes\Lambda)$ by \eqref{H1-L-eq2}.
As $H^1(\BQ,T(1)\hat\otimes\BI\hat\otimes\Lambda)$ has no $T_v$-torsion (since $H^0(\BQ,T(1)\hat\otimes(\BZ_p\oplus\BZ_p(\chi_L))\hat\otimes\Lambda)=0$),
it follows that $\CBF \in H^1_{\mathrm{rel},\ord}(\BQ,T(1)\hat\otimes\BI\hat\otimes\Lambda)$, as claimed.

If furthermore \eqref{Ind-eq} does not hold, then both (*) and (**) are the zero map, and the same arguments now applied to $\CBF$ in place of $\CBF'$ show
that $\CBF \in T_vH^1_{\mathrm{rel},\ord}(\BQ,T(1)\hat\otimes\BI\hat\otimes\Lambda)$.
\end{proof}

We can now complete the proof of Theorems \ref{BT-thm1} and \ref{BT-thm2}.
If \eqref{Ind-eq} holds, then there is nothing more to prove: Theorem \ref{BT-thm2} follows from Proposition \ref{BF-I-prop}.
So suppose \eqref{Ind-eq} does not hold. Then by the same proposition, $\CBF\in T_v H^1_{\mathrm{rel},\ord}(\BQ,T(1)\hat\otimes\BI\hat\otimes\Lambda)$.
We will use the Explicit Reciprocity Law II to deduce a contradiction.

If $\CBF\in T_v H^1_{\mathrm{rel},\ord}(\BQ,T(1)\hat\otimes\BI\hat\otimes\Lambda)$, then 
$\loc_p(\CBF) \in T_v H^1(\BQ_p,T^+(1)\hat\otimes(\BI/\BI_v)\otimes\Lambda)$. On the other hand, since $\BI$ is as in Lemma \ref{ind-I-lem},
the inclusion $\BI\subset\widetilde\BH$ induces an isomorphism $\BI/\BI_v\simeq T_v(\widetilde\BH/\widetilde\BH^+)$. Consequently,
$\loc_p(\CBF) \in T_v^2 H^1(\BQ_p,T^+(1)\hat\otimes\widetilde\BH^-\otimes\Lambda)$. It follows that we must have
\begin{equation}\label{Ind-thm-eq1}
\sL^{Gr}_p(g/L)  = \sL^{\mathrm{int}}(\loc_p(\CBF)) \in T_v^2\sR^\ur.
\end{equation}
Since this must hold for {\em any} ordinary newform $g$ if \eqref{Ind-eq} does not hold, the following proposition yields the desired contradiction and hence completes the proofs
of Theorems \ref{BT-thm1} and \ref{BT-thm2}.

\begin{prop}\label{aux-g-prop} There exists a positive integer $N$ and a newform $g\in S_2(\Gamma_0(N))$ such that 
\begin{itemize}
\item[(i)] $(N,p) = 1$,
\item[(ii)] $g$ is $p$-ordinary, that is, there exists a prime $\lambda\mid p$ of the Hecke field $F_g$ such that $a_p(g)$ is a $\lambda$-adic unit,
\item[(iii)] $\sL_p^{Gr}(g/L) \not \in T_v^2\sR^{\ur}$.
\end{itemize}
\end{prop}
\noindent The proof of this proposition is contained in the next section.

\subsubsection{An auxiliary newform}\label{aux-new} To prove Proposition \ref{aux-g-prop}
we first prove:
\begin{prop}\label{anrk-prop}
There exists a positive integer $N$ and a newform $g\in S_{2}(\Gamma_{0}(N))$ such that
\begin{itemize}
\item[(i)] $p\nmid N$,
\item[(ii)] every prime $\ell\mid N$ splits in $L$,
\item[(iii)] $g$ is $p$-ordinary, that is, there exists a prime $\lambda\mid p$ of the Hecke field $F_g$ such that $a_p(g)$ is a $\lambda$-adic unit,
 \item[(iv)] $\ord_{s=1}L(s, g_{/L})=1$.
\end{itemize}
\end{prop}
\noindent Note that condition (ii) is just the usual Heegner condition for the pair $(g,L)$.

\begin{proof}
Let $K$ be an imaginary quadratic field with discriminant $D_K$ coprime to $pD_L$ and such that $p$ splits in $K$ and every prime
$\ell\mid D_K$ splits in $L$. 
Let $\lambda$ be a Hecke character over $K$ with infinity type $(-1,0)$ such that 
\begin{itemize}
\item $\lambda$ is conjugate-dual: $\lambda^\tau = \lambda^{-1}|\cdot|_K$, where the superscript $\tau$ denotes composition with the action of complex conjugation 
$\tau$ on $K$,
\item $p \nmid \cond(\lambda)$, and 
\item the conductor of $\lambda$ is only divisible by primes whose residue characteristic is split in $L$. 
\end{itemize}
In particular, the CM modular modular form corresponding to the Hecke character $\lambda$ satisfies the classical Heegner hypothesis with respect to $L$,
and
$\epsilon(1/2,\lambda) \cdot \epsilon(1/2, \lambda\chi_{L})=-1$.

Let $\ell\nmid 2p\cdot D_{K}D_{L}N(\cond(\lambda))$ be any prime that splits in both $K$ and $L$.
Let $\mathfrak{X}_{K,\ell}^{ac}$ denote the set of finite order anticyclotomic Hecke characters over $K$ having
$\ell$-power conductor.  If $\epsilon(\frac{1}{2},\lambda)=+1$, then by the main result of \cite{Ro} (also see \cite[\S1]{Gr})
$
L(1,\lambda \cdot \chi) \cdot L'(1,\lambda\chi_{L}\cdot \chi) \neq 0
$
for all but finitely many $\chi \in \mathfrak{X}_{K,\ell}^{ac}$. 
Otherwise, 
$
L'(1,\lambda \cdot \chi) \cdot L(1,\lambda\chi_{L}\cdot \chi) \neq 0
$
for all but finitely many $\chi \in \mathfrak{X}_{K,\ell}^{ac}$.
In either case, let $\chi_{0} \in \mathfrak{X}_{K,\ell}^{ac}$ be a Hecke character for which the non-vanishing holds.
We may then take $g$ to be the CM modular form corresponding to the Hecke character $\lambda\chi_{0}$.
\end{proof}

\begin{remark}\label{aux-g-rmk}
It is natural to ask: Does there exist a non-CM elliptic newform $g$ satisfying Proposition \ref{anrk-prop}(iv)?
\end{remark}

We now explain how Proposition \ref{aux-g-prop} follows from Proposition \ref{anrk-prop}.

\begin{proof}[Proof of Proposition \ref{aux-g-prop}]
Let $g$ be as in Proposition \ref{anrk-prop}. Since the pair $(g,L)$ satisfies the Heegner hypothesis by
Proposition \ref{anrk-prop}(ii), it folllows from Proposition \ref{anrk-prop}(iv) and the Gross--Zagier formula for the pair $(g,L)$ that 
the image of the Heegner point $y_L\in J_0(N)(L)$ has infinite order in the abelian variety quotient associated to $g$.
In particular, the $p$-adic logarithm $\log_{\omega_g}(y_L)$ is non-zero. Combined with the formula
of Theorem \ref{BDPformula}, this implies that $T_\ac\nmid \sL_v^{BDP}(g/L)$. On the other hand, Proposition \ref{GRL=BDPL}
implies that if $T_v^2\mid \sL_p^{Gr}(g/L)$ then $T_\ac \mid \sL_v^{BDP}(g/L)$. Together, this shows
$T_v^2 \nmid \sL_p^{Gr}(g/L)$, proving Proposition \ref{aux-g-prop}.
\end{proof}

\subsection{A two-variable zeta element}\label{two-variable-zeta}
In light of Theorems \ref{BT-thm1} and \ref{BT-thm2}, we can use the Beilinson--Flach element $\mathcal{BF}(g/L)$ to define a
two-variable zeta element for $g$, as explained in the following.

\subsubsection{The zeta element $\CZ(g/L)$}
By Theorem \ref{BT-thm1}, \eqref{Ind-eq} holds. In particular $\BT_1$ is identified with the induction from $G_L$ to $G_\BQ$ of the $\Lambda_L^v[G_L]$-module
$\BT^+$. Recall that $\BT^+$ is a free $\Lambda_L^v$-module of rank one on which $G_L$ acts via the canonical character $\Psi_L^v$ and, furthermore,
we have a preferred $\Lambda_L^v$-isomorphism $\BT^+\isoarrow \Lambda_L^v(\Psi_L^v)$ arising from the map $\omega_{\bh_v}$. The latter determines via 
Shapiro's Lemma identifications
$$
H^1(\BQ,T(1)\hat\otimes\BT_1\hat\otimes\Lambda) \isoarrow H^1(L,T(1)\hat\otimes\BT^+\hat\otimes\Lambda) \isoarrow H^1(L,T(1)\hat\otimes\Lambda_L^v(\Psi_L^v)\hat\otimes\Lambda)
$$
of $\sR$-modules.
Similarly, $H^1(\BQ_p,T(1)\hat\otimes\BT^+\hat\otimes\Lambda)$ is identified with $H^1(L_v,T(1)\hat\otimes\Lambda_L^v(\Psi_L^v)\hat\otimes\Lambda)$
and $H^1(\BQ_p,T(1)\hat\otimes(c\cdot \BT^+)\hat\otimes\Lambda)$ is identified with $H^1(L_{\bar v},T(1)\hat\otimes\Lambda_L^v(\Psi_L^v)\hat\otimes\Lambda)$,
and also with $T(1)$ replaced by $T^\pm(1)$.
It follows that we have an identification
$
H^1_{\rel,\ord}(\BQ,T(1)\hat\otimes\BT_1\hat\otimes\Lambda) \isoarrow H^1_{\rel,\ord}(L,T(1)\hat\otimes\Lambda_L^v(\Psi_L^v)\hat\otimes\Lambda)
$
where the subscripts `$\rel$' and `$\ord$' on the right-hand side denote the submodule of classes 
$c$ such that no condition is imposed on $\loc_v(c)$ but we require $\loc_{\bar v}(c)\in H^1(L_{\bar v},T^+(1)\hat\otimes\Lambda_v^L(\Psi_L^v)\hat\otimes\Lambda)$.  We may thus view 
$\CBF(g/L)$ as an element of $H^1_{\rel,\ord}(L,T(1)\hat\otimes\Lambda_v^L(\Psi_L^v)\hat\otimes\Lambda)$.

We now consider the composition of isomorphism 
\begin{equation}\label{Lgal-twist-eq}
\Gamma_L \isoarrow \Gamma_L^v\times\Gamma \stackrel{\gamma_v\mapsto\gamma_v^{-1}}{\isoarrow} \Gamma_L^v\times\Gamma,
\end{equation}
this determines an isomorphism $\Lambda_L \isoarrow \Lambda_L^v\hat\otimes\Lambda$ 
and hence also isomorphism $\theta:\Lambda_{L,\CO_\lambda}\isoarrow \sR$ and 
$\theta^\ur:\Lambda_{L,\CO_\lambda^\ur}\isoarrow\sR^\ur$.  
The isomorphism \eqref{Lgal-twist-eq} also induces 
$
T(1)\hat\otimes \Lambda_L(\Psi_L^{-1}) \isoarrow T(1)\hat\otimes \Lambda_L^v(\Psi_L^v)\hat\otimes\Lambda(\Psi^{-1}),
$
which is compatible with $\theta$. We thus obtain an identification
$
H^1_{\rel,\ord}(L,T(1)\hat\otimes\Lambda_L) \simeq H^1_{\rel,\ord}(L,T(1)\hat\otimes \Lambda_L^v(\Psi_L^v)\hat\otimes\Lambda),
$
that is compatible with $\theta$, where $G_L$ acts on $\Lambda_L$ in the left-hand side via the inverse of the canonical character, and 
the subscripts `$\rel$' and `$\ord$' on the left-hand side denote the submodule of classes $c\in H^1_{\rel,\ord}(L,T(1)\hat\otimes\Lambda_L)$ such that $\loc_{\bar v}(c)\in H^1(L_{\bar v},T^+(1)\hat\otimes\Lambda_L)$. We let 
$$
\CZ(g/L) \in H^1_{\rel,\ord}(L,T(1)\hat\otimes\Lambda_L)
$$
be identified with $\CBF(g/L)$ under the preceding isomorphism.
This is the two-variable zeta element associated with $g$ and $L$ in the ordinary case.

We let $\CC_v: H^1(L_v,T^-(1)\hat\otimes\Lambda_L) \hookrightarrow \Lambda_{L,\CO_\lambda}$
be the composition
$$
\CC_v: H^1(L_v,T^-(1)\hat\otimes\Lambda_L)
\simeq H^1(\BQ_p,T^-(1)\hat\otimes\BT^+\hat\otimes\Lambda)
\stackrel{\sC^\mathrm{int}}{\hookrightarrow} \sR \stackrel{\theta^{-1}}{\isoarrow}\Lambda_{L,\CO_\lambda}.
$$
This is a $\Lambda_{L,\CO_\lambda}$-injection.  Let
$$
\CL_v(g/L) = \CC_v(\loc_v(\CZ(g/L))) \in \Lambda_{L,\CO_\lambda}.
$$
Note that $\CL_v(g/L) = \theta^{-1}(\sL_p(g/L))$.

We similarly let $\CL_{\bar v}: H^1(K_{\bar v},T^+(1)\hat\otimes\Lambda_L)\hookrightarrow \Lambda_{L,\CO_\lambda^\ur}$ be the composition
$$
\CL_{\bar v}: H^1(K_{\bar v},T^+(1)\hat\otimes\Lambda_L) \simeq H^1(\BQ_p,T^+(1)\hat\otimes\BT_1^-\hat\otimes\Lambda) \stackrel{\frac{1}{T_v}\sL^\mathrm{int}}{\hookrightarrow}\sR^\ur \stackrel{\theta^{\ur,-1}}{\isoarrow}\Lambda_{L,\CO_\lambda^\ur}.
$$
Note that the inclusion $\BT_1 = \widetilde\BT \subset \widetilde\BH$ induces
an isomorphism $\BT_1^- = T_v\widetilde\BH^-$, and so $\sL^\mathrm{int}$ maps
$H^1(\BQ_p,T^+(1)\hat\otimes\BT_1^-\hat\otimes\Lambda)$ into $T_v\sR^\ur$. In particular,
the middle arrow of the composition defining $\CL_{\bar v}$ is well-defined.
We also let 
$$
\CL_p^{Gr}(g/L) = \CL_{\bar v}(\loc_{\bar v}(\CZ(g/L))) \in \Lambda_{L,\CO_\lambda^\ur}.
$$
So $\CL_p^{Gr}(g/L) = \theta^{\ur,-1}(\sL_p^{Gr}(g/L))$.

\subsubsection{Connections with cyclotomic $L$-functions and cyclotomic zeta elements}
From Propositions \ref{Col-cycsp-prop} and \ref{ERLIint-prop} we immediately conclude:
\begin{prop}\label{2varZ-prop} \hfill
\begin{itemize}
\item[(i)] The reduction of $\CC_v$ modulo $\gamma_\ac-1$ equals the composition
$$
H^1(L_v,T^-(1)\hat\otimes\Lambda_L)\stackrel{\mod (\gamma_v-1)}{\twoheadrightarrow}
H^1(L_v,T^-(1)\hat\otimes\Lambda)\stackrel{Col_{\eta_{\omega},v}}{\hookrightarrow}
\Lambda_{\CO_\lambda}.
$$
\item[(ii)] For $0\neq \omega\in S_F$, $\gamma\in V_{F,g}$, $\gamma'\in V_{F,g'}$, 
and $c(\omega,\gamma,\gamma')\in F^\times$ as in 
Proposition \ref{ERLIint-prop},
$$
\CL_v(g/L) \mod (\gamma_\ac-1) = c(\omega,\gamma,\gamma')\mathfrak{g}(\chi_L)^{-1}
\CL_{\alpha,\omega,\gamma,\gamma'}(g/L) \in  \Lambda_{\CO_\lambda} = \Lambda_{L,\CO_\lambda}/(\gamma_\ac-1)\Lambda_{L,\CO_\lambda}.
$$
\end{itemize}
\end{prop} 

We exploit this proposition to prove that the Beilinson--Kato element $\bz^\ord_{\alpha,\omega,\gamma,\gamma'}(g_{/L})$
is essentially the cyclotomic specialisation of the two-variable zeta element $\CZ(g/L)$.
\begin{thm}\label{2varZ-1varZ-thm} Let $g' = g\otimes\chi_L$. Let $0\neq \omega\in S_{F,g}$, and let $\gamma\in V_{F,g}$ and $\gamma'\in V_{F,g'}$
such that $\gamma^\pm\neq 0$ and $(\gamma')^\pm\neq 0$, and let $c(\omega,\gamma,\gamma') \in F^\times$ be as in Proposition \ref{ERLIint-prop}.
\begin{itemize}
\item[(i)] The image of $\CZ(g/L)$ under the map
$H^1_{\rel,\ord}(\CO_L[\frac{1}{p}],T(1)\otimes\Lambda_L) \rightarrow H^1_{\rel,\ord}(\CO_L[\frac{1}{p}],T(1)\otimes_{\BZ_p}\Lambda)\otimes_{\BZ_p}\BQ_p$
induced by the projection $\Lambda_L/(\gamma_\ac-1)\Lambda_L \isoarrow \Lambda_L^\ac\isoarrow\Lambda$
equals $c(\omega,\gamma,\gamma')\mathfrak{g}(\chi_L)^{-1} \bz^\ord_{\alpha,\omega,\gamma,\gamma'}$.
\item[(ii)] If (Van$_\BQ$) holds, $\omega\in S_{g,\CO}$ is good, and $\gamma\in T_{g,\CO}$ and $\gamma'\in T_{g',\CO}$
are such that $\gamma^\pm$ is an $\CO$-basis of $T_{\CO,g}$ and $(\gamma')^\pm$ is an $\CO$-basis of $T_{\CO,g'}^\pm$
(so $c(\omega,\gamma,\gamma') \in \CO^\times$), 
then the equality in \rm{(i)} holds in $H^1_{\rel,\ord}(\CO_L[\frac{1}{p},T(1)\otimes_{\BZ_p}\Lambda)$.
\end{itemize}
\end{thm}

\begin{proof} It follows from Proposition \ref{2varZ-prop}(i) that the diagram
$$
\begin{tikzcd}[row sep=2.5em, column sep=5em]
H^1_{\rel,\ord}(\CO_L[\frac{1}{p}],T(1)\otimes\Lambda_L) \arrow[d,"\loc_v"] \arrow[r,"\mod (\gamma_\ac-1)"] &
H^1_{\rel,\ord}(\CO_L[\frac{1}{p}],T(1)\otimes_{\BZ_p}\Lambda)\otimes_{\BZ_p}\BQ_p \arrow[d,"\loc_v"] \\
H^1(L_v,T^-(1)\otimes_{\BZ_p}\Lambda_L) \arrow[d,"\CC_v"] \arrow[r,"\mod (\gamma_\ac-1)"] &
H^1(L_v,T^-(1)\otimes_{\BZ_p}\Lambda)\otimes_{\BZ_p}\BQ_p \arrow[d,"Col_{\eta_{\omega},v}"] \\ 
\Lambda_{L,\CO_\lambda} \arrow[r,"\mod (\gamma_\ac-1)"] & \Lambda_{\CO_\lambda}\otimes_{\BZ_p}\BQ_p.
\end{tikzcd}
$$
commutes. 
By Theorem \ref{cycIwL-Box}(a), $H^1_{\rel,\ord}(\CO_L[\frac{1}{p}],T(1)\otimes_{\BZ_p}\Lambda)\otimes_{\BZ_p}\BQ_p$
is a free $(\Lambda_{\CO_\lambda}\otimes_{\BZ_p}\BQ_p)$-module of rank one. Since 
$Col_{\eta_\omega,v}(\bz^\ord_{\alpha,\omega,\gamma,\gamma'})\neq 0$ (cf.~the proof of Theorem \ref{cycIwL-Box-II}(a)), 
it follows that the composition of maps on the right of the diagram is an injection. So to prove part (i) of the theorem it suffices
to prove that $\CZ(g/L)$ and $c(\omega,\gamma,\gamma')\mathfrak{g}(\chi_L)^{-1} \bz^\ord_{\alpha,\omega,\gamma,\gamma'}$ have the same image in the 
bottom right corner of the diagram.  But this is a consequence of the explicit reciprocity laws in the guise of Proposition \ref{2varZ-prop}(ii)
and Theorem \ref{ERLBKII}(i). 

If (Van$_\BQ$) holds, then all this holds without $\otimes_{\BZ_p}\BQ_p$, which yields part (ii) of the theorem.
\end{proof}
\begin{remark}\label{EC-Katoelement-overL-int}
Suppose $g$ corresponds to the isogeny class of an elliptic curve $E_\bullet$ as in Remark \ref{EC-optperiod-rmk} but (irr$_\BQ$) does not hold. 
In light of Theorem~\ref{2varZ-1varZ-thm}, 
for $\omega\in S_{g,\CO}$ good, and $\gamma\in T_{g,\CO}$ and $\gamma'\in T_{g',\CO}$
such that $\gamma^\pm$ is an $\CO$-basis of $T_{\CO,g}$ and $(\gamma')^\pm$ is an $\CO$-basis of $T_{\CO,g'}^\pm$, it follows that 
$\bz^\ord_{\alpha.\omega,\gamma,\gamma'}(g_{/L})\in H^1_{\rel,\ord}(\CO_L[\frac{1}{p},T(1)\otimes_{\BZ_p}\Lambda)$. 
\end{remark}
\subsubsection{Connections with anti-cyclotomic $L$-functions and Heegner points}

From Propositions \ref{GRL=BDPL} and \ref{GRLvan-prop} and Theorem \ref{BDPformula} we conclude:
\begin{prop}\label{GRL=BDPL-II}\hfill
\begin{itemize}
\item[(i)] Suppose \eqref{Heeg} holds.
The image of $\CL_p^{Gr}(g/L)$ modulo $\gamma_+-1$ equals $-\sL_v^{BDP}(g/L)$. In particular, 
the image of $\CL_p^{Gr}(g/L)$ under $\phi_{\bf 1}: \Lambda_{L,\CO_\lambda^\ur} \twoheadrightarrow \Lambda_{L,\CO_\lambda^\ur}/(\gamma_+-1,\gamma_--1)\Lambda_{L,\CO_\lambda^\ur} = \CO_\lambda^\ur$ is
$$
\phi_{\bf 1}(\CL_p^{Gr}(g/L)) = - (1-a(p)p^{-1}+p^{-1})^2 (\log_{\omega_g}(y_L))^2.
$$
\item[(ii)] Suppose $\epsilon(g/L) = +1$. The image of $\CL_p^{Gr}(g/L)$ modulo $\gamma_+-1$
is $0$.
\end{itemize}
\end{prop}

This proposition allows us to relate the image of $\CZ(g/L)$ under Perrin-Riou's regulator map (or `big logarithm') to 
Heegner points, providing a key link in our subsequent proof of the Perrin-Riou Conjecture.

Recall that Perrin-Riou's regulator map for $H^1(L_{\bar v},T^+(1)\hat\otimes\Lambda_L^\ac)$ is the composition
\begin{equation*}\begin{array}{ccl}
\CL^{PR}: H^1(L_{\bar v}, T^+(1)\hat\otimes\Lambda_L^\ac) & \stackrel{\Psi_L^\ac(g)\mapsto\eps(g)^{-1}\Psi_L^\ac(g)}{\longrightarrow} &
H^1(L_{\bar v}, T^+\hat\otimes\Lambda_L^\ac) \\ & \stackrel{res}{\hookrightarrow}  & H^1(L_{\bar v}^\ur, T^+\hat\otimes\Lambda^\ac_L)  \\
& \stackrel{\CL_{T^+}}{\longrightarrow} & D(T^+)\hat\otimes W(\ov{\BF}_p)\hat\otimes\Lambda_L^\ac,
\end{array}
\end{equation*}
where for the third map we have used that as a $G_{L_{\bar v}^\ur}$-module $\Lambda^\ac_L$ is naturally isomorphic to the cyclotomic 
algebra $\Lambda$.
Let 
$$
\CL^{PR}_g = ([\omega_g,-]\otimes id\otimes id)\circ \CL^{PR}: H^1(L_{\bar v}, T^+(1)\hat\otimes\Lambda_L^\ac)\rightarrow 
\Lambda_{L,\CO_\lambda^\ur}^{\ac}.
$$
 From the specialisation properties of $\CL_{T^+}$ (see \cite[Thm.~B.5]{LZ0}) it follows that
\begin{equation}\label{PRLog=BKlog}
\phi_{\bf 1}\circ \CL_g^{PR} = (1-\alpha/p)(1-1/\alpha)^{-1}\log_{BK}^{\eta_g},
\end{equation}
where $\log_{BK}$ denotes the Bloch--Kato logarithm for $T^+$ and $\eta_g = \eta_{\omega_g}$ (so 
$\log_{BK}(-) = \log_{BK}^{\eta_g}(-)\cdot\eta_g \in D_\cris(V^+)$). 

\begin{lem}\label{GRL=PRLog}
There exists a unit $U_L \in (\Lambda_L^{\ac,\ur})^\times$ depending only on $L$ such that 
$$
U_L \cdot \CL_{\bar v} \mod (\gamma_+-1) = \CL_g^{PR}. 
$$
In particular, $u_L \cdot \phi_{\bf 1}\circ \CL_{\bar v} = (1-\alpha/p)(1-1/\alpha)^{-1}\log_{BK}^{\eta_g}$,
for $u_L \in \BZ_p^{\ur,\times}$ the image of $U_L$ under the specialization map 
$\Lambda_L^{\ac,\ur}\twoheadrightarrow\Lambda_L^{\ac,\ur}/(\gamma_\ac-1)\Lambda^{\ac,\ur}  = \BZ_p^\ur$.
\end{lem}

\begin{proof} By Theorem \ref{BT-thm1}, $\BT_1 = \BT^+\oplus c\cdot\BT^+$. In particular, $c\cdot \BT^+$ projects isomorphically onto
$\BT_1^-= \BT_1/\BT_1^+$. Let $\lambda_{\bh_v} \in \BT^+_1$ be the $\Lambda_L^v$-basis determined by $\omega_{\bh_v}$. Then 
$c\cdot\lambda_{\bh_v}$ is a $\Lambda_L^v$-basis of $c\cdot\BT_1^-$. Let 
$\CU_L = \frac{1}{T_v}\CH_v\cdot\eta_{\bh_v}(\lambda_{\bh_v}\mod \BT_1^+)\in (\Lambda_L^\ur)^\times$ and let 
$$
U_L = \CU_L^{-1}\mod (\gamma_+-1) \in (\Lambda_L^{\ac,\ur})^\times.
$$
The lemma then follows directly from the identification $H^1(L_{\bar v}, T^+(1)\hat\otimes\Lambda_L) = H^1(\BQ_p, T^+(1)\hat\otimes\BT_1\hat\otimes\Lambda)$ determined by the basis $\lambda_{\bh_v}$ and comparing the definitions of $\CL_{\bar v}$ and $\CL_g^{PR}$;
the conclusion for $\phi_{\bf 1}\circ\CL_{\bar v}$ the follows from \eqref{PRLog=BKlog}.
\end{proof}

From Lemma \ref{GRL=PRLog} and Proposition \ref{GRL=BDPL-II} we conclude:
\begin{lem}\label{GRL=logheegner} \hfill
\begin{itemize} 
\item[(i)] Suppose (Heeg) holds. Then
$$
(1-\alpha/p)(1-1/\alpha)^{-1}\log_{BK}^{\eta_g}(\loc_{\bar v}(\phi_{\bf 1}(\CZ(g/L))) = - u_L (1-a(p)p^{-1}+p^{-1})^2 (\log_{\omega_g}(y_L))^2.
$$
\item[(ii)] If $\eps(g/L) = +1$, then $\loc_{\bar v}(\phi_{\bf 1}(\CZ(g/L)))= 0$ in $H^1(L_{\bar v},V^+(1))$.
\end{itemize}
\end{lem}

\begin{remark} It is natural to ask whether the units $U_L$ and especially $u_L$ can be explicitly identified. This might be
possible by an explicit comparision of $\BT_1$ with a familiy of CM motives, but we do not pursue this here as it is not needed for our purposes
or applications. However, in the course of our proof of the Perrin-Riou's Conjecture (see section \ref{PRC} below), we do show that $u_L$ is
algebraic and, for suitably chosen $L$, belongs to $\BZ_{(p)}^\times$. The proof of this involves
auxiliary input from the arithmetic of CM forms, in the form of a proof of a version of Perrin-Riou's conjecture for certain ordinary CM forms
(see Appendix \ref{A-PRConj-CM}).
\end{remark}

\section{Zeta elements over imaginary quadratic fields: the supersingular case}\label{BFss} 
In this section we construct the plus/minus $p$-adic zeta element for a weight two newform for supersingular primes $p$ 
and describe its explicit reciprocity laws.

The construction relies on the Beilinson--Flach element associated to $p$-stabilisations of the newform and the canonical CM Hida family constructed  
by Loeffler and Zerbes \cite{LZ}, as initiated in their work with Lei \cite{LLZa}, \cite{LLZb}, as well as its variants by Buyukboduk and Lei \cite{BL}. Our approach is similar to the ordinary case treated in the previous section. By definition, the Beilinson--Flach element arises from the Tate lattice of \'etale cohomology of a tower of open modular curves. 
The principal result of this section is that it lives in the Tate lattice of closed modular curves (cf.~Theorem~\ref{BT-thm2-ss}). 
Based on \eqref{Ind-eq}, established in the previous section, 
we then define the sought after zeta element.

We continue with the notation and conventions introduced in \S\S\ref{NotationPrelim}--\ref{CMHF}.
Throughout this section we assume the newform $g\in S_{2}(\Gamma_{0}(N))$ to be supersingular at $p$.

\subsection{The Beilinson--Flach element} 
For $M = \BT_1, \tilde\BT_1, \BH_1$, or $\tilde\BH_1$ as in \S\ref{CMHF}, let 
$$
H^1(\BZ[\frac{1}{p}], T(1)\hat\otimes  M\hat\otimes \Lambda) = 
H^1(G_{\BQ,\Sigma}, T(1)\hat\otimes  M\hat\otimes \Lambda)
$$
where $\Sigma$ is any finite set of primes containing all $\ell\mid ND_Lp\infty$.

Let $\gamma\in\{\alpha,\beta\}$ be a root of the Hecke polynomial $x^2+p$.  
Let $g_\gamma(z) = g(z) - \gamma^{-1}pg(pz)\in S_2(\Gamma_0(Np))$ be the $p$-stabilisation of
$g$ corresponding to $\gamma$. In the notation of \cite{KLZ} there is an isomorphism 
$(Pr^\gamma)^*:M_{F_\lambda}(g_\gamma)^* \isoarrow M_{F_\lambda}(g)^* = V(1)$ 
that maps $M_{\cO_\lambda}(g_\gamma)^*\isoarrow M_{\cO_\lambda}(g)^* = T(1)$.

Since $a_{p}(g)=0$, the Hecke eigenform $g_\gamma$ is specialisation of a Coleman family ${\bf g}_{\gamma}$.

Recall that the canonical CM family $\bh_v$ corresponds to a branch of a Hida family $\bh$ of tame level $D_L$,
and it follows from the definition of $\BH_1$ and $M(\mathbf{h})^*$ that there is a specialisation map 
\begin{equation}\label{CM-sp-II}
M({\bf h})^*\otimes_{\Lambda_{\bf h}}\Lambda^v_L \twoheadrightarrow  \BH_1.
\end{equation}
More precisely, the left-hand side of \eqref{CM-sp-II} is just $\BH$ and the map the quotient modulo $\Lambda_L^v$-torsion.

Fix an integer $c >1$ such that $(c,6ND_Lp) = 1$.
Let
$$
_{c}\mathcal{BF}^{\gamma}(g_{/L}) \in H^1(\BZ[\frac{1}{p}], T(1)\hat\otimes  \BH_1\hat\otimes \mathfrak{K}_{1,F_\lambda(\gamma)})
$$
be the base layer of the Beilinson--Flach Euler system \cite{LZ} associated to the pair $(g_{\gamma}, {\bf h}_{v})$, obtained from the system in {\it{ibid}.} associated to the pair $({\bf g}_{\gamma}, {\bf h}_{v})$ with ${\bf g}_{\gamma}$ specialised  to $g_\gamma$. Note that these elements are unbounded. 

\begin{remark}
In \cite{LLZa}, \cite{KLZ}, \cite{LZ}, \cite{BL} it is assumed that $p>3$. For odd primes $p$, the hypothesis arises only when  considering explicit reciprocity laws, being due to its occurrence in the work of Ohta \cite{Oh1}, \cite{Oh2} on $p$-adic Eichler--Shimura isomorphism. Since the central character of the canonical CM Hida family 
${\bf h}_v$ is non-trivial, {\color{blue}} the pertinent Eichler--Shimura isomorphism is proved in \cite{SV-S-Ohta} (with many cases already covered in 
 Cais \cite{Ca}) for $p\geq 3$, and so the explicit reciprocity laws hold in our setting for any odd prime $p$.
\end{remark}
\subsubsection{The plus/minus Beilinson--Flach element}
Following Pollack \cite{Po}, put\footnote{While our labelling of signs is opposite to \cite{Po}, it is consistent with the formulation of main conjectures in \cite{Ko} and its generalisations.} 
$$
\log_{p}^{+}(T)=\frac{1}{p}\prod_{n=1}^{\infty}\frac{\Phi_{p^{2n}}(1+T)}{p}, \quad
\log_{p}^{-}(T)=\frac{1}{p}\prod_{n=1}^{\infty}\frac{\Phi_{p^{2n-1}}(1+T)}{p},
$$
and 
$$
M=\begin{pmatrix}
\log_p^{+}(1+T) & \alpha \log_{p}^{-}(1+T)\\
\log_{p}^{+}(1+T) & \beta \log_{p}^{-}(1+T)\\
\end{pmatrix}
.
$$

A principle of Kobayashi and Pollack leads to the following integral variant of the unbounded Beilinson--Flach elements: 

\begin{thm}\label{BF-pm}
There exist
$$
_{c}\mathcal{BF}^{\pm}(g_{/L}) \in H^1(\BZ[\frac{1}{p}], T(1)\hat\otimes  \BH_1\hat\otimes \Lambda)
$$
such that
$$
\begin{pmatrix}
_{c}\mathcal{BF}^{\alpha}\\
_{c}\mathcal{BF}^{\beta}\\
\end{pmatrix}
=
M \cdot \begin{pmatrix}
_{c}\mathcal{BF}^{-}\\
_{c}\mathcal{BF}^{+}\\
\end{pmatrix}
.
$$
\end{thm}
\begin{proof}
This is essentially due to Buyukboduk and Lei \cite{BL}. 

More precisely, it follows from the proof of \cite[Thm.~3.7]{BL}. 
First note that $$H^{0}(\BQ(\zeta_{p^{\infty}}),T(1)\hat{\otimes} \BH_{1})=0,$$ i.e. the hypothesis \cite[(18)]{BL} holds. In our case $k=2$ and by the explicit definition of $M$ above, the integer $s_0$ in the proof of \cite[Thm.~3.7]{BL} may be\footnote{In  \cite{BL} it is assumed that the $p$-distinguished hypothesis holds for the underlying CM Hida family. However, the constructions in subsections 3.1 and 3.2 
only rely on \cite[(18)]{BL}, which holds in our setting as noted above.} taken to be $0$. 
\end{proof}

\subsubsection{Some more Iwasawa cohomology}

Recall that $M\in\{\BT_1, \tilde\BT_1, \BH_1, \tilde\BH_1\}$
 is a $\Lambda_L^v$-module and there is a $\Lambda_L^v[G_{\BQ_p}]$-filtration 
$$0\rightarrow M^+\rightarrow M \rightarrow M^-\rightarrow 0.$$

 Let $\sR=\CO_\lambda\hat\otimes \Lambda_L^v\hat\otimes \Lambda$ and for $\gamma\in\{\alpha,\beta\}$, let
 $\tilde{\sR}_{\gamma}=\CO_\lambda\hat\otimes \Lambda_L^v\hat\otimes \mathfrak{K}_{1,F_\lambda(\gamma)}$.

As outlined below, for $\circ\in\{+,-\}$, there is a Coleman map 
$$
Col_{v,\circ}: H^{1}(\BQ_{p},T(1)\hat{\otimes} \BH_{1}^{+}\hat{\otimes} \Lambda) \ra J_{g} \otimes_{\cO_{\lambda}}\sR. 
$$
Define $H^{1}_{\circ}(\BQ_{p},T(1)\hat{\otimes} \BH_{1}^{+}\hat{\otimes} \Lambda) \subset H^{1}(\BQ_{p},T(1)\hat{\otimes} \BH_{1}^{+}\hat{\otimes} \Lambda)$ to be its kernel. Likewise, we introduce $\sR$-submodules 
$H^{1}_{\circ}(\BQ_{p},T(1)\hat{\otimes} \BH_{1}^{-}\hat{\otimes} \Lambda) \subset H^{1}(\BQ_{p},T(1)\hat{\otimes} \BH_{1}^{-}\hat{\otimes} \Lambda)$  
and $H^{1}_{\circ}(\BQ_{p},T(1)\hat{\otimes} \BT_{1}^{-}\hat{\otimes} \Lambda)$.
For $\cdot\in\{\emptyset,+\}$, put $$H^1_{/\circ}(\BQ_p, T(1)\hat\otimes \BH_1^{\cdot}\hat\otimes \Lambda)=\frac{H^1(\BQ_p, T(1)\hat\otimes \BH_1^{\cdot}\hat\otimes \Lambda)}{H^1_{\circ}(\BQ_p, T(1)\hat\otimes \BH_1^+\hat\otimes \Lambda)},$$ 
where we utilize the fact that $H^1(\BQ_p, T(1)\hat\otimes \BH_1^+\hat\otimes \Lambda)$ injects into 
$H^1(\BQ_p, T(1)\hat\otimes \BH_1\hat\otimes \Lambda)$. Also put 
$$H^1_{/\circ}(\BQ_p, T(1)\hat\otimes \BH_1^{-}\hat\otimes \Lambda)=\frac{H^1(\BQ_p, T(1)\hat\otimes \BH_1^{-}\hat\otimes \Lambda)}{H^1_{\circ}(\BQ_p, T(1)\hat\otimes \BH_1^-\hat\otimes \Lambda)}.$$

Let
$$
H^1_{\mathrm{rel},\circ}(\BZ[\frac{1}{p}], T(1)\hat\otimes  M\hat\otimes \Lambda) = \{ 
\kappa\in H^1(\BZ[\frac{1}{p}], T(1)\hat\otimes M \hat\otimes \Lambda) \ : \ \loc_p(\kappa) \in 
H^1_{\circ}(\BQ_p,T(1)\hat\otimes  M^- \hat\otimes \Lambda)\},
$$
where $\loc_p$ refers to the image in 
$H^1(\BQ_p,T(1)\hat\otimes  N^- \hat\otimes \Lambda)$. 
The reason for using the subscripts `$\mathrm{rel}$' and `$\circ$' will be made clear later (see Section \ref{two-variable-zeta-ss} below). 
\subsubsection{Local properties at $p$.} 

\begin{lem}\label{BF-locp-lem-ss} \hfill
\begin{itemize}
\item[(i)] The image of $\loc_p(_{c}\mathcal{BF}^\circ(g_{/L}))$ 
in  $H^1(\BQ_p, T(1)\hat\otimes \BH_1^-\hat\otimes \Lambda)$ is contained in
$H^1_{\circ}(\BQ_p, T(1)\hat\otimes \BH_1^-\hat\otimes \Lambda)$. In particular, 
$_{c}\mathcal{BF}^\circ(g_{/L})\in H^1_{\mathrm{rel},\circ}(\BZ[\frac{1}{p}], T(1)\hat\otimes \BH_1\hat\otimes \Lambda)$.
\item[(ii)] The image of $\loc_p(_{c}\mathcal{BF}^\circ(g_{/L}))$ in $H^1_{/\circ}(\BQ_p, T(1)\hat\otimes \BH_1\hat\otimes \Lambda)$ is contained in $H^1_{/\circ}(\BQ_p, T(1)\hat\otimes \BH^+\hat\otimes \Lambda)$.
\end{itemize}
\end{lem}
This is essentially the content\footnote{Our notation $\{v,\ov{v}\}$ corresponds to $\{\fp,\fp^c\}$ of {\it{loc. cit.}}} of \cite[Cor.~3.15]{BL} (see also \cite[Thm.~7.1.2]{LZ}). We explain the argument in subsection \ref{ss:ERLI_pf}.

\subsection{Explicit reciprocity law I}\label{ss:ERL_I}  In this section we describe an explicit reciprocity law 
for Beilinson--Flach elements based on \cite{LZ} and \cite{BL}, which connects
$_{c}\mathcal{BF}^\circ(g_{/L})$ to a certain two-variable $p$-adic $L$-function.

 We will consider two sets $\Xi^{(I)}$ and $\Xi^{(II)}$ of continuous characters 
 $\chi: \Gamma_L^v\hat\otimes\Gamma \rightarrow \overline{\BQ}^\times_p$ as in \eqref{eq:ch-1} and \eqref{eq:ch-2}. 
   The latter will be utilised in section \ref{s:ERL_II}.  
 Each such character determines
 a continuous $\CO_\lambda$-homomorphism $\phi_\chi:\sR\rightarrow \overline{\BQ}_p$, where
 $\CO_\lambda$ is identified with the completion of $\iota_p(\CO)$ in $\overline{\BQ}_p$.

\subsubsection{Result} As we explain in subsection \ref{ss:Col}, from the constructions in 
\cite{LZ} and \cite{BL}, there exists an injective
$\sR$-morphism 
\begin{equation}\label{ERI-map}
\sC_{\circ}:H^1_{/\circ}(\BQ_p,T(1)\hat\otimes \BH^+\hat\otimes \Lambda)\hookrightarrow {J_g}\otimes_{\CO_\lambda}\sR,
\end{equation}
where 
$$
\text{$J_g  = \frac{1}{\lambda_N(g) c_g}\CO$},
$$
for $c_g$ the congruence number of $g$ as in Section \ref{congruence} and 
$\lambda_N(g) = \pm 1$ is the eigenvalue of $g$ for the usual Atkin--Lehner involution $w_N$ of level $N$.

In view of ~Lemma \ref{BF-locp-lem-ss}(ii), 
we let 
$$
\sL_{p,c}^{\circ}(g/L) = \sC_{\circ}(\loc_p(_{c}\mathcal{BF}^\circ(g/L))) \in {J_g}\otimes_{\CO_\lambda}\sR.
$$
It is related to $L$-values as follows. 
 \begin{thm}\label{ERLI-thm-ss}
For $\chi\in \Xi^{(I)}$, 
$$
\phi_\chi(\sL_{p,c}^{\circ}(g/L)) = (c^2-\langle c\rangle \psi_\zeta^2(c)\chi_L(c))\cdot
e_{p}^{\circ}(\zeta)^2 
\frac{L(1,g\otimes \psi_\zeta^{-1})L(1,g\otimes\chi_L\psi_\zeta^{-1})}{\pi^2(-i)2^3\langle g,g\rangle},
$$
where $e_p^\circ(\zeta)$ is as in Theorem \ref{pcycQss} and $\langle c\rangle = (1+p)^{\log_p(c)}\in 1+p\BZ_p$.
\end{thm}
\begin{remark}\label{ERLI-rmk}
$\sL_{p,c}^{\circ}(g/L)$ satisfies an interpolation formula as in 
Theorem \ref{ERLI-thm-ss} for all characters $\chi:\Gamma_L^v\hat\otimes\Gamma \rightarrow \ov{\BQ}_p^\times$ of finite order. 
However, the general interpolation formula is  
inessential for the purposes of this paper.
\end{remark}

 A relation to the Coleman map $Col_{\omega}^{\circ}$ is given by the following.

\begin{lem}\label{Col-cycsp-prop-ss}
The reduction of the map $\sC_{\circ}$ modulo $\gamma_v-1$ equals the composition 
$$
H^1_{/\circ}(\BQ_p,T(1)\hat\otimes \BH^+\hat\otimes \Lambda)
\stackrel{\mod \gamma_v-1}{\twoheadrightarrow} H^1_{/\circ}(\BQ_p, T(1)\hat\otimes \Lambda)
\stackrel{Col^{\circ}_{\omega}}{\longrightarrow} J_g\otimes_{\CO_\lambda}\Lambda.
$$
\end{lem}

\noindent The first map in this lemma depends on an identification $\BH^+/(\gamma_v-1)\BH^+ \simeq \BZ_p$ determined by part of the data defining
the map $\sC_\gamma$ in subsection \ref{ss:Col} (the reduction modulo $\gamma_v-1$ of the map $\omega_{{\bf h}_v}$ therein).

The proofs of Lemma \ref{Col-cycsp-prop-ss} and Theorem \ref{ERLI-thm-ss} are given in subsection \ref{ss:ERLI_pf}.

\subsubsection{Coleman maps}\label{ss:Col} For $\gamma \in\{\alpha,\beta\}$, we first introduce the two-variable Coleman map $\sC_\gamma$ below, whose linear combination leads to $\sC_\circ$.

Put $\sS = \CO_\lambda\hat\otimes \Lambda_L^v = \Lambda_{L,\CO_\lambda}^v$.
\vskip2mm
{\it{Unbounded Coleman maps}}. 
We introduce the Coleman map $\sC_\gamma$ (cf.~\eqref{Col_PR}). 

Let ${\bf D}(\BI)$ denote the $(\varphi,\Gamma)$-module associated to a $p$-adic representation $\BI$ of $G_{\BQ_{p}}$ (cf.~\cite{BC}). By a result of Kisin \cite{Kis}, as interpreted by Colmez \cite{Co}, there is a filtration
$$
0 \ra \mathcal{F}^{+}_{\gamma}{\bf D}(T(1)) \ra {\bf D}(T(1)) \ra \mathcal{F}^{-}_{\gamma}{\bf D}(T(1)) \ra 0
$$
of $(\varphi,\Gamma)$-modules dependent on $\gamma\in\{\alpha,\beta\}$,  
with $\mathcal{F}_{\gamma}^\pm$ being of rank one.
 We often omit the subscript from this notation. 
The ordinary filtration on $\BH_1$ induces a filtration on the associated $(\varphi,\Gamma)$-module.
Put 
$$\mathcal{F}^{-+}(T(1)\hat\otimes \BH_{1})=\mathcal{F}^{-}{\bf D}(T(1))\hat\otimes \mathcal{F}^{+}{\bf D}(\BH_1)=\mathcal{F}^{-}{\bf D}(T(1))\hat\otimes{\bf D}(\BH^{+}),$$
and analogously define $\mathcal{F}^{\emptyset+}(T(1)\hat\otimes \BH_{1})$.

Let 
$$
\CC_{g_{\gamma},\bh_v}: H^{1}(\BQ_{p},\mathcal{F}^{\emptyset+}(T(1)\hat\otimes \BH_{1}) \hat\otimes \mathfrak{K}_{1,F_\lambda(\gamma)})
\ra D(\mathcal{F}^{-+}(T(1)\hat\otimes \BH_{1})) \hat\otimes \mathfrak{K}_{1,F_\lambda(\gamma)}$$
be the specialisation at $(g_{\gamma},\bh_v)$ of the Perrin-Riou regulator map\footnote{ In our notation the role of $(D,{\bf D})$ is opposite to that in \cite{LZ}.}
} associated to the pair $({\bf g}_{\gamma}, \bh_v)$ as in \cite[Thm.~7.1.4]{LZ}. 
More precisely, this is the restriction of the map denoted $\CL$ in {\it op.~cit.} to the $\Lambda_{\CG}^{(0)}$-summand. It factors through 
$H^{1}(\BQ_{p},\mathcal{F}^{-+}(T(1)\hat\otimes \BH_{1}) \hat\otimes \mathfrak{K}_{1,F_\lambda(\gamma)})$. 
As for the definition of $D(\mathcal{F}^{-+}(T(1)\hat\otimes \BH_{1}))$, 
for an unramified $(\varphi,\Gamma)$-module $M$, 
recall that $D(M) = M^{\Gamma}$ and that a twist of $\mathcal{F}^{-+}(T(1)\hat\otimes \BH_{1})$ as in \cite[p.~41]{LZ} is unramified, which is employed in the definition.

Pre-composing with the isomorphism
$$H^1(\BQ_p,\mathcal{F}^{-+}(T(1)\hat\otimes \BH)\hat\otimes \mathfrak{K}_{1,F_\lambda(\gamma)})\isoarrow H^1(\BQ_p,\mathcal{F}^{-+}(T(1)\hat\otimes \BH(\epsilon^{-1}\Psi_D^{-1}))\hat\otimes \mathfrak{K}_{1,F_\lambda(\gamma)})$$ 
coming from \eqref{unram-twist} yields 
a homomorphism  
$$
\CC_{g_{\gamma},\bh_v}: H^{1}(\BQ_{p},\mathcal{F}^{-+}(T(1)\hat\otimes \BH_{1})) \hat\otimes \mathfrak{K}_{1,F_\lambda(\gamma)})
\ra D(\mathcal{F}^{-}{\bf D}(T(1))\hat\otimes {\bf D}(\BH^+(\epsilon^{-1}\Psi_D^{-1}))) \hat\otimes \mathfrak{K}_{1,F_\lambda(\gamma)}. 
$$

The map $\sC_\gamma$ is then defined to be the composition of $\CC_{g_{\gamma},\bh_v}$ with an isomorphism arising from
 $$
 \psi_{g_{\gamma},{\bf h_v}}:D(\mathcal{F}^{-}{\bf D}(T(1)))\hat\otimes  
 {\bf D}(\BH^+(\epsilon^{-1}\Psi_D^{-1})))\hat\otimes \Lambda \isoarrow {J_g}\otimes_{\CO_\lambda}\sR,
 $$
 defined as follows.

The map $\tilde\eta_g$ is defined to be 
the isomorphism given by
\begin{equation}\label{eta-eq-II}
\tilde\eta_g = [\cdot,\frac{1}{\gamma(1-p\gamma^{-2})(1-\gamma^{-2})\lambda_N(g)}\eta_{\omega_g}]: 
D^{1}_{dR}(V)=
D(\mathcal{F}^{-}{\bf D}(T(1)))\otimes_{\CO_\lambda}F_\lambda \rightarrow F_\lambda,
\end{equation}
where $[\cdot,\cdot]$ is the pairing as in \eqref{CanPai}. 
We define $J_g$ to be the image of $D(\mathcal{F}^{-}{\bf D}(T(1)))$ under
$\tilde\eta_g$.  Recall that the image of 
$S_{\CO_\lambda}$ under the de Rham-\'etale comparison map $S_{F_\lambda} = D^{1}_{dR}(V)
\simeq D(\mathcal{F}^{-}{\bf D}(V(1)))$ equals $D(\mathcal{F}^{-}{\bf D}(T(1)))$.
Hence,
$J_g$ is generated by $\tilde\eta_g(\omega_g/c_g) = (\gamma(1-p\gamma^{-2})(1-\gamma^{-2})\lambda_N(g)c_g)^{-1}$.
Let 
$$
\omega_{{\bf h}_v}: {\bf D}(\BH^+(\epsilon^{-1}\Psi_D^{-1}))\isoarrow \Lambda_L^v.
$$
be the $\Lambda_L^v$-homomorphism induced by \eqref{eq:D-sub},
where we utilise the fact that $D({\bf D}(M))\simeq D(M)$ for an unramified $\BZ_{p}[G_{\BQ_{p}}]$-module $M$.

The map $\psi_{g,{\bh}_v}$ is defined to be the composition of isomorphisms
$$
 \psi_{g,{\bf h_v}}:D(\mathcal{F}^{-}{\bf D}(T(1))\hat\otimes  {\bf D}(\BH^+(\Psi_D^{-1})))\hat\otimes \Lambda  = D(\mathcal{F}^{-}{\bf D}(T(1)))\hat\otimes  D(\BH^+(\Psi_D^{-1}))\hat\otimes \Lambda
 \stackrel{\tilde\eta_g\otimes\omega_{{\bf h}_v}\otimes id}{\longrightarrow} {J_g}\otimes_{\CO_\lambda}\sR \stackrel{id\otimes\nu^{-1}}{\longrightarrow} {J_g}\otimes_{\CO_\lambda}\sR,
 $$
and its composition with $\mathcal{C}_{g_{\gamma},{\bf h}_{v}}$ defines
\begin{equation}\label{Col_PR}
 \sC_{\gamma}: H^{1}(\BQ_{p},\mathcal{F}^{-+}(T(1)\hat\otimes \BH_{1}) \hat\otimes \mathfrak{K}_{1,F_\lambda(\gamma)}) \ra
\tilde{\sR}_{\gamma}.
 \end{equation}
 It is an injective $\tilde{\sR}_\gamma$-homomorphism.
\vskip2mm 
{\it Plus/minus Coleman maps}.\label{ss:Col-pm}

\begin{prop}\label{Col-pm}
There exist $\sR$-module homomorphisms 
$$
\sC_{\pm}: 
 H^{1}(\BQ_{p},T(1)\hat\otimes \BH_{1}^{+} \hat\otimes \Lambda) \ra
\sR.
$$
such that
$$
\begin{pmatrix}
\sC_{\alpha}\\
\sC_{\beta}\\
\end{pmatrix}
=
M \cdot \begin{pmatrix}
\sC_{-}\\
\sC_{+}\\
\end{pmatrix}
.
$$
\end{prop}
\begin{proof}
This follows from the proof of \cite[Thm.~1.1]{BL} in \cite[\S2.3]{BL}. 
\end{proof}

Let $H^{1}_{\circ}(\BQ_{p},T(1)\hat{\otimes} \BH_{1}^{+}\hat{\otimes} \Lambda)$ denote the kernel of $Col_{v,\circ}$.
The map $\sC_\circ$ induces an $\sR$-module injection
$$
Col_{v,\circ}: H^{1}_{/\circ}(\BQ_{p},T(1)\hat{\otimes} \BH_{1}^{+}\hat{\otimes} \Lambda) \ra J_{g} \otimes_{\cO_{\lambda}}\sR.
$$
 
 \subsubsection{Unbounded Beilinson--Flach elements and Coleman map}
 \begin{thm}\label{ERLI-thm_0}
 \begin{itemize}
\item[(i)] The image of $_{c}\mathcal{BF}^\gamma(g/L)$ in $\mathcal{F}^{-\cdot}(T(1)\hat\otimes \BH)\hat\otimes\mathfrak{K}_{1,F_\lambda(\gamma)}$ belongs to 
$\mathcal{F}^{-+}(T(1)\hat\otimes \BH)\hat\otimes \mathfrak{K}_{1,F_\lambda(\gamma)}$. 
\item[(ii)] The element $\sL_{p,c}^{\gamma}(g/L):=\sC_{\gamma}(_{c}\mathcal{BF}^\gamma(g/L))$ satisfies:
For $\chi\in \Xi^{(I)}$, 
$$
\phi_\chi(\sL_{p,c}^{\gamma}(g/L)) = (c^2-\langle c\rangle \psi_\zeta^2(c)\chi_L(c))\frac{\mathcal{E}(\zeta)}{(1-p\gamma^{-2})(1-\gamma^{-2})\lambda_N(g)}
\frac{L(1,f\otimes \psi_\zeta^{-1})L(1,f\otimes\chi_L\psi_\zeta^{-1})}{\pi^2(-i)2^3\langle g,g\rangle},
$$
where
$$
\mathcal{E}(\zeta) = \begin{cases} 
\gamma^{-2(t+1)} \frac{p^{2(t+1)}}{\mathfrak{g}(\psi_\zeta^{-1})^2} & \zeta \neq 1 \\
(1 - \frac{1}{\gamma})^4 & \text{else}.
\end{cases}
$$
\end{itemize}
\end{thm}
This is a special case of \cite[Thm.~7.1.2~and~7.1.5]{LZ}.

\begin{remark}\label{ERLI-rmk_0}
The element $\sL_{p,c}^{\gamma}(g/L)$ is essentially the specialization at $g_\gamma$ of the $p$-adic Rankin-Selberg $L$-function
for
the Coleman family ${\bg}_{\gamma}$ and the Hida family 
${\bh_v}$. It satisfies an interpolation formula as in 
Theorem \ref{ERLI-thm-ss} for all characters $\chi:\Gamma_L^v\hat\otimes\Gamma \rightarrow \ov{\BQ}_p^\times$ of finite order. 
\end{remark}

\subsubsection{Proofs}\label{ss:ERLI_pf}
\begin{proof}[Proof of Lemma \ref{BF-locp-lem-ss}(ii)]
This readily follows from Theorem \ref{ERLI-thm_0}(i).
\end{proof} 

\begin{proof}[Proof of Theorem \ref{ERLI-thm-ss}] 
Note that the weight one specialisation of $\bh_v$ is the ordinary stabilisation of the weight one Eisenstein series $E_1(1,\chi_L)$
for the characters $1$ and $\chi_L$, which means that the Rankin-Selberg $L$-function
$L(s,g,E_1(1,\chi_L),\psi_\zeta^{-1})$ equals $L(s,g\otimes\psi_\zeta^{-1})L(s,g\otimes\chi_L\psi_\zeta^{-1})$. 
Hence the assertion\footnote{See also the comment about weight one specialisations following \cite[Thm.~7.7.2]{KLZ} and \cite[Rem.~7.1.6]{LZ}.} is a consequence of 
Theorem \ref{ERLI-thm_0}(ii) and Proposition \ref{Col-pm}.
\end{proof}

\begin{proof}[Proof of Lemma \ref{Col-cycsp-prop-ss}]
In view of the definition of the $\circ$-Coleman maps it suffices to prove analogous compatibility for the Coleman maps $\sC_\gamma$.

The map $\omega_{\bh_v}$ factors through a $\Lambda_L^v$-isomorphism $D(\BH^+(\epsilon^{-1}\Psi_D^{-1}))\cong \BH^+(\epsilon^{-1}\Psi_D^{-1})$
(see also \cite[Prop.~1.7.6]{FK}).  Since $\BH^+$ is a free $\Lambda_L^v$-module of rank one, $\omega_{\bh_v}$ determines a $\Lambda_L^v$-basis of 
$\BH^+(\epsilon^{-1}\Psi_D^{-1})$. In particular,  $\omega_{\bh_v}$ determines isomorphisms
$$\BH^+/(\gamma_v-1)\BH^+ \simeq \BH^+(\Psi_D^{-1})/(\gamma_v-1) \BH^+(\Psi_D^{-1})\simeq \BZ_p.$$ 
It then follows from functoriality that $\sC_\gamma \,\mod \gamma_v-1$ factors as 
the composition 
$$
H^{1}(\BQ_{p},\mathcal{F}^{-+}(T(1)\hat\otimes \BH_{1}) \hat\otimes \mathfrak{K}_{1,F_\lambda(\gamma)})
\twoheadrightarrow H^1(\BQ_p,\mathcal{F}^{-}T(1)\hat\otimes \mathfrak{K}_{1,F_\lambda(\gamma)}) \stackrel{\CL_{\mathcal{F}^{-}T(1)}}{\rightarrow} 
D(\mathcal{F}^{-}{\bf D}(T(1))\hat\otimes \mathfrak{K}_{1,F_\lambda(\gamma)}\stackrel{\tilde\eta_g\otimes id}{\rightarrow} J_g\hat\otimes \mathfrak{K}_{1,F_\lambda(\gamma)},$$
where the first map is just the projection induced by the isomorphism $\BH^+/(\gamma_v-1)\BH^+ \simeq \BZ_p$ from above.
The characterization \eqref{eta-eq-II} of $\tilde\eta_g$ in terms of $\eta_{\omega_g}$ combined with 
the fact that $\CL_{\mathcal{F}^-T(1)}$ is just Perrin-Riou's `big logarithm map'  for $V(1)$ 
shows -- by comparing with the definition of the Coleman map $Col_{\eta_{\omega_g}}$ in terms of this big logarithm map -- that the above displayed composition of maps equals 
$\frac{1}{\gamma(1-p\gamma^{-2})(1-\gamma^{-2})\lambda_N(g)} Col_{\eta_{\omega_g}}$.
\end{proof}

\subsection{The explicit reciprocity law II}\label{s:ERL_II}
The second explicit reciprocity law essentially arises from exchanging the roles of $g$ and $\bf{h}_v$ in the preceding section.

\subsubsection{Result} As we explain below, there is an injective $\sR$-homomorphism
$$
\sL_{\circ}: H^1_{\circ}(\BQ_p,T^+(1)\otimes\tilde\BH^-\hat\otimes\Lambda) \hookrightarrow \tilde I_{\bf h_v}\otimes_{\Lambda_L^v} \sR
$$
defined analogously to $\sC_{\circ}$ (cf.~\S\ref{ss:Log}).  Here $I_{\bf{h}_v}$ is the congruence ideal for the canonical CM family introduced in Section \ref{CMcong} and
$\tilde I_{\bh_v}$ is its reflexive closure.

In view ~Lemma \ref{BF-locp-lem-ss}(i), 
we let 
$$
\sL_{p,c}^{Gr}(g/L)= \sL_{\circ}(\loc_p(_c\mathcal{BF}^{\circ}(g/L))) \in \tilde I_{\bf h_v}\otimes_{\Lambda_L^v} \sR.
$$
These satisfy:

\begin{thm}[Expicit Reciprocity Law II]\label{ERLII-thm-ss}
For $\chi\in \Xi^{(II)}$, we have
$$
\phi_\chi(\sL_{p,c}^{Gr}(g/L)) = \lambda_{D_L}(h_{v,\chi_1}^0)(c^2-\psi_\zeta^2(c)\chi_L(c)\langle c\rangle^{2n-m+1)} ) \CE(\chi) \frac{n!(n-1)!}{\pi^{2n+1}(-i)^{m-1}2^{2n+m+1}}\frac{L(1+n, g, h_{v,\chi_1}^0,\psi_\zeta^{-1}\omega^n)}{\langle h_{v,\chi_1}^0,h_{v,\chi_1}^0\rangle},
$$
where 
$$
\CE(\chi) = \begin{cases}  
\frac{(1-\frac{p^n}{\chi_1(\varpi_{\bar v})\alpha})(1-\frac{p^n}{\chi_1(\varpi_{\bar v})\beta})
(1-\frac{p^{m}\alpha}{p^{n+1}\chi_1(\varpi_{\bar v})\alpha})(1-\frac{p^{m}\beta}{p^{n+1}\chi_1(\varpi_{\bar v})\alpha})}
{(1-p^{m-1}\chi_1^{-2}(\varpi_{\bar v}))(1-p^{m}\chi_1^{-2}(\varpi_{\bar v}))} & \zeta = 1, n\equiv 0\mod p-1 \\
\frac{(p^{t+1}/\mathfrak{g}(\psi_\zeta^{-1}\omega^n))^2 (p^{2n-1}/\chi_1(\varpi_{\bar v}))^{t+1}}
{(1-p^{m-1}\chi_1^{-2}(\varpi_{\bar v}))(1-p^{m}\chi_1^{-2}(\varpi_{\bar v}))}
 & \text{else}.
\end{cases}
$$
\end{thm}

\noindent Here $h_{v,\chi_1}^0$ is the newform of level $D_L$ and weight $2m+1$ whose ordinary $p$-stabilisation is $h_{v,\chi_1}$.
The $L$-function $L(s,g,h_{v,\chi_1}^0,\psi)$ is the $\psi$-twist of the usual Rankin--Selberg $L$-functions (cf.~\cite[\S2.7]{KLZ}).
Also, we identify the Galois character $G_L\rightarrow\Gamma_L^v\stackrel{\chi_1}{\rightarrow} \BQ_p^\times$ 
with an algebraic Hecke character of infinity type $(-2m,0)$ as in Section \ref{Hecke-char}, also denoted by $\chi_1$.
The hypotheses on $\chi_1$ ensure that this algebraic Hecke character is unramified at each prime above $p$, and we have 
denoted by $\varpi_{\bar v}$ a uniformiser at $\bar v$.
Here $t\geq 0$ is such that $\zeta$ is a primitive $p^t$th root of unity. 

\begin{remark}\label{ELRII-rmk}
Just as for Theorem \ref{ERLI-thm-ss}, the proof of Theorem \ref{ERLII-thm-ss} 
essentially identifies $\sL_{p,c}^{Gr}(g/L)$ as the specialization at $g_\alpha$ of a $p$-adic Rankin-Selberg $L$-function
for ${\bf h_v}$ and Coleman family ${\bf g}_\gamma$. 
It also satisfies an interpolation formula 
for a larger collection of characters than just $\Xi^{(II)}$. 
\end{remark}

\subsubsection{Logarithm maps}\label{ss:Log}
The definition of $\sL_\circ$ is analogous to that of $\sC_\circ$. 

Let $\iota_\epsilon:\Lambda_\CG\isoarrow \Lambda_\CG$ be the isomorphism such that $g\mapsto \epsilon(g)g$ for all $\gamma\in\CG$.
For $\cdot\in\{\emptyset,+,-\}$, put 
$$\mathcal{F}^{\cdot-}_{\gamma}{\bf D}(T(1)\hat\otimes \BH_{1})=\mathcal{F}_{\gamma}^{\cdot}{\bf D}(T(1))\hat\otimes \mathcal{F}^{-}(\BH_1)=\mathcal{F}_{\gamma}^{\cdot}{\bf D}(T(1))\hat\otimes{\bf D}(\tilde{\BH}^{-}).$$
We occasionally drop the subscript $\gamma$ in the above notation.
\vskip2mm
{\it{Coleman map, bis}}. 
Let
\begin{equation}\label{Col_PR'}
\ov{ \sC}_{\gamma}: H^{1}(\BQ_{p},\mathcal{F}_{\gamma}^{\emptyset-}(T(1)\hat\otimes \BH_{1}) \hat\otimes \mathfrak{K}_{1,F_\lambda(\gamma)}) \ra
\tilde{\sR}_{\gamma}
 \end{equation}
 be defined analogously to the Perrin-Riou regulator \eqref{Col_PR}. Its kernel is $H^{1}(\BQ_{p},\mathcal{F}^{+-}(T(1)\hat\otimes \BH_{1}) \hat\otimes \mathfrak{K}_{1,F_\lambda(\gamma)})$. 
 
 \begin{prop}\label{Col-pm'}
There exist $\sR$-module homomorphisms 
$$
\ov{\sC}_{\pm}: 
 H^{1}(\BQ_{p},T(1)\hat\otimes \BH_{1}^{-} \hat\otimes \Lambda) \ra
\sR.
$$
such that
$$
\begin{pmatrix}
\ov{\sC}_{\alpha}\\
\ov{\sC}_{\beta}\\
\end{pmatrix}
=
M \cdot \begin{pmatrix}
\ov{\sC}_{-}\\
\ov{\sC}_{+}\\
\end{pmatrix}
.
$$
\end{prop}
\begin{proof} One may proceed just as 
 in the proof of Proposition \ref{Col-pm}. 
\end{proof}
Let  $H^{1}_{\circ}(\BQ_{p},T(1)\hat\otimes \BH_{1}^{-} \hat\otimes \Lambda)$ denote the kernel of $\ov{\sC}_\circ$. 
\vskip2mm

{\it{Perrin-Riou logarithm}}. We introduce the logarithm map $\sL_\gamma$ (cf.~\eqref{Log_PR}).

The isomorphism $\mathcal{F}^+{\bf D}(T(1))\otimes\Lambda_{\CG} \stackrel{id\otimes\iota_{\epsilon}}{\longrightarrow} \mathcal{F}^+{\bf D}(T(1))\otimes\Lambda_{\CG}$ induces
$$
H^1(\BQ_p,\mathcal{F}^{+-}{\bf D}(T(1)\hat\otimes \BH_{1})\hat\otimes \mathfrak{K}_{1,F_\lambda(\gamma)})\isoarrow H^1(\BQ_p,\mathcal{F}^{+-}{\bf D}(T(1)\hat\otimes \BH_{1})\hat\otimes \mathfrak{K}_{1,F_\lambda(\gamma)}^{(-1)}), 
$$
composing it with 
specialisation at $(g_{\gamma},\bh_v)$ of the Perrin-Riou regulator map associated to the pair $({\bf g}_{\gamma}, \bh_v)$ yields a homomorphism
$$
\CL_{g_{\gamma},{\bf h}_v}: H^1(\BQ_p,\mathcal{F}^{+-}{\bf D}(T(1)\hat\otimes \BH_{1})\hat\otimes \mathfrak{K}_{1,F_\lambda(\gamma)})
\stackrel{\CL}{\hookrightarrow} D(\mathcal{F}^{+-}{\bf D}(T(1)\hat\otimes \BH_{1}))\hat\otimes\mathfrak{K}_{1,F_\lambda(\gamma)}^{(-1)}\cong 
D(\mathcal{F}^{+-}{\bf D}(T(1)\hat\otimes \BH_{1}))\hat\otimes \mathfrak{K}_{1,F_\lambda(\gamma)}.
$$
Then $\sL_{\gamma}$ is the composition of $\CL_{g_{\gamma},{\bf h}_v}$ with an isomorphism
$$
\xi_{g,{\bf h}_v}: D(\mathcal{F}^{+-}{\bf D}(T(1)\hat\otimes \BH_{1}))\hat\otimes\Lambda \isoarrow \tilde I_{\bh_v}\otimes_{\Lambda_L^v} \sR
$$
defined as follows.

We let $\eta_{\bh_v}:D(\BH_1^-)\isoarrow I_{\bh_v}$ be the $\Lambda_{L}^v$-map as in \cite[Prop.~10.1.1 part 2(b)]{KLZ}.
It induces a map on reflexive closures 
$\tilde\eta_{\bh_v}: \widetilde{D(\BH^-_1)}\isoarrow \tilde I_{\bh_v}$. Since 
$\widetilde{D(\BH^-_1)}  = D(\tilde\BH^-_1) = D(\tilde\BH^-)$
by functoriality\footnote{This is easily seen from the fact that there is a natural identification $D(M)\simeq M$ for profinite unramified
$\BZ_p[G_{\BQ_p}]$-modules $M$ that is functorial in $M$ \cite[Prop.~1.7.6]{FK}.}, this yields a 
$\Lambda_L^v$-isomorphism $\tilde\eta_{\bh_v}:D(\tilde\BH^-) \isoarrow \tilde I_{\bh_v}$
of free $\Lambda_L^v$-modules of rank one
and a commutative diagram
\begin{equation}\label{eta-hv-eq-II}
\begin{tikzcd}[row sep=2.5em]
D(\mathscr{F}^-M(\bh)^*)\otimes_{\Lambda_{\bh}}\Lambda_L^v \arrow[r,"{\eta_{\bh_v}}"] \arrow[d,] & I_{\bh_v} \arrow[d,hook] \\
 D(\tilde\BH^-) \arrow[r,"\tilde \eta_{\bh_v}"] &  \tilde I_{\bh_v},
\end{tikzcd}
\end{equation}
where the left vertical arrow is induced by functoriality. 

Let $\omega_{g_\gamma,\alpha}: D(\mathcal{F}^+{\bf D}(V)) \rightarrow F_\lambda$ denote the map
$[\omega_{g_\gamma},-]:D(\mathcal{F}^+{\bf D}(V)) = D_{cris}^0(V) \isoarrow F_\lambda$,
where $[-,-]$ is the pairing from \eqref{CanPai}. The restriction to $D(\mathcal{F}^+{\bf D}(T))$ is mapped isomorphically onto $\CO_\lambda$. 

The maps $\xi_{g,{\bf h}_v}$ and $\sL_{\gamma}$ are then defined to be 
$$
 \xi_{g,{\bf h_v}}: D(\mathcal{F}^{+-}{\bf D}(T(1)\hat\otimes \BH_{1}))\hat\otimes\Lambda
 \stackrel{\omega_{g_\gamma}\otimes\tilde\eta_{{\bf h}_v}\otimes id}{\longrightarrow} 
 {\CO_\lambda}\hat\otimes  \tilde I_{\bh_v}\hat\otimes \Lambda\stackrel{id\otimes id \otimes\iota_\epsilon^{-1}}{\longrightarrow}
 {\CO_\lambda}\hat\otimes  \tilde I_{\bh_v}\hat\otimes \Lambda  = \tilde I_{\bh_v}\otimes_{\Lambda_L^v} \sR
 $$
 and
\begin{equation}\label{Log_PR}
\sL_{\gamma} =  \xi_{g,{\bf h_v}}\circ\CL_{g_{\gamma},\bf{h}_v} : H^1(\BQ_p,\mathcal{F}^{+-}{\bf D}(T(1)\hat\otimes \BH_{1})\hat\otimes \mathfrak{K}_{1,F_\lambda(\gamma)})
\rightarrow 
 \tilde I_{\bh_v}\otimes_{\Lambda_L^v} \tilde{\sR}_{\gamma}.
 \end{equation}
 The latter is an injective $\sR$-homomorphism.
\vskip2mm 
{\it Plus/minus logarithm maps}. 
\begin{lem} 
For $\circ\in\{+,-\}$, pick
$z_\circ \in
H^{1}_{\circ}(\BQ_{p},T(1)\hat{\otimes} \BH_{1}^{-}\hat{\otimes} \Lambda)$. 
For $\gamma\in\{\alpha,\beta\}$, define 
$$z_{\gamma} \in H^1(\BQ_p,\mathcal{F}^{\emptyset-}_{\gamma}{\bf D}(T(1)\hat\otimes \BH_{1})\hat\otimes\Lambda)$$ by 
$$
\begin{pmatrix}
z_{\alpha}\\
z_{\beta}\\
\end{pmatrix}
=
M \cdot \begin{pmatrix}
z_{-}\\
z_{+}\\
\end{pmatrix}. 
$$
Then $z_{\gamma}\in H^1(\BQ_p,\mathcal{F}^{+-}_{\gamma}{\bf D}(T(1)\hat\otimes \BH_{1})\hat\otimes\Lambda)$. 
\end{lem}
\begin{proof} 
Note that $\sC_{\pm}(z_\pm)=0$.
Thus, we have $\ov{\sC}_\gamma(z_\gamma)=0$ 
by Proposition \ref{Col-pm'}. 
\end{proof}
\begin{defn}\label{Log-pm}
Let
$$
\sL_{\circ}: H^{1}_{\circ}(\BQ_{p},T(1)\hat{\otimes} \BH_{1}^{-}\hat{\otimes} \Lambda) \ra
 \tilde I_{\bh_v}\otimes_{\Lambda_L^v} \sR.
$$
be the $\sR$-module homomorphisms such that
$$
\begin{pmatrix}
\sL_{+} \ 
\sL_{-}\\
\end{pmatrix}
=
\begin{pmatrix}
\sL_{\beta} \
\sL_{\alpha}\\
\end{pmatrix}
\cdot M.
$$
\end{defn}

 \subsubsection{Unbounded Beilinson--Flach elements and logarithm map}
 \begin{thm}\label{ERLII-thm_0}
 \begin{itemize}
\item[(i)] The image of $_{c}\mathcal{BF}^\gamma(g/L)$ in $\mathcal{F}^{\cdot - }{\bf D}(T(1)\hat\otimes \BH)\hat\otimes \mathfrak{K}_{1,F_\lambda(\gamma)}$ belongs to 
$ \mathcal{F}^{+ -}{\bf D}(T(1)\hat\otimes \BH)\hat\otimes \mathfrak{K}_{1,F_\lambda(\gamma)}$. 
\item[(ii)] The element $\sL_{p,c}^{Gr,\gamma}(g/L):=\sL_{\gamma}(_{c}\mathcal{BF}^\gamma(g/L))$ satisfies the same interpolation formula as in Theorem \ref{ERLII-thm-ss}.
\end{itemize}
\end{thm}
 
This is a special case of \cite[Thm.~7.1.2~and~7.1.5]{LZ}.

\subsubsection{Proofs}
\begin{proof}[Proof of Lemma \ref{BF-locp-lem-ss}(i)]
By Theorem \ref{ERLII-thm_0}(i), we have $
\ov{\sC}_{\gamma}(\Im(_{c}\mathcal{BF}^\gamma(g/L)))=0$. 
Moreover, 
$$
\ov{\sC}_{\alpha}(\Im(_{c}\mathcal{BF}^\beta(g/L)))=-\ov{\sC}_{\beta}(\Im(_{c}\mathcal{BF}^\alpha(g/L))). 
$$
by \cite[Prop.~3.14]{BL}. 
Hence, the assertion follows by the same argument as for the proof of \cite[Cor.~3.15]{BL}.
\end{proof}

\begin{proof}[Proof of Theorem \ref{ERLII-thm-ss}] 
This is a simple consequence of Theorem \ref{ERLII-thm_0}(ii), Proposition \ref{BF-pm} and Definition \ref{Log-pm}.
\end{proof}

\subsection{Integral normalisations} 
We introduce certain normalisations of the Beilinson--Flach elements, and regulator maps. These will be subsequently used in regards to zeta elements, and reciprocity laws. 

Put 
$$
\sC^{\mathrm{int}}_{\circ}= c_g \cdot \sC_{\circ}: H^1_{/\circ}(\BQ_p, T(1)\hat\otimes \Lambda_L^v\hat\otimes\Lambda) \rightarrow \CO_\lambda\hat\otimes \Lambda_L^v\hat\otimes\Lambda = \sR.
$$
for $c_g\in \CO$ a congruence number as in Section \ref{congruence}.

Let $\CH_v \in \Lambda_L^{\ur}$ be as in \eqref{nrm}. Let $\CO_\lambda^{\ur}$ be the completion of the ring of integers of the 
maximal unramified extension of $F_\lambda$ (this is just the compositum of $\CO_\lambda$ and $W(\bar{\BF}_p)$).
We put
$$
\sL^\mathrm{int}_{\circ} = \CH_v\cdot \sL_{\circ}:H^1_{\circ}(\BQ_p,T(1)\hat\otimes \tilde \BH^-\hat\otimes\Lambda)\otimes_{\Lambda_{L,\CO_\lambda}^v}
\Lambda_{L,\CO_\lambda^\ur}^v \rightarrow \sR^\ur,
$$
where  $\sR^\ur= \CO_\lambda^\ur\hat\otimes\Lambda_L^v\hat\otimes\Lambda = \sR\otimes_{\Lambda_{L,\CO_\lambda}^v} \Lambda_{L,\CO_\lambda^\ur}^v$.

We also normalise the elements $_c\mathcal{BF}^{\circ}(g/L)$.  Put $r_c = \log_p(c) \in \BZ_p$ and 
$$
\mathbf{c} = (c^2 - \chi_L(c)\otimes\langle c\rangle \gamma_{v}^{-r_c}\otimes\gamma_{\cyc}^{2r_c}) \in \sR.
 $$
Modulo the maximal ideal of $\sR$, $\mathbf{c}$ is congruent to $c^2-\chi_L(c)$. 
As $p$ is odd, we can therefore choose $c$ so that $\mathbf{c}\in \sR^\times$. Henceforth we assume that $c$ satisfies this.
Define 
$$
\mathcal{BF}^{\circ}(g/L) = \mathbf{c}^{-1}\cdot {_c\mathcal{BF}^{\circ}(g/L)} \in H^1(\BQ,T(1)\hat\otimes \tilde\BH\hat\otimes \Lambda),
$$
and
$$
\sL_p^{\circ}(g/L) = \sC^\mathrm{int}_{\circ}(\loc_p(\mathcal{BF}^{\circ}(g/L)))\in \sR  \ \ \text{and} \ \ \sL_p^{Gr}(g/L)=\sL^{\mathrm{int}}_{\circ}(\loc_p(\mathcal{BF}^{\circ}(g/L))) \in \sR^{\ur}.
$$

We record the following versions of the earlier explicit reciprocity laws (Theorems \ref{ERLI-thm-ss} and \ref{ERLII-thm-ss}).

\begin{thm}[Explicit Reciprocity Law I$'$]\label{ERLIint-thm-ss}
The element $\sL_{p}(g/L)\in \sR$ satisfies:
For $\chi\in \Xi^{(I)}$ with $\chi_1=1$, 
$$
\phi_\chi(\sL_{p}(g/L)) = 
e_{p}^{\circ}(\zeta)^2 
\frac{L(1,g\otimes \psi_\zeta^{-1})L(1,g\otimes\chi_L\psi_\zeta^{-1})}{\pi^2(-i)2^3\Omega_g},
$$
where
$
\Omega_g = \frac{\langle g,g\rangle}{\tilde c_g}$
and $e_p^\circ(\zeta)$ is as in Theorem \ref{pcycQss}.
\end{thm}

Given $\chi\in \Xi^{(II)}$ let 
$\psi_\chi$ be the algebraic Hecke character of $L$ with infinity type $(n-m,n)$ and such that
$\sigma_{\psi_\chi}$ is the composition 
$G_L\twoheadrightarrow \Gamma_L^v\times\Gamma
\stackrel{\chi_1\times \chi_2^{-1}}{\rightarrow}\overline{\BQ}_p^\times$.
Then the second explicit reciprocity law can be rewritten as follows.

\begin{thm}[Explicit Reciprocity Law II$'$]\label{ERLIIint-thm-ss}
The element $\sL_p^{Gr}(g/L)\in \sR^\ur$ satisfies: 
For $\chi\in \Xi^{(II)}$, 
$$
\phi_\chi(\sL_{p}^{Gr}(g/L)) = 
- w_L (\chi_1(\gamma_v)-1)\CE'(\chi) \frac{n!(n-1)!\pi^{2m-2n-1}}{2^{2n-2m+2}D_L^{m/2}}\cdot \Omega_p^{2m}\frac{L(1, g, \psi_\chi)}{\Omega_\infty^{2m}},
$$
where 
$$
\CE'(\chi) = 
\begin{cases}
(1-a(p)\psi_\chi(\varpi_{\bar v})^{-1}p^{-1} + \psi_\chi(\varpi_{\bar v})^{-2}p^{-1})^2 & \zeta = 1, n\equiv 0\mod p-1 \\
(p^{t+1}/\mathfrak{g}(\psi_\zeta^{-1}\omega^n))^2 (p^{2n-1}/\chi_1(\varpi_{\bar v}))^{t+1}
& \text{else}.
\end{cases}
$$
\end{thm}

\noindent Here $(\Omega_p,\Omega_\infty)$ are the CM periods as in Theorem \ref{pKatzL},
and 
$$
L(s,g,\psi_\chi) = \sum_{\mathfrak{a}, (\mathfrak{a},\mathfrak{f}_{\psi_\chi})=1}
a_g(N(\mathfrak{a})) \psi_\chi(x_{\mathfrak{a}}) N(\mathfrak{a})^{-s} = L(s,g,h_{\chi_1}^0, \psi_{\zeta}^{-1}\omega^n).
$$
The key to this rewrite of the second explicit reciprocity law is the following formula (essentially due to Shimura -- see \cite[\S 7]{HT1}), which expresses ${\langle h_{v,\chi_1}^0,h_{v,\chi_1}^0\rangle}$
in terms of $L(1,\chi_1/\chi_1^c)$:
$$
{\langle h_{v,\chi_1}^0,h_{v,\chi_1}^0\rangle} = \frac{h_L m! D_L^{1/2} }{w_L 2^{m-1}(2\pi)^{m+1}} L(1,\chi_1/\chi_1^c).
$$
Combined with the interpolation formula for the Katz $p$-adic $L$-function $\sL_v(L)$ from Theorem \ref{pKatzL}, this easily yields
the formula in Theorem \ref{ERLIIint-thm-ss}.

\begin{remark}\label{BFint-rmk} The class $\mathcal{BF}^{\circ}(g/L)$ does not depend on the choice of an auxiliary integer $c$. 
This follows from 
the second explicit reciprocity law, the injectivity of $\sL^{\mathrm{int}}_{\circ}$, and the Zariski density of the specialisations $\phi_\chi$, $\chi\in \Xi^{(II)}$. 
\end{remark}

\subsubsection{Comparisons with other $p$-adic $L$-functions}
Our application of the Beilinson--Flach elements to the conjecture of Perrin-Riou and other problems stems in large part from being able to realize various $p$-adic $L$-functions as specialisations of the elements $\sL_p^{\circ}(g/L)$ and $\sL_p^{Gr}(g/L)$. 
\medskip

\noindent{\it Comparison with cyclotomic $L$-functions.} From Theorem \ref{ERLIint-thm-ss} together with Proposition \ref{Col-cycsp-prop-ss} we conclude the following:

\begin{prop}\label{ERLIint-prop-ss} 
Let $g' = g\otimes\chi_L$. Let $0\neq \omega\in S_{F,g}$, and let $\gamma\in V_{F,g}$ and $\gamma'\in V_{F,g'}$
such that $\gamma^\pm\neq 0$ and $(\gamma')^\pm\neq 0$. 
For $\circ\in\{+,-\}$, there exists a constant $c^{\circ}(\omega,\gamma,\gamma') \in F^\times$ such that 
\begin{itemize}
\item[(i)] $\sL_p^{\circ}(g/L) \mod (\gamma_v-1) = c^{\circ}(\omega,\gamma,\gamma')\mathfrak{g}(\chi_L)^{-1}
\CL_{\omega,\gamma,\gamma'}^{\circ}(g/L) \in  \Lambda_{\CO_\lambda} = \sR/(\gamma_v-1)\sR$,
\item[(ii)] if (irr$_\BQ$) holds, $\omega\in S_{g,\CO}$ is good, and $\gamma\in T_{g,\CO}$ and $\gamma'\in T_{g',\CO}$
are such that $\gamma^\pm$ is an $\CO$-basis of $T_{\CO,g}^\pm$ and $(\gamma')^\pm$ is an $\CO$-basis of $T_{\CO,g'}^\pm$,
then $c^{\circ}(\omega,\gamma,\gamma') \in \CO^\times$.
\end{itemize}
\end{prop}
\noindent In part (i), $\CL^{\circ}_{\omega,\gamma,\gamma'}(g/L)$ is the cyclotomic $p$-adic $L$-function for $g$ over $L$ as in 
Section \ref{sspLcy}.

\medskip

\noindent{\it Comparison with the $L$-function of Bertolini--Darmon--Prasanna.}
Comparing Theorems \ref{pBDPL} and \ref{ERLIIint-thm-ss} yields:
\begin{prop}\label{GRL=BDPL-ss} Suppose \eqref{Heeg} holds. 
The image of $\sL_p^{Gr}(g/L)$ under the map
$\phi_\ac:\sR^\ur \twoheadrightarrow \Lambda_{L,\CO^\ur}^{ac}$ induced from the homomorphism
$\Gamma_L^v\times\Gamma\twoheadrightarrow \Gamma_L^\ac$, $\gamma_v \mapsto \gamma_\ac^2$ and 
$\gamma\mapsto \gamma_\ac^{h_p}$, equals $-T_\ac^2$ times the image of $\sL_v^{BDP}(g/L)$ under the involution 
$\iota_\ac$ of $\Lambda_L^\ac$ induced by $\gamma_\ac\mapsto \gamma_\ac^{-1}$:
$$
\phi_\ac(\sL_p^{Gr}(g/L)) = -T_\ac^2 \cdot \iota_\ac(\sL_v^{BDP}(g/L)).
$$
\end{prop}

Recall that if the Heegner hypothesis \eqref{Heeg} holds, then the root number 
$\epsilon(g/L)$ equals $-1$. When this root number is $+1$ we have:
\begin{prop}\label{GRLvan-prop-ss} Suppose $\epsilon(g/L)=+1$. Let $\phi_\ac$ be
as in Proposition \ref{GRL=BDPL}. Then 
$
\phi_\ac(\sL^{Gr}(g/L)) = 0.
$
\end{prop}
\begin{proof}
This is identical to the proof of Proposition~\ref{GRLvan-prop}.  
\end{proof}

\subsection{$\CBF^{\circ}(g/L)$ arises from $\BT_1$} 

\begin{thm}\label{BT-thm2-ss}
The Beilinson--Flach element $\CBF^{\circ}(g/L)$ belongs to the submodule
$H^1_{\mathrm{rel},\circ}(\BQ,T(1)\hat\otimes\BT_1\hat\otimes\Lambda)$ of 
$H^1_{\mathrm{rel},\circ}(\BQ,T(1)\hat\otimes\BH_1\hat\otimes\Lambda)$.
\end{thm}
\begin{proof}

By Proposition \ref{Sat-prop3}(iii),(iv), the $\Lambda_L^v = \BZ_p[\![T_v]\!]$-module $\widetilde\BH/\widetilde\BT$ is 
annihilated by $T_v$. It follows that as a $\BZ_p[G_\BQ]$-modules $\widetilde\BH/\widetilde\BT$
is a quotient of $\BZ_p\oplus\BZ_p(\chi_L)$.
Consequently, as $H^0(\BQ,T(1)\hat\otimes\Lambda) = 0 = H^0(\BQ,T(1)\hat\otimes\Lambda(\chi_L)),$
the inclusion $\widetilde\BT \subset \widetilde\BH$ induce an inclusion
$$H^1(\BQ,T(1)\hat\otimes\widetilde\BT\hat\otimes\Lambda) \hookrightarrow
H^1(\BQ,T(1)\hat\otimes\widetilde\BH\hat\otimes\Lambda),$$
by which we view the former as a submodule of the latter.

Let $\CBF^{\circ}\in H^1(\BQ,T(1)\hat\otimes\widetilde\BH\hat\otimes\Lambda)$ be the image of $\CBF^{\circ}(g/L)$ under the map induced by the inclusion
$\BH_1\subset \widetilde\BH$. 
Put $$H^1_{/\circ}(\BQ_p,T(1)\hat\otimes(\widetilde\BH/\widetilde\BH_v)\hat\otimes\Lambda)
= \frac{H^1(\BQ_p,T(1)\hat\otimes(\widetilde\BH/\widetilde\BH_v)\hat\otimes\Lambda)}
{H^1_{\circ}(\BQ_p,T(1)\hat\otimes \BH_{1}^{-}\hat\otimes\Lambda)},
$$
where we utilise the fact that $H^1(\BQ_p,T(1)\hat\otimes \BH_{1}^{-}\hat\otimes\Lambda)$ 
injects into $H^1(\BQ_p,T(1)\hat\otimes(\widetilde\BH/\widetilde\BH_v)\hat\otimes\Lambda)$. 
Then the image of $\loc_p(\CBF^{\circ})$ in $H^1_{/\circ}(\BQ_p,T(1)\hat\otimes(\widetilde\BH/\widetilde\BH_v)\hat\otimes\Lambda)$
is the image of $\loc_p(\CBF^\circ(g/L))$ under the induced map $H^1_{/\circ}(\BQ_p,T(1)\hat\otimes \BH_1^-\hat\otimes\Lambda)
\rightarrow H^1_{/\circ}(\BQ_p,T(1)\hat\otimes(\widetilde\BH/\widetilde\BH_v)\hat\otimes\Lambda).$  
So it follows from Lemma \ref{BF-locp-lem-ss} that 
$$
\CBF^{\circ}\in H^1_{\mathrm{rel},\circ}(\BQ,T(1)\hat\otimes\widetilde\BH\hat\otimes\Lambda)
= \{ c\in H^1(\BQ,T(1)\hat\otimes\widetilde\BH\hat\otimes\Lambda) \ : \ \loc_p(c) = 0 \in H^1_{/\circ}(\BQ_p,T(1)\hat\otimes(\widetilde\BH/\widetilde\BH_v)\hat\otimes\Lambda)\}.
$$

Since $T_v$ annihilates $\widetilde\BH/\widetilde{\BT}$, we have $$\CBF' := T_v \cdot \CBF^{\circ}\in H^1(\BQ,T(1)\hat\otimes\widetilde{\BT}\hat\otimes\Lambda).$$
As $G_{\BQ_p}$ acts trivially on $\widetilde\BH/\widetilde{\BT}$, note that
$H^1_{/\circ}(\BQ_p,T\hat\otimes\widetilde{\BT}/\BT^{+}\hat\otimes\Lambda)\hookrightarrow H^1_{/\circ}(\BQ_p,T\hat\otimes(\widetilde\BH/\widetilde\BH_v)\hat\otimes\Lambda)$. 
Consequently, $\CBF'$ belongs to the submodule $H^1_{\mathrm{rel},\circ}(\BQ,T(1)\hat\otimes\widetilde{\BT}\hat\otimes\Lambda)$.

We have $H^1(\BQ, T(1)\hat\otimes\widetilde{\BT}\hat\otimes\Lambda) = H^1(L,T(1)\hat\otimes\BT^+\hat\otimes\Lambda)$ by \eqref{Ind-eq} (cf.~Theorem \ref{BT-thm1}). 
This induces 
an identification
$$H^1_{\mathrm{rel},\circ}(\BQ,T(1)\hat\otimes\BT^+\hat\otimes\Lambda) = H^1_{\mathrm{rel},\circ}(L,T(1)\hat\otimes\BT^+\hat\otimes\Lambda)$$
where 
$H^1_{\mathrm{rel},\circ}(L,T(1)\hat\otimes\BT^+\hat\otimes\Lambda) = \{ c\in H^1_{\mathrm{rel},\circ}(L,T(1)\hat\otimes\BT^+\hat\otimes\Lambda) \ : \
\loc_p(c) = 0 \in H^1_{/\circ}(L_{\bar v}, T(1)\hat\otimes\BT^+\hat\otimes\Lambda)\}$. 
In turn, we have 
an injection
\begin{equation}\label{H1-L-eq2-ss}
H^1_{\mathrm{rel},\circ}(\BQ,T(1)\hat\otimes\BT_{1}\hat\otimes\Lambda)/T_v H^1_{\mathrm{rel},\circ}(\BQ,T(1)\hat\otimes\BT_{1}\hat\otimes\Lambda) 
\hookrightarrow H^1_{\mathrm{rel},\circ}(L,T(1)\hat\otimes\Lambda)
\end{equation}
after choosing an isomormophism $\BT^+\simeq \Lambda_L^v$ (and hence an isomorphism $\BT^+/T_v\BT^+\simeq\BZ_p$).
Let $\ov{\CBF}'\in H^1_{\mathrm{rel},\circ}(L,T(1)\hat\otimes\Lambda)$ denote the image of $\CBF'$ under \eqref{H1-L-eq2-ss}. 

We now consider the image of $\loc_{v}(\ov{\CBF}')$ under the Coleman map $Col_{\eta_\omega,v}^{\circ}$.
Since $\BH^+ = \BT^+$, we have a commutative diagram
\begin{equation}\label{H1-L-eq3-ss}
\begin{tikzcd}[row sep=2.5em]
H^1_{\mathrm{rel},\circ}(\BQ,T(1)\hat\otimes\BT^{+}\hat\otimes\Lambda) \arrow[r,"\eqref{H1-L-eq2-ss}"]  \arrow[d,"\loc_p"] & H^1_{\mathrm{rel},\circ}(L,T(1)\hat\otimes\Lambda) \arrow[d,"\loc_v"] & \\
 H^1_{/\circ}(\BQ_p,T(1)\hat\otimes\BT^{+}\hat\otimes\Lambda) \arrow[r,"/T_v"] & H^1_{/\circ}(L_v,T(1)\hat\otimes\Lambda) \arrow[r,"Col_{\eta_\omega,v}^{\circ}"] & 
 \Lambda_{\CO_\lambda}\otimes_{\BZ_p}\BQ_p \\
 H^1_{/\circ}(\BQ_p,T(1)\hat\otimes\BH^+\hat\otimes\Lambda) \arrow[r,"/T_v"] \arrow[u,]  & H^1_{/\circ}(\BQ_p,T(1)\hat\otimes\Lambda) \arrow[r,"Col_{\eta_\omega}^{\circ}"]  \arrow[u,"(*)"] & 
 \Lambda_{\CO_\lambda}\otimes_{\BZ_p}\BQ_p \arrow[u,"(**)"] . 
\end{tikzcd}
\end{equation}
Here we have fixed the isomorphism $\BT^+\simeq\Lambda_L^v$ to restrict to the isomorphism $\BT^+=\BH^+\simeq\Lambda_L^v$ determined
by $\omega_{\bh_v}$ as in the proof of Proposition \ref{Col-cycsp-prop}. This in turn induces an isomorphism $\BT^+/T_v\BT^+\simeq \BZ_p$,
which we take as the isomorphism leading to \eqref{H1-L-eq2-ss}. The lower veritical arrows in \eqref{H1-L-eq3-ss} are induced from the equality
$\BT^+=\BH^+ $. In particular, $(*)$ and $(**)$ are both the identity map.
Since $\loc_p(\CBF') \in H^1_{/\circ}(\BQ_p,T(1)\hat\otimes\widetilde{\BT}\hat\otimes\Lambda)$ is the image of 
$T_v\cdot\loc_p(\CBF^{\circ}(g/L)) \in  H^1_{/\circ}(\BQ_p,T(1)\hat\otimes\BH^+\hat\otimes\Lambda)$, it follows from the commutativity of \eqref{H1-L-eq3-ss} that
$Col_{\eta_\omega,v}^{\circ}(\loc_{v}(\ov{\CBF}')) = 0$.  

By Theorem \ref{cycIwL-Box}, $H^1_{\mathrm{rel},\circ}(L,T(1)\hat\otimes\Lambda)$ is a free $\Lambda_{\CO_\lambda}$-module  of rank one.
As $Col_{\eta_\omega,v}^{\circ}\circ\loc_v$ is non-zero on
$H^1_{\mathrm{rel},\circ}(L,T(1)\hat\otimes\Lambda)$ (see the proof of Theorem \ref{cycIwL-Box-II}), we have an injection
$$
Col_{\eta_\omega,v}^{\circ}\circ\loc_v : H^1_{\mathrm{rel},\circ}(L,T(1)\hat\otimes\Lambda)  \hookrightarrow \Lambda_{\CO_\lambda}.$$

Since $Col_{\eta_\omega,v}^{\circ}(\loc_{v}(\ov{\CBF}')) = 0$, it then follows that $\ov{\CBF}'=0$ and hence that 
$\CBF'\in T_v H^1_{\mathrm{rel},\circ}(\BQ,T(1)\hat\otimes\widetilde{\BT}\hat\otimes\Lambda)$ by \eqref{H1-L-eq2-ss}. Recall that $\widetilde{\BT}=\BT_1$ by \eqref{Ind-eq}.
As $H^1(\BQ,T(1)\hat\otimes\BT_{1}\hat\otimes\Lambda)$ has no $T_v$-torsion\footnote{since $H^0(\BQ,T(1)\hat\otimes(\BZ_p\oplus\BZ_p(\chi_L))\hat\otimes\Lambda)=0$},
it follows that $\CBF \in H^1_{\mathrm{rel},\circ}(\BQ,T(1)\hat\otimes\BT_{1}\hat\otimes\Lambda)$, as claimed.
\end{proof}
\begin{remark}
The proof does not rely on the second reciprocity law. 
\end{remark}
\begin{cor} For $\gamma\in\{\alpha,\beta\}$, we have 
$$
_{c}\mathcal{BF}^{\gamma}(g_{/L}) \in H^1(\BZ[\frac{1}{p}], T(1)\hat\otimes  \BT_1\hat\otimes\mathfrak{K}_{1,F_\lambda(\gamma)}).
$$
\end{cor}
\begin{proof}
This follows from Theorem \ref{BF-pm}.
\end{proof}

\subsection{Two-variable zeta element}\label{two-variable-zeta-ss}
In light of Theorem \ref{BT-thm2-ss} and \eqref{Ind-eq}, the Beilinson--Flach element $\mathcal{BF}^\circ(g/L)$ leads to a
two-variable zeta element for $g$, as explained below.

\subsubsection{The zeta element $\CZ^{\circ}(g/L)$}
By \eqref{Ind-eq}, $\BT_1$ is identified with the induction from $G_L$ to $G_\BQ$ of the $\Lambda_L^v[G_L]$-module
$\BT^+$. Recall that $\BT^+$ is a free $\Lambda_L^v$-module of rank one on which $G_L$ acts via the canonical character $\Psi_L^v$ and, furthermore,
we have a preferred $\Lambda_L^v$-isomorphism $\BT^+\isoarrow \Lambda_L^v(\Psi_L^v)$ arising from the map $\omega_{\bh_v}$. The latter determines via 
Shapiro's Lemma identifications
$$
H^1(\BQ,T(1)\hat\otimes\BT_1\hat\otimes\Lambda) \isoarrow H^1(L,T(1)\hat\otimes\BT^+\hat\otimes\Lambda) \isoarrow H^1(L,T(1)\hat\otimes\Lambda_L^v(\Psi_L^v)\hat\otimes\Lambda)
$$
of $\sR$-modules.
Similarly, $H^1(\BQ_p,T(1)\hat\otimes\BT^+\hat\otimes\Lambda)$ is identified with $H^1(L_v,T(1)\hat\otimes\Lambda_L^v(\Psi_L^v)\hat\otimes\Lambda)$
and $H^1(\BQ_p,T(1)\hat\otimes(c\cdot \BT^+)\hat\otimes\Lambda)$ is identified with $H^1(L_{\bar v},T(1)\hat\otimes\Lambda_L^v(\Psi_L^v)\hat\otimes\Lambda)$,
and also with $H^1$ replaced by $H^1_\circ$.
It follows that we have an identification
$$
H^1_{\rel,\circ}(\BQ,T(1)\hat\otimes\BT_1\hat\otimes\Lambda) \isoarrow H^1_{\rel,\circ}(L,T(1)\hat\otimes\Lambda_L^v(\Psi_L^v)\hat\otimes\Lambda)
$$
where the subscripts `$\rel$' and `$\circ$' on the right-hand side denote the submodule of classes 
$c$ such that no condition is imposed on $\loc_v(c)$ but we require $\loc_{\bar v}(c)\in H^1_{\circ}(L_{\bar v},T(1)\hat\otimes\Lambda_L^v(\Psi_L^v)\hat\otimes\Lambda)$.  We may thus view 
$\CBF^{\circ}(g/L)$ as an element of $H^1_{\rel,\circ}(L,T(1)\hat\otimes\Lambda_L^v(\Psi_L^v)\hat\otimes\Lambda)$.

We now consider the composition of isomorphism 
\begin{equation}\label{Lgal-twist-eq-ii}
\Gamma_L \isoarrow \Gamma_L^v\times\Gamma \stackrel{\gamma_v\mapsto\gamma_v^{-1}}{\isoarrow} \Gamma_L^v\times\Gamma,
\end{equation}
this determines an isomorphism $\Lambda_L \isoarrow \Lambda_L^v\hat\otimes\Lambda$ 
and hence also isomorphism $\theta:\Lambda_{L,\CO_\lambda}\isoarrow \sR$ and 
$\theta^\ur:\Lambda_{L,\CO_\lambda^\ur}\isoarrow\sR^\ur$.  
The isomorphism \eqref{Lgal-twist-eq-ii} also induces 
$
T(1)\hat\otimes \Lambda_L(\Psi_L^{-1}) \isoarrow T(1)\hat\otimes \Lambda_L^v(\Psi_L^v)\hat\otimes\Lambda(\Psi^{-1}),
$
which is compatible with $\theta$. We thus obtain an identification
$
H^1_{\rel,\circ}(L,T(1)\hat\otimes\Lambda_L) \simeq H^1_{\rel,\circ}(L,T(1)\hat\otimes \Lambda_L^v(\Psi_L^v)\hat\otimes\Lambda),
$
that is compatible with $\theta$, where $G_L$ acts on $\Lambda_L$ in the left-hand side via the inverse of the canonical character, and 
the subscripts `$\rel$' and `$\circ$' on the left-hand side denote the submodule of classes $c\in H^1_{\rel,\circ}(L,T(1)\hat\otimes\Lambda_L)$ such that $\loc_{\bar v}(c)\in H^1_{\circ}(L_{\bar v},T(1)\hat\otimes\Lambda_L)$. 

We let 
$$
\CZ^{\circ}(g/L) \in H^1_{\rel,\circ}(L,T(1)\hat\otimes\Lambda_L)
$$
be identified with $\CBF^\circ(g/L)$ under the preceding isomorphism.
This is the two-variable zeta element associated with $g$ and $L$ in the supersingular case.

We let $\CC^{\circ}_v: H^1_{/\circ}(L_v,T(1)\hat\otimes\Lambda_L) \hookrightarrow \Lambda_{L,\CO_\lambda}$
be the composition
$$
\CC_v^{\circ}: H^1_{/\circ}(L_v,T(1)\hat\otimes\Lambda_L)
\simeq H^1_{/\circ}(\BQ_p,T(1)\hat\otimes\BT^+\hat\otimes\Lambda)
\stackrel{\sC^\mathrm{int}_{\circ}}{\hookrightarrow} \sR \stackrel{\theta^{-1}}{\isoarrow}\Lambda_{L,\CO_\lambda}.
$$
This is a $\Lambda_{L,\CO_\lambda}$-injection.  Let
$$
\CL_v^{\circ}(g/L) = \CC_v^{\circ}(\loc_v(\CZ^{\circ}(g/L))) \in \Lambda_{L,\CO_\lambda}.
$$
Note that $\CL_v^{\circ}(g/L) = \theta^{-1}(\sL_p^{\circ}(g/L))$.

We similarly let $\CL_{\bar v}^{\circ}: H^1_{\circ}(L_{\bar v},T(1)\hat\otimes\Lambda_L)\hookrightarrow \Lambda_{L,\CO_\lambda^\ur}$ be the composition
$$
\CL_{\bar v}^{\circ}: H^1_{\circ}(L_{\bar v},T(1)\hat\otimes\Lambda_L) \simeq H^1_{\circ}(\BQ_p,T(1)\hat\otimes\BT_1^-\hat\otimes\Lambda) \stackrel{\frac{1}{T_v}\sL^\mathrm{int}_{\circ}}{\hookrightarrow}\sR^\ur \stackrel{\theta^{\ur,-1}}{\isoarrow}\Lambda_{L,\CO_\lambda^\ur}.
$$
Note that the inclusion $\BT_1 = \widetilde\BT \subset \widetilde\BH$ induces
an isomorphism $\BT_1^- = T_v\widetilde\BH^-$, and so $\sL^\mathrm{int}_\circ$ maps
$H^1_\circ(\BQ_p,T(1)\hat\otimes\BT_1^-\hat\otimes\Lambda)$ into $T_v\sR^\ur$. In particular,
the middle arrow of the composition defining $\CL_{\bar v}^\circ$ is well-defined.
We also let 
$$
\CL_p^{Gr}(g/L) = \CL_{\bar v}^{\circ}(\loc_{\bar v}(\CZ^{\circ}(g/L))) \in \Lambda_{L,\CO_\lambda^\ur}.
$$
So $\CL_p^{Gr}(g/L) = \theta^{\ur,-1}(\sL_p^{Gr}(g/L))$.

\subsubsection{Connections with cyclotomic $L$-functions and cyclotomic zeta elements}
From Propositions \ref{Col-cycsp-prop-ss} and \ref{ERLIint-prop-ss} we immediately conclude:
\begin{prop}\label{2varZ-prop-ss} \hfill
\begin{itemize}
\item[(i)] The reduction of $\CC_v^{\circ}$ modulo $\gamma_\ac-1$ equals the composition
$$
H^1_{/\circ}(L_v,T(1)\hat\otimes\Lambda_L)\stackrel{\mod (\gamma_v-1)}{\twoheadrightarrow}
H^1_{/\circ}(L_v,T(1)\hat\otimes\Lambda)\stackrel{Col_{\eta_{\omega},v}}{\hookrightarrow}
\Lambda_{\CO_\lambda}.
$$
\item[(ii)] For $0\neq \omega\in S_F$, $\gamma\in V_{F,g}$, $\gamma'\in V_{F,g'}$, 
and $c^{\circ}(\omega,\gamma,\gamma')\in F^\times$ as in 
Proposition \ref{ERLIint-prop-ss},
$$
\CL_v^{\circ}(g/L) \mod (\gamma_\ac-1) = c^{\circ}(\omega,\gamma,\gamma')\mathfrak{g}(\chi_L)^{-1}
\CL_{\omega,\gamma,\gamma'}^{\circ}(g/L) \in  \Lambda_{\CO_\lambda} = \Lambda_{L,\CO_\lambda}/(\gamma_\ac-1)\Lambda_{L,\CO_\lambda}.
$$
\end{itemize}
\end{prop} 

We exploit this proposition to prove that the Beilinson--Kato element $\bz^\circ_{\omega,\gamma,\gamma'}(g_{/L})$
is essentially the cyclotomic specialisation of the two-variable zeta element $\CZ^{\circ}(g/L)$.
\begin{thm}\label{2varZ-1varZ-thm-ss} Let $g' = g\otimes\chi_L$. Let $0\neq \omega\in S_{F,g}$, and let $\gamma\in V_{F,g}$ and $\gamma'\in V_{F,g'}$
such that $\gamma^\pm\neq 0$ and $(\gamma')^\pm\neq 0$, and let $c^{\circ}(\omega,\gamma,\gamma') \in F^\times$ be as in Proposition \ref{ERLIint-prop-ss}.
\begin{itemize}
\item[(i)] The image of $\CZ^{\circ}(g/L)$ under the map
$H^1_{\rel,\circ}(\CO_L[\frac{1}{p}],T(1)\otimes\Lambda_L) \rightarrow H^1_{\rel,\circ}(\CO_L[\frac{1}{p}],T(1)\otimes_{\BZ_p}\Lambda)\otimes_{\BZ_p}\BQ_p$
induced by the projection $\Lambda_L/(\gamma_\ac-1)\Lambda_L \isoarrow \Lambda_L^\ac\isoarrow\Lambda$
equals $c^{\circ}(\omega,\gamma,\gamma')\mathfrak{g}(\chi_L)^{-1} \bz^\circ_{\omega,\gamma,\gamma'}$.
\item[(ii)] If (irr$_\BQ$) holds, $\omega\in S_{g,\CO}$ is good, and $\gamma\in T_{g,\CO}$ and $\gamma'\in T_{g',\CO}$
are such that $\gamma^\pm$ is an $\CO$-basis of $T_{\CO,g}$ and $(\gamma')^\pm$ is an $\CO$-basis of $T_{\CO,g'}^\pm$
(so $c(\omega,\gamma,\gamma') \in \CO^\times$), 
then the equality in \rm{(i)} holds in $H^1_{\rel,\circ}(\CO_L[\frac{1}{p},T(1)\otimes_{\BZ_p}\Lambda)$.
\end{itemize}
\end{thm}
\begin{proof}
One may proceed just as in the proof of Theorem~\ref{2varZ-1varZ-thm}. 
\end{proof}
\subsubsection{Connections with anti-cyclotomic $L$-functions and Heegner points}

From Propositions \ref{GRL=BDPL-ss} and \ref{GRLvan-prop-ss} and Theorem \ref{BDPformula} we conclude:
\begin{prop}\label{GRL=BDPL-II-ss}\hfill
\begin{itemize}
\item[(i)] Suppose \eqref{Heeg} holds.
The image of $\CL_p^{Gr}(g/L)$ modulo $\gamma_+-1$ equals $-\sL_v^{BDP}(g/L)$. In particular, 
the image of $\CL_p^{Gr}(g/L)$ under $\phi_{\bf 1}: \Lambda_{L,\CO_\lambda^\ur} \twoheadrightarrow \Lambda_{L,\CO_\lambda^\ur}/(\gamma_+-1,\gamma_--1)\Lambda_{L,\CO_\lambda^\ur} = \CO_\lambda^\ur$ is
$$
\phi_{\bf 1}(\CL_p^{Gr}(g/L)) = - (1-a(p)p^{-1}+p^{-1})^2 (\log_{\omega_g}(y_L))^2.
$$
\item[(ii)] Suppose $\epsilon(g/L) = +1$. The image of $\CL_p^{Gr}(g/L)$ modulo $\gamma_+-1$
is $0$.
\end{itemize}
\end{prop}

This proposition allows us to relate the image of $\CZ^{\circ}(g/L)$ under Perrin-Riou's regulator map (or `big logarithm') to 
Heegner points, providing a key link in our subsequent proof of the Perrin-Riou Conjecture.

 Recall that Perrin-Riou's regulator map for $H^1(L_{\bar v},\CF^{+}_{\gamma}{\bf D}(T(1))\hat\otimes\Lambda_L^\ac)$ is the composition
\begin{equation*}\begin{array}{ccl}
\CL^{PR}_{\gamma}: H^1(L_{\bar v}, \CF^{+}_{\gamma}{\bf D}(T(1))\hat\otimes  \mathfrak{K}_{1,F_\lambda(\gamma)}^{\ac}) & \stackrel{\Psi_L^\ac(g)\mapsto\eps(g)^{-1}\Psi_L^\ac(g)}{\longrightarrow} &
H^1(L_{\bar v}, \CF^{+}_{\gamma}{\bf D}(T(1))\hat\otimes  \mathfrak{K}_{1,F_\lambda(\gamma)}^{\ac}) \\ & \stackrel{res}{\hookrightarrow}  & H^1(L_{\bar v}^\ur, \CF^{+}_{\gamma}{\bf D}(T(1))\hat\otimes  \mathfrak{K}_{1,F_\lambda(\gamma)}^{\ac})  \\
& \stackrel{\CL_{\CF^{+}_{\gamma}{\bf D}(T(1))}}{\longrightarrow} & D(\CF^{+}_{\gamma}{\bf D}(T(1)))\hat\otimes W(\ov{\BF}_p)\hat\otimes \mathfrak{K}_{1,F_\lambda(\gamma)}^{\ac},
\end{array}
\end{equation*}
where $\mathfrak{K}_{1,F_\lambda(\gamma)}^{\ac}$ is the anticyclotomic counterpart of $\mathfrak{K}_{1,F_\lambda(\gamma)}$ and for the third map we have used that as a $G_{L_{\bar v}^\ur}$-module $\Lambda^\ac_L$ is naturally isomorphic to the cyclotomic 
algebra $\Lambda$.
Let 
$$
\CL^{PR}_{g_\gamma} = ([\omega_g,-]\otimes id\otimes id)\circ \CL^{PR}_{\gamma}: 
H^1(L_{\bar v}, \CF^{+}_{\gamma}{\bf D}(T(1))\hat\otimes \mathfrak{K}_{1,F_\lambda(\gamma)}^{\ac}) 
\rightarrow 
\mathfrak{K}_{1,F_\lambda(\gamma)}^{\ac,\ur}.
$$

As in Proposition~\ref{Col-pm}, 
there exist $\sR^\ac$-module homomorphisms 
$$
\sL_{\pm}^{PR}: 
 H^{1}_{\pm}(L_{\ov{v}},T(1)\hat\otimes \Lambda_{L}^{\ac}) \ra
\Lambda_{L}^{\ac,\ur}.
$$
such that
$$
\begin{pmatrix}
\CL_{g_{\alpha}}^{PR}\\
\CL_{g_{\beta}}^{PR}\\
\end{pmatrix}
=
M \cdot \begin{pmatrix}
\CL^{PR}_{-}\\
\CL^{PR}_{+}\\
\end{pmatrix}
.
$$

\begin{lem}\label{GRL=PRLog-ss}
There exists a unit $U_L \in (\Lambda_L^{\ac,\ur})^\times$ depending only on $L$ such that 
$$
U_L \cdot \CL_{\bar v}^{\circ} \mod (\gamma_+-1) = \CL^{PR}_{\circ}. 
$$
In particular, we have
$$u_L \cdot \phi_{\bf 1}\circ \CL_{\bar v}^{\circ} 
= \begin{cases} 
\frac{p(p-1)}{p+1}\cdot \log_{BK}^{\eta_g} & \text{ if $\circ=-$} \\
\frac{-2p}{p+1}\cdot \log_{BK}^{\eta_g} & \text{if $\circ=+$}, 
\end{cases}
$$
where $u_L \in \BZ_p^{\ur,\times}$ is the image of $U_L$ under the specialization map 
$\Lambda_L^{\ac,\ur}\twoheadrightarrow\Lambda_L^{\ac,\ur}/(\gamma_\ac-1)\Lambda^{\ac,\ur}  = \BZ_p^\ur$.
\end{lem}

\begin{proof} 
This is analogous to the proof of Lemma~\ref{GRL=PRLog}. 

As therein, put 
$\CU_L = \frac{1}{T_v}\CH_v\cdot\eta_{\bh_v}(\lambda_{\bh_v}\mod \BT_1^+)\in (\Lambda_L^\ur)^\times$ and 
$$
U_L = \CU_L^{-1}\mod (\gamma_+-1) \in (\Lambda_L^{\ac,\ur})^\times.
$$
The lemma then follows directly from the identification
$H^1_{\circ}(L_{\bar v},T(1)\hat\otimes\Lambda_L) = H^1_{\circ}(\BQ_p,T(1)\hat\otimes\BT_1^-\hat\otimes\Lambda)$ determined by the basis $\lambda_{\bh_v}$ and comparing the definitions of $\CL_{\bar v}^{\circ}$ and $\CL^{PR}_{\circ}$. 
 
  From the specialisation properties of $\CL_{\CF^{+}_{\gamma}{\bf D}(T(1))}$ (see \cite[Thm.~B.5]{LZ0}) it follows that
\begin{equation}\label{PRLog=BKlog}
\phi_{\bf 1}\circ \CL_{g_\gamma}^{PR} = (1-\gamma/p)(1-1/\gamma)^{-1}\log_{BK}^{\eta_g},
\end{equation}
where $\log_{BK}$ denotes the Bloch--Kato logarithm for $\CF^{+}_{\gamma}{\bf D}(T(1))$ and $\eta_g = \eta_{\omega_g}$ (so 
$\log_{BK}(-) = \log_{BK}^{\eta_g}(-)\cdot\eta_g \in D_\cris(\CF^{+}_{\gamma}{\bf D}(T(1)))$).  
This implies the conclusion for $\phi_{\bf 1}\circ\CL_{\bar v}^{\circ}$.
\end{proof}
From Lemma \ref{GRL=PRLog-ss} and Proposition \ref{GRL=BDPL-II} we conclude:
\begin{lem}\label{GRL=logheegner-ss} \hfill
\begin{itemize} 
\item[(i)] Suppose (Heeg) holds. Then
$$
 - u_L (1-a(p)p^{-1}+p^{-1})^2 (\log_{\omega_g}(y_L))^2 = 
 \begin{cases} 
\frac{p(p-1)}{p+1}\cdot \log_{BK}^{\eta_g}(\loc_{\bar v}(\phi_{\bf 1}(\CZ^{-}(g/L))) &  \\
\frac{-2p}{p+1}\cdot \log_{BK}^{\eta_g}(\loc_{\bar v}(\phi_{\bf 1}(\CZ^{+}(g/L))). &
\end{cases} 
$$
\item[(ii)] If $\eps(g/L) = +1$, then $\loc_{\bar v}(\phi_{\bf 1}(\CZ^{\circ}(g/L)))= 0$.
\end{itemize}
\end{lem}

\section{The Perrin-Riou Conjecture}\label{PRC} 
In this section we complete our proofs of results towards the Perrin-Riou conjecture. In particular, we prove the conjecture for elliptic curves at primes 
$p\geq 5$ of good reduction. 
 
The eponymous conjecture predicts a link between Beilinson--Kato element associated to a weight two newform and rational points on its
associated $\GL_2$-type abelian variety. It is a $p$-adic Beilinson conjecture for the $p$-adic Beilinson--Kato elements (see \cite{PR} for more on this perspective). Its weak version (stated below) can be regarded as a $p$-adic Leopoldt-style conjecture.

\subsection{The Perrin-Riou Conjecture}  
Let $g \in S_{2}(\Gamma_{0}(N))$ be a newform and $p$ a prime. We keep to the notation of the preceding sections,
especially Section \ref{NotationPrelim}. In particular, $T=T_{\CO_{\lambda}}$ is a lattice in  the $p$-adic Galois representation 
$V=V_{F_\lambda}$ associated with $g$ and a prime $\lambda\mid p$ of the Hecke field $F$.

Let $A_{/\BQ}$ be a $\GL_2$-type abelian variety in the isogeny class of such varieties associated 
with $g$ (see \S\ref{newforms-AV}).

\subsubsection{The conjecture}
Let $\gamma \in V_F$ with $\gamma^{\pm}\neq 0$ and let $\bz_\gamma(g) \in H^1(\BZ[\frac{1}{p}],V(1)\otimes_{\BZ_p}\Lambda)$
be the Beilinson--Kato element as in Section \ref{BK-rationals}. Let 
$$z_{\Kato}(g) \in H^{1}(\BQ,V(1))$$ 
be the image of $\bz_\gamma(g)$ under the map induced by the specialisation map $\Lambda\twoheadrightarrow \BZ_p$, $\gamma_\cyc\mapsto 1$.
Kato's explicit reciprocity law \cite[Thm. 12.5]{K}
implies:
$$
\loc_{p}(z_{\Kato}(g)) \in H^{1}_{f}(\BQ_{p},V(1)) \iff L(1,g_{/\BQ})=0.
$$ 
The following conjecture, connecting $z_\Kato(g)$ to the arithmetic of $A_g$, especially when $0\neq \loc_p(z_\Kato(g))\in H^1_f(\BQ_p,V(1))$,
is essentially due to Perrin-Riou \cite[\S3.3-3.4]{PR}.

\begin{conj}[The Perrin-Riou Conjecture] \label{PR} 
Suppose that $L(1,g) = 0$ $($equivalently, $\loc_p(z_\Kato(g))\in H^1_f(\BQ_p,V(1))$\,$)$. 
Let $\lambda\mid p$ be a prime of the Hecke field $F$ of $g$.
There exists $P \in A(\BQ)\otimes\BQ$ such that 
\begin{itemize}
\item[(1)]  $\log_{BK}^{\eta_{\omega}}(\loc_{p}(z_{\rm Kato}(g))) \doteq \log_{\omega_{A}}(P)^{2}$, where
{\begin{itemize}
\item $0\neq \omega\in S_F$ and $0\neq \omega_A\in \Omega^1(A/\BQ)$ are $F$-bases,
\item $\log_{BK}^{\eta_{\omega}}: H^{1}_{f}(\BQ_{p},V_{F_{\lambda}}(1)) \ra F_{\lambda}$ is the Bloch--Kato logarithm 
arising from $\omega$, and
\item $\log_{\omega_{A}}:A(\BQ_p)\otimes_{\BZ_p}\BQ_p\rightarrow F_\lambda$ is the logarithm associated with $\omega_{A}$ and $\lambda$,
\item `$\doteq$' denotes an equality up to an element of $F^\times$;
\end{itemize}}
\item[(2)] $P \neq 0 \iff  \ord_{s=1}L(s,f_{/\BQ})=1$.
\end{itemize}
In particular, 
$$
\loc_{p}(z_{Kato}(g)) \neq 0 \iff \ord_{s=1}L(s,g)=1.
$$
\end{conj}
\noindent We refer to the `In particular' part as the {\em weak version} of the Perrin-Riou Conjecture.

\begin{remark}
The conjecture in \cite{PR}  is stated for elliptic curves over the rationals and for odd primes $p$ of good reduction. 
For $p\geq 5$ this is proved below as a special case of the main results of this section. The generalisation of the statement of the conjecture 
to newforms is straightforward. However, as stated, the point $P\in A(\BQ)\otimes\BQ$ could depend on the prime $\lambda\mid p$. A slightly
stronger version of the conjecture would be that there exists $P$ such that (1) holds for all $\lambda\mid p$ (as well as (2)). 
This is equivalent to the conjecture as stated holding for {\em all} $\lambda\mid p$. 
\end{remark}

\begin{remark}
The $p$-adic logarithm of a non-torsion point in $A(\BQ)$ is expected to be transcendental\footnote{We are grateful to G. Wustholz for reminding us of this}. 
Conjecture \ref{PR} thus implies a transcendence result for the Beilinson--Kato elements with non-zero localisation at $p$. 
\end{remark}

\subsection{Main result}\label{sMR} 
The main result of this section is a proof of the Perrin-Riou conjecture for all newforms $g\in S_2(\Gamma_0(N))$, all primes $p\nmid 2N$,
and all $\lambda\mid p$ such that $g$ is ordinary with respect to $\lambda$ (so $\lambda\nmid a_p(g)$) and all $\lambda\mid p$ if $a_p=0$.
Along the way we prove a more precise (integral) version of the conjecture, which has strong arithmetic consequences.

\begin{thm}\label{mPR}
Let $g \in S_{2}(\Gamma_{0}(N))$ be a newform. Let $p\nmid 2N$ be a prime and let $\lambda\mid p$ be a prime of the 
Hecke field of $g$. Suppose that either $a_p(g)=0$ or $\lambda\nmid a_p(g)$ (that is, $g$ is ordinary with respect to $\lambda$). 
Then Conjecture \ref{PR} holds.
\end{thm} 

As every elliptic curve $E_{/\BQ}$ is modular, we conclude:
\begin{cor}\label{PR-CM}
Let $E_{/\BQ}$ be an elliptic curve with conductor $N$. Let $p\nmid 6N$ be a prime. Then Conjecture \ref{PR0} holds.
\end{cor}

Since for a CM form $g\in S_2(\Gamma_0(N))$ and $p\nmid 2N$, either $\lambda\nmid a_p(g)$ for all $\lambda\mid p$ (when $p$ splits in the CM field) or $a_p(g) = 0$ (when $p$ is inert in the CM field), this theorem also implies:
\begin{cor}\label{PRCM}
Let $g \in S_{2}(\Gamma_{0}(N))$ be a CM newform. Let $p\nmid 2N$ be a prime.
Then Conjecture \ref{PR} holds.
\end{cor} 
However, a proof of a version of Conjecture \ref{PR} in the case that $g$ is a CM form that is ordinary at $\lambda\mid p$ 
(see Theorem \ref{A-PRConj-CMthm}), based on results in \cite{LLZcr}, \cite{BDP2}, and \cite{PR},
plays a role in our proof of the general case of Theorem \ref{mPR}.

\begin{remark}
For CM elliptic curves of analytic rank one and $p>2$ an ordinary prime, a result towards Conjecture \ref{PR} is due to Rubin \cite{Ru1}.  
Rubin's result preceded Conjecture \ref{PR}, and the formulation of Conjecture \ref{PR} was in fact inspired by Rubin's theorem.
\end{remark}

\subsubsection{Proof of Theorem \ref{mPR}: ordinary case}\label{ss:PRconj-proof-ord}
We first give the proof in the ordinary case, that is, in the case $\lambda\nmid a_p(g)$.

Suppose $L(1,g) = 0$.
Let $L$ be an imaginary quadratic field of discriminant $-D_L<0$ such that 
\begin{itemize}
\item $(D_L,N) = 1$,
\item the prime $p$ splits in $L$: $p=v\bar v$,
\item $\eps(g') = +1$ and  $\ord_{s=1} L(s,g') = 0$, $g' = g\otimes\chi_L$.
\end{itemize}
The existence of $L$ follows from \cite{BFH}.
Note that we do not (yet) require that $L$ also satisfies (Heeg). In fact, if $\epsilon(g) =+1$ then this would not be possible.
Let $(\alpha$, $\omega$, $\gamma,\gamma')$ be a quadruple as in \S\ref{BK-overL}. 
If (irr$_{\BQ}$) holds, then we choose $\omega\in S_{g,\CO}$ to be good, and $\gamma\in T_{g,\CO}$ and $\gamma'\in T_{g',\CO}$
such that $\gamma^\pm$ is an $\CO$-basis of $T_{\CO,g}^\pm$ and $(\gamma')^\pm$ is an $\CO$-basis of $T_{\CO,g'}^\pm$.

Let $z(g/L) \in H^1(L,V(1))$ be the image of $\CZ(g/L)$ under the map induced by specializing $\Lambda_L$ at the trivial character.
Then it follows from Theorem \ref{2varZ-1varZ-thm} and the choice of $L$ that we have 
\begin{equation}\label{comp-with-Kato}
z(g/L) = c(\omega,\gamma,\gamma')\mathfrak{g}(\chi_L)(1-1/\alpha)^2 \frac{L(1,g')}{\Omega^-} z_\Kato(g) \in H^1(L,V(1)),
\end{equation}
where $c(\omega,\gamma,\gamma') \in F^\times$ is as in Proposition \ref{ERLIint-prop}. 
Since $z_\Kato(g) \in H^1(\BQ,V(1)) = H^1(L,V(1))^+$, both
$H^1(L_v,V(1))$ and $H^1(L_{\bar v}, V(1))$ are canonically identified with $H^1(\BQ_p,V(1))$. 
In particular, we see that 
\begin{equation}\label{zkato-zg}
\loc_p(z_\Kato(g))\neq 0 \iff \loc_{\bar v} (z(g/L)) \neq 0.
\end{equation}
\medskip

\noindent{\underline{\em Proof of the weak version}:}
Suppose first that $\eps(g) = +1$, so $\ord_{s=1}L(s,g)$ is even (and $\neq 1$, in particular). 
We then need to show that $\loc_p (z_\Kato(g))= 0$.
By the choice of $L$, $\eps(g/L) = \eps(g)\eps(g') = +1$. But then 
$\loc_{\bar v}(z(g/L)) = 0$ by Lemma \ref{GRL=logheegner}(ii), which implies $\loc_p(z_\Kato(g)) = 0$ by \eqref{zkato-zg}.

Suppose next that $\eps(g) = -1$.  In this case, we can - and do - require that $L$ also satisfies the Heegner hypothesis \eqref{Heeg}.
It then follows from Lemma \ref{GRL=logheegner}(i) that 
$$
(1-\alpha/p)(1-1/\alpha)^{-1} \log_{BK}^{\eta_g}(\loc_{\bar v}(z(g/L))) = - u_L (1-a(p)p^{-1}+p^{-1})^2 (\log_{\omega_g}(y_L))^2.
$$
Here $y_L \in J_0(N)(L)$ is the Heegner point associated with $L$ and $\log_{\omega_g}$ is the $p$-adic logarithm
on $J_0(N)(L_{\bar v})$ associated with the differential $\omega_g$ (which corresponds to the differential
on $X_0(N)$ that pulls back to $2\pi i g(z) dz$ under the usual complex uniformization).  
Combining this with \eqref{comp-with-Kato} yields
\begin{equation}\label{log-Kato-elt-1}
\log^{\eta_{\omega_g}}_{BK}(\loc_p z_\Kato(g)) = - u_L \frac{(1-a(p)p^{-1}+p^{-1})}{c(\omega,\gamma,\gamma')\mathfrak{g}(\chi_L) L(1,g')/\Omega^-}
(\log_{\omega_g}(y_L))^2.
\end{equation}
It follows immediately that $\loc_p(z_\Kato(g)) \neq 0$ if and only if $y_L$ is non-torsion, and by the Gross--Zagier formula this happens if and only if
$\ord_{s=1}L(s,g) = 1$, which completes the proof of the weak form of the Perrin-Riou conjecture.
\medskip

\noindent{\underline{\em Proof of the stronger form}:}
To prove the stronger form of the conjecture, we can and do assume that $L\neq \BZ(\sqrt{-1}), \BQ(\sqrt{-2})$.
We fix a quotient map $$\phi_A:J_0(N)\rightarrow A,$$ which also determines a quotient map $J_1(N)\rightarrow A$
by composition with the natural map $J_1(N)\rightarrow J_0(N)$. We then let $P_L = \phi_A(y_L) \in A(L)$.
Then $\phi_A^*(\omega_A) = c_g\cdot \omega_g$ for some $0\neq c_g\in F^\times$, and so 
$\log_{\omega_A}(P_L) = c_g\log_{\omega_A}(y_L)$.   
Without loss of generality we may assume that $\omega=\omega_g$ and $\omega$ and $\omega_A$ are identified by 
the quotient map $J_1(N)\rightarrow A$. 

We can then rewrite \eqref{log-Kato-elt-1} as 
\begin{equation*}\label{log-Kato-elt-2}
\log^{\eta_{\omega}}_{BK}(\loc_p z_\Kato(g)) = - u_L c_g^{-2}\frac{(1-a(p)p^{-1}+p^{-1})}{c(\omega,\gamma,\gamma')\mathfrak{g}(\chi_L) L(1,g')/\Omega^-}
(\log_{\omega_A}(P_L))^2.
\end{equation*}
Let $P_L^\pm = \frac{1\pm c}{2} P_L$.
Then $P_L^- \in (A(L)\otimes_{\CO_F} F)^- = A'(\BQ)\otimes_{\CO_F} F$, where $A'$  is the $\GL_2$-type abelian variety associated with $g'$ that is just the $L$-twist of $A$.  
Since $L(1,g') \neq 0$ by the choice of $L$, it follows (say from Kato \cite[Thm.~14.2]{K}) that $A'(\BQ)\otimes_{\CO_F} F = 0$, hence $P_L^- = 0$
and so $\log_{\omega_A} P_L = \log_\omega P_L^+$. Let $P = P_L^+\in (A(L)\otimes_{\CO_F} F)^+ = A(\BQ)\otimes_{\CO_F} F$. We then have
\begin{equation}\label{up-to-uL}
\log_{BK}^{\eta_\omega} (\loc_p z_\Kato(g)) = c \cdot (\log_{\omega_A}(P))^2,
\end{equation}
with
$$
c = - u_L \frac{(1-a(p)p^{-1}+p^{-1})}{c_g^2c(\omega,\gamma,\gamma')\mathfrak{g}(\chi_L)L(1,g')/\Omega^-}.
$$
Since $\frac{(1-a(p)p^{-1}+p^{-1})}{c_g^2c(\omega,\gamma,\gamma')\mathfrak{g}(\chi_L) L(1,g')/\Omega^-} \in F^\times$,
to prove part (1) of the conjecture, it remains to note that $u_L \in F^\times$. We in fact show that $u_L\in\BZ_{(p)}^\times$; this is explained below.
Part (2) of the conjecture is an immediate consequence of $P\neq 0 \iff y_L \neq 0$ and the Gross--Zagier formula for $y_L$.
\medskip

\noindent{\underline{\em Proof that $u_L \in \BZ_{(p)}^\times$}:}
To determine $u_L$ and complete the proof of the stronger form of the Perrin-Riou Conjecture, we exploit
a version of the conjecture proved for CM forms in Theorem \ref{A-PRConj-CMthm}.   Theorem \ref{A-PRConj-CMthm}
applies to the CM forms associated with the characters $\psi\in \mathfrak{X}$, for $\mathfrak{X}$ as in \S\ref{B-auxiliary} for 
the chosen field $L$. It then follows from comparing Theorem \ref{A-PRConj-CMthm} with \eqref{up-to-uL} for $g=g_\psi$ that
$u_L \in F_\psi^\times$ for all $\psi\in\mathfrak{X}$. In particular, $$u_L \in \cap_{\psi\in \mathfrak{X}} F_\psi.$$ The latter equals
$\BQ$ by Lemma \ref{B-aux-lem}. This shows that $u_L \in \BQ^\times\cap \BZ_p^{\ur,\times} = \BZ_{(p)}^\times$.
\medskip

This completes the proof of Theorem \ref{mPR} in the ordinary case.

\begin{remark} The two-variable zeta element underlies the above proof. 
\end{remark}

\subsubsection{Proof of Theorem \ref{mPR}: supersingular case}
We now consider the supersingular case, that is, assume $a_{p}(g)=0$. The proof is analogous to the ordinary case.

Suppose $L(1,g) = 0$. Let $L$ be an imaginary quadratic field as in the second paragraph of subsection \ref{ss:PRconj-proof-ord}. 
Let $z^{+}(g/L) \in H^1(L,V(1))$ be the image of the zeta element $\CZ^{+}(g/L)$ under the map induced by specializing $\Lambda_L$ at the trivial character. (Alternatively, one may consider $z^{-}(g/L)$.)

In view of  Theorem \ref{2varZ-1varZ-thm-ss} and the choice of $L$ we have 
\begin{equation}\label{comp-with-Kato-ss}
z^{+}(g/L) = c^{+}(\omega,\gamma,\gamma')\mathfrak{g}(\chi_L) \frac{L(1,g')}{\Omega^-} z_\Kato(g) \in H^1(L,V(1)),
\end{equation}
where $c^{+}(\omega,\gamma,\gamma') \in F^\times$ is as in Proposition \ref{ERLIint-prop-ss}. 
Since $z_\Kato(g) \in H^1(\BQ,V(1)) = H^1(L,V(1))^+$, both
$H^1(L_v,V(1))$ and $H^1(L_{\bar v}, V(1))$ are canonically identified with $H^1(\BQ_p,V(1))$. 
In particular, we see that 
\begin{equation}\label{zkato-zg-ss}
\loc_p(z_\Kato(g))\neq 0 \iff \loc_{\bar v} (z^{+}(g/L)) \neq 0.
\end{equation}
\medskip

\noindent{\underline{\em Proof of the weak version}:} 
Suppose first that $\eps(g) = +1$, so $\ord_{s=1}L(s,g)$ is even (and $\neq 1$, in particular). 
We then need to show that $\loc_p (z_\Kato(g))= 0$.
By the choice of $L$, $\eps(g/L) = \eps(g)\eps(g') = +1$. But then 
$\loc_{\bar v}(z^{+}(g/L)) = 0$ by Lemma \ref{GRL=logheegner-ss}(ii), which implies $\loc_p(z_\Kato(g)) = 0$ by \eqref{zkato-zg}.

Suppose next that $\eps(g) = -1$.  In this case, we further suppose that $L$ also satisfies the Heegner hypothesis \eqref{Heeg}.
It then follows from Lemma \ref{GRL=logheegner-ss}(i) that 
$$
\frac{-2p}{p+1}\cdot \log_{BK}^{\eta_g}(\loc_{\bar v}(\phi_{\bf 1}(\CZ^{+}(g/L)))
= - u_L (1-a(p)p^{-1}+p^{-1})^2 (\log_{\omega_g}(y_L))^2.
$$
Here $y_L \in J_0(N)(L)$ is the Heegner point associated with $L$ and $\log_{\omega_g}$ is the $p$-adic logarithm
on $J_0(N)(L_{\bar v})$ associated with the differential $\omega_g$. 
Combining this with \eqref{comp-with-Kato-ss} yields
\begin{equation}\label{log-Kato-elt-1}
\log^{\eta_{\omega_g}}_{BK}(\loc_p z_\Kato(g)) = 
\frac{u_L}{2p/p+1} \cdot \frac{(1-a(p)p^{-1}+p^{-1})^{2}}{c^{+}(\omega,\gamma,\gamma')\mathfrak{g}(\chi_L) L(1,g')/\Omega^-}
(\log_{\omega_g}(y_L))^2.
\end{equation}
It follows immediately that $\loc_p(z_\Kato(g)) \neq 0$ if and only if $y_L$ is non-torsion, and by the Gross--Zagier formula this happens if and only if
$\ord_{s=1}L(s,g) = 1$, which completes the proof of the weak form of the Perrin-Riou conjecture.
\medskip

\noindent{\underline{\em Proof of the stronger form}:} Given the preceding analysis, one may proceed  exactly as in the ordinary case (cf.~subsection \ref{ss:PRconj-proof-ord}). 

\begin{remark} 

The supersingular case can also be approached via Kobayashi's $p$-adic Gross--Zagier formula \cite{Ko1} (cf.~\cite[App.~B]{BKO2}). This approach does not rely on the two-variable zeta element. 
\end{remark}

{\appendix

\section{The ordinary CM case}\label{A-PRConj-CM}
In this appendix we present a proof of a version of the Perrin-Riou Conjecture for certain CM modular forms at good ordinary primes $p>2$. Our proof essentially pieces together results from \cite{LLZcr}, \cite{BDP2}, and \cite{PR}. In particular, it does not rely on the two-variable zeta elements for auxiliary imaginary quadratic fields that are the focus of Part I of this paper or on the $p$-adic Gross--Zagier formula, which undergirds Rubin's earlier work \cite{Ru1}.  The truth of the Perrin-Riou Conjecture for the cases considered in this appendix is used in the proof of the general case given in Part I of this paper (it is used to show that the constants $u_L$ associated with the auxiliary imaginary quadratic fields $L$ belong to $\BZ_{(p)}^\times$). 

\subsection{The set-up}\label{A-setup}
Let $\overline\BQ$ be a separable algebraic closure of $\BQ$ and let $\iota_\infty:\overline\BQ\hookrightarrow \BC$ 
be an embedding. Let $K\subset\overline\BQ$ be an imaginary quadratic field of discriminant $-D_K<0$
and $\CO_K$ its ring of integers. The embedding $\iota_\infty$ identifies $K\otimes\BR$ with $\BC$.
Let $\tau\in G_\BQ$ be the involution induced via $\iota_\infty$ by complex conjugation on $\BC$ (so $\tau$ restricts to the non-trivial automorphism of $K$).

Let $\psi: K^\times\bs \BA_K^\times\rightarrow \BC^\times$ be a Hecke character with infinity type $(-1,0)$ (meaning that $\psi(z) = z^{-1}$ for $z\in K\otimes\BR = \BC$). The values $\psi(\BA_K^{\infty,\times})$ of $\psi$ on the finite ideles generate a finite extension $F_\psi\subset\BC$ of $K$, which we view as a subfield
of $\overline\BQ$ via $\iota_\infty$. Let $\mathfrak{f}_\psi$ be the conductor of $\psi$, and let $\mathfrak{f}\subset \mathfrak{f}_\psi$ be an ideal as in \cite[Thm.~2.2]{LLZcr}. Then there is a unique CM pair $(E,\alpha)$ over the ray class field $K(\mathfrak{f})$: $\End_{K(\mathfrak{f})}(E) \simeq\CO_K$ such that the induced action on $\mathrm{coLie}(E/K(\mathfrak{f}))$ is the natural action of $K$, the annihilator of $\alpha\in E(K(\mathfrak{f}))_{\tor}$ in $\CO_K$ is 
$\mathfrak{f}$, and there is an isomorphism $E(\BC)\isoarrow \BC/\mathfrak{f}$ that maps $\alpha$ to $1$.   Following
\cite[\S15.8]{K} we let $V(\psi) = H^1(E(\BC),\BQ)\otimes_K F_\psi$ and $S(\psi) = H^0(\Gal(K(\mathfrak{f})/K),\mathrm{coLie}(E/K(\mathfrak{f}))\otimes_KF_\psi)$, where the action of $\Gal(K(\mathfrak{f})/K)$ is defined in {\it op.~cit.} These are one-dimensional $F_\psi$-spaces. 
For any $F_\psi$-algebra $A$, let $V(\psi)_A = V(\psi)\otimes_{F_\psi} A$ and $S(\psi)_A = S(\psi)\otimes_{F_\psi} A$. We let
$\mathrm{per}_\psi: S(\psi)\hookrightarrow V(\psi)_\BC$ be the period map induced from the usual 
period map $\mathrm{coLie}(E/K(\mathfrak{f}))\rightarrow H^1(E(\BC),\BC)$.  

Let $p$ be a prime and $\lambda\mid p$ a prime of $F_\psi$. We fix an embedding $\iota_p:\overline\BQ\hookrightarrow \overline\BQ_p$ that induces
$\lambda$. We let $G_K$ act on $V(\psi)_\lambda = V(\psi)_{F_{\psi,\lambda}}$ as described in \cite[\S15.8]{K} and denote the corresponding
character $G_K\rightarrow F_{\psi,\lambda}^\times$ by $\psi_\lambda$ (this action is the inverse of that defined in \S \ref{CMHF}). Let
$$
\widetilde{V(\psi)}_\lambda = \Ind_{K}^{\BQ} V(\psi)_\lambda.
$$
This is an irreducible two-dimensional $F_{\psi,\lambda}$-representation of $G_\BQ$. Concretely, 
$\widetilde{V(\psi)}_\lambda = V(\psi)_\lambda\oplus V(\psi)_\lambda^\tau$, where $V(\psi)_\lambda^\tau$ has the same underlying
space as $V(\psi)_\lambda$ but the action of $\sigma\in G_K$ is multiplication by $\psi_\lambda(\tau^{-1}\psi\tau)$; the action
of $\tau$ just swaps the two summands.  We similarly let
$\widetilde{V(\psi)} = \Ind_{\BC}^{\BR} V(\psi)$, which is a two-dimensional $F_\psi$-space with an action of $G_\BR = \Gal(\BC/\BR)$.

Let $g = g_\psi$ be the weight $2$ newform associated with $\psi$. For our purposes we will assume that
$\psi$ is conjugate self-dual, that is,
\begin{equation}\label{A-psi-res}
\psi|_{\BA^\times} = \chi_K |\cdot|^{-1},
\end{equation}
where $\chi_K$ is the quadratic character corresponding to the quadratic extension $K/\BQ$. This implies that
$g$ has trivial Nebentypus. In particular, $g \in S_2(\Gamma_0(N))$ is a newform of level $N = D_KN_{K/\BQ}(\mathfrak{f}_\psi)$.
We therefore freely use the notations of \S\ref{NotationPrelim}.

Note that the Hecke field $F$ of $g$ is contained in $F_\psi$, and by a mild abuse of notation we also denote by $\lambda$ the prime
of $F$ under the chosen prime $\lambda\mid p$ of $F_\psi$. We let $V(g)_\lambda = V_{F_\psi,\lambda} = V_\lambda\otimes_{F_\lambda}F_{\psi,\lambda}$.
Then $$V(g)_\lambda \simeq \widetilde{V(\psi)}_\lambda = \Ind_{G_K}^{G_\BQ} V(\psi)_\lambda.$$
Such an isomorphism can be normalized as in \cite[Lem.~15.11]{K}: Fixing an isomorphism
$s: S(\psi)\isoarrow S_{F_\psi}$ of one-dimensional $F_\psi$-spaces, there is a unique $F_{\psi,\lambda}$-isomorphism $\widetilde{V(\psi)}_\lambda\isoarrow V(g)_\lambda$ of $G_\BQ$-representations such that the isomorphism $S(\psi)_{F_{\psi,\lambda}} \isoarrow S_{F_{\psi,\lambda}}$ induced by the functoriality of $D_{dR}$ is just that  induced from $s$.  There also exists an unique $F_\psi$-isomorphism $\widetilde{V(\psi)} \isoarrow V_{F_\psi}$ of 
$G_\BR$-representations inducing via $s$ an identification of the period maps $per_\psi: S(\psi) \rightarrow V(\psi)_\BC \subset
\widetilde{V(\psi)}\otimes_{F_\psi}\BC$ and $per:S_{F_\psi}\rightarrow V_\BC$.  From now on we assume that we have fixed an isomorphism $s$ as above and freely appeal to these resulting isomorphisms and identifications.

\subsection{The main result} 
Let $0\neq \gamma\in V(\psi)$ and let $0\neq \gamma'\in V_{F_\psi}$ be the corresponding element. Note that $(\gamma')^\pm\neq 0$. 
Let $\bz_{\gamma'}(g) \in H^1(\BZ[\frac{1}{p}],V(g)_\lambda(1) \otimes_{\BZ_p}\Lambda)$ be the Beilinson--Kato element as in 
\S\ref{Beilinson--Kato-elements} (with the field of coefficients extended from $F_\lambda$ to $F_{\psi,\lambda}$). 
Let $z_{\gamma'}(g) \in H^1(\BZ[\frac{1}{p}],V(g)_\lambda(1))$ be the image of this element via the specialisation
map $\Lambda\mapsto \BZ_p$, $\gamma_\cyc\mapsto 1$.

Suppose now that 
\begin{equation}\label{A-split}
\text{$p$ splits in $K$}
\end{equation}
and
\begin{equation}\label{A-good}
(p,\mathfrak{f}_\psi) = 1.
\end{equation}
This means in particular that $p\nmid N$ and $g$ is ordinary with respect to $\lambda$.
One of the key results of \cite{LLZcr} identifies the image of $\bz_{\gamma'}(g)$ under 
Perrin-Riou's $p$-adic regulator in terms of cyclotomic specialisations of Katz's two-variable $p$-adic $L$-functions for $K$,
as we now recall.

Let $p = \mathfrak{p}\bar{\mathfrak{p}}$ be the factorisation of $p$ in $K$ with $\mathfrak{p}$ the prime determined by the fixed
embedding $\iota_p$. 
Let $K_\infty = K(\mathfrak{f}\bar{\mathfrak{p}}^\infty)$ be the union of the ray class extensions of $K$ 
of conductor $\mathfrak{f}\bar{\mathfrak{p}}^n$, $n>0$. Let $\hat\CO_{K_\infty}$ be the completion of the ring of integers of $K_\infty$ 
at the prime above $\mathfrak{p}$ determined by $\iota_p$. From the measure $\mu(\mathfrak{f}\bar{\mathfrak{p}}^\infty)$ 
of \cite[Thm.~II.4.14]{deShalit} (equivalently, the element $\BL_{\mathfrak{f}p^\infty}$ of \cite[Def.~2.2.1]{LLZcr}) we obtain two specialisations 
$\BL_1, \BL_2\in \Lambda\otimes_{\BZ_p}{\hat\CO_{K_\infty}}$ that are characterised by
$$
\BL_1(\chi) = \BL_{\mathfrak{f}p^\infty}(\chi\psi_\lambda^{-1}) = \phi_{\chi\psi_\lambda^{-1}}(\mu_{\mathfrak{f}\bar{\mathfrak{p}}^\infty}) \ \ \text{and} \ \ \BL_2(\chi) = \BL_{\mathfrak{f}p^\infty}(\chi(\psi_\lambda^{\tau})^{-1}) = \phi_{\chi(\psi_\lambda^{\tau})^{-1}}(\mu_{\mathfrak{f}\bar{\mathfrak{p}}^\infty})
$$
for all finite order characters $\chi:\Gamma\rightarrow \overline\BQ_p^\times$. Recall that $\phi_?$ is the continuous homomorphism of the Iwasawa algebra
defined to the linear extension of the character $\chi$.  The values $\BL_1(\chi)$ are multiples of $L(1,g\otimes\chi^{-1})$.
More precisely, $\BL_1$ is the specialisation of the two-variable Katz $p$-adic $L$-function of $K$ (and the prime $\mathfrak{p}$) along the cyclotomic line passing through $\psi_\lambda$, and so is a multiple - by a $p$-adic period - of the $p$-adic $L$-function of $g$. 
However, the values $\BL_2(\chi)$ are not interpolating critical values: $\BL_2$ is the specialisation of
the two-variable Katz $p$-adic $L$-function at the cyclotomic line passing through $\psi_\lambda^\tau$, which is outside the region of interpolation.

Let $\CL:H^1(\BQ_p,V(g)_\lambda(1)\otimes_{\BZ_p}\Lambda) \rightarrow \CH(\Gamma)\otimes_{\BQ_p} D_\cris(V(g)_\lambda(1))$
be Perrin-Riou's regulator (cf.~\cite[App.~B]{LZ0}).  By \cite[Thm.~3.2]{LLZcr}, 
the image of $\loc_p(\bz_{\gamma'}(g))$ under this regulator is
$$
\CL(\loc_p(\bz_{\gamma'}(g))) = L_{p,1} + L_{p,2},
$$
where 
\begin{itemize}
\item[$\bullet$]  $L_{p,1} = \BL_1\cdot t\otimes\gamma'$
and $L_{p,2} = \ell_0 \cdot \BL_2\otimes \tau\gamma'$ in $\Lambda\otimes_{\BZ_p} B_{\cris}\otimes_{\BQ_p} V(1)$
with $\ell_0 = \log(\gamma_\cyc)/\log\eps(\gamma_\cyc) \in \Lambda$ and $t \in B_{\cris}$ as usual;
\item under the identification $D_{\cris}(V(g)_\lambda(1)) = D_\cris(V(\psi)_\lambda(1)) \oplus D_\cris(V(\psi)_\lambda^\tau(1))$
induced from the fixed isomorphism $\widetilde{V(\psi)}_\lambda \simeq V(g)_\lambda$,
$L_{p,1} \in \Lambda\otimes_{\BZ_p} D_\cris(V(\psi)_\lambda(1))$ and $L_{p,2} \in \Lambda\otimes_{\BZ_p}D_\cris(V(\psi)_\lambda^\tau(1))$;
\item via the de Rham-crystalline comparison isomorphism, $L_{p,1} \in \Lambda\otimes \mathrm{Fil}^0D_{dR}(V(g)_\lambda(1))$.
\end{itemize}

Let $0\neq \omega\in S_{F}$ and let $\omega_\psi = s^{-1}(\omega)\in S(\psi)$. 
Let $(\Omega_\infty,\Omega_p)\in \BC^\times\times\hat\CO_{K_\infty}^\times$ be the pair of CM periods such that 
$per_\psi(\omega_\psi) = \Omega_\infty\gamma$ in $V(\psi)_\BC$ and $\omega_\psi = \Omega_p \gamma$ in $D_{dR}(V(\psi)_\lambda)$. 
Then $per(\omega) = \Omega_\infty\gamma$ in $V_\BC$, $\omega = \Omega_p\gamma'$ in $D_{dR}(V(g)_\lambda(1))$, and $\BL_1 = \Omega_p Col_{\eta_\omega}(\loc_p(\bz_{\gamma'}(g))) \in \Lambda\otimes_{\BZ_p} \hat K_\infty$. Note that $Col_{\eta_\omega}(\loc_p(\bz_{\gamma'}(g)))\in \Lambda\otimes F_\psi$ is just the usual $p$-adic $L$-function of $g$ for the period $\Omega_\infty$.

Let $\eta\in D_{dR}(V(g)_\lambda(1))$ be such that $[\eta,\omega] = 1$, where $[-,-]$ is the pairing induced via the de Rham-crystalline isomorphism
from the Weil-pairing on $V(g)_\lambda(1)$ (which is just the pairing induced by the Poincare pairing $(,)$ on $V_F$).
Let $c_\gamma = (\tau \gamma',\gamma')^{-1} \in F_\psi^\times$. Then it follows that $\eta = c_\gamma\Omega_p^{-1} \tau\gamma'$.

Suppose now that $L(1,g) = L(1,\psi) = 0$. This means that $\BL_1$ vanishes on the identity character and that
$\loc_p(z_{\gamma'}(g)) \in H^1_f(\BQ_p,V(g)_\lambda(1))$. It then follows from 
\cite[Prop.~2.2.2]{PR} and the listed properties of $\CL(\loc_p(\bz_{\gamma'}(g)))$ that 
\begin{equation}\label{A-PRformula}
(1-1/\alpha)(1-1/\beta)^{-1}  \BL_2(1) c_\gamma\Omega_p^{-1} = \log_{BK}^\eta(\loc_p(z_{\gamma'}(g))),
\end{equation}
where $\alpha = \psi(\varpi_{\bar{\mathfrak{p}}})$, with $\varpi_{\bar{\mathfrak{p}}} \in \CO_{K_{\mathfrak{p}}}$ a uniformiser,
is the unit root of $x^2 - a_p(g) x + p$ with respect to $\iota_p$ and $\beta$ is the other root, and $\log_{BK}^\eta(\loc_p(z_{\gamma'}(g)))\cdot \eta = \log_{BK}(z_{\gamma'}(g)) \in D_{dR}(V(g)_\lambda(1))$ is the Bloch--Kato logarithm. 

Suppose the root number $\epsilon(\psi)$, which is just $\epsilon(g)$, is $+1$. 
Then $$\BL_2(1)=0.$$ This is because $\BL_2(1)$ is the $p$-adic limit of the values
$\BL_{\mathfrak{f}\bar{\mathfrak{p}}^\infty}((\psi_\lambda/\psi_\lambda^\tau)^{-(p-1)j}(\psi_\lambda^\tau)^{-1})$ for integers $j>0$
tending to $0$ $p$-adically, and because $\BL_{\mathfrak{f}\bar{\mathfrak{p}}^\infty}((\psi_\lambda/\psi_\lambda^\tau)^{-(p-1)j}(\psi_\lambda^\tau)^{-1})$
is a multiple of $L(1,\psi^\tau (\psi/\psi^\tau)^{(p-1)j})$, which is readily seen to vanish as the corresponding root number is then $-1$.
Combined with \eqref{A-PRformula}, this shows that if $L(1,g) = 0$ and $\ord_{s=1}L(s,g)$ is even, then $\loc_p(z_{\gamma'}(g)) = 0$.

Suppose that $\epsilon(\psi)= \epsilon(g) = -1$. Suppose additionally that
\begin{equation}\label{A-BDP2-hyp}
\text{$D_K$ is odd and $\mathfrak{d}_K$ exactly divides $\mathfrak{f}_\psi$.}
\end{equation}
That is, we assume hypotheses (1-7), (1-8), and (1-9) of \cite{BDP2} hold for the character $\psi$. 
That (1-6) also holds is a consequence of $g$ having trivial Nebentypus (see \eqref{A-psi-res}).
Then we have the important formula \cite[Thm.~2]{BDP2}
\begin{equation}\label{A-BDP-CM}
\BL_2(1) =_{F_{\psi}^\times} \Omega_p^{-1}\log_{\omega}(P_\psi)^2,
\end{equation}
where the subscript `$F_{\psi}^\times$' denotes equality up to an $F_{\psi}^\times$-multiple, 
$A_g$ is an abelian variety in the isogeny class of $\GL_2$-type abelian varieties associated with $g$ such that $\CO_{F_\psi} \subset \End_K(A_g)$, 
and $P_\psi\in A_g(K)\otimes_{\CO_{F_\psi}} F_\psi$ is a rational point. Here we are using that there is a natural $F_\psi$-linear
isomorphism $\Omega^1(A_g/K)  \simeq S_{F_\psi}$, so $\omega$ is identified with a differential on $A_g$. 
We also have
\begin{equation}\label{A-BDP-CM2}
\text{$P_\psi$ has infinite order if and only if $\ord_{s=1}L(s,g) = 1$}.
\end{equation}
To deduce \eqref{A-BDP-CM} and \eqref{A-BDP-CM2} from {\em op.~cit.} we have used: 
\begin{itemize}
\item $\BL_2(1) = \CL_p(\psi^*)$, where the right-hand side is as in {\it op.~cit.} (this follows immediately from $\BL_2(1)$ and $\CL_p(\psi^*)$ being the $p$-adic limit of $L$-values as described above in the case $\epsilon(\psi)=+1$);
\item $A_g$ can be taken to be the abelian variety $B_\psi$ associated with $\psi$ as in \cite[Thm.~2.5]{BDP2};
\item by their definitions, both $\Omega_p$ and the $p$-adic period $\Omega_p(\psi^*)$ of \cite[Def.~2.13]{BDP2} are $F_\psi^\times$-multiples of a $p$-adic period
$\Omega_p(E)$ associated with the CM pair $(E,\alpha)$.
\end{itemize}
Combined with \eqref{A-PRformula} this shows, at least under the hypotheses \eqref{A-BDP2-hyp} that if $\epsilon(g) = -1$, then 
$\loc_p(z_{\gamma'}(g)) \neq 0$ if and only if $\ord_{s=1}L(s,g) = 1$.

In fact, we have proved
\begin{thm}\label{A-PRConj-CMthm}
Suppose that \eqref{A-psi-res}, \eqref{A-split}, \eqref{A-good} and \eqref{A-BDP2-hyp} hold. 
Let $g=g_\psi$ and suppose $L(1,g) = 0$.
Let $0\neq \omega \in \Omega^1(A_g/K)\otimes_{\CO_{F_\psi}}F_\psi$. 
Then 
there exists a rational point $P_\psi\in A_g(K)\otimes_{\CO_{F_\psi}}F_\psi$ such that
$$
\log_{BK}^\eta(\loc_p(z_{\gamma'}(g))) =_{F_\psi^\times} \Omega_p^{-1}\log_\omega(P_\psi)^2,
$$
and $P_\psi\neq 0$ if and only if $\loc_p(z_{\gamma'}(g))\neq 0$. In particular,
$$
\loc_p(z_{\gamma'}(g)) \neq 0 \ \iff \ \ord_{s=1}L(s,g) = 1.
$$
\end{thm}

\begin{remark} It should be possible to dispense with the hypotheses \eqref{A-BDP2-hyp}, as noted in the introduction to \cite{BDP2}. Much of the work needed to do this has been carried out in \cite{Brooks} and \cite{LZZ}. However, we have not pursued this here. In part because, Theorem \ref{A-PRConj-CMthm} suffices for our purposes: the proof of Theorem \ref{mPR}, which also covers the more general CM case.
\end{remark}

\section{Values of Hecke characters}
In the proof of the Perrin-Riou Conjecture given in \S\ref{PRC} it is necessary to show that a certain constant -- $u_L$ in the notation of the proof -- belongs to $\BQ$. This is achieved by appealing to Theorem \ref{A-PRConj-CMthm} and by knowing that is possible to choose suitable Hecke characters $\psi$ whose values generate disjoint fields. The result on disjoint fields is provided by this appendix.

Let $\overline\BQ$ be a separable algebraic closure of $\BQ$ and let $\iota_\infty:\overline\BQ\hookrightarrow \BC$ 
be an embedding. Let $K\subset\overline\BQ$ be an imaginary quadratic field of discriminant $-D_K$
and $\CO_K$ its ring of integers. The embedding $\iota_\infty$ identifies $K\otimes\BR$ with $\BC$.
Let $\tau\in G_\BQ$ be the involution induced via $\iota_\infty$ by complex conjugation on $\BC$ (so $\tau$ restricts to the non-trivial automorphism of $K$).

Recall that a Hecke chararacter $\psi:K^\times\bs \BA_K^\times\rightarrow \BC^\times$ has infinity type $(-1,0)$ if $\psi(z) = z^{-1}$ for all $z\in (K\otimes\BR)^\times  = \BC^\times$. For such a character, the values $\psi(\BA_K^{\infty,\times})$ generate a finite extension $F_\psi\subset \BC$ of $\BQ$. 
We consider $F_\psi$ as a subfield of $\overline\BQ$ via $\iota_\infty$.

\subsection{Values of canonical characters}  Suppose $-D_K<-4$.
Let $\mathfrak{d}_K$ be the different of $K$. The canonical characters of $K$ are the Hecke characters $\psi$ of infinity type $(-1,0)$ and
conductor $\mathfrak{d}_K$ such that $\psi(\tau(x)) = \tau(\psi(x))$ for all $x\in \BA_K^{\infty,\times}.$ Such characters were first considered by Rohrlich in \cite{Ro-CM}.
They exist precisely when
\begin{equation}\label{B-Dcong}
D_K\equiv 3 \mod 4 \ \ \text{or} \ \ 8\mid D_K.
\end{equation}
Moreover, it is clear that a canonical Hecke character $\psi$ must satisfy
\begin{equation}\label{B-centralchar}
\psi|_{\BA_\BQ^\times} = \chi_K|\cdot|^{-1},
\end{equation}
where $\chi_K$ is the quadratic character associated with the extension $K/\BQ$.
Using this, it is easy to describe all the canonical characters as follows.

Let $U = K^\times\hat\CO_K^\times \subset \BA_K^{\infty,\times}$. The quotient $\BA_K^{\infty,\times}/U$ is isomorphic to the class group of $K$ and, in particular, is finite. Via the isomorphisms $$(\BZ/D_K\BZ)^\times \isoarrow (\CO_K/\mathfrak{d}_K)^\times \isoarrow \hat\CO_K^\times/(1+\hat{\mathfrak{d}}_K)^\times,$$ we can view $\chi_K$ as a character of $\hat\CO^\times$ of conductor $\mathfrak{d}_K$. 
Let $\psi_U:U\rightarrow\BC^\times$ be the character given by $\psi_U(\alpha x) = \alpha\chi_K(x)$ for $\alpha\in K^\times$ and $x\in \hat\CO^\times$.
This is well-defined as $K^\times\cap \hat\CO_K^\times = \{\pm 1\}$ since $-D_K<-4$ by hypothesis and since $\chi_K(-1) = -1$. 
Let $\psi_f:\BA_K^{\infty,\times}\rightarrow\BC^\times$ be any character extending $\psi_U$. Then $\psi(x) = x_\infty^{-1}\psi_f(x_f)$ is a 
canonical Hecke character, and clearly all canonical characters are obtained in this way. In particular, there are $h_K$ (the class number of $K$) canonical characters and the ratio of any two is a character of the class group.

Let $\mathfrak{X}_K^\can$ be the set of canonical Hecke characters of $K$. The fields $F_\psi$, $\psi\in \mathfrak{X}_K^\can$, admit a simple description.
Since $\psi_U$ is the identity on $K^\times$, we certainly have $K\subset F_\psi$.
Let $x_1,...,x_r\in \BA_K^{\infty,\times}$ be such that their images $\bar x_1,...,\bar x_r$ in the quotient $\BA_K^{\infty,\times}/U$ 
satisfy $$\BA_K^{\infty,\times}/U = \langle \bar x_1\rangle \oplus \cdots \oplus \langle \bar x_r\rangle$$ and (for convenience later)
the order $h_i$ of each $\bar x_i$ is a power of a prime $\ell_i$.
Since $\BA_K^{\infty,\times}/U$ is isomorphic to the class group, $h_i\mid h_K$. 
Suppose $x_i^{h_i} = \alpha_i u_i$ with $\alpha_i\in K^\times$ and $u_i\in \hat\CO_K^\times$. Then $\psi(x_i)^{h_i} = \psi(\alpha_i u_i) = \psi_U(\alpha_i u_i) = \alpha_i\chi_K(u_i) = \pm \alpha_i$. Replacing $\alpha_i$ with $-\alpha_i$ and $u_i$ with $-u_i$ if necessary, we may then assume that $\psi(x_i)^{h_i} = \alpha_i$. 
That is, $\beta_i = \psi(x_i)$ is an $h_i$th root of $\alpha_i$. Let $\beta_\psi = (\beta_1,...,\beta_r)$. Then $F_\psi$ is just the radical extension
$$F_\psi=K(\beta_\psi): = K(\beta_1,...,\beta_r).$$
Furthermore, the fields $F_\psi$, $\psi$ running over all characters in $\mathfrak{X}_K^\can$, 
are exactly those fields $K(\beta)$, $\beta=(\beta_1,...,\beta_r)$ with the $\beta_i$ running over all possible $h_i$th roots of the $\alpha_i$.

Let $\mathfrak{a}_i$ be the ideal corresponding to $x_i$. Then by the choice of the $x_i$:
$\mathfrak{a_i}^{h_i} = (\alpha_i)$, $h_i$ is the order of the image $\bar{\mathfrak{a}}_i$ of $\mathfrak{a}_i$ in the 
class group $Cl(K)$ of $K$, and $Cl(K) = \langle \bar{\mathfrak{a}}_1\rangle \oplus \cdots \oplus \langle \bar {\mathfrak{a}}_r\rangle$.

\begin{prop}\label{B-canonicalvalues-prop} 
Let $K$ be an imaginary quadratic field with discriminant $-D_K<-4$ satisfying \eqref{B-Dcong} and
suppose the class number $h_K$ of $K$ is odd (so $D_K$ is a prime).
Then $\cap_{\psi\in\mathfrak{X}_K^\can} F_\psi = K$.
\end{prop}

\begin{proof} Suppose $h_K\neq 1$ (else there is nothing to prove).
Let $\mathfrak{B} = \{\beta=(\beta_1,...,\beta_r): \beta_i^{h_i} = \alpha_i\}$. Then $\cap_{\psi\in\mathfrak{X}_K^\can} F_\psi = 
\cap_{\beta\in\mathfrak{B}} K(\beta)$, so we want to show that $F=\cap_{\beta\in\mathfrak{B}} K(\beta)$ equals $K$.

The collection of fields $K(\beta)$, $\beta\in\mathfrak{B}$, is stable under $G_K$:  for $\sigma \in G_K$, 
$\sigma(K(\beta))) = K(\sigma(\beta))$, with $\sigma(\beta) = (\sigma(\beta_1),...,\sigma(\beta_r))$. 
To prove that $F=K$ it then suffices to show that $K(\beta)\neq K(\beta')$ for $\beta\neq\beta'$ in $\mathfrak{B}$:
For then it will follow that the conjugates $\sigma(K(\beta))$, $\sigma\in G_K$, over $K$ of each $K(\beta)$ are distinct. This can only happen if the intersection
of these conjugates is $K$. So $F$, which is contained in this intersection, must also equal $K$.

Suppose $\beta = (\beta_1,...,\beta_r), \beta' = (\beta_1',....,\beta_r')\in \mathfrak{B}$ are such that $\beta\neq \beta'$ but $K(\beta) = K(\beta')$. 
Since $\beta_i/\beta_i'$ is an $h_i$th root of unity, it follows that $K(\beta)$ must contain some nontrivial $h_j$th root of unity for some $j$.
Since each $h_j$ is odd and $K\neq \BQ(\mu_3)$, this non-trivial root of unity is not contained in $K$.
In particular, $K(\beta) \cap K(\mu_h) \neq K$ for $h = [h_1,...,h_r]$ the least common multiple
of the $h_i$. 

Let $W = \langle \alpha_1^{h/h_1},...,\alpha_r^{h/h_r}\rangle\subset K^\times$ and let $\overline{W}$ be the image of $W$ in $K^\times/(K^\times)^h$. 
We claim that if $\prod_i \alpha_i^{m_ih/h_i} = \alpha^h$ then $h_i\mid m_i$.
For then $\prod_i\mathfrak{a}_i^{m_ih} = \prod_i(\alpha_i^{m_ih/h_i}) = (\alpha)^h$, and so $\prod_i \mathfrak{a}_i^{m_i} = (\alpha)$.
From the description of the class group $Cl(K)$ preceding the statement of the proposition, it then follows that 
we must have $h_i\mid m_i$. In particular, $\overline{W}$ is isomorphic to the class group of $K$.

Suppose $K(\beta)\cap K(\mu_h) \neq K$. Then $[K(\beta,\mu_h):K(\mu_h)] < [K(\beta):K]\leq h_K = \#\overline{W}$.
However, $K(\beta,\mu_h)$ corresponds via Kummer theory with the image of $\overline{W}$ in $K(\mu_h)^\times/(K(\mu_h)^\times)^h$
and so this image, whose order equals $[K(\beta,\mu_h):K(\mu_h)]$, must be smaller than $\overline{W}$. Hence
$\overline{W}$ must have non-trivial intersection with the kernel of the natural map $K^\times/(K^\times)^h\rightarrow K(\mu_h)^\times/(K(\mu_h)^\times)^h$.
Let $w\in W$ be an element with image $\bar w$ in $\overline{W}$ in this kernel and having order a prime, say $\ell$ (necessarily odd).   
Let $I_\ell = \{ 1\leq i\leq r \ ; \ \ell_i =\ell\}$. Then 
$w = \prod_{i\in I_\ell} \alpha_i^{m_ih/h_i}$ with each $h_i/m_i = \ell$ or $1$. Let $I_w\subset I_\ell$ be the subset of those
$i$ such that $h_i/m_i = \ell$. Then $w = \prod_{i\in I_w}\alpha_i^{h/\ell}$. Let $u = \prod_{i\in I_w} \alpha_i$.
Then $\omega = \prod_{i\in I_w}\beta_i^{h_i/\ell}$ is an $h$th-root of $w$ and an $\ell$th-root of $u$.

Since $\overline{w}$ is in the kernel of $K^\times/(K^\times)^h\rightarrow K(\mu_h)^\times/(K(\mu_h)^\times)^h$,
$K(\omega) \subset K(\mu_h)$. In particular, $K(\omega)$ equals each of its $G_K$-conjugates. 
As $u\not\in (K^\times)^\ell$ (else $w = u^{h/\ell} \in (K^\times)^h)$ it follows that $K(\omega)$ contains
a primitive $\ell$th root of unity. As $\ell$ is odd and $K\neq \BQ(\mu_3)$, this means $K(\omega)\cap K(\mu_\ell) \neq K$ and so, by Kummer theory again, $u\in (K(\mu_\ell)^\times)^\ell$. It follows that
$w= u^{h/\ell} \in (K(\mu_\ell)^\times)^h$.  That is, $\prod_{i\in I_w} \alpha_i^{h/\ell} = \gamma^h$ for 
some $\gamma\in K(\mu_\ell)$.  But in term of ideals this becomes 
$\prod_{i\in I_w} \mathfrak{a}_i^{hh_i/\ell}\CO_{K(\mu_\ell)} = (\gamma)^h$, so 
$\prod_{i\in I_w} \mathfrak{a}_i^{h_i/\ell}\CO_{K(\mu_\ell)} = (\gamma)$. This last equality means that the ideal
$\mathfrak{a} = \prod_{i=I_w} \mathfrak{a}_i^{h_i/\ell}$ capitulates in the extension $K(\mu_\ell)$,
which in turn implies $\mathfrak{a}^{[K(\mu_\ell):K]}$ is principal.
As $[K(\mu_\ell):K]\mid (\ell-1)$ but $\mathfrak{a}$ has order $\ell$ in the class group of $K$, this is impossible.
This contradiction completes the proof of the proposition.
\end{proof}

\subsection{An auxiliary result}\label{B-auxiliary}
Suppose $p>2$ is a fixed prime and $L$ is a fixed imaginary quadratic field of discriminant $-D_L<0$
Consider the set $\mathfrak{X}$ of Hecke characters $\psi: K^\times\bs \BA_K^\times\rightarrow \BC^\times$ of 
varying imaginary quadratic fields such that 
\begin{itemize}
\item[(i)] the prime $p$ splits in $K$;
\item[(ii)] the discriminant $-D_K$ of $K$ is odd;
\item[(iii)] the infinity type of $\psi$ is $(-1,0)$;
\item[(iv)] $\psi|_{\BA^\times} = \chi_K|\cdot|^{-1}$;
\item[(v)] $\mathfrak{f}_\psi = \mathfrak{d}_K\mathfrak{f}$ with $(\mathfrak{d_K},\mathfrak{f}) = 1$; 
\item[(vi)] the root number $w(\psi)$ of $\psi$ is $-1$;
\item[(vii)] every prime dividing $D_KN_{K/\BQ}(\mathfrak{f}_\psi)$ splits in $L$;
\item[(viii)] $\ord_{s=1}L(s,\psi) = 1$.
\end{itemize}
Note that conditions (i)-(vi) imply that Theorem \ref{A-PRConj-CMthm} holds for $\psi$

\begin{lem} \label{B-aux-lem} If $L\neq \BQ(\sqrt{-1}), \BQ(\sqrt{-2})$, then $\cap_{\psi\in\mathfrak{X}} F_\psi = \BQ$.
\end{lem}

\begin{proof} We will define a fairly explicit subset $\mathfrak{X}' \subset \mathfrak{X}$ such that 
$\cap_{\psi\in\mathfrak{X}'} F_\psi = \BQ$.

Let $\mathfrak{Q}$ be the set of primes $q$ such that 
\begin{itemize}
\item $q\equiv 3 \mod 8$;
\item $-q$ is a square modulo $p$;
\item $-D_L$ is a square modulo $q$.
\end{itemize}
The second and third conditions are equivalent to $q$ belonging to an index two subgroup of $(\BZ/p\BZ)^\times$
and $(\BZ/D_L\BZ)^\times$, respectively. In the case of the third condition, this subgroup --
the kernel of the character $\chi_L$ -- is not induced from any index two subgroup of any
$(\BZ/D\BZ)^\times$ with $D$ a proper divisor of $D_L$.
Since $p$ is odd and $D_L$ has a prime factor not equal to $2$ (this is where we use
the hypothesis that $L\neq \BQ(\sqrt{-1}), \BQ(\sqrt{-2})$), it is then easy to see that these conditions are satisfied 
by infinitely many primes. That is, $\mathfrak{Q}$ is an infinite set.

Let $\mathfrak{K}$ be the set of imaginary quadratic fields $K= \BQ(\sqrt{-q})$ for
$q\in\mathfrak{Q}$. This is then an infinite set of imaginary quadratic fields. The first condition on $q$ implies
that the discriminant of $K$ is $-D_K=-q$, which means that the class number of $K$ is odd. The second
condition on $q$ is then just the condition that $p$ splits in $K$, while the third condition is that every
prime dividing $D_K$ (which is just $q$!) splits in $L$.

For $K\in \mathfrak{K}$ let $\mathfrak{X}_K^\can$ be the set of canonical Hecke characters of $K$ of infinity type
$(-1,0)$. The condition that $-D_K = -q \equiv 3 \mod 8$ means that the root number of each 
$\psi\in \mathfrak{X}_K^\can$ is $-1$.
If $\ell$ is a prime that splits in both $K$ and $L$ (which is an infinite set), then 
for every anticyclotomic Hecke characters $\chi$ of $K$ of finite $\ell$-power
order and $\ell$-power conductor the root number of $\psi\chi$ is also $-1$. 
Moreover, by a theorem 
of Rohrlich \cite{Ro} for all but finitely many such $\chi$,
$\ord_{s=1}L(s,\psi\chi) = 1$ for all $\psi\in \mathfrak{X}_K^\can$;
let $\mathfrak{X}(K,\ell)$ be the set of such $\chi$ for which this holds. 

Let $\mathfrak{S}(K)$ be the set of primes that split in both $K$ and $L$.
We then have 
$$
\mathfrak{X}': = \{\psi\chi \ : \ \psi\in \mathfrak{X}_K^\can, \chi\in\mathfrak{X}(K,\ell), K\in \mathfrak{K}, \ell\in\mathfrak{S}(K)\}\subset \mathfrak{X}.
$$
We will next show that $\cap_{\psi\in\mathfrak{X}'} F_\psi = \BQ$, which implies the lemma.

If $\chi\in \mathfrak{X}(K,\ell)$ then $F_{\psi\chi} \subset F_\psi(\mu_{\ell^\infty})$. 
As $\mathfrak{S}(K)$ is infinite it follows that
$$
\cap_{\ell\in\mathfrak{S}(K)} \cap_{\chi\in\mathfrak{X}(K,\ell)} F_{\psi\chi}  \subset \cap_{\ell\in\mathfrak{S}(K)} F_\psi(\mu_{\ell^\infty}) \subset F_\psi.
$$
So 
$$
\cap_{\psi\in \mathfrak{X}_K^\can}
\cap_{\ell\in\mathfrak{S}(K)} \cap_{\chi\in\mathfrak{X}(K,\ell)} F_{\psi\chi}  \subset \cap_{\psi\in \mathfrak{X}_K^\can} F_\psi = K,
$$
the final equality being by Proposition \ref{B-canonicalvalues-prop}.
Hence
$$
\cap_{K\in \mathfrak{K}} \cap_{\psi\in \mathfrak{X}_K^\can}
\cap_{\ell\in\mathfrak{S}(K)} \cap_{\chi\in\mathfrak{X}(K,\ell)} F_{\psi\chi}  \subset \cap_{K\in \mathfrak{K}} K = \BQ.
$$
\end{proof}

}

\part{Iwasawa main conjectures and applications}

\setcounter{section}{8}

\section{Main conjectures: backdrop} 

This section presents main conjectures underlying Iwasawa theory of a weight two elliptic newform over the rationals and imaginary quadratic fields, and interrelations among them. We also recall the prior results towards the main conjectures.

\subsection{Selmer groups} \label{sSel} The subsection introduces Selmer groups with varying local conditions (cf.~\cite[\S3]{Sk}). 

\subsubsection{Over the rationals}
Let $g \in S_{2}(\Gamma_{0}(N))$ be an elliptic newform, $F$ the Hecke field and $\cO$ the integer ring.
 Let $A=A_{g}$ be an associated $\GL_2$-type abelian variety\footnote{In fact there is an abelian variety $A_{0}$ over $\BQ$ with $\BZ_{g}\hookrightarrow \End(A_{0})$ such that $A_{g}=A_{0}\otimes_{\BZ_{g}}\cO$, where $\BZ_{g} \subset \cO$ is the Hecke order.}
 over $\BQ$ with $\cO\hookrightarrow \End(A_{g})$ as in \S\ref{newforms-AV}. 
 Let $p \nmid 2N$ be a prime and $\Sigma$ the set of primes dividing $Np$. 
For a prime $\lambda$  of the Hecke field $F$ above $p$, 
we may choose $A_{g}$ and $$\pi: X_{0}(N)\ra A_{g}$$ to be an $(\cO,\lambda)$-optimal paramterisation in the sense of \cite[\S3.7]{Zh}.
Let $T$ be a lattice in the attached $p$-adic Galois representation $V=V_{F_{\lambda}}$.
Often, we consider the lattice arising from the $\lambda$-adic Tate module of $A=A_{g}$.
Define   
$$
W=V/T, 
$$
which is a discrete $\cO_{\lambda}$-divisible $G_{\BQ}$-module. 
In the Tate lattice case it is isomorphic to the $p$-divisible group 
$A[\lambda^{\infty}]$.
Let $\Sel(g)$ be the Bloch--Kato Selmer group associated to $W(1)$, $\Sel_{\lambda^{\infty}}(A)$ the Selmer group associated to $A[\lambda^{\infty}]$ and $\Sha(A)[\lambda^{\infty}]$ the Tate--Shafarevich group. 

Let 
$$
M=T(1) \otimes_{\BZ_{p}} \Lambda^{\vee}
$$
be a discrete $\Lambda_{\cO_{\lambda}}$-module with the $G_{\BQ}$-action on $\Lambda^{\vee}$ via the inverse of the canonical character $\Psi$ (cf.~\eqref{cycuniv}).
\vskip2mm
{\it{Ordinary and signed Selmer groups}.} 
For an ordinary prime $p\nmid N$, there is an $\cO_{\lambda}[G_{\BQ_{p}}]$-filtration 
$$0 \subset T^{+} \subset T$$
 with $\rank_{\cO_{\lambda}}T^{+}=1.$
Define
\begin{equation}\label{ordQ}
S(g) = \ker \big{\{} H^{1}(G_{\Sigma},M) \ra \prod_{v \in \Sigma, v \nmid p} H^{1}(\BQ_{v},M) 
\times H^{1}(I_{p},T/T^{+} \otimes_{\BZ_{p}} \Lambda^{\vee})  \big{\}},
\end{equation}
a discrete $\Lambda_{\cO_{\lambda}}$-module. 
Its Pontryagin dual 
$$
X(g)=\Hom_{\cont}(S(g),\BQ_{p}/\BZ_{p})
$$
is a finitely generated compact $\Lambda_{\cO_{\lambda}}$-module.

\begin{remark}\label{rpin}
In the above definition $\prod_{v \in \Sigma, v \nmid p} H^{1}(\BQ_{v},M)$ may be replaced with 
$\prod_{v \in \Sigma, v \nmid p} H^{1}(I_{v},M)$.
\end{remark}

\vskip2mm
We now consider the supersingular case. Let $p\nmid 2N$ be a non-ordinary prime so that 
$a_{g}(p)=0$, which we refer to as a supersingular prime. 
For $\circ\in \{+,-\}$, 
let $$H^{1}_{\circ}(\BQ_{p},M) \subset H^{1}(\BQ_{p},M)$$ be the annihilator of Kobayashi's signed submodule 
$H^{1}_{\circ}(\BQ_{p}, T(1) \otimes_{\BZ_{p}} \Lambda)\subset H^{1}(\BQ_{p},T(1)\otimes_{\BZ_{p}}\Lambda)$ under the Pontryagin duality pairing (cf.~\S\ref{IwCoh}). 
Define
\begin{equation}\label{ssQ}
S_{\circ}(g) = \ker \big{\{} H^{1}(G_{\Sigma},M) \ra \prod_{v \in \Sigma, v \nmid p} H^{1}(\BQ_{v},M) 
\times \frac{H^{1}(\BQ_{p},M)}{H^{1}_{\circ}(\BQ_{p},M)}  \big{\}}
\end{equation} and let $X_{\circ}(g)$ denote its Pontryagin dual.

\vskip2mm
{\it{Strict Selmer groups}.} 
For a newform $g\in S_{2}(\Gamma_{0}(N))$ and a prime $p$, 
define
\begin{equation}\label{strQ}
S_{\st}(g) = \ker \big{\{} H^{1}(G_{\Sigma},M) \ra \prod_{v \in \Sigma, v \nmid p} H^{1}(\BQ_{v},M) 
\times H^{1}(\BQ_{p},M)  \big{\}}
\end{equation} 
and let $X_{\st}(g)$  denote its Pontryagin dual.

\subsubsection{Over imaginary quadratic fields}\label{ss:Sel-L}
Let $L$ be an imaginary quadratic field satisfying the conditions \eqref{ord} and \eqref{coprime},  
and so $(p)=v\ov{v}$ with $v$ determined via the embedding $\iota_{p}:\ov{\BQ}\hookrightarrow \ov{\BQ}_{p}$.

Let $\Sel(g_{/L})$ denote the Bloch--Kato Selmer group associated to $W(1)$ when viewed as a $G_{L}$-module and 
$\Sel_{\lambda^{\infty}}(A_{/L})$ the  
$\lambda^{\infty}$-Selmer group over $L$. 
For $\cdot \in\{ \emptyset, \cyc, \ac\}$, 
let 
$$
\mathcal{M}^{\cdot}=T(1) \otimes_{\BZ_{p}} \Lambda_{L}^{\cdot,\vee} 
$$
be a discrete $\Lambda_{\cO_{L}}^\cdot$-module 
with the $G_{L}$-action on $\Lambda_{L}^{\cdot,\vee}$ via the inverse of the canonical character $\Psi_{L}^\cdot$ (cf.~\S\ref{MoreIwasawa}). 
\vskip2mm
{\it{Ordinary and signed Selmer groups}.} For $p\nmid 2N$ an ordinary prime, define 
a discrete $\Lambda_{L,\cO_{\lambda}}^{\cdot}$-module
\begin{equation}\label{ordL}
S^{\cdot}(g_{/L}) = \ker \big{\{} H^{1}(G_{L,\Sigma}, \mathcal{M}^{\cdot}) \ra \prod_{v \in \Sigma, v \nmid p} H^{1}(L_{v},\mathcal{M}^{\cdot}) 
\times \prod_{w|p}H^{1}(I_{w}, T(1)/T^{+}(1) \otimes_{\BZ_{p}} \Lambda_{L}^{\cdot,\vee})  \big{\}},
\end{equation}
where we let $\Sigma=\Sigma_{L}$ also denote the set of places of $L$ over $\Sigma$.
Let 
$
X^{\cdot}(g_{/L})
$
be the Pontryagin dual.

For $p\nmid 2N$ a supersingular prime, define 
\begin{equation}\label{ssL}
S_{\circ}^{\cdot}(g_{/L}) = \ker \big{\{} H^{1}(G_{L,\Sigma}, \mathcal{M}^{\cdot}) \ra \prod_{v \in \Sigma, v \nmid p} H^{1}(L_{v},\mathcal{M}^{\cdot}) 
\times \prod_{w|p}\frac{H^{1}(L_{w},\CM^{\cdot})}{H^{1}_{\circ}(L_{w},\CM^{\cdot})}  \big{\}}
\end{equation}
and let $X_{\circ}^{\cdot}(g_{/L})$ denote its Pontryagin dual.

\vskip2mm
{\it{Strict Selmer groups}.} 
For a newform $g \in S_{2}(\Gamma_{0}(N))$ and an ordinary or a supersingular prime $p\nmid 2N$, let 
$$
S_{\st,\bullet}^\cdot(g_{/L}) \subset S_{\bullet}^\cdot(g_{/L}), \ \ \bullet = 
\begin{cases} \ord  & \text{$g$ is ordinary} \\ \circ & \text{$g$ is supersingular}
\end{cases}
$$
be the submodule consisting of $\kappa \in S^{\cdot}_{\bullet}(g_{/L})$ such that
$$
\loc_{v}\kappa = 0.
$$
Let $X^{\cdot}_{\st,\bullet}(g_{/L})$ denote its Pontryagin dual.

\vskip2mm
{\it{Greenberg Selmer groups}.}  Define 
\begin{equation}\label{GrL}
\Sel_{\rm{Gr}}(g_{/L}) = \ker \big{\{} H^{1}(G_{L,\Sigma}, W(1)) \ra \prod_{w \in \Sigma, v \nmid p} H^{1}(L_{w},W(1)) 
\times H^{1}(L_{v},W(1))  
\times \frac{H^{1}(L_{\ov{v}},W(1))}{H^{1}(L_{\ov{v}},W(1))_{\div}}
\big{\}}
\end{equation}
and 
\begin{equation}\label{GrL}
S_{\rm{Gr}}^\cdot(g_{/L}) = \ker \big{\{} H^{1}(G_{L,\Sigma}, \mathcal{M}^\cdot) \ra \prod_{w \in \Sigma, v \nmid p} H^{1}(L_{w},\mathcal{M}^\cdot) 
\times H^{1}(L_{v},\mathcal{M}^\cdot)  \big{\}}, 
\end{equation}
the latter a discrete $\Lambda_{L,\cO_{\lambda}}^{\cdot}$-module.  
Let 
$X_{\rm{Gr}}^\cdot(g_{/L})$ denote the Pontryagin dual. 

\begin{remark}\label{rpin'}
\noindent
\begin{itemize}
\item[(i)] In the cyclotomic definition $\prod_{w \in \Sigma, v \nmid p} H^{1}(L_{w},\mathcal{M}^{\rm cyc})$ may be replaced with 
$\prod_{w \in \Sigma, v \nmid p} H^{1}(I_{w},\mathcal{M}^{\rm cyc})$.
\item[(ii)] The Greenberg Selmer groups arise from interpolation of Bloch--Kato Selmer groups  associated to twists of $T(1)$ by certain infinite order Hecke characters over $L$.
\end{itemize}
\end{remark}

\subsection{The main conjectures}\label{sIMC} 
\subsubsection{Over the rationals}\label{ssIMC1} 
\noindent\
\vskip2mm

{\it{Cyclotomic main conjectures with $p$-adic $L$-functions}.} 
\begin{conj}\label{KatopL} 
Let $g \in S_{2}(\Gamma_{0}(N))$ be an elliptic newform, $p\nmid N$ a prime of ordinary reduction and $\alpha$ the $p$-unit root of the Hecke polynomial at $p$.
Let $0\neq \omega \in S_F$, $\gamma \in V_{F}$ with $\gamma^{\pm} \neq 0$ and $\mathcal{L}_{\alpha,\omega,\gamma}(g)$ be the associated $p$-adic $L$-function.
\begin{itemize}
\item[(a)]$X(g)$ is $\Lambda_{\cO_{\lambda}}$-torsion.
\item[(b)] For $\omega$ good (cf.~\S\ref{ModularForms}) and $\gamma=\gamma_{g}$ as in Lemma \ref{GorPer}, we have  an equality of ideals 
$$
(\mathcal{L}_{\alpha,\omega,\gamma}(g))
=\xi(X(g)),
$$
 in $\Lambda_{\cO_{\lambda}} \otimes_{\BZ_{p}} \BQ_{p}$ and even in $\Lambda_{\cO_{\lambda}}$ if  (irr$_\BQ$) holds for $p \neq 2$.
\end{itemize}
\end{conj}

 In the supersingular case
Kobayashi \cite[\S1]{Ko} proposed the following:

\begin{conj}\label{KatopLss} 
Let $g \in S_{2}(\Gamma_{0}(N))$ be an elliptic newform, $p\nmid 2N$ a prime of supersingular reduction and $\circ\in\{+,-\}$. 
\begin{itemize}
\item[(a)]$X_{\circ}(g)$ is $\Lambda_{\cO_{\lambda}}$-torsion. 
\item[(b)]  
For $\gamma=\gamma_{g}$ as in Lemma \ref{GorPer}, we have an equality of ideals
$$
(\mathcal{L}_{\gamma}^{\circ}(g))
=\xi(X_{\circ}(g))
$$
in $\Lambda_{\cO_{\lambda}}$.
\end{itemize}
\end{conj}
\begin{remark}
Kobayashi's original formulation concerns supersingular elliptic curves. The above generalisation to weight two newforms is due to Lei \cite[Conj.~1.2]{L}.
\end{remark}

For the choice of $\gamma$ as in Conjecture \ref{KatopLss}(b), we let $(\CL_{p}^{\circ}(g))$ denote $(\CL_{\gamma}^{\circ}(g))$.

\vskip2mm

{\it{Kato's main conjecture}.} The following 
links Kato's zeta element with the strict Selmer group 
(cf.~\cite[Conj.~12.10]{K}). 
\begin{conj}\label{Kato} 
Let $g \in S_{2}(\Gamma_{0}(N))$ be an elliptic newform and $p$ a prime. 
\begin{itemize}
\item[(a)] $X_{\st}(g)$ is $\Lambda_{\cO_{\lambda}}$-torsion. 
\item[(b)] For 
$\gamma=\gamma_{g}$ as in Lemma \ref{GorPer}, we have an equality of ideals
$$
\xi\big{(}H^{1}(\BZ[\frac{1}{p}], T \otimes_{\BZ_{p}} \Lambda)/ \Lambda_{\cO_{\lambda}}\cdot{{\bf{z}}_{\gamma}(g)}\big{)}
=\xi(X_{\st}(g))
$$
in $\Lambda_{\cO_{\lambda}}\otimes_{\BZ_{p}} \BQ_{p}$ and even in $\Lambda_{\cO_{\lambda}}$ if (irr$_\BQ$) holds for $p \neq 2$.
\end{itemize}
\end{conj}

\begin{remark}\label{KaPR}
\noindent
\begin{itemize}
\item[(i)] The conjecture is uniformly formulated for all primes $p$. 
\item[(ii)] The lower bound for $X_{\st}(g)$ as predicted by Conjecture \ref{Kato}(b) has an application to Conjecture \ref{PR}: In combination with \cite{Ko'} (or \S\ref{Shaf}), the lower bound leads to the implication 
$$
\ord_{s=1}L(s,g)=1 \implies 
\loc_{p}(z_{\rm{Kato}}(g)) \neq 0
$$
for any prime $p$. 
\end{itemize}
\end{remark}

\subsubsection{Over imaginary quadratic fields}\label{ssIMC2} 
Let $(g,p,L)$ be as above 
and $g'=g \otimes \chi_L$ denote the quadratic twist.

For an ordinary or a supersingular prime $p\nmid N$ and $\circ\in\{+,-\}$, put
\begin{equation}\label{n-red}
\Box = 
\begin{cases} \emptyset  & \text{$g$ is ordinary} \\ \circ & \text{$g$ is supersingular},
\end{cases}
\ \bullet=
\begin{cases} \ord  & \text{$g$ is ordinary} \\ \circ & \text{$g$ is supersingular.}
\end{cases}
\end{equation}

\vskip2mm

{\it{Cyclotomic main conjectures with standard $p$-adic $L$-functions}.} 
\begin{conj}\label{cycBC} 
Let $g \in S_{2}(\Gamma_{0}(N))$ be an elliptic newform, $p\nmid N$ a prime of ordinary reduction and $\alpha$ the $p$-unit root of the Hecke polynomial at $p$. 
Let $0\neq \omega \in S_F$, $\gamma \in V_{F}$, $\gamma' \in V'_{F}$ and $\CL_{\alpha,\omega,\gamma,\gamma'}(g_{/L})$ be the associated $p$-adic $L$-function. 
\begin{itemize}
\item[(a)] $X^{\rm cyc}(g_{/L})$ is $\Lambda_{L,\cO_{\lambda}}^{\cyc}$-torsion.  
\item[(b)] For $\omega$ good (cf.~\S\ref{ModularForms}), $\gamma=\gamma_{g}$ and $\gamma'=\gamma_{g'}$ as in Lemma \ref{GorPer}, we have an equality of ideals 
$$(\mathcal{L}_{\alpha,\omega,\gamma,\gamma'}(g_{/L}))=\xi(X(g_{/L}))$$
in $\Lambda_{L,\cO_{\lambda}}^{\cyc}\otimes_{\BZ_{p}} \BQ_{p}$ and even in 
$\Lambda_{L,\cO_{\lambda}}^{\cyc}$ if (irr$_L$) holds for $p \neq 2$.
\end{itemize}
\end{conj}

\begin{conj}\label{cycBCss} 
Let $g \in S_{2}(\Gamma_{0}(N))$ be an elliptic newform, $p\nmid 2N$ a prime of supersingular reduction and 
$\circ\in\{+,-\}$. 

\begin{itemize}
\item[(a)] $X_{\circ}(g_{/L})$ is $\Lambda_{L,\cO_{\lambda}}^{\cyc}$-torsion. 
\item[(b)] For $\gamma=\gamma_{g}$ and $\gamma'=\gamma_{g'}$ as in Lemma \ref{GorPer}, we have an equality of ideals 
$$(\mathcal{L}_{\gamma,\gamma'}^{\circ}(g_{/L}))=\xi(X_{\circ}(g_{/L}))$$
in  
$\Lambda_{L,\cO_{\lambda}}^{\cyc}$.
\end{itemize}
\end{conj}

{\it{Main conjectures with standard $p$-adic $L$-functions, bis}.} 
\begin{conj}\label{St} Let $g \in S_{2}(\Gamma_{0}(N))$ be an elliptic newform, $p\nmid 2N$ 
an ordinary or a supersingular prime, and $\cdot \in \{\emptyset, \cyc, \ac\}$. In the case $\cdot=\ac$ suppose that the root number of $E$ over $L$ equals $+1$. 
\begin{itemize}
\item[(a)] $X^{\cdot}_{\bullet}(g_{/L})$ is $\Lambda_{L,\cO_{\lambda}}^{\cdot}$-torsion. 
\item[(b)] We have an equality of ideals 
$$(\mathcal{L}_{p}^{\Box,\cdot}(g_{/L}))=\xi(X_{\bullet}^{\cdot}(g_{/L})),$$ 
 in $\Lambda_{L,\cO_{\lambda}}^{\cdot}\otimes_{\BZ_{p}} \BQ_{p}$ and even in $\Lambda_{L}^{\cdot}$ if $({\rm irr}_{L})$ holds.
\end{itemize}
\end{conj} 
\begin{remark} 
\noindent
\begin{itemize}
\item[(i)] Note that the cyclotomic case of the above conjecture is nothing but Conjecture \ref{cycBC} (resp.~Conjecture \ref{cycBCss}) in the ordinary (resp.~supersingular) case. 
\item[(ii)] A related two-variable main conjecture for supersingular elliptic curves has been proposed by Kim \cite{Ki}. Its anticyclotomic counterpart goes back to Darmon--Iovita \cite{DI}.
\end{itemize}
\end{remark}
{\it{Greenberg main conjecture}.} 
\begin{conj}\label{Greenberg} Let $g \in S_{2}(\Gamma_{0}(N))$ be an elliptic newform, $p\nmid N$ a prime and $\cdot \in \{\emptyset, \cyc, \ac\}$. In the case $\cdot=\ac$ suppose that the root number of $E$ over $L$ equals $-1$. 
\begin{itemize}
\item[(a)] $X_{\rm{Gr}}^{\cdot}(g_{/L})$ is $\Lambda_{L,\cO_{\lambda}}^{\cdot}$-torsion. 
\item[(b)] We have an equality of ideals  
$$(\mathcal{L}_{p}^{\rm{Gr},\cdot}(g_{/L}))=\xi(X_{\rm{Gr}}^{\cdot}(g_{/L})),$$ 
in $\Lambda_{L,\cO_{\lambda}}^{\ur,\cdot}\otimes_{\BZ_{p}} \BQ_{p}$ and even in $\Lambda_{L}^{\ur,\cdot}$ for $p \neq 2$.
\end{itemize}
\end{conj} 
\begin{remark} 
\noindent
\begin{itemize}
\item[(i)] The conjecture is uniformly formulated for all primes $p$ of good reduction. 
It is also independent of the choice of lattice (cf.~\cite[Prop.~2.9]{KO}). 
\item[(ii)] 
The cyclotomic projection $\mathcal{L}_{p}^{\rm{Gr},\cyc}(g_{/L})$ does not interpolate classical $L$-values. Yet the cyclotomic Greenberg main conjecture is elemental in our approach to the $p$-part of the BSD formula  (cf.~\S\ref{pBSD}) and $p$-converse to the Gross--Zagier and Kolyvagin theorem (cf.~\S\ref{s:pcv}). 
\item[(iii)] The non-vanishing of 
$\mathcal{L}_{p}^{\rm{Gr},\cyc}(g_{/L})$
 is an open problem. 
It can be verified in certain  rank one situations (cf.~\eqref{locnv}).
\end{itemize}
\end{remark}
\vskip2mm 
{\it{Zeta element main conjecture}.} 

\begin{conj}\label{LLZ} 
Let $g \in S_{2}(\Gamma_{0}(N))$ be an elliptic newform and $p \nmid 2N$ an ordinary or a supersingular prime. 
For  $\bullet, \Box$ as in \eqref{n-red} and 
$\cdot \in \{\emptyset, \cyc, \ac\}$, 
let $$\CZ^{\Box,\cdot}(g_{/L}) \in H^{1}_{\rm{rel},\bullet}(L, T \otimes \Lambda_{L}^{\cdot})$$ be the 
corresponding projection of the two-variable zeta element $\CZ^{\Box}(g_{/L})$.

\begin{itemize}
\item[(a)] $X_{\st,\bullet}^{\cdot}(g_{/L})$ is $\Lambda_{L,\cO_{\lambda}}^{\cdot}$-torsion. 
\item[(b)] We have an equality of ideals 
$$
\xi\big{(}H^{1}_{\rm{rel},\bullet}(L, T \otimes \Lambda_{L}^{\cdot})/ \Lambda_{L,\cO_{\lambda}}^{\cdot}
\cdot{\CZ^{\Box,\cdot}(g_{/L})}\big{)}
=\xi(X_{\st,\bullet}^{\cdot}(g_{/L})), 
$$
in $\Lambda_{L,\cO_{\lambda}}^{\cdot}\otimes_{\BZ_{p}} \BQ_{p}$ and even in 
$\Lambda_{L,\cO_{\lambda}}^{\cdot}$ if $({\rm van}_{L})$ holds.
\end{itemize}
\end{conj}
\vskip2mm
{\it{Heegner main conjecture}}.
Let $g \in S_{2}(\Gamma_{0}(N))$ be an elliptic newform and $p$ an odd prime of good ordinary reduction. Suppose that $({\mathrm{van}}_{\BQ})$ holds. 
Let $L$ be an imaginary quadratic field satisfying the conditions \eqref{Heeg} and $({\mathrm{van}}_{L})$, and so that $p\nmid D_L$.

In view of the Heegner hypothesis \eqref{Heeg} and the ordinarity assumption, the Kummer image of CM points on $X_{0}(N)$ over ring class fields of $p$-power conductor of $L$ under an $(\cO,\lambda)$-optimal modular parametrisation $\pi:X_{0}(N)\ra A_g$ give rise to an Iwasawa-theoretic Heegner class $$\kappa_g \in H^{1}_{\ord}(L,T_{g}\otimes_{\BZ_{p}}\Lambda_{L,\cO_{\lambda}}^{\ac})$$ 
(cf.~\cite[(10)]{BCK}). The Iwasawa theory of varying Heegner points is encaptured by the following conjecture of Perrin-Riou \cite{PR0}.

\begin{conj}\label{HMC}
Let $g\in S_{2}(\Gamma_{0}(N))$ be an elliptic newform and $p$ an odd prime of good ordinary reduction. 
Let $L$ be an imaginary quadratic field satisfying \eqref{Heeg} and ~$({\mathrm{van}}_{L})$, and so that $p\nmid D_L$.
\begin{itemize}
\item[(a)] The Heegner class $\kappa_{g} \in H^{1}_{\ord}(L,T_{g}\otimes_{\BZ_{p}}\Lambda_{L,\cO_{\lambda}}^{\ac})$ is $\Lambda_{L,\cO_{\lambda}}^{\ac}$-non-torsion. Moreover, 
$$\rank_{\Lambda_{L,\cO_{\lambda}}^{\ac}}H^{1}_{\ord}(L,T_{g}\otimes_{\BZ_{p}}\Lambda_{L,\cO_{\lambda}}^{\ac})=\rank_{\Lambda_{L,\cO_{\lambda}}}^{\ac}X^{\ac}(g_{/L})=1.$$
\item[(b)] We have
$$\xi_{\Lambda_{L,\cO_{\lambda}}^{\ac}}\bigg{(}
\frac{H^{1}_{\ord}(L,T_{g}\otimes_{\BZ_{p}}\Lambda_{L,\cO_{\lambda}}^{\ac}{)}}{\Lambda_{L,\cO_{\lambda}}^{\ac}\cdot \kappa_{g}}\bigg{)}\cdot \xi_{\Lambda_{L,\cO_{\lambda}}^{\ac}}\bigg{(}\frac
{H^{1}_{\ord}(L,T_{g}\otimes_{\BZ_{p}}\Lambda_{L,\cO_{\lambda}}^{\ac}}{\Lambda_{L,\cO_{\lambda}}^{\ac}\cdot\kappa_{g}^{\iota}}\bigg{)}=\xi_{\Lambda_{L,\cO_{\lambda}}^{\ac}}(X(g)_{\tor}).$$
\end{itemize}
\end{conj}

\subsection{Main conjectures and zeta elements}\label{seqv} We describe connections between some of the main conjectures in section \ref{sIMC} based on zeta elements. 

\subsubsection{Via Kato's zeta element} 
\noindent

In light of Theorem \ref{ERLBKI} we have the following (cf.~\cite[Thm. 7.4]{Ko}).
\begin{lem}\label{KatoEq}
Let $g \in S_{2}(\Gamma_{0}(N))$ be an elliptic newform. 
\begin{itemize}
\item[(i)] For a prime $p\nmid N$ of ordinary reduction, a one-sided divisibility in Conjecture \ref{KatopL} implies the analogous divisibility in Conjecture
\ref{Kato} and conversely. In particular, Conjectures \ref{KatopL} amd \ref{Kato} are equivalent. 
\item[(ii)] For a prime $p\nmid 2N$ of supersingular reduction, the assertions as in part (i) hold for 
Conjectures \ref{KatopLss} and 
 \ref{Kato}.
\end{itemize}
\end{lem}

A relation between cyclotomic main conjectures over the rationals and imaginary quadratic fields: 
\begin{lem}\label{Eq'}
Let $g \in S_{2}(\Gamma_{0}(N))$ be an elliptic newform and $p$ a prime. 
Let $L$ be an imaginary quadratic field satisfying \eqref{ord} and 
\begin{equation}\label{coprime}
(D_{L},N)=1.
\end{equation}
\begin{itemize}
\item[(i)] For a prime $p\nmid N$ of ordinary reduction, a one-sided divisibility in 
Conjecture \ref{KatopL} for $g$ and the quadratic twist $g'=g\otimes \chi_L$ imply the analogous divisibility in Conjecture \ref{cycBC} for $g$. In particular, Conjecture \ref{KatopL} for $g$ and $g'$ imply Conjecture \ref{cycBC} for $g$.
\item[(ii)] For a prime $p\nmid 2N$ of supersingular reduction, the assertions as in part (i) hold for 
Conjectures \ref{KatopLss} and \ref{cycBCss}.
\end{itemize}
\end{lem}
\begin{proof} We only consider the ordinary case, an analogous argument applies in the supersingular case. 

Note that
\begin{equation}\label{sel-fac}
\xi(X^{\cyc}(g_{/L}))=
\xi(X(g)) \cdot \xi(X(g'))
\end{equation}
(cf.~\cite[Lem. 3.1.5 \& Prop. 3.2.3]{SU}). 
In view of the definition of the $p$-adic $L$-functions over $L$ as in \S\ref{sspLcy} the assertion follows.
\end{proof}

\subsubsection{Via two-variable zeta element} 

\begin{prop}\label{Eq}
Let $g \in S_{2}(\Gamma_{0}(N))$ be an elliptic newform and $p \nmid 2N$ an ordinary or a supersingular prime. 
Let $L$ be an imaginary quadratic field satisfying the conditions \eqref{ord}, \eqref{coprime} and $({\rm van}_{L})$ so that $(p)=v\ov{v}$ with $v$ determined via the embedding $\iota_p$. 
For  $\bullet, \Box$ as in \eqref{n-red} and 
$\cdot \in \{\emptyset, \cyc, \ac\}$, 
let $$\CZ^{\Box,\cdot}(g_{/L}) \in H^{1}_{\rm{rel},\bullet}(L, T \otimes \Lambda_{L}^{\cdot})$$ be the projection of the two-variable zeta element $\CZ^{\Box}(g_{/L})$.  
In the case $\cdot=\cyc$ suppose that 
\begin{equation}\label{nv}\tag{nv}
\text{$\loc_{v}(\CZ^{\Box,\cyc}(g_{/L}))$ and $\loc_{\ov{v}}(\CZ^{\Box,\cyc}(g_{/L}))$ are $\Lambda_{L,\cO_{\lambda}}^{\cyc}\otimes_{\BZ_{p}} \BQ_{p}$-non-torsion.}
\end{equation}
Then for $\cdot\in\{\emptyset,\cyc\}$ a one-sided divisibility in one of the Conjectures \ref{St}, \ref{Greenberg} and \ref{LLZ} 
implies the analogous divisibility in the other conjectures. 
Moreover, in the case $\cdot=\ac$ and $\epsilon(E_{/L})=+1$ (resp.~$\epsilon(E_{/L})=-1$) a one sided-divisibility in Conjecture~\ref{LLZ} implies the analogous divisibility in Conjecture~\ref{St} (resp.~Conjecture~\ref{Greenberg}). 
In particular, the corresponding subset of Conjectures \ref{St}, \ref{Greenberg} and \ref{LLZ} are equivalent. 
Without the condition $({\rm van}_{L})$, these assertions still hold in $\Lambda_{L,\cO_{\lambda}}^{\cdot}\otimes_{\BZ_{p}} \BQ_{p}$.

In particular, Kato's main Conjecture \ref{Kato} for $g$ and its quadratic twist $g'=g\otimes \chi_L$  imply the cyclotomic Greenberg main Conjecture \ref{Greenberg} for $g$ over $L$.

\end{prop} 
\begin{proof} 
We present the ordinary case 
assuming the hypothesis $(\rm van_{L})$, and leave the supersingular case to the interested reader.

To begin, the Poitou--Tate global duality (cf.~\cite[Thm. 2.3.2]{JSW} and \cite[Thm. 2.3.4]{MR}) yields an exact sequence 
\begin{equation}\label{exvc1}
0 \ra \frac{H^{1}_{\rm{rel},\ord}(L,T \otimes_{\BZ_{p}} \Lambda_{L}^{\cdot})}{\Lambda_{L,\cO_{\lambda}}^{\cdot}\cdot {\CZ^{\cdot}(g_{/L})}} 
\ra
\frac{\Im(H^{1}_{\ord}(L_{\ov{v}},T \otimes_{\BZ_{p}} \Lambda_{L}^{\cdot}))}{ \Lambda_{L,\cO_{\lambda}}^{\cdot}\cdot \loc_{\ov{v}}(\CZ^{\cdot}(g_{/L}))}
\ra
X_{\rm{Gr}}^\cdot(g_{/L})
\ra
X_{\st,\ord}^\cdot(g_{/L})
\ra
0
\end{equation}
of $\Lambda_{L,\cO_{\lambda}}^{\cdot}$-modules, 
 where the second (resp.~third) map arises from the localisation at $v$ (resp.~the dual of the localisation at $\ov{v}$).
Note that the exactness on the left is a consequence of
\begin{itemize}
\item[-] The non-vanishing \eqref{nv},
\item[-] The fact that
\begin{equation}\label{rk}\tag{rk}
\text{$H^{1}_{\rm{rel},\ord}(L,T \otimes \Lambda_{L}^{\cdot})$ is $\Lambda_{L,\cO_{\lambda}}^{\cdot}$ torsion-free with rank one,}
\end{equation}
\item[-] Moreover,  in the case $\cdot=\ac$, the hypothesis that $\epsilon(E_{/L})=-1$.
\end{itemize}
In regards to the fact \eqref{rk}, it is simply the content of Theorem \ref{cycIwL}(b) for $\cdot=\cyc$, which in turn implies the case $\cdot=\emptyset$. 
By the second explicit reciprocity law (cf.~\S\ref{two-variable-zeta}), we have an analogous exact sequence 
\begin{equation}\label{exvc2}
0 \ra \frac{H^{1}_{\rm{rel},\ord}(L,T \otimes_{\BZ_{p}} \Lambda_{L}^{\cdot,\ur})}{\Lambda_{L,\cO_{\lambda}}^{\cdot,\ur}\cdot {\CZ^{\cdot}(g_{/L})}}
\ra
\frac{\Lambda_{L}^{\cdot,\ur}}{ \Lambda_{L}^{\cdot,\ur}\cdot \mathcal{L}_{p}^{\rm{Gr},\cdot}(g_{/L})}
\ra
X_{\rm{Gr}}^{\cdot,\ur}(g_{/L})
\ra
X_{\st,\ord}^{\cdot,\ur}(g_{/L})
\ra
0
\end{equation}
of $\Lambda_{L,\cO_{\lambda}}^{\cdot}$-modules. 

On the other hand, switching the role of $v$ and $\ov{v}$ in the preceding analysis and utilising the first explicit reciprocity law, we obtain the exact sequence   
\begin{equation}\label{exv}
0 \ra \frac{H^{1}_{\rm{rel},\ord}(L,T \otimes_{\BZ_{p}} \Lambda_{L}^{\cdot})}{\Lambda_{L,\cO_{\lambda}}^{\cdot}\cdot \CZ^{\cdot}(g_{/L})}
\ra
\frac{\Lambda_{L,\cO_{\lambda}}^{\cdot}}{ \Lambda_{L,\cO_{\lambda}}^{\cdot}\cdot \mathcal{L}_{p}^{\cdot}(g_{/L})}
\ra
X^{\cdot}(g_{/L})
\ra
X_{\st,\ord}^{\cdot}(g_{/L})
\ra
0 
\end{equation}
of $\Lambda_{L,\cO_{\lambda}}^{\cdot}$-modules. 
Here the exactness on the left is a consequence of  
\begin{itemize}
\item[-] Proposition \ref{2varZ-prop}, 
\item[-] the fact \eqref{rk}, 
\item[-] Moreover, in the case $\cdot=\cyc$, we utilise Remark \ref{NVcyc}, and in the case $\cdot=\ac$, the hypothesis that $\epsilon(E_{/L})=+1$.
\end{itemize}

By the exact sequence (\ref{exv}), $X_{\st,\ord}^{\cdot}(g_{/L})$ is 
$\Lambda_{L,\cO_{\lambda}}^{\cdot}$-torsion. 
Hence, in light of (\ref{exvc2}) a divisibility in the Greenberg main conjecture for $g$  is equivalent to the analogous divisibility in Conjecture \ref{LLZ}. 
The latter is also equivalent to the analogous divisibility in Conjecture \ref{St}
by
\eqref{exv}.

Without the condition (van$_{L}$), the above analysis still applies by considering exact sequences of $\Lambda_{L,\cO_{\lambda}}^{\cdot}\otimes \BQ_{p}$-modules.
Note that the `In particular' part just follows by Lemma \ref{Eq'}. 

\end{proof}
\begin{remark} 
\noindent
\begin{itemize}
\item[(i)] As seen in the proof, $X_{\st,\ord}^{\cdot}(g_{/L})$ is $\Lambda_{L,\cO_{\lambda}}^{\cdot}$-torsion. 
\item[(ii)] Anticyclotomic variants of the above main conjectures are also equivalent. In this setting \eqref{rk} is a consequence of the existence of an Euler system.
\end{itemize}
\end{remark}

\subsubsection{Via Heegner points}
\begin{prop}\label{Eq-He}
Let $g \in S_{2}(\Gamma_{0}(N))$ be an elliptic newform and $p$ an odd prime of good ordinary reduction. 
Let $L$ be an imaginary quadratic field satisfying \eqref{spl}, \eqref{Heeg} and ~$({\mathrm{van}}_{L})$.
Then a one-sided divisibility in the anticyclotomic case of Conjecture \ref{Greenberg} implies an analogous divisibility in Conjecture \ref{HMC}, and conversely. 
In particular, the anticyclotomic counterpart of Conjectures \ref{Greenberg} and \ref{HMC} are equivalent. 
\end{prop}
The proposition is a consequence of the $p$-adic Waldspurger formula \cite{BDP1}, interpreted as an explicit reciprocity law for the Heegner class $\kappa_g$ (cf.~\cite[Thm.~5.2]{BCK}).
\subsection{Towards the Main conjectures}
We describe key results towards the main conjectures.
\subsubsection{Kato's main conjecture}
\begin{thm}\label{KaMC_r}
Let $g \in S_{2}(\Gamma_{0}(N))$ be an elliptic newform and $p$ a prime. 
\begin{itemize}
\item[(a)] $X_{\st}(g)$ is $\Lambda_{\cO_{\lambda}}$-torsion. 
\item[(b)] For 
$\gamma=\gamma_{g}$ as in Lemma \ref{GorPer}, we have a divisibility of ideals 
$$
\xi(X_{\st}(g)) \big{|} 
\xi\big{(}H^{1}(\BZ[\frac{1}{p}], T \otimes_{\BZ_{p}} \Lambda)/ \Lambda_{\cO_{\lambda}}\cdot{{\bf{z}}_{\gamma}(g)}\big{)}
$$
 in $\Lambda_{\cO_{\lambda}}\otimes_{\BZ_{p}} \BQ_{p}$ and even in $\Lambda_{\cO_{\lambda}}$ for odd primes $p$ if the following holds:
\begin{equation}\label{im}\tag{im}
\text{There exists $\sigma \in G_\BQ$ fixing $\BQ_\infty$ such that $T/(\sigma-1)$ is a free $\cO_\lambda$-module of rank $1$.}
\end{equation}
\item[(c)] Let $p\nmid 6N$ be an ordinary prime such that $({\rm irr}_\BQ)$ holds. Then we have an equality of ideals
$$
\xi\big{(}H^{1}(\BZ[\frac{1}{p}], T \otimes_{\BZ_{p}} \Lambda)/ \Lambda_{\cO_{\lambda}}\cdot{{\bf{z}}_{\gamma}(g)}\big{)}
=\xi(X_{\st}(g))
$$
in $\Lambda_{\cO_{\lambda}}\otimes_{\BZ_{p}} \BQ_{p}$, which is an equality in $\Lambda_{\cO_{\lambda}}$  for primes $p \nmid 2N$ if the following holds: 
\begin{equation}\label{ram}\tag{ram}
\text{There exists a prime $\ell || N$ with $\ov{\rho}$ ramified at $\ell$.}
\end{equation} 
Moreover, 
the equality also holds for the quadratic twist $g_{K}:=g\otimes \chi_K$, where $K/\BQ$ is a quadratic field extension with $p\mid\disc(K)$ so that  
\begin{equation}\label{ram$_{K}$}\tag{ram$_{K}$}
\text{There exists a prime $\ell \nmid D_{K}$ as in \eqref{ram}.}
\end{equation} 
In the CM case, the equality holds for any prime $p\nmid 2N$. 
\end{itemize}
\end{thm} 
\begin{proof}
The CM case is due to Rubin \cite[\S1]{Ru} (see also \cite[\S7]{PoRu}). 

In the non-CM case parts (a) and (b) are due to Kato ~\cite[Thm~12.4 and~Thm.~12.5]{K}. 
As for part (c) the equality for the newform $g$
 in $\Lambda_{\cO_{\lambda}}$  is due to Skinner--Urban \cite[Thm.~3.29]{SU}, and that in $\Lambda_{\cO_{\lambda}}\otimes \BQ_{p}$ due to Wan \cite[Thm.~4]{W'}. 

We now consider the quadratic twist $g_{K}$. 
Note that $g_K$ is nearly ordinary at $\lambda$. 
The Eisenstein congruence method of Skinner--Urban \cite{SU} applies to this setting.
 Let $L$ be an imaginary quadratic field with $q$ as in \eqref{ram$_{K}$} to be inert, and all other primes dividing $N$ split. 
Let ${\mathbf{g}}_\alpha$ be the Hida family passing through the $p$-ordinary stabilisation of
 $g$.   
 Then the Eisenstein congruence divisibility in the three-variable main conjecture
 ~\cite[3.4.5]{SU}  for $g_K$ over $L$ may be shown following \cite{SU}.
 
 To begin, there exists a $\Lambda_{\bf D}$-adic Eisenstein series ${\bf E_{D}}$ on $U(2,2)$ whose constant term is closely related to the underlying three-variable $p$-adic $L$-function (cf.~\cite[\S12]{SU}), where $\Lambda_{\bf D}$ is a three-variable Iwasawa algebra as in \cite[p.~11]{SU}. 
 To implement the strategy of {\it{op. cit.}} it suffices to show that a non-constant term of the  Fourier--Jacobi expansion of ${\bf E_{D}}$ is non-zero modulo $p$. 
 As in {\it{op.~cit.}} 
 a certain $\Lambda_{\bf D}$-combination of the non-constant coefficients factorises as 
 $$\mathcal{A}_{\mathbf{D},\mathbf{g}} \mathcal{B}_{\mathbf{D},\mathbf{g}}$$ 
 for ${\bf g}$ an auxiliary $p$-adic family of CM forms. Hence, it suffices to show that $\mathcal{A}_{\mathbf{D},\mathbf{g}}$ and $\mathcal{B}_{\mathbf{D},\mathbf{g}}$ are non-zero modulo $p$. 
 
  The calculation of Fourier--Jacobi coefficients in
\cite[\S11]{SU} applies to our setting. 
Note that $\mathcal{A}_{\mathbf{D},\mathbf{g}}$ in Proposition 13.5 of \emph{op.~cit.} satisfies the same properties upon replacing the underlying data with $\chi_K$-twist.
 As for the study of $\mathcal{B}_{\mathbf{D},\mathbf{g}}$,  \cite[Lem.~11.36]{SU} still applies for its evaluation.
 
  However, the analysis in \cite[Prop.~13.5~and~Prop.~13.6]{SU} needs a slight modification. 
 Proceeding as in the proof of Proposition 13.6 of {\it{op.~cit.}} 
 the specialisation of $\mathcal{B}_{\mathbf{D},\mathbf{g}}$ 
 to an arithmetic point is a triple product period\footnote{
  Indeed, the triple product is a special case
   of the construction of triple product $p$-adic $L$-functions in \cite{Hs1}.   
 Note that the test vectors in \cite[\S3]{Hs1} are the same as ours. 
 The triple product period has also been investigated in \cite{CLW}.
  }
 instead of the Rankin--Selberg period therein. 
 To show that $\mathcal{B}_{\mathbf{D},\mathbf{g}}$ is $p$-indivisible  
 for some choice of ${\bf g}$, we then pick ${\bf g}$  and apply Finis' non-vanishing result \cite{Fi}
 just as in the proof of \cite[Prop.~13.6]{SU}.
To conclude the $p$-indivisibility of 
$\mathcal{A}_{\mathbf{D},\mathbf{g}}$, instead of resorting to the vanishing of the anticyclotomic $\mu$-invariant \cite{Va} in \cite[12.3.5]{SU}, we utilise a non-vanishing result of Hung \cite[Thm.~B]{Hu} (see also \cite{CLW}). The latter does not explicitly cover our setting, however the same argument applies:
  the key is that the modular form $\mathbf{f}_1$ as in \cite[p.~205]{Hu} is non-zero modulo $p$. 
\end{proof}
\begin{remark}\label{ram-im}
Note that the condition \eqref{ram} implies \eqref{im}.
\end{remark}

We record the following consequence of Theorem \ref{KaMC_r} and Proposition \ref{Eq}. 
\begin{cor}\label{GKSU} 
Let $g \in S_{2}(\Gamma_{0}(N))$ be an elliptic newform and $p$ a prime. Let $L$ be an imaginary quadratic satisfying \eqref{ord}, and \eqref{coprime}. Suppose that the non-vanishing \eqref{nv} holds. 
\begin{itemize}
\item[(a)] Let $p\nmid 2N$ be an ordinary or a supersingular prime. Then we have a divisibility of ideals
$$\xi(X_{\rm{Gr}}^{\rm cyc, ur}(g_{/L}))\big{|}(\mathcal{L}_{p}^{\rm{Gr},\cyc}(g_{/L})),$$ 
 in $\Lambda_{L,\cO_{\lambda}}^{\ur,\cyc}\otimes_{\BZ_{p}} \BQ_{p}$ and even in $\Lambda_{L}^{\ur,\cyc}$ if \eqref{im} holds.
\item[(b)] Let $p\nmid 2N$ be a prime of ordinary reduction satisfying $({\rm irr}_{L})$ and \eqref{ram}.
Then 
$$\xi(X_{\rm{Gr}}^{\rm cyc, ur}(g_{/L}))=(\mathcal{L}_{p}^{\rm{Gr},\cyc}(g_{/L})).$$
\end{itemize}
\end{cor}

\subsubsection{Greenberg main conjecture}
\begin{thm}\label{GMC_r} 
Let $g\in S_{2}(\Gamma_{0}(N))$ be an elliptic newform with $N$ square-free and $p\nmid 2N$ a prime. Let $L$ be an imaginary quadratic field satisfying \eqref{ord} and $({\rm irr}_L)$. 
Write $N=N^{+}N^{-}$ for $N^+$ precisely divisible by split primes in $L$. 
Suppose that
\begin{equation}\label{spl}\tag{{\rm spl}}
\text{There exists a prime $q|N^-$, and if $2$ does not split in $L$, then $2|N^-$.}
\end{equation}

Then we have a divisibility of ideals
\begin{equation}\label{Gr-div}
\mathcal{L}_{p}^{\rm{Gr}}(g_{/L})\big{|}
\xi(X_{\rm{Gr}}^{\rm ur}(g_{/L}))
\end{equation}
in $\Lambda_{L,\cO_{\lambda}}^{\ur}\otimes_{\Lambda_{L,\cO_{\lambda}}^{\cyc,\ur}} 
\Frac(\Lambda_{L,\cO_{\lambda}}^{\cyc,\ur})$.
For $g$ and $p$ as above, let $K$ be a quadratic extension of $\BQ$ with discriminant prime to $Np$, and $g_{K}:=g\otimes \chi_K$ the quadratic twist. Then the divisibility \eqref{Gr-div} also holds for $g_K$. 
\end{thm}
\begin{proof}
This is essentially the content of \cite[Thm.~8.2.3]{CLW} (see also \cite{W1}). 

In view of Theorem \ref{excz} the proof of \cite[Prop.~8.2.2]{CLW} applies for the case $\xi={\bf 1}_L$, henceforth the notation being as in {\it loc. cit.} Thus, we have an integral $p$-adic $L$-function $$\CL_{\pi,L,{\bf 1}_L} \in \Lambda_{L,\cO_{\lambda}}^{\ur}$$ for $\pi$ the automorphic representation generated by $g$. Considering the interpolation formulas, $\CL_{\pi,L,{\bf 1}_L}$ is nothing but the Greenberg $p$-adic $L$-function $\CL_{p}^{\rm Gr}(g_{/L})$. Hence the divisibility \eqref{Gr-div} holds in $\Lambda_{L,\cO_{\lambda}}^{\ur}\otimes_{\Lambda_{L,\cO_{\lambda}}^{\cyc,\ur}} 
\Frac(\Lambda_{L,\cO_{\lambda}}^{\cyc,\ur}) 
$ by \cite[Thm.~8.2.3]{CLW}. 

The conductor of the quadratic twist $g_K$ is not square-free. So 
this case is not explicitly covered by the results of \cite{W1,CLW}, but essentially the same argument applies. The square-free-ness was assumed in {\it{loc. cit.}} for explicit calculation\footnote{The periods appear in the computation of Fourier--Jacobi expansion of an Eisenstein series on $GU(3,1)$ whose constant term is linked with  
$\CL_{p}^{\rm Gr}(g_{/L})$.} of local triple product periods.
However, at any prime dividing $D_K$, the local automorphic representation associated to $g$ is principal series, and the computation of the local period has already appeared in \cite[\S8D]{W1}.
\end{proof}
\begin{remark}
While the method of \cite{CLW} uniformly treats the ordinary and non-ordinary primes, the ordinary case  goes back to \cite{W1}.
\end{remark}

\subsubsection{Heegner main conjecture}
\begin{thm}\label{HMC-ub}
Let $g \in S_{2}(\Gamma_{0}(N))$ be an elliptic newform, and $p\nmid 2N$ an ordinary prime. 
Let $L$ be an imaginary quadratic field satisfying \eqref{ord}, \eqref{coprime}, \eqref{Heeg} and $({\rm van}_{L})$. 
Then we have a divisibility of ideals 
$$
\xi_{\Lambda_{L,\cO_{\lambda}}^{\ac}}(X(g)_{\tor})\big{|}
\xi_{\Lambda_{L,\cO_{\lambda}}^{\ac}}(
\frac{H^{1}_{\ord}(L,T_{g}\otimes_{\BZ_{p}}\Lambda_{L,\cO_{\lambda}}^{\ac})}{\Lambda_{L,\cO_{\lambda}}^{\ac}\cdot \kappa_{g}})\cdot \xi_{\Lambda_{L,\cO_{\lambda}}^{\ac}}(\frac
{H^{1}_{\ord}(L,T_{g}\otimes_{\BZ_{p}}\Lambda_{L,\cO_{\lambda}}^{\ac})}{\Lambda_{L,\cO_{\lambda}}^{\ac}\cdot\kappa_{g}^{\iota}}).
$$
in $\Lambda_{L,\cO_{\lambda}}^{\ac}\otimes \BQ_{p}$, and even in 
$\Lambda_{L,\cO_{\lambda}}^{\ac}$ if the follows holds: 
\begin{equation}\label{sur}\tag{sur}
\text{The image of $\rho:G_{\BQ}\ra  \GL_{2}(\cO_{\lambda})$ equals the subgroup of matrices with determinant in $\BZ_{p}^{\times}\subset \cO_{\lambda}^\times$.}
\end{equation}
\end{thm}
While the integral divisibility\footnote{In \cite{Ho,Ho1} it is assumed that $p\nmid h_{L}$, however as explained in \cite[\S4]{CGLS} the hypothesis is inessential.} is due to Howard \cite{Ho,Ho1}, the rational version is more recent \cite[Thm.~5.5.1]{CGS}.

\section{Kobayashi's main conjecture}\label{s:Kob} 
The aim of this section is to prove Conjecture \ref{KatopLss} in the semistable case. 
\subsection{Main results}
\subsubsection{}
\begin{thm}\label{KoMC_r}
Let $g \in S_{2}(\Gamma_{0}(N))$ be an elliptic newform with $N$ square-free, and $p\nmid 2N$ a prime of supersingular reduction. Then Kobayashi's main Conjecture \ref{KatopLss} is true, that is, 
$$
(\mathcal{L}_{p}^{\circ}(g))
=\xi(X_{\circ}(g))
$$
for $\circ\in\{+,-\}$.
Moreover, the same holds for $g_{K}:=g\otimes \chi_K$ for any quadratic field extension $K/\BQ$ with discriminant coprime to $Np$.
\end{thm} 
Theorem \ref{KoMC_r} is proved is subsection \ref{ss:KMC_pf}.
\subsubsection{$p$-part of the BSD formula}
As first observed by Kobayshi~\cite{Ko,Ko1}, his main conjecture has the following
application to the $p$-part of the conjectural BSD formula: 
\begin{cor}\label{BSD_f_Ko}
Let $g \in S_{2}(\Gamma_{0}(N))$ be an elliptic newform with $N$ square-free, and $F$ the Hecke field with degree $d$.
Let $A_{g}$ be an associated $\GL_2$-type abelian variety over $\BQ$.
Let $p\nmid 2N$ be a prime so that $a_{p}(g)=0$. 
If $\ord_{s=1}L(s,g)=r\leq 1$, then the $p$-part of the Birch and Swinnerton-Dyer conjecture for $A_{g}$ is true, that is, $\rank_{\BZ}A_{g}(\BQ)=rd$, $\Sha(A_{g})[p^{\infty}]$ is finite and 
 $$
\bigg{|} \frac{L^{(rd)}(1,A_{g})}{d!\cdot \Omega_{A_{g}}R(A_{g})}
\bigg{|}^{-1}_{p}
=
\big{|}
\# \Sha(A_{g})[p^{\infty}] \cdot \prod_{\ell | N} c_{\ell}(A_{g})
\big{|}^{-1}_{p}.
$$ 
Moreover, the same holds for $g\otimes \chi_K$ for quadratic field extensions $K/\BQ$ with discriminant coprime to $Np$.
\end{cor}
\begin{proof} 
Note that $$L(s,A_{g})=\prod_{\sigma: F \hookrightarrow \BC}L(s,g^{\sigma}).$$ 
Since the BSD conjecture is isogeny invariant, we may assume that $\cO_{F}\hookrightarrow \End{A_{g}}$. 
It suffices to consider the $\lambda$-part of the BSD formula for any prime $\lambda$ of $F$ above $p$. Choose an embedding $\iota_{p}:\ov{\BQ}\hookrightarrow \ov{\BQ}_{p}$ so that it induces the place $\lambda$.

In view of Theorem~\ref{KoMC_r}
 Corollary \ref{KatoEq}, Kato's main Conjecture \ref{Kato} is true. 
In the $r=0$ case  
it implies the $\lambda$-part of the BSD conjecture for $A_{g}$ over $\BQ$, as observed by Kato  
(cf.~\cite[\S14.20]{K},~see also the proof of \cite[Thm.~3.6.13]{SU} and \cite[\S7.2]{JSW}). 
In the $r=1$ case $\rank_{\BZ} A_{g}(\BQ)=rd$ and $\Sha(A)[\lambda^{\infty}]$ is finite by the Gross--Zagier and Kolyvagin theorem (alternatively see Theorems \ref{MWrk} and \ref{boundSha}).
In the $r=1$ non-ordinary case, 
Kato's main conjecture in combination with the $\lambda$-adic Gross--Zagier formula implies the $\lambda$-part of the Birch and Swinnerton-Dyer formula, 
as observed by Kobayashi
(cf.~\cite{Ko1},~\cite{Ko2},~\cite[Cor.~A.5]{BKO2}). 
\end{proof}
\begin{remark}\noindent
\begin{itemize}
\item[(i)] The rank zero BSD formula also follows by  descent of Kobayashi's main Conjecture \ref{KatopLss}. 
\item[(ii)] As for the rank one case, Jetchev--Skinner--Wan \cite{JSW} have proposed a different approach. 
It is based on the $p$-adic Waldspurger formula \cite{BDP1}, yet 
relies on the $r=0$ case of Corollary \ref{BSD_f_Ko}. Indeed their approach first proves the rank one BSD formula over an imaginary quadratic field, and relies on the $r=0$ case of Corollary \ref{BSD_f_Ko} to isolate the desired rank one formula over $\BQ$. 
\item[(iii)] If $g$ has CM and $r=0$, then the full Birch and Swinnerton-Dyer conjecture is recently proved ~\cite{BF}, building on the work of Rubin \cite{Ru}.
\end{itemize}
\end{remark}

\subsection{Elements of the proof}
\subsubsection{An Euler system divisibility}
\begin{thm}\label{KoMC_ub}
Let $g \in S_{2}(\Gamma_{0}(N))$ be an elliptic newform, and $p\nmid 2N$ a prime of supersingular reduction. 
\begin{itemize}
\item[(a)] $X_{\circ}(g)$ is $\Lambda_{\cO_{\lambda}}$-torsion for $\circ\in\{+,-\}$. 
\item[(b)]  We have a divisibility of ideals 
$$
\xi(X_{\circ}(g)) | (\CL^{\circ}_{p}(g))
$$
 in $\Lambda_{\cO_{\lambda}}\otimes_{\BZ_{p}} \BQ_{p}$ and even in $\Lambda_{\cO_{\lambda}}$ under the condition \eqref{im}.
\end{itemize}
\end{thm}
This theorem is due to Kobayashi \cite[Theorems 1.2 and 1.3]{Ko} for elliptic curves, and Lei \cite[Prop.~6.4 and Rem.~6.10]{L}. It is a consequence of Kato's Theorem \ref{KaMC_r}.

\subsubsection{An Eisenstein congruence divisibility}
\begin{thm}\label{KoMC'_lb} 
Let $g\in S_{2}(\Gamma_{0}(N))$ be an elliptic newform with $N$ square-free and $p\nmid 2N$ a  prime. Let $L$ be an imaginary quadratic field satisfying $(D_{L},2N)=1$ and \eqref{ord}. In the ordinary case suppose that $({\rm irr}_{L})$ holds. 
Write $N=N^{+}N^{-}$ for $N^+$ precisely divisible by split primes in $L$. Suppose that  either
\begin{equation}\label{def}\tag{def}
\text{Each prime dividing $N^-$ satisfies \eqref{ram} and $\nu(N^-)$ is odd,}
\end{equation}
or 
\begin{equation}\label{indef}\tag{indef}
\text{Each prime dividing $N^{-}\neq 1$ satisfies \eqref{ram} and $\nu(N^-)$ is even.}
\end{equation}

Then for $\circ\in\{+,-,\emptyset\}$, one  has  
$$
\mathcal{L}_{p}^{\circ}(g_{/L})\big{|}
\xi(X_{\circ}(g_{/L}))
$$
 in $\Lambda_{L,\cO_{\lambda}}$. Moreover, the same holds for $g_{K}:=g\otimes \chi_K$ for any quadratic field extension $K/\BQ$ with discriminant coprime to $Np$.
\end{thm}
\begin{proof}
In the supersingular case $\bar{T}$ is an irreducible $k_\lambda[G_{\BQ_{p}}]$-module, and the hypothesis $({\rm irr}_{L})$ holds by \eqref{spl}. So the hypotheses of Theorem \ref{GMC_r}(a) are  satisfied, and we have the divisibility
$$
\mathcal{L}_{p}^{\rm{Gr}}(g_{/L})\big{|}
\xi(X_{\rm{Gr}}^{\rm ur}(g_{/L}))
$$
in $\Lambda_{L,\cO_{\lambda}}^{\ur}\otimes_{\Lambda_{L,\cO_{\lambda}}^{\cyc,\ur}} 
\Frac(\Lambda_{L,\cO_{\lambda}}^{\cyc,\ur})$. 
In turn, by the proof of Proposition \ref{Eq}, we have the divisibility 
\begin{equation}\label{tpm-div}
\mathcal{L}_{p}^{\circ}(g_{/L})\big{|}
\xi(X_{\circ}(g_{/L}))
\end{equation}
in $\Lambda_{L,\cO_{\lambda}}\otimes_{\Lambda_{L,\cO_{\lambda}}^{\cyc}} 
\Frac(\Lambda_{L,\cO_{\lambda}}^{\cyc})$. 

Suppose that \eqref{def} holds. Then $\epsilon(g_{/L})=+1$, and
$$
\mu(\CL^{\circ}_{\mathscr{W}}(g_{/L}))=0
$$
by \cite[Thm.~1.2]{PW}. 
Here $\CL^{\circ}_{\mathscr{W}}(g_{/L})$ is the $\circ$-anticyclotomic $p$-adic $L$-function whose construction is based on the Waldspurger formula on the definite quaternion algebra ramified at the primes dividing $N^-{}\infty$ (see also \cite{Va}). 
In view of the interpolation formulas for the underlying $p$-adic $L$-functions, it follows that
$$(\CL^{\circ,\ac}_{p}(g_{/L}))=(\prod_{q|N^{-}}c_q(g) \cdot \CL^{\circ}_{\mathscr{W}}(g_{/L})).$$ Note that the Tamagawa numbers $c_q(g)$ are $p$-indivisible under the hypothesis \eqref{ram}, 
and so
$$
\mu(\CL^{\circ}_{p}(g_{/L}))=0.
$$
Hence $\CL^{\circ}_{p}(g_{/L})$ is coprime to height one prime ideals of 
$\Lambda_{L,\cO_{\lambda}}^{\cyc}$, and the divisibility \eqref{tpm-div} holds in $\Lambda_L$.

In the same vein, the $\mu$-invariant of the BDP $p$-adic $L$-function vanishes \cite{Bu} under \eqref{indef}, and the argument in the preceding paragraph applies. 
\end{proof}
\begin{remark} 
The hypothesis \eqref{indef} may be generalised  as in \cite[Thm.~8.2.3(2)]{CLW}. 
\end{remark}
\subsubsection{A control theorem}
\begin{prop}\label{ctl_st}
Let $g \in S_{2}(\Gamma_{0}(N))$ be an elliptic newform and $p\nmid 2N$ 
a supersingular prime. Let $L$ be an imaginary quadratic field such that \eqref{ord} holds. 
Then we have
 $$X_{\circ}(g_{/L})/(\gamma_{\ac}-1)\simeq X^{\cyc}_{\circ}(g_{/L}).$$
 In particular $$\xi(X_{\circ}(g_{/L})) \mod(\gamma_{\ac}-1) \big{|} \xi(X_{\circ}(g))\cdot \xi(X_{\circ}(g\otimes \chi_{L})).$$ 
 Moreover, an analogous divisibility holds along the anticyclotomic tower. 
\end{prop}
\begin{proof}
One may proceed as in the proof of \cite[Prop.~3.9]{SU} (see also \cite[\S 9]{Ko}). Note that `In particular' part follows from supersingular analogue of the factorsiation \eqref{sel-fac}  (cf.~\cite[Cor.~3.8(ii)]{SU}). 
\end{proof}

\subsection{Proof of Theorem \ref{KoMC_r}}\label{ss:KMC_pf} 
\noindent
\begin{proof}
Since $g$ is semistable, there exists a prime $q|N$ satisfying the condition \eqref{ram} by Ribet's level raising. 
Pick an imaginary quadratic field $L$ such that \eqref{ord} holds and 
$(D_{L},2N)=1$. Suppose also that either $q$ is inert in $L$ and the primes dividing $N/q$ split, or that \eqref{def} holds. 

Recall that 
$\CL_{p}^{\circ,\cyc}(g_{/L})=\CL_{\gamma}(g)\cdot\CL_{\gamma'}(g')$ 
for $g'=g\otimes \chi_L$, and 
$\gamma,\gamma'$ as in Lemma \ref{GorPer}. 
So Theorem \ref{KoMC'_lb} in combination with Proposition \ref{ctl_st} implies that 
$$
\CL_{\gamma}(g)\cdot\CL_{\gamma'}(g')
 \big{|} \xi(X_{\circ}(g))\cdot \xi(X_{\circ}(g')).
$$
On the other hand,  we have 
$\xi(X_{\circ}(g))| \CL_{\gamma}^{\circ}(g)$ and $\xi(X_{\circ}(g'))| \CL_{\gamma'}^{\circ}(g')$ 
by Theorem \ref{KoMC_ub}(b) (cf.~Remark~\ref{ram-im}). 
Therefore, 
$$
\xi(X_{\circ}(g))= (\CL_{\gamma}^{\circ}(g)).
$$

The same argument applies for the quadratic twist $g_K$.
\end{proof}
\subsection{Complements}
\subsubsection{Supersingular main conjecture, bis}
\begin{thm}\label{KoMC'_eq} 
Let $g\in S_{2}(\Gamma_{0}(N))$ be an elliptic newform with $N$ square-free and $p\nmid 2N$ a supersingular prime. Let $L$ be an imaginary quadratic field such that $(D_{L},2N)=1$, and \eqref{ord} holds. 
Suppose also that either the condition \eqref{def} or \eqref{indef} holds. 
Then Conjecture \ref{St} is true, that is, 
for $\circ\in\{+,-\}$ and $\cdot\in\{\emptyset,\cyc\}$, we have
$$
\xi(X_{\circ}^{\cdot}(g_{/L}))=(\mathcal{L}_{p}^{\circ,\cdot}(g_{/L})).
$$
Moreover, the same holds for $\cdot=\ac$ under the condition \eqref{def}.
In particular Conjecture \ref{Greenberg} is true for $\cdot=\emptyset$

Moreover, the same holds for $g_{K}:=g\otimes \chi_K$ for any quadratic field extension $K/\BQ$ with discriminant coprime to $Np$.
\end{thm}
\begin{proof}
Since Theorem \ref{KoMC_r} applies to $g$ as well as $g\otimes \chi_L$, we have  
$$
(\CL_{p}^{\circ,\cyc}(g_{/L}))=\xi(X_{\circ}^{\cyc}(g_{/L}))
$$
by Lemma~\ref{Eq'}. 
On the other hand, 
Theorem \ref{KoMC'_lb} gives the two-variable divisibility $\CL_{p}^{\circ}(g_{/L})| \xi(X_{\circ}(g_{/L}))$.  
Hence, noting the non-vanishing of 
$\CL_{p}^{\circ,\cyc}(g_{/L})$, 
the two-variable main conjecture 
$$
(\CL_{p}^{\circ}(g_{/L}))=\xi(X_{\circ}(g_{/L}))
$$
follows (cf.~\cite[Lem.~3.2]{SU}). 
In turn, noting the non-vanishing of $\CL_{p}^{\circ,\ac}(g_{/L})$ under \eqref{def}, 
the anticyclotomic main conjecture follows by descent (cf.~Proposition \ref{ctl_st}).

Note that the `In particular' part is a consequence of Proposition \ref{Eq}. 
\end{proof}
\begin{remark}
If the condition \eqref{indef} holds, then the anticyclotomic $p$-adic $L$-function $\CL_{p}^{\circ,\ac}(g_{/L})$ vanishes by the interpolation property. 
\end{remark}
\subsubsection{Ordinary main conjecture}
\begin{thm}\label{IMC_ord} 
Let $g\in S_{2}(\Gamma_{0}(N))$ be an elliptic newform with $N$ square-free and $p\nmid 2N$ an ordinary  prime such that $(\rm irr_{\BQ})$ holds. 
\begin{itemize}
\item[(a)] Conjecture \ref{KatopL} is true.
\item[(b)] Let $L$ be an imaginary quadratic field satisfying $(D_{L},2N)=1$, \eqref{ord} and $({\rm irr}_{L})$. Suppose that either the condition \eqref{def} or \eqref{indef} holds. Then Conjectures \ref{St} and \ref{Greenberg} are true. 
\end{itemize}
Moreover, the same holds for $g_{K}:=g\otimes \chi_K$ for any quadratic field extension $K/\BQ$ with $(D_{K},Np)=1$.
\end{thm}
\begin{proof}
One may proceed just as in the proof of Theorem \ref{KoMC'_eq}.  
\end{proof}
\begin{remark}
Part (a) gives a different proof of a special case of Theorem \ref{KaMC_r}(c). In the \eqref{def} case part (b) is also a special case of a result \cite[Thm.~3.30]{SU} of Skinner--Urban towards Conjecture \ref{St}. In the remaining cases Theorem \ref{IMC_ord} presents new evidence towards Conjectures \ref{St} and \ref{Greenberg}.
\end{remark}
\subsubsection{The Birch and Swinnerton-Dyer formula: a rank zero quadratic twist family}

\begin{thm}\label{theoretic}
Let $E$ be a semistable elliptic curve defined over $\mathbb{Q}$ with conductor $N$, and $M>1$ a square-free integer with $(M,N)=1$. Let $E^{(M)}$ denote the quadratic twist of $E$ by the character associated to the quadratic extension $\mathbb{Q}(\sqrt{M})/\BQ$. Suppose that the following conditions hold. 
\begin{itemize}
\item[(i)] We have $$L(1,E^{(M)})\not=0,$$ 
\item[(ii)] The $2$-part of the BSD formula holds for $E^{(M)}$, 
\item[(iii)] $a_{3}(E)=0$, 
\item[(iv)]  For all odd primes $p$, $E[p]$ is an absolutely irreducible $G_{\BQ}$-representation, 
\item[(v)] For any prime $p|N$, there exists a multiplicative prime $q\neq p$  at which $E[p]$ is ramified, 
\item[(vi)] $E$ has ordinary reduction at prime divisors of $M$. 
\end{itemize}

Then the Birch and Swinnerton-Dyer conjecture is true for $E^{(M)}$, that is, $E^{(M)}(\BQ)$ and $\Sha(E^{(M)})$ are finite, and 
 $$
 \frac{L(1,E^{(M)})}{\Omega_{E^{(M)}}}
=
\frac{\# \Sha(E^{(M)}) \cdot \prod_{\ell \nmid \infty} c_{\ell}(E^{(M)})}{\#E^{(M)}(\BQ)_{\tor}^2}.
$$ 
\end{thm}
\begin{proof}
In view of the condition (i), note that 
$E^{(M)}(\BQ)$ and $\Sha(E^{(M)})$ are finite by the Gross--Zagier and Kolyvagin theorem. 

As for the Birch and Swinnerton-Dyer formula, the archimedian part holds by the non-negativity of the central $L$-value $L(1,E^{(M)})$ due to Kohnen--Zagier \cite{KZ}. 
In view of (ii), it suffices to consider the $p$-part of the BSD formula for odd primes $p$. 
For such primes of good ordinary, multiplicative or supersingular reduction, the formula is the content of 
\cite[Thm.~2]{SU}, \cite[Thm.~C]{Sk0} and Corollary \ref{BSD_f_Ko}, respectively.
Lastly, for odd primes $p$ dividing $M$ it is a consequence of Theorem \ref{KaMC_r}(c) (cf.~the proof of Corollary \ref{BSD_f_Ko}). 
\end{proof}
\begin{remark}\noindent
\begin{itemize}
\item[(i)] As for the condition (vi), recall that $E$ has ordinary reduction at primes of density one since $E$ is non-CM. 
In practice the condition may be verified by considering Tamagawa numbers of the quadratic twist. 
It holds for square-free products of an explicit set of primes of positive density (cf.~\cite{CLZ}).
\item[(ii)] The $2$-part of the BSD formula as in (ii) is known to hold for example under the conditions of \cite[Thm.~1.5]{CLZ}.
\end{itemize} 
\end{remark}

\begin{example}\label{exam} 
For semistable elliptic curves $E$ up to conductor 150, 
the conditions of Theorem \ref{theoretic} are satisfied by\footnote{We are grateful to Shuai Zhai for his assistance in finding these examples.} the curves denoted in the LMFDB database by 
\begin{itemize}
\item 46a1, 69a1, 77c1, 94a1, 114b1, 141b1 and 142c1,  
\item 62a1, 66b1, 105a1, 106d1, 115a1, 118c1, 118d1, 141c1 and 141e1
\end{itemize}
for infinitely many $M$. 
Here for the elliptic curves in the first bullet point we rely on ~\cite[Thm.~1.5]{CLZ} for the conditions (i)-(ii), and for the curves in the second bullet point on \cite[Thm.~1]{Zi}. 
 This gives the first infinite families of non-CM elliptic curves satisfying the Birch and Swinnerton-Dyer  conjecture.
\end{example}

\section{Gross--Zagier and Kolyvagin theorem revisited}\label{MW} 
This section presents a new approach to the Gross--Zagier and Kolyvagin theorem based on the Euler system of Beilinson--Kato elements and the two-variable zeta element. It leads to an optimal upper bound for the Tate--Shafarevich group in terms of index of a Heegner point as predicted by the BSD conjecture. Moreover, in combination with cyclotomic Greenberg main conjecture it leads to rank one cases of the $p$-part of the BSD formula. 

\subsection{Cyclotomic Greenberg Selmer groups} 
We begin with some preliminaries.
\subsubsection{A control theorem} 

\begin{prop}\label{GrCtl}
Let $g \in S_{2}(\Gamma_{0}(N))$ be an elliptic newform and $p \nmid 2N$ a prime.
For $\lambda$ a prime of the Hecke field above $p$,  
let $(\rho,V)$ be the associated Galois representation, $T\subset V$ a Galois stable lattice and $$W=V/T.$$ 
Suppose that $({\rm van}_{\BQ})$ holds. 
Let $L$ be an imaginary quadratic field satisfying the conditions \eqref{ord}, \eqref{coprime} and $({\rm van}_{L})$. 
Suppose that the Selmer group $\Sel_{\rm{Gr}}(g_{/L})$ is finite.  

\begin{itemize}
\item[(a)] $X_{\rm{Gr}}^{\cyc}(g_{/L})$ is a torsion $\Lambda_{L,\cO_{\lambda}}^{\cyc}$-module.
\item[(b)] We have an exact sequence
$$
0 \ra \Sel_{\rm{Gr}}(g_{/L}) \ra S_{\rm{Gr}}^{\cyc}(g_{/L})[\gamma_{\cyc}-1] \ra 
\prod_{w | N} H^{1}_{\ur}(I_{w}, W) \times H_{v} \times H_{{\ov{v}}} \ra 0, 
$$
where 
$$
H^{1}_{\ur}(I_{w}, W)=\ker\big{\{}H^{1}(L_{w},W) \ra H^{1}(I_{w},W)\big{\}}
$$
and 
$$
H_{v}=\ker\big{\{}H^{1}(L_{v},W)  \ra H^{1}(L_{v}, W \otimes_{\BZ_{p}} \Lambda_{L}^{\cyc}[\gamma_{\cyc} - 1])\big{\}} \simeq 
\frac{(W \otimes_{\BZ_{p}} \Lambda_{L,\cO_{\lambda}}^{\cyc})^{G_{L_{v}}} } {[\gamma_{\cyc}-1] (W \otimes_{\BZ_{p}} \Lambda_{L,\cO_{\lambda}}^{\cyc})^{G_{L_{v}}}   }, 
$$
$$
H_{{\ov{v}}}=\frac{H^{1}(L_{{\ov{v}}},W)}{H^{1}(L_{{\ov{v}}},W)_{\div}} 
\simeq H^{0}(L_{{\ov{v}}}, W)^{\vee}.
$$
\end{itemize}
\end{prop}
\begin{proof} 
The following is based on \cite[Thm.~3.3.1]{JSW} and \cite[Prop.~12]{Sk} (see also \cite[\S4]{Gr2}). 

Put
$
\mathcal{P}_{\Sigma} = \prod_{w\in \Sigma} \mathcal{P}_{w}
$
for 
$$
\mathcal{P}_{w}=
\begin{cases}
H^{1}(I_{w},\mathcal{M}^{\cyc})^{G_{L_{w}}} & \text{$w \nmid p$}\\
H^{1}(L_{v},\mathcal{M}^{\cyc}) & \text{$w=v$}\\
0 & \text{$w=\ov{v}$}.\\
\end{cases}
$$
By definition, 
$
S_{\rm{Gr}}^{\cyc}(g_{/L})=\ker\big{\{}H^{1}(L,\mathcal{M}^{\cyc}) \ra \mathcal{P}_{\Sigma}\big{\}}.
$
Put 
$
P_{\Sigma} = \prod_{w\in \Sigma} P_{w} 
$
for 
$$
P_{w}=
\begin{cases}
H^{1}(L_{w},W) & \text{$w \nmid p$}\\
H^{1}(L_{v},W)_{\div} & \text{$w=v$}\\
\frac{H^{1}(L_{{\ov{v}}},W)}{H^{1}(L_{{\ov{v}}},W)_{\div}}  & \text{$w=\ov{v}$}.\\
\end{cases}
$$
We then have an exact sequence
\begin{equation}\label{ctr1}
0 \ra \Sel_{\rm{Gr}}(g_{/L}) \ra S_{\rm{Gr}}^{\cyc}(g_{/L})[\gamma-1] \ra 
\Im \big{\{} H^{1}(G_{L,\Sigma},W) \ra P_{\Sigma} \big{\}} \cap 
\ker \big{\{} P_{\Sigma} \ra \mathcal{P}_{\Sigma}[\gamma-1]     \big{\}}.
\end{equation}
\begin{itemize}
\item[(a)] In view of the finiteness of $\Sel_{\rm{Gr}}(g_{/L})$ the Poitou--Tate global duality implies that the 
localisation $$H^{1}(G_{L,\Sigma},W) \ra P_{\Sigma}$$ is a surjection. 
Hence  
$S_{\rm{Gr}}^{\cyc}(g_{/L})$ is a torsion 
$\Lambda_{L,\cO_{\lambda}}^{\cyc}$-module 
by (\ref{ctr1}) and Nakayama lemma. 

\item[(b)] The Poitou--Tate duality yields a surjection
$H^{1}(G_{L,\Sigma},\mathcal{M}^{\cyc}) \twoheadrightarrow \mathcal{P}_{\Sigma}$
and so 
$$
S_{\rm{Gr}}^{\cyc}(g_{/L})[\gamma_{\cyc}-1] = 
\ker\big{\{}H^{1}(L,W) \ra \mathcal{P}_{\Sigma}[\gamma_{\cyc}-1]\big{\}}.
$$
Now in light of (\ref{ctr1}) there is an exact sequence
\begin{equation}\label{ctr2}
0 \ra \Sel_{\rm{Gr}}(g_{/L}) \ra S_{\rm{Gr}}^{\cyc}(g_{/L})[\gamma-1] \ra 
 \ker \big{\{} P_{\Sigma} \ra \mathcal{P}_{\Sigma}[\gamma-1]     \big{\}}
 \ra 0,
\end{equation}
and the proof concludes.
\end{itemize}
\end{proof}

In conjunction with Remark \ref{rpin'} we note the following.
\begin{cor}\label{GrCtlsz}
Let the notation and conditions be as in Proposition \ref{GrCtl}. 
Then 
$$
\# S_{\rm{Gr}}^{\cyc}(g_{/L})[\gamma_{\cyc}-1] 
= 
\big{|}\#\Sel_{\rm{Gr}}(g_{/L}) \cdot \prod_{\ell | N} c_{\ell}(g)^{2} \cdot   \delta_{p}^{2}\big{|}_{\lambda}^{-1}
$$
for
$$
\delta_{p}=
\begin{cases} 
1+p-a_{g}(p) & \text{if $\lambda\nmid a_{p}(g)$}\\
1 & \text{else.}\\
\end{cases}
$$
\end{cor}

\subsubsection{No pseudo-null submodules} 

\begin{prop}\label{GrPNS}
Let $g \in S_{2}(\Gamma_{0}(N))$ be an elliptic newform 
and $p \nmid 2N$ a prime.
For $\lambda$ a prime of the Hecke field of $g$ above $p$,  
suppose that $({\rm van}_{\BQ})$ holds for the associated Galois representation and that either $\lambda\nmid a_{p}(g)$ or $a_{p}(g)=0$. 
Let $L$ be an imaginary quadratic field satisfying \eqref{ord}, \eqref{coprime} and $({\rm van}_{L})$. 
 If the condition \eqref{nv} holds, then $X_{\rm{Gr}}^{\cyc}(g_{/L})$ has no non-zero pseudo-null submodule.
\end{prop}
\begin{proof}
We consider the ordinary case, and an analogous argument applies for the supersingular case (see also Remark \ref{nPN}).

As seen in the proof of Proposition \ref{Eq}, 
$
X_{\rm{Gr}}^{\cyc}(g_{/L})
$
is a torsion $\Lambda_{L,\cO_{\lambda}}^{\cyc}$-module under the condition \eqref{nv}.
Suppose that $X_{\rm{Gr}}^{\cyc}(g_{/L})$ has a non-zero pseudo-null submodule. 
Then so does $X_{\st,\ord}^{\cyc}(g_{/L})$ by\footnote{Note that $\Lambda_{L,\cO_{\lambda}}^{\cyc,\ur}/ (\mathcal{L}_{p}^{\rm{Gr}}(g_{/L})^{\cyc})$ 
does not have a non-zero pseudo-null submodule.} the exact sequence (\ref{exvc2}), 
 and in turn $X^{\cyc}(g_{/L})$ also does by the exact sequence \eqref{exv}.
The latter contradicts \cite[Prop.~3.3.19]{SU}.
\end{proof}
\begin{remark}\label{nPN}
A special case as in \cite[Prop.~3.3.12]{JSW}  suffices for applications in this paper.
More generally, if $X_{\rm{Gr}}^{\cyc}(g_{/L})$
is a torsion $\Lambda_{L,\cO_{\lambda}}^{\cyc}$-module, then 
an argument of Greenberg \cite{Gr3} shows that it does not have a non-zero pseudo-null submodules. 
\end{remark}
A consequence of Corollary \ref{GrCtlsz} and Proposition \ref{GrPNS}:
\begin{cor}\label{Grsz}
Let the notation and conditions be as in Proposition \ref{GrPNS}. 
Then for a generator $h$ of $\xi(X_{\rm{Gr}}^{\cyc}(g_{/L}))$, we have
$$
|h(0)|_{p}^{-1}=
\# S_{\rm{Gr}}^{\cyc}(g_{/L})[\gamma_\cyc-1] 
= 
\big{|}\#\Sel_{\rm{Gr}}(g_{/L}) \cdot \prod_{\ell | N} c_{\ell}(g)^{2} \cdot   \delta_{p}^{2}\big{|}_{\lambda}^{-1}.
$$
\end{cor}

\subsection{Gross--Zagier and Kolyvagin theorem}
\subsubsection{Mordell--Weil groups} 

\begin{thm}\label{MWrk} 
Let $g \in S_{2}(\Gamma_{0}(N))$ be an elliptic newform, $F$ the Hecke field and $A_{g}$ an associated $\GL_2$-type abelian variety over $\BQ$.
Then
$$
\ord_{s=1}L(s,g)=1 \implies \rank_{\BZ} A_{g}(\BQ)=[F:\BQ].
$$
\end{thm}
\begin{proof}
Let $p \nmid 2N$ be an ordinary\footnote{The existence of positive density of such primes is well-known (cf.~\cite[\S7]{Hi4}).} prime. 
Let $\iota_{\infty}:\ov{\BQ}\hookrightarrow \BC$ and $\iota_{p}:\ov{\BQ}\hookrightarrow \BC_p$ be embeddings, and $\lambda$ the prime of the Hecke field $F$ above $p$ arising from $\iota_p$. Let $\rho: G_{\BQ} \ra \GL_{2}(F_{\lambda})$ be the attached Galois representation, $V$ the underlying vector space and $T$ a Galois stable lattice. 

Let $L$ be an imaginary quadratic field satisfying \eqref{ord} and \eqref{coprime}, in particular $(p)=v\ov{v}$ with $v$ determined via the embedding $\iota_p$.
Suppose that $L$ also satisfies: 
\begin{itemize}
\item The classical Heegner hypothesis \eqref{Heeg}, 
\item $\ord_{s=1}L(s, g_{/L})=1$.
\end{itemize}
The existence of such an $L$ is a special case of the main result of \cite{FH}. 

By the Gross--Zagier formula \cite{GZ}, the Heegner point 
$
P_{L} \in A_{g}(L)
$
arising from $y_{L}\in J_{0}(N)(L)$
is non-torsion. In conjunction with the $p$-adic Waldspurger formula the non-triviality implies that 
$\phi_{{\bf 1}_{L}}(\CL_{p}^{\rm Gr}(g_{/L}))\neq 0$ 
(cf.~Proposition~\ref{GRL=BDPL-II}). 
 Hence
\begin{equation}\label{locnv}
\phi_{{\bf 1}_{L}}(\loc_{\ov{v}}(\CZ(g_{/L})^{\cyc})) \neq 0, 
\end{equation}
and the condition \eqref{nv} holds.  
Then $
\phi_{{\bf 1}_{L}}(\xi(X_{\rm{Gr}}(g_{/L}))) \neq 0
$
by Corollary \ref{GKSU}, 
and so 
$
\Sel_{\rm{Gr}}(g_{/L})
$
is finite (cf.~\eqref{ctr1}). 
In particular, 
$
H^{1}_{\st}(L,V)=0
$
for 
$$
H^{1}_{\st}(L,V):=\ker \big{\{} H^{1}(G_{L,\Sigma},V) \ra H^{1}(L_{v},V) \times H^{1}(L_{\ov{v}},V) \big{\}}. 
$$
In view of \cite[Lem. 2.3.1]{Sk'} it then follows that 
\begin{equation}\label{rkL}
\dim_{F_{\lambda}} H^{1}(G_{L,\Sigma},V) = 2.
\end{equation} 

Since $L(1,g')\neq 0$, 
 Theorem \ref{ERLBKI}
implies that 
$
0 \neq z_{\rm{Kato}}(g') \in H^{1}(G_{L,\Sigma},V)
$. 
Note that 
$z_{\rm{Kato}}(g') \notin H^{1}_{\rm{f}}(G_{L,\Sigma},V)$
by the explicit reciprocity law \cite[Thm.~12.5]{K}. 
In view of (\ref{rkL}) it thus follows that 
$$
\dim_{F_{\lambda}} H^{1}_{\rm{f}}(L,V) \leq 1.
$$  
On the other hand, 
$
\dim_{F_{\lambda}} H^{1}_{\rm{f}}(L,V) \geq 1
$
since the Heegner point $P_L$ is non-torsion. 
\end{proof}
\begin{remark} 
The above rank one approach is akin to that of Kato in the analytic rank zero case \cite[\S13]{K}. 
It is independent from the anticyclotomic Kolyvagin system of Heegner points \cite{Ko}. The Heegner point 
presents itself only in the form of reciprocity laws: $p$-adic Waldspurger and Gross--Zagier formulas.
\end{remark}

\subsubsection{Tate--Shafarevich groups}\label{Shaf}

\begin{thm}\label{boundSha} 
Let $g \in S_{2}(\Gamma_{0}(N))$ be an elliptic newform, $F$ the Hecke field and $\cO$ the integer ring. Let $A_{g}$ be an associated $\GL_2$-type abelian variety over $\BQ$ with $\cO\hookrightarrow \End(A_{g})$.
Let $p\nmid 2N$ be a prime and $\lambda$ a prime of the 
Hecke field $F$ above $p$. 
 Suppose that either $\lambda\nmid a_{p}(g)$ or $a_p(g)=0$. 
\begin{itemize}
\item[(a)] We have
$$
\ord_{s=1}L(s,g)=1 \implies \#\Sha(A_{g})[\lambda^{\infty}]<\infty.
$$
\item[(b)] 
Suppose that $\ord_{s=1}L(s,g)\leq 1$. Suppose also that $({\rm van}_{\BQ})$ and \eqref{im} hold for the associated $\lambda$-adic Galois representation. 
Let $L$ be an imaginary quadratic field satisfying \eqref{ord}, \eqref{coprime}, \eqref{Heeg} and $({\rm irr}_{L})$ so that $\ord_{s=1}L(s,g_{/L})=1.$
Let $P_{L} \in A_{g}(L)$ be the Heegner point arising from an $(\cO,\lambda)$-optimal modular parametrisation of $A_{g}$.
Then 
$$
\big{[}A_{g}(L) \otimes_{\cO} \cO_{\lambda} : \cO_{\lambda}\cdot P_{L}\big{]}^{2}
\geq
\big{|}\#\Sha(A_{g/L})[\lambda^{\infty}] \cdot \prod_{q|N}c_{q}(A_{g})^{2}
\big{|}_{\lambda}^{-1}.
$$
\end{itemize}
\end{thm}
\begin{proof} 
Let $\iota_{\infty}:\ov{\BQ}\hookrightarrow \BC$ and $\iota_{p}:\ov{\BQ}\hookrightarrow \ov{\BQ}_{p}$ be embeddings, the latter so that it induces the prime $\lambda$ of the Hecke field. 
 \begin{itemize} 
\item[(a)] Let $L$ be an imaginary quadratic field 
as in the proof of Theorem \ref{MWrk}. 
As seen in the proof, the Selmer group  
$\Sel_{\rm{Gr}}(g_{/L})$ is finite and then so is 
$
\Sha(A_{g})[\lambda^{\infty}]
$
(see also \cite[\S3.2, Rem. 3.2.2]{JSW}).

\item[(b)] Note that $L$ satisfies the hypotheses appearing in the proof of Theorem \ref{MWrk}.
So the Heegner point $P_{L} \in A_{g}(L)$ is non-torsion and 
\begin{equation}\label{index}
[A_{g}(L): \cO\cdot P_{L}] < \infty.
\end{equation}
Moreover, 
$
\phi_{{\bf 1}_{L}}(\mathcal{L}_{p}^{\rm{Gr}}(g_{/L}))\neq 0
$
and 
$
\Sel_{\rm{Gr}}(g_{/L})
$
is finite. 
Thus, Corollary \ref{GKSU}(a) gives rise to the divisibility
\begin{equation}\label{GrSelub}
\xi(X_{\rm{Gr}}^{\cyc,\ur}(g_{/L})) \big{|}
(\mathcal{L}_{p}^{\rm{Gr},\cyc}(g_{/L}))
\end{equation}
in $\Lambda_{L}^{\cyc,\ur}$. 
We now consider its specialisation 
at the identity Hecke character over $L$. 

By Corollary \ref{GrCtlsz} we have
$$
\big{|}\#S_{\rm{Gr}}^{\cyc}(g_{/L})[\gamma_{\cyc}-1]\big{|}_{\lambda}^{-1}=
\big{|}\#\Sel_{\rm{Gr}}(g_{/L}) \cdot \prod_{\ell | N} c_{\ell}(g)^{2} \cdot   \delta_{p}^{2}\big{|}_{\lambda}^{-1}, 
$$
while the $p$-adic Waldspurger formula gives 
$$
\big{|} \phi_{{\bf 1}_{L}}(\mathcal{L}_{p}^{\rm{Gr}}(g_{/L})^{\cyc}))\big{|}_{\lambda}^{-1}=
\big{|}  \frac{1+p-a_{p}}{p} \cdot \log_{A_{g}(L_{\ov{v}})} P_{L}               \big{|}_{\lambda}^{-2}
$$
(cf.~Proposition~\ref{GRL=BDPL-II} and \cite[\S5.1.5]{JSW}).

So in light of (\ref{GrSelub}) and Corollary \ref{Grsz} it follows that 
$$
\big{|}\#\Sel_{\rm{Gr}}(g_{/L}) \cdot \prod_{\ell | N} c_{\ell}(g)^{2} \cdot   \delta_{p}^{2}\big{|}_{\lambda}^{-1}
\geq
\big{|}  \frac{1+p-a_{p}}{p} \cdot \log_{A_{g}(L_{\ov{v}})} P_{L}               \big{|}_{\lambda}^{-2}.
$$
To conclude the proof, note that 
\begin{equation}\label{pre-1}
\#\Sel_{\rm{Gr}}(g_{/L})
=
\# \Sha(A_{g/L})[\lambda^{\infty}] \cdot 
\frac
{\big{[}A_{g}(L_{\ov{v}})/ A_{g}(L_{\ov{v}})_{\tor} \otimes_{\cO} \cO_{\lambda}: \cO_{\lambda}\cdot P_{L} \big{]}^{2}}
{\big{[}A_{g}(L) \otimes_{\cO} \cO_{\lambda} : \cO_{\lambda} \cdot P_{L}\big{]}^{2}}
\end{equation}
(cf.~\cite[Prop. 3.2.1 \& (3.5.b)]{JSW}) and 
\begin{equation}\label{pre-2}
\big{[}A_{g}(L_{\ov{v}})/ A_{g}(L_{\ov{v}})_{\tor} \otimes_{\cO} \cO_{\lambda} : \cO_{\lambda} \cdot P_{L}\big{]} \cdot \delta_{p}=
\bigg{|}  \frac{1+p-a_{p}}{p} \cdot \log_{A_{g}(L_{\ov{v}})} P_{L}               \bigg{|}_{\lambda}^{-1}
\end{equation}
(cf.~\cite[p. 398]{JSW}). 
\end{itemize}
\end{proof} 
\begin{remark} 
\noindent
The $p$-part of the conjectural BSD formula for $A_{g}$ over $L$ predicts the upper bound as in Theorem \ref{boundSha}(b)
 to be an equality.
The above upper bound is finer than the one arising from the Kolyvagin system of Heegner points. Specifically, $p$ is allowed to divide the Tamagawa numbers, unlike the results of \cite{Ko}, \cite{Ho1}. 
\end{remark}

In combination with Theorem \ref{MWrk} we note the following.
\begin{cor}
Let $E_{/\BQ}$ be an elliptic curve of conductor $N$ and $p\nmid 2N$ a prime. If $p=3$ is non-ordinary, suppose that $a_{p}=0$. Then 
$$\ord_{s=1}L(s,E)=1 \implies \rank_{\BZ}E(\BQ)=1,\ \#\Sha(E)[p^{\infty}]< \infty.$$
Moreover, an upper bound for \#$\Sha(E_{/L})[p^{\infty}]$ over imaginary quadratic fields $L$ as in Theorem \ref{boundSha}(b) holds. 
\end{cor}
\subsection{$p$-part of the Birch--Swinnerton-Dyer formula}\label{pBSD}

\subsubsection{Kato's main conjecture and the BSD formula}
\begin{prop}\label{p-BSD-prop}
Let $g \in S_{2}(\Gamma_{0}(N))$ be an elliptic newform, $F$ the Hecke field with degree $d$ and $\cO$ the integer ring. 
Let $A_{g}$ be an associated $\GL_2$-type abelian variety over $\BQ$ with $\cO\hookrightarrow \End(A_{g})$. 
Let $p\nmid 2N$ be a prime, $\lambda$ a prime of the 
Hecke field $F$ above $p$ and $T$ the $\lambda$-adic Tate module of $A_{g}$.
 Suppose that
 \begin{itemize}
 \item $\ord_{s=1}L(s,g)=1$,  
 \item Either $\lambda\nmid a_{p}(g)$ or $a_p(g)=0$, 
\item The $\lambda$-adic Galois representation  $\rho: G_{\BQ}\ra \Aut_{\cO_{\lambda}}T$ satisfies $({\rm van}_{\BQ})$.
\end{itemize}
Let $L$ be an imaginary quadratic field satisfying \eqref{ord}, \eqref{coprime}, \eqref{Heeg} and $({\rm van}_{L})$ so that $\ord_{s=1}L(s,g_{/L})=1.$

Then the $\lambda$-adic Kato's main Conjecture \ref{Kato} for $g$ and the quadratic twist $g'=g\otimes \chi_L$ 
 imply the $\lambda$-part of the Birch and Swinnerton-Dyer conjecture for $A_{g}$, that is, $\rank_{\BZ}A_{g}(\BQ)=d$, $\Sha(A_{g})[\lambda^{\infty}]$ is finite and 
 $$
\bigg{|} \frac{L^{(d)}(1,A_{g})}{d!\cdot \Omega_{A_{g}}R(A_{g})}
\bigg{|}^{-1}_{\lambda}
=
\big{|}
\# \Sha(A_{g})[\lambda^{\infty}] \cdot \prod_{\ell | N} c_{\ell}(A_{g})
\big{|}^{-1}_{\lambda}.
$$ 
 \end{prop} 
\begin{proof} 
Note that $$L(1,g')\neq 0.$$ The $\lambda$-part of the BSD formula for $A_{g'}$ over 
$\BQ$ thus follows from Kato's main conjecture for $g'$ (cf.~\cite[\S14.20]{K},~see also the proof of \cite[Thm.~3.6.13]{SU} and \cite[\S7.2]{JSW}), which is assumed.
Hence in view of  the proof\footnote{The {\it{loc. cit.}} considers the case of elliptic curves, the argument also applies to $\GL_2$-type abelian varieties.} of \cite[Thm.~5.3.1]{CGLS} 
the $\lambda$-part of the BSD formula for $A_{g}$ over $\BQ$ is a consequence of
 that for $A_{g}$ over $L$. Henceforth we consider the latter. 

Recall that $L(s, A_{g})=\prod_{\sigma: F\hookrightarrow \BC}L(s,g^{\sigma}).$
Since $\ord_{s=1}L(s,g)=1$, 
the Gross--Zagier formula implies that the Heegner point $y_{L} \in A_{g}(L)$ is non-torsion and 
so $\ord_{s=1}L(s,A_{g})=d$. 
Moreover, 
$
[A_{g}(L): \cO \cdot y_{L}] < \infty
$
by the proof of Theorem \ref{MWrk}. 
Then as  in the proof of Theorem \ref{boundSha}, the non-vanishing hypothesis \eqref{nv} holds and 
so Proposition \ref{Eq} yields the cyclotomic Greenberg main conjecture: 
$$
\xi(X_{\rm{Gr}}^{\cyc,\ur}(g_{/L})) =
(\mathcal{L}_{p}^{\rm{Gr},\cyc}(g_{/L})). 
$$
We now consider its specialisation at the identity Hecke character over $L$. 

Proceeding as in the proof of Theorem \ref{boundSha}, we have 
$$
\big{|}\#\Sel_{\rm{Gr}}(g_{/L}) \cdot \prod_{\ell | N} c_{\ell}(A_{g})^{2} \cdot   \delta_{p}^{2}\big{|}_{\lambda}^{-1}
=
\bigg{|}  \frac{1+p-a_{p}}{p} \cdot \log_{A_{g}(L_{\ov{v}})} y_{L}               \bigg{|}_{\lambda}^{-2}.
$$
In combination with \eqref{pre-1} and \eqref{pre-2} it follows that 
$$
\big{[}A_{g}(L) \otimes_{\cO} \cO_{\lambda} : \cO_{\lambda} \cdot y_{L}\big{]}^{2}
=
\big{|}\# \Sha(A_{g/L})[\lambda^{\infty}] \cdot\prod_{\ell | N} c_{\ell}(A_{g})^{2} \big{|}_{\lambda}^{-1}.
$$
Then the Gross--Zagier formula and the hypothesis $({\rm van}_{L})$\footnote{The hypothesis leads to equality of the 
$\Gamma_{0}$ and $\Gamma_{1}$-periods associated to $g$ up to $p$-units (cf.~\cite[Lem. 9.4]{SZ}).} imply that 
$$
\bigg{|} \frac{L^{(d)}(1,A_{g/L})}{d!\cdot \Omega_{g}R(A_{g_{/L}})}
\bigg{|}^{-1}_{\lambda}
=
\big{|}
\# \Sha(A_{g/L})[\lambda^{\infty}] \cdot c_{g}^{-2}\cdot \prod_{\ell | N} c_{\ell}(A_{g})^{2}
\big{|}^{-1}_{\lambda}
$$
for $\Omega_{g}:=\prod_{\sigma: F\hookrightarrow \BC} \Omega_{g^{\sigma}}^{\rm{cong}}$ and $c_{g}$ the Manin constant. Therefore the $p$-indivisibility of the Manin constant \cite[Cor.~3.8]{ARS} concludes the proof. 
\end{proof}
\subsubsection{Main result}
\begin{thm}\label{p-BSD}
Let $g \in S_{2}(\Gamma_{0}(N))$ be an elliptic newform, $F$ the Hecke field with degree $d$ and $\cO$ the integer ring. 
Let $A_{g}$ be an associated $\GL_2$-type abelian variety over $\BQ$ with $\cO\hookrightarrow \End(A_{g})$. 
Let $p\nmid 2N$ be a prime, $\lambda p$ a prime of the 
Hecke field $F$ above $p$ and $T$ the $\lambda$-adic Tate module of $A_{g}$.
 Suppose that
 \begin{itemize}
 \item Either  $\lambda\nmid a_{p}(g)$ or $a_{p}(g)=0$, 
\item If $\lambda\nmid a_{p}(g)$, then the $\lambda$-adic Galois representation  $\rho: G_{\BQ}\ra \Aut_{\cO_{\lambda}}T$ satisfies 
$({\rm irr}_{\BQ})$ and \eqref{ram}. If $a_{p}(g)=0$, then $N$ is square-free.
\end{itemize}

If $\ord_{s=1}L(s,g)=1$, then 
the $\lambda$-part of the Birch and Swinnerton-Dyer conjecture for $A_{g}$ holds, that is, $\rank_{\BZ}A_{g}(\BQ)=d$, $\Sha(A_{g})[\lambda^{\infty}]$ is finite and 
 $$
\bigg{|} \frac{L^{(d)}(1,A_{g})}{d!\cdot \Omega_{A_{g}}R(A_{g})}
\bigg{|}^{-1}_{\lambda}
=
\big{|}
\# \Sha(A_{g})[\lambda^{\infty}] \cdot \prod_{\ell | N} c_{\ell}(A_{g})
\big{|}^{-1}_{\lambda}.
$$ 
In the CM case the same holds for any ordinary or a supersingular prime $p\nmid2N$.
\end{thm}
\begin{proof} 
We consider the ordinary case, and an analogous argument applies for the supersingular case. 

Let $L$ be an imaginary quadratic field satisfying \eqref{ord}, \eqref{coprime} and \eqref{Heeg} so that 
$$\ord_{s=1}L(s,g_{/L})=1.$$ The existence is a special case of the main result of \cite{FH}. 
Observe that such an $L$ satisfies (irr$_{L}$). Indeed, $\ov{\rho}|_{G_{L}}$ contains a unipotent element of order a power $p$ by \eqref{ram} and \eqref{Heeg}, and one may then proceed just as in the proof of \cite[Lem.~2.8.1]{Sk'}.

By Theorem \ref{KaMC_r} Kato's main conjecture holds for $g$ and the quadratic twist $g'$. Hence, the assertion is a consequence of Proposition \ref{p-BSD-prop}.
\end{proof}

\begin{remark}
\noindent
\begin{itemize} 
\item[(i)] The ordinary case generalises the result of Jetchev--Skinner--Wan \cite{JSW} for semistable curves. The supersingular case coincides with \cite{JSW} and Corollary \ref{BSD_f_Ko}, albeit a different approach. 
\item[(ii)] For CM elliptic curves over $\BQ$ with analytic rank one, the $p$-part of the BSD formula for $p\nmid 2N$ is due to Rubin \cite{Ru1} and Kobayashi \cite{Ko1} for ordinary and supersingular primes $p$ respectively. 
Their approach relies on the $p$-adic Gross--Zagier formula, and the non-vanishing of $p$-adic height of non-torsion points, which is an open problem in the ordinary non-CM case.  
The above approach instead relies on the $p$-adic Waldspurger formula, uniformly treating the ordinary and supersingular cases. An independent approach in the ordinary case is due to Castella \cite{Ca-f}.
\end{itemize}
\end{remark}

\section{$p$-converse to the Gross--Zagier and Kolyvagin theorem}\label{s:pcv}
This section presents some $p$-converse theorems. 
We begin with a preliminary form of the strategy under the hypothesis \eqref{inj}, and then refine it. 
\subsection{$p$-converse and Kato's main conjecture}
\subsubsection{A cyclotomic criterion}
\begin{prop}\label{p-converse-prop}
Let $g \in S_{2}(\Gamma_{0}(N))$ be an elliptic newform. Let $p\nmid 2N$ be a prime, $\lambda$ a prime of the Hecke field $F$ of $g$ above $p$ and $V$ the associated $p$-adic Galois representation. Suppose that either 
$\lambda\nmid a_{p}(g)$ or $a_p(g)=0$. 
Let $L$ be an imaginary quadratic field satisfying \eqref{ord}, \eqref{coprime} and \eqref{Heeg}. 
Suppose also the divisibility 
$$
\xi\big{(}H^{1}(\BZ[\frac{1}{p}], T \otimes_{\BZ_{p}} \Lambda)/ \Lambda_{\cO_{\lambda}}\cdot{{\bf{z}}_{\gamma}(h)}\big{)}
\big{|} 
\xi(X_{\st}(h))
$$
in $\Lambda_{L,\cO_{\lambda}}\otimes_{\BZ_{p}} \BQ_{p}$ 
for $h \in\{g,g\otimes \chi_{L}\}$ (cf.~Conjecture \ref{Kato}(b)).
  If
 \begin{equation}\label{inj}\tag{inj}
 \text{The localisation $H^{1}_{\rm{f}}(L,V) \ra \prod_{w|p}H^{1}_{\rm{f}}(L_{w},V)$ is an injection,}
 \end{equation}
then 
 $$
 \dim_{F_{\lambda}} H^{1}_{\rm{f}}(L,V)=1 \implies \ord_{s=1} L(s,g_{/L})=1. 
 $$

\end{prop} 
\begin{proof} We present the ordinary case. 

The hypotheses \eqref{inj} and  $\dim_{F_{\lambda}} H^{1}_{\rm{f}}(L,V)=1$ imply that 
$$
\dim_{F_{\lambda}} \Im  \big{\{}H^{1}_{\rm{f}}(L,V) \ra \prod_{w|p}H^{1}_{\rm{f}}(L_{w},V) \big{\}}=1.
$$
Hence  the Selmer groups 
$
\Sel_{\rm{Gr}}(g_{/L})$ and $\Sel_{\st,\ord}(g_{/L})$
are finite by \cite[Lem.~2.3.2]{Sk'}. 

Note that the lower bound for Selmer group predicted by the cyclotomic main Conjecture \ref{cycBC} 
holds by our hypotheses and Lemma \ref{Eq'}. Hence the exact sequence \eqref{exv} in combination with the finiteness of 
$\Sel_{\st,\ord}(g_{/L})$ implies that 
$$
\phi_{{\bf 1}_{L}}(\CZ(g_{/L})^{\cyc}) \neq 0. 
$$
In turn the hypothesis \eqref{nv} holds by \eqref{inj}.

We now resort to Proposition \ref{Eq}, which yields the cyclotomic Greenberg main Conjecture \ref{Greenberg}. 
In view of the finiteness of $\Sel_{\rm{Gr}}(g_{/L})$ the latter implies that 
$$
\phi_{{\bf 1}_{L}}(\CL_{p}^{\rm{Gr,cyc}}(g_{/L})) \neq 0.
$$
Hence  the Heegner point $y_{L}\in A_{g}(L)$ is non-torsion by the $p$-adic Waldspurger formula
(cf.~Proposition~\ref{GRL=BDPL-II}). 
 The Gross--Zagier formula \cite{GZ} concludes the proof. 
\end{proof}
\begin{remark}
For non-ordinary prime $p\nmid N$, the $p$-converse is also a consequence of lower bound for the strict Selmer group as predicted by Kato's main conjecture (cf.~\cite{BKO2}). 
\end{remark}

\subsubsection{$p$-converse I}

\begin{thm}\label{p-converse}
Let $g \in S_{2}(\Gamma_{0}(N))$ be an elliptic newform, $F$ the Hecke field and $\cO$ the integer ring. 
Let $A_{g}$ be an associated $\GL_2$-type abelian variety over $\BQ$ with $\cO\hookrightarrow \End(A_{g})$. 
Let $p\nmid 2N$ be a prime, $\lambda$ a prime of the 
Hecke field $F$ above $p$ and $T$ the $\lambda$-adic Tate module of $A_{g}$.
 Suppose the following:
 \begin{itemize}
 \item Either $a_{p}(g)=0$ and $N$ is square-free or $\lambda \nmid a_{p}(g)$, 
\item $\rho: G_{\BQ}\ra \Aut_{\cO_{\lambda}}T$ satisfies $({\rm irr}_{\BQ})$ if $g$ is non-CM, and \eqref{ram} holds if $p=3$.
\end{itemize}

Then  
 $$
 \corank_{\cO_{\lambda}} \Sel_{\lambda^{\infty}}(A_{g})=1, \ \#\Sha(A_{g})[\lambda^{\infty}] < \infty
  \implies \ord_{s=1} L(s,A_{g})=[F:\BQ].
 $$
\end{thm}
\begin{proof} 
Since $\corank_{\cO_{\lambda}} \Sel_{\lambda^{\infty}}(A_{g})=1$, note that $\varepsilon(A_{g})=-1$ by the proof of parity conjecture \cite{N1}.

Let $L$ be an imaginary quadratic field satisfying \eqref{ord}, \eqref{coprime} and \eqref{Heeg} so that 
$$\ord_{s=1}L(s,g_{/L})=\ord_{s=1}L(s,g).$$ 
The existence of $L$ is a special case of the main result of \cite{FH}. 

As $L(1,g') \neq 0$, the Selmer group $\Sel_{\lambda^{\infty}}(A_{g'})$ is finite (cf.~\cite[Thm.~14.2]{K}). 
In view of the splitting $$\Sel_{\lambda^{\infty}}(A_{g/L})\simeq \Sel_{\lambda^{\infty}}(A_{g})\oplus \Sel_{\lambda^{\infty}}(A_{g'})$$ it follows that $\corank_{\cO_{\lambda}}\Sel_{\lambda^{\infty}}(A_{g/L})=1$ and $\Sha(A_{g/L})[\lambda^{\infty}]$ is finite. In particular, the hypothesis \eqref{inj} holds. 
By Theorem \ref{KaMC_r}(c) and Theorem \ref{KoMC_r} 
Kato's main conjecture holds for $g$ and the quadratic twist $g'$ in $\Lambda_{L,\cO_{\lambda}}\otimes\BQ_{p}$.
Hence, Proposition \ref{p-converse-prop} concludes the proof.
\end{proof}
\begin{remark}
\noindent
\begin{itemize}
\item[(i)] For $N$ square-free and $p$ ordinary, the above $p$-converse goes back to \cite{Sk'}. The approach in {\it loc.cit.} relies on the anticyclotomic Greenberg main conjecture.
\item[(ii)] For CM curves over $\BQ$, the $p$-converse as above is due to Rubin for ordinary primes \cite{Ru1}. The non-ordinary case is more recent \cite{BKO2}. 
\end{itemize}
 \end{remark}

\begin{cor}\label{pcv-w}
Let $g \in S_{2}(\Gamma_{0}(N))$ be an elliptic newform and $F$ the Hecke field. Let $A_{g}$ be an associated $\GL_2$-type abelian variety. 
For $r\in\{0,1\}$, 
 $$
 \rank_{\BZ}A_{g}(\BQ)=r[F:\BQ], \ \ \#\Sha(A_{g})< \infty
  \implies \ord_{s=1} L(s,A_{g})=r[F:\BQ].
 $$
\end{cor}
\begin{proof}
In light of \cite{Ru1,BuTi} it suffices to consider the non-CM case. We may suppose that 
$\cO\hookrightarrow \End(A_{g})$.

Let $p$ be a prime. Let $\iota_{\infty}:\ov{\BQ}\hookrightarrow \BC$ and $\iota_{p}:\ov{\BQ}\hookrightarrow \BC_p$ be embeddings and $\lambda$ the prime of $F$ above $p$  arising from $\iota_p$. Let $\rho: G_{\BQ} \ra \GL_{2}(F_{\lambda})$ be the attached Galois representation. 
From now, pick\footnote{For all but finitely many $\lambda$, the residual representation $\ov{\rho}$ satisfies (irr$_{\BQ}$) (cf.~\cite{Mo},~\cite[Thm.~2.1]{Ri}).}
an ordinary prime $p\nmid 2N$ so that $\rho$ satisfies (irr$_{\BQ}$). 
By the hypothesis,  note that 
$\corank_{\cO_{\lambda}}\Sel_{\lambda^{\infty}}(A_{g})=r$ and $\Sha(A_{g})[\lambda^{\infty}]$ is finite. Hence the $r=1$ case follows from Theorem \ref{p-converse}, and the $r=0$ case from Theorem \ref{KaMC_r}(c).
\end{proof}
\begin{remark}
The semistable case goes back to \cite{Sk'}, and the CM case to \cite{Ru1}. More recently, the above result is also independently obtained by Kim \cite{Ki1}.
\end{remark}

\subsection{$p$-converse and Heegner main conjecture}\label{HMC} 
\subsubsection{Kato's main conjecture and Heegner main conjecture}
\begin{prop}\label{HMC-prop}
Let $g \in S_{2}(\Gamma_{0}(N))$ be a non-CM elliptic newform, $F$ the Hecke field and $\cO$ the integer ring. 
Let $A_{g}$ be an associated $\GL_2$-type abelian variety over $\BQ$ with $\cO\hookrightarrow \End(A_{g})$. 
Let $p\nmid 2N$ be a prime, $\lambda$ a prime of the 
Hecke field $F$ above $p$ and $T$ the $\lambda$-adic Tate module of $A_{g}$.
 Suppose that
 \begin{itemize} 
 \item $\lambda\nmid a_{p}(g)$, 
\item The $\lambda$-adic Galois representation  $\rho: G_{\BQ}\ra \Aut_{\cO_{\lambda}}T$ satisfies $({\rm van}_{\BQ})$.
\end{itemize}
Let $L$ be an imaginary quadratic field satisfying \eqref{ord}, \eqref{coprime}, \eqref{Heeg} and $({\rm van}_{L})$ such that 
\begin{equation}\label{HMC_ub}
\xi_{\Lambda_{L,\cO_{\lambda}}^{\ac}}(X(g)_{\tor})\big{|}
\xi_{\Lambda_{L,\cO_{\lambda}}^{\ac}}\bigg{(}
\frac{H^{1}_{\ord}(L,T_{g}\otimes_{\BZ_{p}}\Lambda_{L,\cO_{\lambda}}^{\ac})}{\Lambda_{L,\cO_{\lambda}}^{\ac}\cdot \kappa_{g}}\bigg{)}\cdot \xi_{\Lambda_{L,\cO_{\lambda}}^{\ac}}\bigg{(}\frac
{H^{1}_{\ord}(L,T_{g}\otimes_{\BZ_{p}}\Lambda_{L,\cO_{\lambda}}^{\ac})}{\Lambda_{L,\cO_{\lambda}}^{\ac}\cdot\kappa_{g}^{\iota}}\bigg{)}.
\end{equation}

Then the $\lambda$-adic Kato's main Conjecture \ref{Kato} for $g$ and the quadratic twist $g'=g\otimes \chi_L$ imply the Heegner main Conjecture \ref{HMC} for $g$ over $L$, that is
 $$\xi_{\Lambda_{L,\cO_{\lambda}}^{\ac}}\bigg{(}
\frac{H^{1}_{\ord}(L,T_{g}\otimes_{\BZ_{p}}\Lambda_{L,\cO_{\lambda}}^{\ac})}{\Lambda_{L,\cO_{\lambda}}^{\ac}\cdot \kappa_{g}}\bigg{)}\cdot \xi_{\Lambda_{L,\cO_{\lambda}}^{\ac}}\bigg{(}\frac
{H^{1}_{\ord}(L,T_{g}\otimes_{\BZ_{p}}\Lambda_{L,\cO_{\lambda}}^{\ac}}{\Lambda_{L,\cO_{\lambda}}^{\ac}\cdot\kappa_{g}^{\iota}}\bigg{)} = 
\xi_{\Lambda_{L,\cO_{\lambda}}^{\ac}}(X(g)_{\tor}).$$ 
 
\end{prop}
\begin{proof} 
By Proposition \ref{Eq-He} and \eqref{HMC_ub} we have 
\begin{equation}\label{BDP_ub}
\xi(X_{\rm{Gr}}^{\ac}(g_{/L}))\big{|}
(\mathcal{L}_{p}^{\rm{Gr},\ac}(g_{/L})).
\end{equation}
It suffices to show that this divisibility is an equality. 

For a character $\alpha:\Gamma_{L}^{\rm ac} \ra R^\times$ with $R$ the integer ring of a $p$-adic local field, we have 
\begin{equation}\label{BDPt_ub}
\xi(X_{\rm{Gr}}^{\ac}(g_{/L}\otimes\alpha))\big{|}
(\mathcal{L}_{p}^{\rm{Gr},\ac}(g_{/L}\otimes\alpha)) 
\end{equation}
by \eqref{BDP_ub}. Here the Selmer groups arises from the $G_L$-representation $T\otimes_{\cO_{\lambda}}R(\alpha)$, and the $p$-adic $L$-function $\mathcal{L}_{p}^{\rm{Gr},\ac}(g_{/L}\otimes\alpha)$ is just the image of $\mathcal{L}_{p}^{\rm{Gr},\ac}(g_{/L})$ under the map $${\rm Tw}_{\alpha}:\Lambda_{L,\cO_{\lambda}}^\ac \ra \Lambda_{L,\cO_{\lambda}}^\ac$$ induced by $\gamma \mapsto \alpha(\gamma)\gamma$ for $\gamma \in \Gamma_L$. 
Note that the anticyclotomic Greenberg main conjecture for $g$ is equivalent to that for the twist $g\otimes \alpha$.
Henceforth we consider the latter for a well-chosen $$\alpha \equiv {\bf 1}_{L} \mod \varpi^{m},$$ where $m$ is a sufficiently large integer and $\varpi$ a uniformiser of $R$.

Let $\alpha:\Gamma_{L}^{\rm ac} \ra R^\times$ be a character such that
\begin{equation}\label{BDP_nv}
\phi_{\alpha}(\CL_{p}^{\rm{Gr},\ac}(g_{/L}))\neq 0, 
\end{equation}
and
\begin{equation}\label{cong}
\alpha \equiv {\bf 1}_{L} \mod \varpi^{m} \text{ for $m>\mu(X^{\cyc}(g_{/L}))$}.
\end{equation}
The existence of $\alpha$ follows from the non-vanishing of $\CL_{p}^{\rm{Gr},\ac}(g_{/L})$ and $X^{\cyc}(g_{/L})$ being a torsion $\Lambda_{L,\cO_{\lambda}}^{\cyc}$-module. 
In view of \eqref{BDP_nv} and \eqref{BDPt_ub} the desired main conjecture for $g\otimes \alpha$ is equivalent to the formula
\begin{equation}\label{BK_f}
\phi_{{\bf 1}_{L}}(\xi(X_{\rm{Gr}}^{\ac}(g_{/L}\otimes\alpha)))=
\phi_{{\bf 1}_{L}}(\mathcal{L}_{p}^{\rm{Gr},\ac}(g_{/L}\otimes\alpha)).
\end{equation}
In the following we approach it using cyclotomic Iwasawa theory. 

To begin, the $\lambda$-adic Kato's main Conjecture \ref{Kato} for $g$ and the quadratic twist $g'=g\otimes \chi_L$ imply the cyclotomic main conjecture over $L$: 
\begin{equation}\label{cyc_eq}
\xi(X^{\cyc}(g_{/L}))=
(\mathcal{L}_{p}^{\cyc}(g_{/L}))
\end{equation}
(cf.~Lemma~\ref{Eq'}(i)).

Note that $$\mathcal{L}_{p}^{\cyc}(g_{/L}\otimes \alpha) \equiv \mathcal{L}_{p}^{\cyc}(g_{/L}) \mod{\varpi^{m}},$$
and so $\mathcal{L}_{p}^{\cyc}(g_{/L}\otimes \alpha)$ is non-zero. 
Hence, in the non-CM case the Euler system of Beilinson--Flach elements for $T\otimes \alpha$ yields the divisibility
\begin{equation}\label{BF_ub}
\xi(X^{\cyc}(g_{/L}\otimes \alpha))|
(\mathcal{L}_{p}^{\cyc}(g_{/L}\otimes \alpha))
\end{equation}
in $\Lambda_{L,\cO_{\lambda}}^{\cyc}\otimes \BQ_{p}$ (cf.~\cite{LLZa,KLZ}, see also \cite[\S3.3]{CGS}). 
In view of \eqref{cong} we have
\begin{equation}
\text{
$\mu(X^{\cyc}(g_{/L}\otimes \alpha))=\mu(X^{\cyc}(g_{/L}))$ and 
$\mu(\CL_{p}^{\cyc}(g_{/L}\otimes \alpha))=\mu(\CL_{p}^{\cyc}(g_{/L}))$.}
\end{equation}
Therefore 
 the divisibility  \eqref{BF_ub} holds integrally in $\Lambda_{L,\cO_{\lambda}}^{\cyc}$ by \eqref{cyc_eq}. The integral divisibility in combination with the 
 equality \eqref{cyc_eq}  leads to the cyclotomic main conjecture
 \begin{equation}\label{cyct_eq}
\xi(X^{\cyc}(g_{/L}\otimes \alpha))=
(\mathcal{L}_{p}^{\cyc}(g_{/L}\otimes \alpha))
\end{equation} 
(see also \cite[\S6.1]{CGS}).

The hypotheses of Proposition \ref{Eq} hold by \eqref{BDP_nv}. 
Therefore the two-variable zeta element and \eqref{cyc_eq} yield the cyclotomic Greenberg main conjecture: 
$$
\xi(X_{\rm{Gr}}^{\cyc}(g_{/L}\otimes\alpha))=
(\mathcal{L}_{p}^{\rm{Gr},\cyc}(g_{/L}\otimes\alpha)).
$$
Its descent leads to the sough-after Bloch--Kato formula \eqref{BK_f} (cf.~\cite[\S2.3]{CGS}). 
\end{proof}
\begin{remark} 
The above strategy is reversal of the strategy of \cite{CGS} which established an Eisenstein case of Kato's main conjecture via Heegner main conjecture and the two-variable zeta element. 
\end{remark}

\subsubsection{Heegner main conjecture}
\begin{thm}\label{HMC_r}
Let $g \in S_{2}(\Gamma_{0}(N))$ be an elliptic newform, $F$ the Hecke field and $\cO$ the integer ring. 
Let $A_{g}$ be an associated $\GL_2$-type abelian variety over $\BQ$ with $\cO\hookrightarrow \End(A_{g})$. 
Let $p\nmid 2N$ be a prime, $\lambda$ a prime of the 
Hecke field $F$ above $p$ and $T$ the $\lambda$-adic Tate module of $A_{g}$.
 Suppose the following.
 \begin{itemize}
 \item $\lambda\nmid a_{p}(g)$, 
\item The $\lambda$-adic Galois representation  $\rho: G_{\BQ}\ra \Aut_{\cO_{\lambda}}T$ satisfies 
$({\rm irr}_{\BQ})$, \eqref{ram} and \eqref{sur}. 
\end{itemize}
Let $L$ be an imaginary quadratic field satisfying \eqref{ord}, \eqref{coprime}, \eqref{Heeg} and $({\rm irr}_{L})$.
Then Heeger main Conjecture \ref{HMC}
for $g$ over $L$ is true.
\end{thm}
\begin{proof} 
The divisibility \eqref{HMC_ub} holds by Theorem \ref{HMC-ub}. Hence 
the assertion is consequence of Proposition \ref{HMC-prop} and Theorem \ref{KaMC_r}(c).
\end{proof}

\subsubsection{$p$-converse II}

\begin{prop}\label{pcv-prop}
Let $g \in S_{2}(\Gamma_{0}(N))$ be a non-CM elliptic newform, $F$ the Hecke field and $\cO$ the integer ring. 
Let $A_{g}$ be an associated $\GL_2$-type abelian variety over $\BQ$ with $\cO\hookrightarrow \End(A_{g})$. 
Let $p\nmid 2N$ be a prime, $\lambda$ a prime of the 
Hecke field $F$ above $p$ and $T$ the $\lambda$-adic Tate module of $A_{g}$.
 Suppose that
 \begin{itemize}
 \item $\lambda\nmid a_{p}(g)$, 
\item The $\lambda$-adic Galois representation  $\rho: G_{\BQ}\ra \Aut_{\cO_{\lambda}}T$ satisfies $({\rm van}_{\BQ})$.
\end{itemize}
Let $L$ be an imaginary quadratic field satisfying \eqref{ord}, \eqref{coprime}, \eqref{Heeg} and $({\rm van}_{L})$ such that \eqref{HMC_ub} holds.

Then the $\lambda$-adic Kato's main Conjecture \ref{Kato} for $g$ and the quadratic twist $g'=g\otimes \chi_L$ imply the p-converse over $L$: 
$$
\corank_{\cO_{\lambda}}\Sel_{\lambda^{\infty}}(A_{g/L})=1 \implies \ord_{s=1}L(s,A_{g/L})=[F:\BQ].
$$
\end{prop}
\begin{proof}
The $p$-converse is a consequence of the Heegner main conjecture (cf.~\cite{BuTi,W0}), and so the assertion of  Proposition \ref{HMC-prop}.
\end{proof}
Proceeding as in the proof of Theorem \ref{p-converse}, we deduce the following. 
\begin{thm}\label{p-converse-II}
Let $g \in S_{2}(\Gamma_{0}(N))$ be an elliptic newform, $F$ the Hecke field and $\cO$ the integer ring. 
Let $A_{g}$ be an associated $\GL_2$-type abelian variety over $\BQ$ with $\cO\hookrightarrow \End(A_{g})$. 
Let $p\nmid 2N$ be a prime, $\lambda$ a prime of the 
Hecke field $F$ above $p$ and $T$ the $\lambda$-adic Tate module of $A_{g}$.
 Suppose the following:
 \begin{itemize}
 \item $\lambda \nmid a_{p}(g)$, 
\item $\rho: G_{\BQ}\ra \Aut_{\cO_{\lambda}}T$ satisfies \eqref{sur}, \eqref{ram} and \eqref{sur}.
\end{itemize}

Then  
 $$
 \corank_{\cO_{\lambda}} \Sel_{\lambda^{\infty}}(A_{g})=1 
  \implies \ord_{s=1} L(s,A_{g})=[F:\BQ].
 $$
\end{thm}

\begin{remark}
A related $p$-converse for non-CM curves is due to Zhang \cite{Zh} and Skinner--Zhang \cite{SZ}, assuming $p$-indivisibility of Tamagawa numbers. The approach relies on the principle of level raising and rank lowering as well as \cite{SU}. The use of two-variable zeta element bypasses the auxiliary level raising. 
\end{remark}

\end{document}